\documentclass{article}

\usepackage{graphicx} 
\usepackage[margin=1in]{geometry}
\usepackage[section]{placeins}
\usepackage{amsfonts, amssymb, amsmath, amsthm,esint,multicol}
\usepackage[colorlinks=true, pdfstartview=FitV, linkcolor=blue,
            citecolor=blue, urlcolor=blue]{hyperref}
\usepackage[usenames,dvipsnames,table]{xcolor}
\usepackage{accents}
\usepackage{comment}
\usepackage[capitalize,nameinlink,noabbrev]{cleveref}
\usepackage{enumerate}
\usepackage{enumitem}

\usepackage{bm}
\usepackage{authblk} 

\usepackage{bbm}
\usepackage{mathtools} 
\usepackage{tabularx}

\usepackage{listings}
\usepackage[textsize=tiny]{todonotes}
\usepackage{tikz}
\usetikzlibrary{shapes.misc}
\usepackage{etoolbox}
\usepackage{appendix}
\usepackage{subcaption}
\usepackage{wrapfig}
\usepackage{xr}
\usepackage{booktabs}
\usepackage[section]{algorithm}      
\usepackage[noend]{algpseudocode}    

\numberwithin{equation}{section}

\DeclareFontFamily{U}{mathx}{}
\DeclareFontShape{U}{mathx}{m}{n}{<-> mathx10}{}
\DeclareSymbolFont{mathx}{U}{mathx}{m}{n}
\DeclareMathAccent{\widehat}{0}{mathx}{"70}
\DeclareMathAccent{\widecheck}{0}{mathx}{"71}

\definecolor{Red}{rgb}{0.7,0,0.1}
\definecolor{Green}{rgb}{0,0.7,0}
\definecolor{lightblue}{RGB}{185,220,245}
\definecolor{softgreen}{RGB}{190,230,190}
\definecolor{softorange}{RGB}{245,205,160}
\definecolor{darkgreen}{RGB}{0,110,70}
\definecolor{myblue}{RGB}{25,70,180}
\definecolor{myred}{RGB}{180,40,40}
\definecolor{mgGreen}{RGB}{144, 200, 165}      
\definecolor{mgPurple}{RGB}{180, 150, 210}     
\definecolor{mgGold}{RGB}{240, 210, 140}        
\definecolor{matchBlue}{RGB}{26,91,200}     
\definecolor{mismatchRed}{RGB}{180,40,40}   

\newtheorem{Theorem}{Theorem}[section]
\newtheorem{Proposition}[Theorem]{Proposition}
\newtheorem{Lemma}[Theorem]{Lemma}

\newtheorem{Definition}[Theorem]{Definition}
\newtheorem{Remark}[Theorem]{Remark}

\newtheorem{Assumption}[Theorem]{Assumption}

\crefname{Theorem}{Theorem}{Theorems}
\crefname{figure}{Figure}{Figures}

\newcommand{\bx}{\mathbf{x}}
\newcommand{\by}{\mathbf{y}}
\newcommand{\bz}{\mathbf{z}}

\newcommand{\bA}{\mathbf{A}}

\newcommand{\bw}{\mathbf{w}}
\newcommand{\bX}{\mathbf{X}}
\newcommand{\bZ}{\mathbf{Z}}

\newcommand{\baX}{\bar{X}}

\newcommand{\baQ}{\bar{Q}}
\newcommand{\balpha}{\bar{\alpha}}
\newcommand{\baf}{\bar{f}}
\newcommand{\bg}{\mathbf{g}}
\newcommand{\btheta}{\bm{\theta}}
\newcommand{\bI}{\mathbf{I}}

\newcommand{\td}{\tilde{d}}
\newcommand{\tx}{\tilde{x}}
\newcommand{\ty}{\tilde{y}}
\newcommand{\tz}{\tilde{z}}

\newcommand{\tX}{\tilde{X}}
\newcommand{\tY}{\tilde{Y}}
\newcommand{\tu}{\tilde{u}}
\newcommand{\tU}{\tilde{U}}

\newcommand{\tZ}{\tilde{Z}}
\newcommand{\txi}{\tilde{\xi}}
\newcommand{\tXi}{\tilde{\Xi}}
\newcommand{\trho}{\tilde{\rho}}

\newcommand{\tbA}{\tilde{\mathbf{A}}}

\newcommand{\fC}{\mathfrak{C}}

\newcommand{\cB}{\mathcal{B}}
\newcommand{\cC}{\mathcal{C}}
\newcommand{\cE}{\mathcal{E}}
\newcommand{\cF}{\mathcal{F}}

\newcommand{\cI}{\mathcal{I}}
\newcommand{\cM}{\mathcal{M}}
\newcommand{\cN}{\mathcal{N}}
\newcommand{\cP}{\mathcal{P}}

\newcommand{\cU}{\mathcal{U}}

\newcommand{\rr}{\sqrt{1-\rho^2}}

\newcommand{\R}{\mathbb{R}}
\newcommand{\E}{\mathbb{E}\,}
\newcommand{\N}{\mathbb{N}}
\newcommand{\bbP}{\mathbb{P}}

\newcommand{\tr}{\operatorname{tr}}

\newcommand{\Var}{\operatorname{Var}}

\newcommand{\eps}{\varepsilon}

\newcommand{\qsp}{\mathcal{X}} 
\newcommand{\mk}{P} 
\newcommand{\qk}{Q} 
\newcommand{\ar}{\alpha} 
\newcommand{\har}{\hat{\alpha}} 
\newcommand{\Pot}{\Phi} 
\newcommand{\Fbar}{\overline{F}}

\newcommand{\Wass}{W} 
\newcommand{\tbeta}{\tilde{\beta}}

\newcommand{\tJ}{\tilde{J}}

\newcommand{\Lip}{\operatorname{Lip}}

\usetikzlibrary{arrows.meta,decorations.pathreplacing}

\title{On the mixing properties of some preconditioned multiproposal Markov Chain Monte Carlo algorithms}

\author[1]{Giulia Carigi}
\author[1]{Nathan E. Glatt-Holtz}
\author[2]{Cecilia F. Mondaini}
\author[1,3]{Guillermina Senn}

\affil[1]{Department of Statistics, Indiana University}
\affil[2]{Department of Mathematics, Drexel University}
\affil[3]{Department of Mathematical Sciences, Norwegian University of Science and Technology}

\date{\today}

\begin{document}

\maketitle

\begin{abstract}
  We study two recently discovered ``dimension-free'' Monte Carlo sampling algorithms, the multiproposal and multiple-try preconditioned Crank–Nicolson methods (mpCN and MTpCN). These methods were designed to address certain non-parametric (i.e.\ infinite-dimensional) sampling problems, defined relative to a Gaussian reference measure, by combining proposal and acceptance mechanisms that take non-trivial advantage of parallel computing architectures.

We provide the first rigorous analysis of both algorithms, establishing exponential convergence to the target measure through the weak Harris framework, both for a finite number of proposals and in the infinite-proposal limit. The resulting mixing rates are independent of the dimension and uniform in the number of proposals, and apply to targets with bounded, Lipschitz log-likelihoods, without requiring convexity. At the center of the analysis are two new coupling constructions, together with analytical tools of independent interest, yielding Wasserstein contraction estimates, $L^2$ spectral gaps, and associated statistical guarantees (laws of large numbers, central limit theorems, and non-asymptotic concentration bounds) for the corresponding Monte Carlo estimators.

These theoretical results are complemented by a numerical study on benchmark problems with complex posterior geometries and high-dimensional structure, comparing mpCN and MTpCN against standard pCN and independent parallel-chain implementations. The experiments indicate that the multiproposal methods can offer a shorter warm-up phase and greater robustness to the choice of tuning parameters as the number of proposals grows.
\end{abstract}

\medskip
\noindent\textbf{Keywords}: Multiproposal and Preconditioned Markov chain
Monte Carlo (MCMC), weak Harris theorem, preconditioned Crank-Nicolson (pCN).\\
\textbf{MSC 2020 Classifications}: 65C40, 60J05, 60B10, 65Y05, 60J22

\setcounter{tocdepth}{1}
\tableofcontents

\section{Introduction}
Markov Chain Monte Carlo (MCMC) algorithms are an essential computational tool for sampling from complex probability distributions arising in Bayesian statistics, applied mathematics and across the physical and social sciences. While MCMC methods trace their origins to the dawn of computing in the 1940s, this approach remains an active area of fundamental research, driven by a rapidly developing technological landscape and the basic scientific need to accurately resolve ever more complex statistical models.
 
In recent years two developments in MCMC research have been especially notable. The nontrivial intersection of these two developments leads to exciting new questions which we begin to address here. The first development involves a class of infinite dimensional (i.e. non-parametric) probability distributions $\mu$ which are defined relative to a Gaussian reference measure $\mu_0$; in the Bayesian context, $\mu$ is a posterior and $\mu_0$ a Gaussian prior. This class of measures includes a broadly applicable category of statistical models for uncertainty in physics-informed data which can be formulated within the framework of the Bayesian approach to PDE inverse problems; cf. \cite{stuart2010inverse, dashti2017bayesian, borggaard2020bayesian}. For such measures a careful combination of preconditioning and suitable numerical discretizations of an infinite dimensional dynamics related to $\mu$ led to the discovery of the now widely used preconditioned Crank–Nicolson algorithm \cite{NealpCNbeforepCN, beskos2008mcmc, cotter2013mcmc} as well as `gradient informed' variants related to Hamiltonian and (Overdamped) Langevin dynamics defined for such $\mu$ (see \eqref{eq:intro:target} below). The second development is the so-called multiproposal (or sometimes parallel) paradigm described in \cite{liu2000multiple,tjelmeland2004using,calderhead2014general,glatt2024parallel,pozza_zanella_2025} in which a cloud of proposals is drawn at each step, thereby making a delicate, nontrivial use of parallelism.

Our contributions herein are centered on two promising, recently discovered methods at the intersection of these two research directions, namely the multiproposal and multiple-try preconditioned Crank-Nicolson algorithms (mpCN and MTpCN, respectively); see \cite{glatt2024parallel, glatt2025+} as well as \cref{alg:mpcn,alg:mtpcn} below. These new methods are particularly well adapted to the rich class of infinite-dimensional target measures that such preconditioned methods were designed for while taking non-trivial advantage of modern parallel computing architectures. We establish rigorous mixing results for the resulting Markov chains, with rates that are independent of the dimension and uniform in the number of proposals, properties which are key to ensuring vanishing Monte Carlo error, asymptotic confidence intervals, and non-asymptotic concentration guarantees. Our analysis applies to targets with bounded and Lipschitz log-likelihoods---and, in the infinite-proposal limit, merely Lipschitz ones---without requiring convexity or additional smoothness. A central analytical contribution is the development of two novel coupling constructions, one for a finite number of proposals and one for the infinite-proposal limit. Overall we expect the analytical approach developed here to serve as a general blueprint for studying mixing in other sampling algorithms and, more broadly, in general state space Markov chains.

Our theoretical analysis is complemented by numerical experiments which underline significant advantages of mpCN and MTpCN over the standard pCN method. Specifically, these advantages are manifested as a shorter warm-up (or burn-in) phase and an increased robustness to algorithmic parameter tuning. Both features can be decisive for the kind of demanding, large-scale, PDE-informed problems (e.g.~\cite{borggaard2020bayesian, borggaard2023statistical}) that mpCN and MTpCN were designed for, where burn-in and tuning can represent a substantial computational bottleneck in practice.

Regarding the scope of our multiproposal algorithms, it is worth noting the very interesting recent contribution \cite{pozza_zanella_2025}. This work provides a general analysis of multiproposal methods through spectral gap comparisons at stationarity, within a framework that can be shown to cover the algorithms studied here. While \cite{pozza_zanella_2025} offers a rather cautious perspective on the advantages of multiproposal methods, our results complement this perspective by indicating that such methods can nevertheless exhibit substantial advantages, particularly during the burn-in phase. More broadly, our findings highlight several open questions in the comparison between embarrassingly parallel chains and multiproposal methods.

To complete this introduction and place our contributions in context, we next briefly recall the general framework and recent challenges in the theory of MCMC methods. We then provide an overview of our main results, a sketch of the principal elements of our proofs, and a summary of our numerical case studies.

\subsection{Background and Motivation}

Ultimately, MCMC algorithms are stochastic numerical methods to resolve high-dimensional integration problems with costly, analytically intractable integrands. In other words, given a target probability distribution $\mu$ on a parameter space $\qsp$, we want to compute
\begin{equation}\label{eq:intro:target_integral}
    I_\mu(f) = \int_{\qsp} f(x) \, \mu(dx),
\end{equation}
over a suitable class of observables $f:\qsp \to \R$. Successful MCMC algorithms construct a Markov chain $(X_n)_{n\in \N}$ in $\qsp$ that holds $\mu$ invariant so that the ergodic averages
\begin{equation}\label{eq:intro:ergodic_ave}
     \hat{f}_N := \frac{1}{N}\sum_{n = 1}^N f(X_n), \quad N\in \N,
\end{equation}
provide an accurate approximation of $I_\mu(f)$. Results ensuring geometric ergodicity (or mixing) for the constructed chain are fundamental for assessing how the accuracy of the approximation in \eqref{eq:intro:ergodic_ave} depends on the number of `samples' $N$.

An enormous body of work has focused on designing effective MCMC algorithms. The field traces back to the development of the Metropolis--Hastings (MH) approach \cite{metropolis1953equation,hastings1970monte}, in which draws from a proposal kernel $Q$, followed by an accept-reject step with acceptance probabilities $\alpha$, yield a $\mu$-reversible chain. In the ensuing decades many variants of the MH paradigm have been formulated by tailoring the proposal kernel and the acceptance rule, ranging from the original Random Walk Metropolis approach to more modern methods such as the Metropolis-Adjusted Langevin Algorithm (MALA) and Hamiltonian Monte Carlo (HMC), which take advantage of dynamical invariants and the local structure of the target measure; see \cite{liu2008, RobertCasellabook} for general background.

Notwithstanding the tremendous successes of the aforementioned Metropolis-type MCMC methods, high-dimensional settings, often reflecting non-parametric problems such as those arising in Bayesian inverse problems \cite{stuart2010inverse,borggaard2020bayesian}, remain an important frontier. Indeed, classical algorithms, including all of the methods mentioned above, tend to suffer from the curse of dimensionality, with mixing rates which deteriorate as the dimension increases \cite{cotter2013mcmc,hairer2014spectral}. Since the overall computational cost depends critically on the number of steps required to resolve the target $\mu$ with the desired level of accuracy, it is essential to formulate methods whose mixing properties remain stable with respect to an increase in dimension. As already noted, in this work we bring together two active directions in modern MCMC research: methods tailored for non-parametric problems on the one hand, and multiproposal strategies designed to better exploit modern parallel computing on the other.

With this bottleneck in mind, a number of algorithms have been introduced that are well-defined even on functional spaces \cite{cotter2013mcmc, Coullon21} and therefore tend to exhibit the desired dimension-robust behavior. This is an interconnected family of methods \cite{glatt2020accept, glatt2024sacred} often referred to as the preconditioned or Hilbert space approach; the former moniker refers to the preconditioning, by the covariance operator $\cC$ of the Gaussian prior $\mu_0$, of a continuous-time dynamical system related to the target $\mu$. In particular, the Hilbert space approach includes the preconditioned HMC (or $\infty$HMC) \cite{Beskosetal2011, bou2018geometric,glatt2020accept} and the preconditioned Crank–Nicolson (pCN) algorithm \cite{NealpCNbeforepCN, beskos2008mcmc, cotter2013mcmc}, the latter being the method we extend here.  While the former may be expected to mix faster on a per-sample basis than the latter, $\infty$HMC requires costly and sometimes numerically or even analytically intractable gradient evaluations. In any case, both pCN and preconditioned HMC have been rigorously shown to mix at dimension-independent rates under suitable assumptions \cite{hairer2014spectral, bou2020hmc, glatt2021mixing}.

Beyond high-dimensional considerations, another fundamental challenge in algorithm design is the efficient use of modern computational resources. In particular, many MCMC algorithms remain inherently sequential and therefore fail to fully exploit parallel architectures. This motivates the second direction considered here: multiproposal MCMC methods, which generate a `cloud' of candidate states at each iteration and select among them using suitable acceptance probabilities. Ensuring that such algorithms preserve the desired target distribution and possess good convergence properties, however, requires careful design of the acceptance mechanism. A far-from-exhaustive list of contributions on the development, refinement, and implementation of such algorithms includes \cite{liu2000multiple, neal2003slice, frenkel2004speed, tjelmeland2004using, delmas2009does, calderhead2014general, luo2019multiple, glatt2024parallel, pozza_zanella_2025, Lin_quantum25, senn2026mess}. In particular, the constructions introduced in \cite{liu2000multiple} and \cite{tjelmeland2004using} to incorporate multiple proposals in the Random Walk Metropolis algorithm inspired analogous extensions of pCN in \cite{glatt2025+,glatt2024parallel}, leading respectively to the Multiple-Try pCN (MTpCN) and the Multiproposal pCN (mpCN) algorithms we are interested in here.

As regards this family of multiproposal methods, we note that it is often revealing to consider the limiting regime in which the number of proposals $p$ tends to infinity. This $p \to \infty$ limit offers a benchmark for the ultimate attainable performance of a given method as the degree of parallelism increases. Moreover, the $p = \infty$ Markov kernel can illuminate surprising structure in multiproposal chains with a large number of proposals, including asymptotic unbiasedness as $p$ grows and unexpected relationships to other sampling methods. The broad significance of this limit was explored systematically in the recent contribution \cite{glatt2025+} of two of the coauthors, where the $p = \infty$ limiting kernels for mpCN and MTpCN were derived. With this in mind, we analyze both the finite- and infinite-proposal regimes, establishing their mixing properties within a unified framework suitable for high- and infinite-dimensional target measures.

In rigorously assessing the convergence and efficiency of MCMC algorithms, the literature broadly follows two complementary approaches: functional inequalities and Harris theorems (or weak Harris theorems for high-dimensional targets). In the former approach, when the target distribution possesses sufficient structure, for instance log-concavity, strong convexity, or unimodality, one can establish the existence of a \textit{spectral gap} of the Markov kernel in the space $L^2_\mu$ of $\mu$-square-integrable functions. Such results are typically obtained through functional inequalities, including Poincar\'{e} and Cheeger inequalities \cite{Cheeger88, pavliotis, andrieu2024explicit}, and yield, under reversibility conditions, quantitative information on convergence rates and on the dependence of the spectral gap on algorithmic parameters. Consequences of the spectral gap include central limit theorems for all $L^2_\mu$ observables and explicit control of the asymptotic variance \cite{Kipnis86}, both closely related to the efficiency of Monte Carlo estimators.

On the other hand, for more general targets and algorithms, such as Metropolis--Hastings schemes in non-convex settings, establishing an $L^2_\mu$ spectral gap directly is often difficult or intractable. This is where the so-called Harris approach becomes advantageous. In finite dimensions, one instead commonly proves geometric ergodicity, namely convergence of the Markov transition kernels to the invariant measure in total variation. These arguments rely on Lyapunov and minorization conditions, which respectively provide a mechanism for controlling the chain when it is far from and close to the center of the state space. See e.g.~\cite{meyn_tweedie, RobertsRosenthal, Rudolf_ESS21}. While this approach still yields convergence of ergodic averages and associated limit theorems, it generally provides weaker quantitative control on the actual rates of equilibration, particularly as a function of algorithmic or target-measure parameters.

A fundamental additional difficulty in the infinite-dimensional setting we are concerned with here is that total variation distances become poorly suited to the analysis, since the relevant measures (for example, the kernels started from different initial states) are often mutually singular. To address this issue, Harris-type methods were extended to infinite-dimensional settings in \cite{hairer2008spectral, HMS11, glattholtz2025longterm}, leading to convergence results formulated in Wasserstein distances. This weak Harris framework has since proved to be a flexible and powerful tool for the analysis of Markov processes on infinite-dimensional spaces, including stochastic PDEs (e.g. \cite{Kuksin_Shirikyan_2012, butkovsky2020, carigi2022exponential}) and dimension-free MCMC algorithms \cite{hairer2014spectral,glatt2021mixing}.

One of the main advantages of the weak Harris approach is that it applies under comparatively mild structural assumptions on the target measure, while still yielding strong ergodic consequences. In the reversible setting, Wasserstein contraction estimates can moreover be used to recover $L^2_\mu$-spectral gaps, as shown in \cite{hairer2014spectral}. By contrast, approaches based on Poincar\'{e} or Cheeger-type inequalities often provide sharper quantitative information on convergence rates and spectral gaps, but typically require stronger assumptions on the target distribution, which may fail in more complex applications.

In this work, we rely on the weak Harris framework to study convergence properties of the multiproposal extensions of pCN, namely mpCN and MTpCN. This approach allows us to establish convergence under minimal assumptions on the target measure. The resulting Wasserstein contraction estimates imply strong statistical consequences, including strong laws of large numbers and central limit theorems for observables that are locally $\tfrac{1}{2}$-H\"{o}lder continuous with respect to the metric on the state space $\qsp$. In the reversible setting, they also yield an $L^2_\mu$-spectral gap, thereby connecting the probabilistic and functional-analytic perspectives.

In summary, the interplay between dimension-robust MCMC design, multiproposal sampling strategies, and quantitative convergence properties in high- and infinite-dimensional settings forms the core of the present work. We provide a detailed overview of the analytical results and numerical experiments in the following section.

\subsection{Overview of the main results}\label{subsec:overview:results}
We fix our state space $\qsp$ as a separable Hilbert space, and consider any target probability measure $\mu$ on $\qsp$ that is absolutely continuous with respect to a given reference measure $\mu_0 = \cN(0,\cC)$, namely $\mu_0$ is a Gaussian measure with zero mean and covariance operator $\cC$ on $\qsp$. We write any such target measure $\mu$ in the Gibbsian form
\begin{equation}\label{eq:intro:target}
	\mu(dx) = \frac{1}{Z} \exp\left(- \Pot(x)\right)\, \mu_0(dx), \qquad Z = \int_{\qsp} \exp\left(- \Pot(x)\right)\, \mu_0(dx),
\end{equation}
where $\Pot:\qsp\to \R$ is a suitable potential function. In a Bayesian context, the measure $\mu_0$ is called the prior, and $\Phi$ the negative log-likelihood. 

With the aim of sampling from such a target measure $\mu$, the standard preconditioned Crank-Nicolson algorithm has a proposal mechanism derived from an Ornstein-Uhlenbeck dynamics that is appropriately defined so as to leave the reference Gaussian measure $\mu_0$ invariant. Taking a Crank-Nicolson time discretization of such a dynamics yields, for a given initial state $x_0 \in \qsp$, a proposal of the form 
\begin{align}\label{prop:pCN}
	X = \rho x_0 + \sqrt{1 - \rho^2} \xi, \quad \xi \sim \mu_0,
\end{align}
where $\rho \in [0,1)$ is a given parameter depending on the associated time step. This proposed state $X$ is then accepted with probability 
\begin{align}\label{eq:intro:pcn:alpha}
	\alpha(x_0, X) = 1 \wedge \exp ( \Phi(x_0) - \Phi(X) ).
\end{align}

In the multiproposal pCN (mpCN) algorithm, derived in \cite{glatt2024parallel} and inspired by a conditionally independent proposal structure developed earlier in \cite{tjelmeland2004using}, one instead generates a cloud of $p$ proposed states, for a fixed number $p \geq 1$, as follows.  The proposal mechanism consists in first drawing $\bar{X}$ as in \eqref{prop:pCN}, namely $\bar{X} = \rho x_0 + \sqrt{1- \rho^2} \xi_0$, $\xi_0 \sim \mu_0$, and then drawing $p$ independent states $X_j \sim \cN(\rho \bar{X}, (1 - \rho^2)\cC)$, $j=1, \ldots, p$, around $\bar{X}$ so that 
\begin{align*}
	X_j = \rho \bar{X} + \sqrt{1- \rho^2} \xi_j, \quad \xi_j \sim \mu_0, \quad j=1, \ldots, p.
\end{align*}
A state $X_j$ is then selected among $X_1, \ldots, X_p$ according to the Barker-like acceptance probability
\begin{equation}\label{eq:intro:alpha}
	\alpha_j(x_0, X_1, \ldots, X_p) :=\frac{e^{- \Pot(X_j)}}{e^{- \Pot(x_0)}+ \sum_{k = 1}^p e^{- \Pot(X_k)}}, \quad j = 1, \ldots, p,
\end{equation}
or else the cloud of proposals is rejected, and hence the chain stays at the current point $x_0$, with probability 
\begin{equation}\label{eq:intro:alpha0}
\alpha_0(x_0, X_1, \ldots, X_p) = 1 - \sum_{j =1}^p  \alpha_j(x_0, X_1, \ldots, X_p) = \frac{e^{- \Pot(x_0)}}{e^{- \Pot(x_0)}+ \sum_{k = 1}^p e^{- \Pot(X_k)}}.
\end{equation}
The mpCN procedure is summarized in \cref{alg:mpcn}. In \cite{glatt2024parallel}, it was shown that the mpCN is an unbiased algorithm, via a general involutive framework developed therein that in fact yields reversibility to for broad class of MCMC algorithms.

\begin{algorithm}
	\caption{Multiproposal pCN (mpCN)}
	\begin{algorithmic}[1]\label{alg:mpcn}
		\State Select the algorithmic parameter $\rho \in [0,1)$.
		\State Choose an initial point $X^{(0)} \in \qsp$.
		\For{$k \geq 1$}
		\State Draw $\xi_0^{(k)} \sim \mathcal{N}(0,\mathcal{C})$ 
		\State Compute $\baX^{(k)} = \rho X^{(k-1)} + \sqrt{1-\rho^2}\,\xi_0^{(k)}$, 
		\State Draw $\xi_1^{(k)}, \ldots, \xi_p^{(k)}\sim \cN(0,\cC)$ iid. 
		\State Compute $X_0^{(k)} = X^{(k-1)}$ and 
		$X_j^{(k)} = \rho \baX^{(k)} + \sqrt{1-\rho^2}\,\xi_j^{(k)}$, 
		\State Set $X^{(k)} = X_j^{(k)}$, $j = 0, \ldots, p$, with probabilities $\alpha_j\big(X_0^{(k)},\ldots,X_p^{(k)}\big)$ as in \eqref{eq:intro:alpha} or \eqref{eq:intro:alpha0}
		\State $k \rightarrow k+1$
		\EndFor
	\end{algorithmic}
\end{algorithm}

The second algorithm we analyze here, the Multiple Try pCN (MTpCN) from \cite{glatt2025+}, adapts the strategy introduced in \cite{liu2000multiple} to create another unbiased multiproposal extension of pCN. In MTpCN, starting again from a state $x_0 \in \qsp$, a cloud of $p$ independent proposals is first drawn as $X_j \sim \cN(\rho x_0, \sqrt{1 - \rho^2} \cC),$ $j=1, \ldots, p.$ A state $Y = X_j$ is then selected among $X_1, \ldots, X_p$ with probability
\begin{equation}\label{eq:intro:mtpcn:beta}
	\beta_j(X_1, \ldots, X_p) := \frac{\exp(-\Pot(X_j))}{
		\sum_{l=1}^p \exp(-\Pot(X_l))}, \quad j=1, \ldots, p.
\end{equation}
Next, a second cloud of $p-1$ states is generated as $Z_j \sim \cN(\rho Y, \sqrt{1 - \rho^2} \cC),$ $j=1, \ldots, p-1.$ A Metropolis-Hastings accept-reject step is then performed comparing the first and second clouds to determine whether the previously selected state $Y$ is accepted or not. Specifically, $Y$ is accepted with probability
\begin{equation}\label{eq:intro:mtpcn:balpha}
	\balpha(x_0,X_1, \ldots, X_p, Z_1, \ldots, Z_{p-1}) = 1\wedge \frac{\sum_{l=1}^p\exp(-\Pot(X_l))}{\exp(-\Pot(x_0)) + \sum_{l = 1}^{p-1}\exp(-\Pot(Z_l))}
\end{equation}
and rejected otherwise. As shown in \cite{glatt2025+} again under the involutive framework from \cite{glatt2024parallel}, this scheme also ensures reversibility with respect to the target measure $\mu$. We summarize this second procedure in \cref{alg:mtpcn}.

\begin{algorithm}
	\caption{Multiple-Try Preconditioned Crank–Nicolson (MTpCN)}
	\begin{algorithmic}[1]\label{alg:mtpcn}
		\State Select the algorithmic parameter $\rho \in [0,1)$.
		\State Choose an initial point $X^{(0)} \in \qsp$.
		\For{$k \ge 1$}
		\State Draw $\xi_1^{(k)}, \ldots, \xi_p^{(k)}\sim \cN(0,\cC)$ i.i.d.
		\State Compute $X_j^{(k)} = \rho X^{(k-1)} + \rr \xi_j^{(k)}$, $j =1, \ldots, p$.
		\State Select $Y^{(k)} := X_j^{(k)}$ among $X_1^{(k)}, \ldots, X_p^{(k)}$ with probabilities $\beta_j(X_1^{(k)}, \ldots, X_p^{(k)})$ as in \eqref{eq:intro:mtpcn:beta}.
		\State Draw $\txi_1^{(k)}, \ldots, \txi_{p-1}^{(k)}\sim \cN(0,\cC)$ i.i.d.
		\State Compute $Z_j^{(k)} = \rho Y^{(k)} + \rr \txi_j^{(k)}$, $j =1, \ldots, p-1$.
		\State Set $X^{(k)} := Y^{(k)}$ with probability 
		$ \balpha(X^{(k-1)},X_1^{(k)} \ldots, X_p^{(k)}, Z_1^{(k)}, \ldots, Z_{p-1}^{(k)}) $ as in \eqref{eq:intro:mtpcn:balpha}.
		\State Otherwise, take $X^{(k)} := X^{(k-1)}.$
		\State $k \to k+1$
		\EndFor
	\end{algorithmic}
\end{algorithm}

In \cref{thm:intro:spectralgap} below, we summarize our main results for both the mpCN and the MTpCN algorithms. For this purpose, we denote by $\lbrace X^{(n)} \,:\, n\in \N\rbrace$ the Markov chain generated by either one of these algorithms for a given fixed number $p \geq 1$ of proposals and algorithmic parameter $\rho \in [0,1)$ and by $P_p$ its associated transition kernel, so that $X^{(n)} \sim P_p(X^{(n-1)}, \cdot )$, $n \in \N.$ Analytical formulations of the kernels associated to mpCN and MTpCN algorithms are presented in \eqref{eq:2:mpcn:kernel} and \eqref{eq:2:mtpcn:kernel}, respectively below.

Note that in our mixing results we obtain the convergence of  $P_p$ towards the target measure $\mu$ relative to, two different notions of distance. Firstly, we consider a Wasserstein distance $W_{\td}$ on the space of probability measures on $\qsp$, defined relative to a suitable semidistance $\td$ on $\qsp$.  This semidistance $\td$ taken from \cite{HMS11}
is specially adapted to mechanisms coupling two processes at large intermediate and small scales; see \eqref{eq:td} below. Secondly, we show convergence with respect to the norm in $L^2_\mu$, the space of real-valued and $\mu$-square-integrable functions on $\qsp$, namely $f: \qsp \to \R$ with $\int f^2 d \mu$ finite. See \cref{sec:prelim}  below for precise definitions regarding these notions of distance. In addition, we also obtain a Strong Law of Large Numbers (SLLN), a Central Limit theorem (CLT), and a concentration type inequality for the empirical Monte Carlo error
 \begin{equation}\label{eq:intro:error}
	\cE_{N,p}(f) := \frac{1}{N} \sum_{n = 1}^N f(X^{(n)}) - \int f\, d\mu,
\end{equation}
for any suitable test function $f$.

\begin{Theorem}\label{thm:intro:spectralgap}
	Let $\mu$ be the target measure in \eqref{eq:intro:target}, with the corresponding potential $\Phi$ assumed \emph{globally bounded} and \emph{Lipschitz}. Denote by $\lbrace X^{(n)} \; : \; n\in \N\rbrace$ the Markov chain generated by either the mpCN or the MTpCN algorithms, with a fixed number $p \geq 1$ of proposals. Then, there exists a semidistance $\td$ (explicit definition in \cref{thm:our_harris}) such that the following results hold for both mpCN and MTpCN:
	\begin{enumerate}
		\item \emph{Wasserstein convergence}: There exist $\lambda\in (0,1)$, a number of iterations $n_1\in \N$, and, for each initial state $X^{(0)}=x_0\in \qsp$, a positive constant $C= C(x_0)$ such that 
		\begin{equation}\label{eq:intro:harris:exp}
			\sup_{p \geq 1} W_{\td}(P_p^{n}(x_0, \cdot), \mu )\leq C(x_0) \lambda^n, \quad \mbox{for all} \,\,\, n \geq n_1.
		\end{equation}
		\item \emph{$L^2_\mu$ spectral gap}: for $\lambda$ as in item 1,
		\begin{equation}\label{eq:intro:specgapL2}
			\sup_{p \geq 1} \|P_p^n f - \smallint f \, d\mu\|_2 \leq \lambda^n\|f - \smallint f \, d\mu\|_2, \quad \mbox{for all } \,\, f\in  L^2_\mu \,\, \mbox{and} \,\, n\in \N,
		\end{equation}
		where $\|g\|_2^2 = \int g^2 \, d\mu$.
		\end{enumerate}
		Moreover, regarding the Monte Carlo error $\cE_{N,p}(f)$, \eqref{eq:intro:error}, we have:
        \begin{itemize}
		\item[3] \emph{SLLN and CLT}: $\cE_{N,p}(f)$ vanishes as $N\to \infty$ and $\sqrt{N}\cE_{N,p}(f)$ is asymptotically normal, whenever $f: \qsp \to \R$ is a Lipschitz function with respect to $\td$ and $X^{(0)}\sim \delta_{x_0}$, for any $x_0\in \qsp$, or $f\in L^2_\mu$ and $X^{(0)}\sim \mu$; 
		\item[4] \emph{Hoeffding's inequality}: for $f\in L^2_{\mu}$ bounded, $f\in [a,b]$, and $X^{(0)}\sim \mu$,
		\begin{equation*}
			\sup_{p \geq 1} \bbP\left( |\cE_{N,p}(f)| > \varepsilon \right) \leq 2\exp\left(- \frac{1 - \lambda}{1 +\lambda} \frac{2\varepsilon^2 N}{(b-a)^2}\right),\quad \text{for any }\varepsilon>0,
		\end{equation*}
		where $\lambda$ is as in item 1.
	\end{itemize}
\end{Theorem}

These results can also be extended to target measures $\mu$ with an unbounded potential $\Phi$. In this case, however, the spectral gap $\lambda$ and the time $n_1$ in \cref{thm:intro:spectralgap} deteriorate as the number of proposals increases. We present this extension for mpCN in details in \cref{app:unbound}. We expect the results for MTpcN to follow a similar strategy.

We also consider the kernels corresponding to the limiting case of an infinite number of proposals, denoted $P_\infty$, introduced in \eqref{eq:2:inftypCN} for mpCN and in \eqref{eq:2:mtminfty:kernel} for MTpCN. We show that these kernels exhibit Wasserstein contraction, respectively in \cref{thm:3:inftypcn} and \cref{thm:3:inftyMTpcn}. Interestingly, analyzing the infinite-proposal kernel directly allows us to establish this result for target measures with Lipschitz but possibly unbounded potential $\Phi$. This highlights a difference between two analytical perspectives: one may first establish mixing properties for kernels with a finite number of proposals and then pass to the limit as $p \to \infty$, or instead study the limiting kernel itself. In the present setting, the latter viewpoint avoids the deterioration of certain estimates with increasing $p$ that arises in the finite-proposal analysis.
The complete statements and extensions are given in \cref{sec:results}.

\subsection{Elements of the proofs}
The proofs of our main results (which we summarized as \cref{thm:intro:spectralgap}) proceed in two distinct steps. First, we establish exponential mixing properties in a suitable Wasserstein distance $W_{\td}$ for the Markov kernels associated with the proposed algorithms. This Wasserstein mixing then allows us to derive the $L^2_\mu$ spectral gap standard in this literature (see \cref{thm:L2gap}). Second, we leverage these estimates to deduce items 3 and 4 in \cref{thm:intro:spectralgap}, namely statistical properties of the resulting Monte Carlo estimators, including the strong law of large numbers, the central limit theorem, and non-asymptotic concentration bounds.
 
 Note that the latter steps mostly rely on existing results in the literature: once suitable contraction and spectral gap estimates are available, standard arguments yield convergence of ergodic averages and asymptotic normality. For observables in $\Lip(\td)$, these results are summarized for example in \cite[Appendix~A]{glatt2021mixing}. For observables in $L^2_\mu$, the CLT \cite{Kipnis86,latuszynski2013clts} and SLLN follow classically from the derived spectral gap. We provide a proof of the SLLN in \cref{app:L2} for completeness as, while it is standard for Harris recurrent or geometrically ergodic Markov chains \cite{meyn_tweedie}, an explicit formulation based on an $L^2_\mu$ spectral gap does not seem to be readily available in the literature. Finally, regarding the Hoeffding-type bound we rely on recent results \cite{Hoeffding} which allow one to derive concentration inequalities for bounded observables under the same $L^2_\mu$ spectral gap assumptions.
    
The core of the analysis is therefore devoted to establishing the mixing properties in items 1 and 2 in \cref{thm:intro:spectralgap}. Our approach is based on the weak Harris theorem \cite{HMS11}, which provides a flexible framework to obtain exponential convergence in Wasserstein distance for Markov chains on general state spaces. To apply this theorem, we verify three key ingredients for the transition kernels of the algorithms under consideration: the existence of a Lyapunov function, a suitable notion of small set, and a contraction property with respect to an appropriate distance-like function $d$. While the Lyapunov structure follows from relatively standard estimates, establishing contraction requires a more delicate coupling argument. In particular, we construct an explicit coupling between the kernels $P_p(X^{(0)},\cdot)$ and $P_p(\tX^{(0)},\cdot)$ for two initial points $X^{(0)} \neq \tX^{(0)}$ close to each other in $d$, and show that the expected distance between the coupled chains decreases after one step.

Our construction draws on the coupling strategy introduced in \cite{hairer2014spectral} for the single-proposal pCN algorithm, but the multiproposal setting requires us to go well beyond it. A single step of the chain now involves an entire cloud of $p$ proposals together with a Barker-type selection rule \eqref{eq:intro:alpha}, in place of the classical Metropolis--Hastings accept--reject mechanism \eqref{eq:intro:pcn:alpha} of pCN. The coupling we develop must therefore control all $p$ proposals together with the selection step, ensuring that the two chains select ``compatible'' outcomes with sufficiently high probability while keeping the resulting contraction constants uniform in $p$; it is this cloud-and-selection structure, which has no counterpart in the single-proposal analysis, that makes the construction considerably more delicate. We regard it as a significant technical contribution of the present work.

The construction for mpCN is given in details in \cref{prop:contr:fin:p:1} and it proceeds in two stages. First, we couple the proposal clouds synchronously, 
using the same Gaussian random variables $\xi_0^{(1)}, \ldots \xi_p^{(1)}$ to generate the proposals 
\begin{align*}
    X_j^{(1)} &= \rho \left(\rho X^{(0)} + \rr \xi_0^{(1)}\right) + \rr \xi_j^{(1)},  \quad j =1, \ldots, p,
\end{align*}
and $\tX_j^{(1)}$, $j =1, \ldots, p$, respectively with initial condition $\tX^{(0)}$, so that the distance between proposals with the same index $j$ is directly controlled by $\rho^2\|X^{(0)} - \tX^{(0)}\|$. Second, we couple the selection step by constructing a coupling of the Barker-type probabilities \eqref{eq:intro:alpha} and their counterparts for $\tX^{(0)}$, through a shared uniform random variable. This ensures that proposals with the same index are selected whenever possible, and separates the cases where the indices of selected proposals differ. 

This leads to a decomposition of the coupling into two regimes: on the event where both chains select the same index, contraction follows from the synchronous coupling of the proposals; on the complementary event, where different indices are selected, the error is controlled through quantitative bounds on the discrepancy between the acceptance probabilities. These bounds rely on the Lipschitz continuity of $\Phi$ and are sufficiently sharp to ensure an overall contraction, while boundedness of $\Phi$ ensures that the corresponding constants remain uniform in the number of proposals $p$.

The boundedness of the potential $\Phi$ is used in several other steps when verifying the conditions of the weak Harris theorem. In fact, it yields uniform bounds on the acceptance probabilities \eqref{eq:intro:alpha} ensuring that each $\alpha_j$ is bounded away from zero and one uniformly in the proposal cloud, which in turn ensures that the constants appearing in the contraction and Lyapunov conditions remain uniform in the number of proposals.
Removing this assumption is possible, as showed in \cref{app:unbound}, while losing uniform boundedness in $p$ of the relevant constants.

The MTpCN algorithm combines elements of the mpCN proposal mechanism with a Metropolis--Hastings accept--reject structure. The first stage, corresponding to the selection of a proposal from the cloud, is treated using the same ideas developed for mpCN. The second stage, corresponding to the accept–reject step, is handled using techniques closer to the single-proposal setting, where one couples the acceptance decisions through a shared uniform random variable and exploits bounds on the acceptance ratio. The combination of these two components requires additional care, but ultimately yields a contraction estimate of a similar form as in the multiproposal case.

Finally, we also study the limiting kernels corresponding to an infinite number of proposals. We illustrate the strategy starting from the $\infty$-pCN kernel which has the following explicit formula 
\begin{equation}\label{eq:intro:inftypCN0}
    P_\infty(x, dy) = \int_{\qsp^2} \baQ(z,dy) Q(x, d z) \quad \text{with }\baQ(z,dy)= \dfrac{\exp(-\Pot(y)) Q(z, d y)}{\int \exp(-\Pot(u)) Q(z, d u) }.
\end{equation}
Here we recognize the first step according to the proposal, which we referred to as Tjelmeland correction, corresponding to generating $\overline{X}^{(k)}$ in \cref{alg:mpcn}, and an internal kernel $\baQ(z,dy)$ corresponding to taking the limit for $p\to \infty$ of the proposal and acceptance steps after the Tjelmeland correction. Note how, in this regime, the accept–reject structure present in the finite-$p$ algorithms disappears, and the coupling strategy used in that setting cannot be applied. 
Starting from two distinct initial points, we first couple the outer proposal step synchronously, as in the finite-$p$ case, but the main difficulty is then to couple $\baQ(z, \cdot), \, \baQ(\tz, \cdot)$ with different initial points $z\neq \tz$.
To address this, we introduce an auxiliary coupling that effectively reintroduces an accept–reject mechanism at this level. This is achieved by constructing a coupling of $\baQ(z, \cdot), \, \baQ(\tz, \cdot)$ using their Radon--Nikodym derivatives, allowing us to compare the two chains through a controlled accept–reject procedure. The coupling construction is carried through for two generic probability measures \cref{lemma:AR_coupling} making it a versatile tool also for other applications. 

Finally, for the limiting Multiple-Try kernel \eqref{eq:2:mtminfty:kernel}, the limiting kernel still retains an explicit accept–reject step, so the coupling can be built by combining two ingredients.
We treat the internal reweighted proposal step using the same coupling introduced for the $\infty$-pCN kernel, and we couple the outer accept–reject step as in the classical pCN setting.

\subsection{Summary of numerical case studies}
We supplement the theoretical analysis with extensive numerical tests designed to assess the practical performance of the proposed algorithms.
We consider three inverse problems: two low-dimensional examples with non-trivial likelihood structures, and a higher-dimensional toy problem inspired by PDE-based models for fluid flows. In the first two problems, we consider a two-dimensional parameter $x =(x_1, x_2) \in \R^2$, and corresponding forward maps $f_1, f_2 :\mathbb{R}^2 \to \mathbb{R}^2$ that generate multiwell and polar-twist surfaces in $\mathbb{R}^{3}$ (see \cref{sec:multiwell,sec:polar}).
The third problem (\cref{sec:matrix_inversion}) is formulated as a functional inverse problem motivated by sparse indirect observations of a stationary fluid flow under damping and external forcing; in its discretized version, the forward map reduces to a matrix inversion problem.

Across these examples, the simulations are designed to investigate stationary mixing using standard proxies such as the Mean Square Jumping Distance (MSJD) and the Effective Sample Size (ESS). In addition, for the high-dimensional toy solute transport problem, we analyze the warm-up behavior of the algorithms. Although the burn-in phase is brief in this setting, the problem's high dimensionality and geometric complexity make it clearly observable. This makes it a testbed for understanding algorithmic behavior in more challenging, large-scale data-driven models, where warm-up can constitute a significant portion of the total computational cost.
Overall, we observe that the mixing performance of mpCN and MTpCN improves as the number of proposals $p$ increases, in both low- and high-dimensional settings, consistently outperforming the single-proposal pCN in terms of these proxies.

The multiwell posterior (\cref{sec:multiwell}) provides a first test case to visualize the balance between global and local exploration. For a wide range of values of $\rho$, mpCN successfully explores all posterior modes, indicating good global mixing. However, for sufficiently large $\rho$, the proposals do not contain enough randomness and the algorithm loses the ability to transition between modes in directions not informed by the current state, defaulting to essentially local exploration.

The polar twist example (\cref{sec:polar}), characterized by thin and curved high-probability regions, highlights the dependence of mixing on the algorithmic parameter $\rho$. Also in this case, mixing deteriorates as $\rho$ increases, reflecting the need for sufficiently exploratory proposals to reconstruct the geometry of the target. Comparing mpCN and MTpCN, we observe similar qualitative behavior, with mpCN mixing better than MTpCN. Given also the additional likelihood evaluations needed for MTpCN, we focus on mpCN for the higher dimensional problem.

The toy solute transport problem (\cref{sec:matrix_inversion}) allows for a more detailed study of both the warm-up and stationary regimes in comparison with the baseline strategy of running multiple chains in parallel. The results show that mpCN converges to stationarity significantly faster than both single-chain pCN and an embarrassingly parallel strategy. This improved warm-up behavior is reflected in faster stabilization of traceplots (\cref{fig:mpcn_pcn_traceplots}) and in a more rapid decay of the mean square error of ergodic averages \eqref{eq:MSE} as seen in \cref{fig:running_mse}.

In contrast, when assessing mixing at stationarity, parallel pCN chains outperform multiproposal methods in terms of ESS and MSJD (\cref{fig:sweep_solute_ep}), in agreement with observations in \cite{pozza_zanella_2025}. While the stationary mixing of mpCN with $p$ proposals is bounded by that of $p$ independent parallel pCN chains (when compared under an equal computational and storage budget), the improved burn-in behavior of multiproposal methods can lead to significant practical gains in complex, high-dimensional problems. Finally, we observe that increasing $p$ improves the robustness of mpCN with respect to the choice of $\rho$, in the sense that a wider range of parameter values yields near-optimal performance (see \cref{fig:fraction}), allowing more and more flexibility in the tuning of the algorithm parameters. 

\subsection{Organisation of the manuscript}
The remainder of the paper is organized as follows. In \cref{sec:algorithms}, we introduce the multiproposal preconditioned Crank–Nicolson (mpCN) and Multiple-Try preconditioned Crank–Nicolson (MTpCN) algorithms, together with their infinite-proposal limits. We also review the weak Harris framework and the Wasserstein contraction tools that form the basis of our analysis.
\cref{sec:results} contains the statements of the main theoretical results of the paper. 
The proofs of these main results are presented in \cref{sec:proofs}. We first analyze the mpCN algorithm, treating separately the finite- and infinite-proposal regimes, and then turn to the corresponding results for the MTpCN algorithm.
\cref{sec:numerics} provides numerical experiments illustrating the behavior of the proposed methods as described above. 
We conclude with a brief discussion in \cref{sec:outlook} about possible extensions and open questions.

The appendices collect several technical and supplementary results. \cref{app:unbound} is concerned with extending \cref{subsec:proofs_mpcN}, the results for mpCN with bounded and Lipschitz potential, to potentially unbounded ones. \cref{app:proofs} collects proofs of some technical results used in the body of the paper and \cref{ap:numerics} provides additional numerical results. In \cref{app:L2}, we give a complete, self-contained proof that the Wasserstein spectral gap implies the $L^2_\mu$ spectral gap, filling in the details of the elegant but condensed argument first presented in \cite{hairer2014spectral}. Last, in \cref{app:harris} we review the original proof of the weak Harris theorem \cite{HMS11} for generic kernels with discrete times to ensure Wasserstein spectral gaps not only at one time but for a sequence of times.

\section{Preliminaries}\label{sec:algorithms}
Let $\qsp$ be a separable Hilbert space with inner product $\langle \cdot, \, \cdot\rangle$ and associated norm $\|\cdot\|$. Denote by $\cB(\qsp)$ the $\sigma$-algebra of Borel sets in $\qsp$. Take $\mu_0= \cN(0, \cC)$ to be a reference Gaussian measure on $\qsp$ with zero mean and covariance operator $\cC: \qsp\to \qsp$ being linear, symmetric, strictly positive-definite, and trace-class. Under this setting, we consider a class of target measures $\mu$ that are absolutely continuous with respect to $\mu_0$ of the form in \eqref{eq:intro:target}.

We recall that the Markov transition kernel of a standard Metropolis-Hastings algorithm is defined as a mapping $P: \qsp\times \cB(\qsp) \to [0,1]$ written as
\begin{equation*}
	P(x, dy) =   \alpha(x, y)Q(x, dy) + \delta_x(dy)\int (1 - \alpha(x, z)) \,Q(x, dz),
\end{equation*}
where $Q: \qsp\times \cB(\qsp) \to [0,1]$ represents the associated \textit{proposal kernel} and $\alpha: \qsp \times \qsp \to  [0,1]$ the  \textit{acceptance probability}. In explicit terms, given an initial state $x \in \qsp$, a proposal state $Y$ is drawn from $Q(x, \cdot)$, which is then accepted with probability $\alpha(x,Y)$, and rejected otherwise. As such, the next state in the Markov chain generated by the algorithm can also be written as 
\begin{equation}\label{eq:2:pcn:X1}
	X^{(1)} = Y \,\mathbbm{1}_{U\leq \alpha(x, Y)} + x \,\mathbbm{1}_{U> \alpha(x, Y)}
\end{equation}
where $U$ is a uniform random variable on the interval $[0,1]$, i.e.~$U \sim \cU([0,1])$, which is chosen independently of $Y$.

With this notation, we have for the pCN algorithm recalled in \cref{subsec:overview:results} that the proposal kernel is given by
\begin{equation}\label{eq:2:proposal}
	Q(x, dy) = \cN(\rho x, \rr \cC)(dy), \quad x \in \qsp,
\end{equation}
with algorithmic parameter $\rho\in [0,1)$, and the acceptance probability is 
\begin{equation}\label{eq:2:accept}
	\alpha(x,y) =  1 \wedge \exp\left(- \Pot(y)+ \Pot(x)\right), \quad x,y \in \qsp.
\end{equation}
Note that, for $\rho = 0$, the pCN algorithm reduces to an independence sampler from the reference measure $\mu_0$, whereas for $\rho = 1$ the pCN kernel degenerates to $\delta_x$, keeping the chain frozen at the initial point. 

Finally, it is useful to notice that the kernel $Q$ in \eqref{eq:2:proposal} can also be written as the pushforward of the reference measure $\mu_0$ by the function $F: \qsp \times \qsp \to \qsp$ defined as
\begin{equation}\label{eq:2:defF}
	F(x, w) = \rho x + \rr w,
\end{equation}
so that a random variable $Y \sim Q(x, \cdot) = F(x, \cdot)^*\mu_0$ can be written explicitly as 
\begin{equation*}
	Y = F(x, \xi) = \rho x + \rr \xi, \quad \text{with }\xi\sim \mu_0.
\end{equation*}

\subsection{Multiproposal pCN}
Within the setting of \cref{alg:mpcn}, fix $p \geq 1$ the number of proposals per iteration, the algorithmic parameter $\rho\in [0,1),$ and the starting point $x_0\in \qsp$.
Given independent and identically distributed random variables $\xi_j \sim \mu_0$, $j=0,\ldots, p,$ and recalling the notation \eqref{eq:2:defF}, the $p$ proposals generated in one iteration of \cref{alg:mpcn} can be written in terms of the preliminary draw
\begin{align}
	\baX &= F(x_0, \xi_0) = \rho x_0 + \rr \xi_0
\end{align}
as
\begin{align}
	 \begin{split}
		X_j &= F(\baX, \xi_j) = F(F(x_0, \xi_0), \xi_j) \\
		& = \rho^2 x_0 + \rho \rr \xi_0 + \rr \xi_j, \quad j = 1, \ldots, p.
	\end{split}\label{eq:2:proposals}
\end{align}

Then, recalling the definition of the acceptance probabilities $\alpha_j$, $j=0, \ldots,p,$ in \eqref{eq:intro:alpha}, the next state in the constructed Markov chain can be written as follows
\begin{equation}\label{eq:2:mpcn:chain}
    X^{(1)} = x_0 \mathbbm{1}_{U\in I_0} + \sum_{j = 1}^p X_j \mathbbm{1}_{U\in I_j}, \quad I_j = \left[ \sum_{k = 0}^{j-1}\alpha_k (x_0, X_1, \ldots, X_p), \sum_{k = 0}^{j}\alpha_k (x_0, X_1, \ldots, X_p)\right) , 
\end{equation}
 where we set $\alpha_{-1} = 0$, and $U\sim \cU([0,1])$ is independent of the samples $\xi_0, \ldots, \xi_p$ from the reference measure $\mu_0$. Note that the interval $I_j$ 
 has length $\alpha_j,$ ensuring that each $X_j$ has an associated probability $\alpha_j$ of being accepted. 

The mpCN transition kernel associated to the chain $\lbrace X^{(k)}, \; k \in \N\rbrace$ can be expressed in integral form as
\begin{equation}\label{eq:2:mpcn:kernel}
      P_p(x_0, dy) = \sum_{j=0}^p \int_{\qsp^{p+1}}\frac{e^{- \Pot(x_j)}}{\sum_{l=0}^p e^{- \Pot(x_l)}} \delta_{x_j}(dy)\prod_{k= 1}^p Q(z, dx_k) Q(x_0, dz).
\end{equation}

Alternatively, we can write the kernel $P_p$ as the pushforward by a suitable mapping of $\mu_0$ and the uniform distribution on $[0,1]$, thus providing a better connection with the expression \eqref{eq:2:mpcn:chain}. For this purpose, consider the function $\Fbar(x_0, \cdot): \qsp^{p+1} \to \qsp^p$ defined for $\bw = (w_0, w_1, \ldots, w_p)\in \qsp^{p+1}$ as 
\begin{align}\label{eq:mpcn:Fbar}
	\Fbar\left(x_0, \bw \right) = (F(F(x_0,w_0),w_1), \ldots, F(F(x_0,w_0),w_p)),
\end{align}
where $F$ is as in \eqref{eq:2:defF}.
With this notation, we have that
\begin{align*}
	\int_\qsp \prod_{j=1}^p \qk(z, d x_j) \qk(x_0, d z) = \Fbar(x_0, \cdot)^* \mu_0^{\otimes (p+1)} (d x_1, \ldots, d x_p).
\end{align*}
Moreover, define the intervals 
\begin{equation*}
    \bar{I}_j= \bar{I}_j\left(x_0,\bw\right) =  \left[ \sum_{k = 0}^{j-1}\alpha_k \left(x_0, \Fbar(x_0, \bw)\right), \sum_{k = 0}^{j}\alpha_k \left(x_0,\Fbar(x_0, \bw)\right)\right),\quad j = 1, \ldots, p,
\end{equation*}
 and the function $D(x_0, \cdot, \cdot): \qsp^{p+1} \times [0,1] \to \qsp$ as
\begin{align*}
	D(x_0, \bw, \zeta) = \begin{cases}
		F(F(x_0, w_0), w_j)  \quad \mbox{ if } \zeta \in \bar{I}_j, \,\, j = 1 , \ldots, p, \\
		x_0  \hspace{2.55cm}\mbox{ if } \zeta \in \left[ 0, \ar_0 \right).
	\end{cases}
\end{align*}
It is not difficult to show that
\begin{align}\label{eq:P:mpCN:pfwd:mu0:U}
	P_p(x_0, d y) = D(x_0, \cdot, \cdot)^*(\mu_0^{\otimes (p + 1)} \otimes \cU([0,1]))(d y).
\end{align}

In the next lemma, we introduce other two formulations of the mpCN kernel \eqref{eq:2:mpcn:kernel} that will allow us to heuristically illustrate its behavior in the limit of many proposals. Its proof is given in \cref{app:proofs}.

\begin{Lemma}\label{lemma:Pp1}
    The mpCN kernel \eqref{eq:2:mpcn:kernel} can be written in the following equivalent formulations:
    \begin{enumerate}
        \item  The first formulation reads
        \begin{equation}\label{eq:Pp2_AR}
         P_p (x_0, d y) = \int_{\qsp} \alpha_1'(x_0, y, z) Q(z, dy) Q(x_0, dz) + \delta_{x_0}(dy) \int_{\qsp^2} \alpha_0'(x_0, x_1, z) Q(z, dx_1)Q(x_0, dz),
    \end{equation}
    with
    \begin{equation}\label{def:a0p:a1p}
        \alpha_0'(x_0, x_1, z) = e^{-\Pot(x_0)} \gamma_p(x_0, x_1, z) \quad \text{and} \quad \alpha_1'(x_0, x_1, z) = p e^{-\Pot(x_1)} \gamma_p(x_0, x_1, z),
    \end{equation}
    where
    \begin{equation*}
        \gamma_p(x_0, x_1, z) = \int_{\qsp^{p-1}} \left( e^{-\Pot(x_0)} + e^{-\Pot(x_1)} + \sum_{l=2}^p e^{-\Pot(x_l)} \right)^{-1}\prod_{k=2}^p Q(z, dx_k).
    \end{equation*}
    \item The second formulation reads
    \begin{equation*}
        \begin{split}
             P_p (x_0, d y) = T_1(p) + T_2(p) - T_3(p) 
        \end{split}
    \end{equation*}
    with 
    \begin{align*}
        T_1(p) &=  \int_{X}  \frac{ pe^{-\Pot(y)}}
      { e^{-\Pot(x_0)  } + p \int e^{-\Pot(u)}  Q(z, d u)}
     Q(z, d y) Q(x_0, dz),\\
      T_2 (p)&= \int_{X}  \frac{e^{-\Pot(y)} }
      { e^{-\Pot(x_0)  } + p \int e^{-\Pot(u)}  Q(z, d u)} \delta_{x_0}(dy) Q(x_0, dz),\\
       T_3 (p)&=  \sum_{j=0}^p \int_{X} \int_{X^p}  \delta_{x_j} (d y) \frac{e^{-\Pot(x_j)}}{\sum_{k=0}^p e^{-\Pot(x_k)}}
    \frac{ p^{-1} \sum_{l = 1}^p (e^{-\Pot(x_l)} - \int e^{-\Pot(u)}Q(z, d u) )}{p^{-1}e^{-\Pot(x_0)  } + \int e^{-\Pot(u)}Q(z, d u)}
   \prod_{l =1}^p 
   Q(z, d x_l) Q(x_0, dz).
    \end{align*}
    \end{enumerate}
\end{Lemma}

Note that $\alpha_0'$ and $\alpha_1'$ in \eqref{def:a0p:a1p} do not represent acceptance probabilities. In fact, $\alpha_1'$ may be larger than one. However, the following relation holds
\begin{equation*}
    \int_{\qsp^2} (\alpha_0'(x_0, x_1, z) + \alpha_1'(x_0,x_1,z)) Q(z, dx_1)Q(x_0, dz ) =1 \quad \mbox{for all }\,\, x_0 \in \qsp.
\end{equation*}

Formally, given the formulation \eqref{eq:Pp2_AR}, thanks to the strong Law of Large Numbers, we can expect the following convergence
\begin{align*}
    \int_{\qsp} &\alpha_1'(x_0, y, z) Q(z, dy) Q(x_0, dz) \\
    &=  \int_{\qsp} \int_{\qsp^{p-1}} \frac{ e^{-\Pot(y)} }{\frac{1}{p}e^{-\Pot(x_0)} + \frac{1}{p}e^{-\Pot(y)} + \frac{1}{p}\sum_{l=2}^p e^{-\Pot(x_l)}}\prod_{k=2}^p Q(z, dx_k) Q(z, dy) Q(x_0, dz)\\
    &\to  \int_{\qsp} \frac{ e^{-\Pot(y)}}{\int e^{-\Pot(u)} Q(z, du)} Q(z, dy) Q(x_0, dz), \quad \text{ as } p\to \infty,
\end{align*}
while the term representing the rejection will vanish as the number of proposals goes to infinity, namely
\begin{align*}
   \delta_{x_0}(dy) \int_{\qsp^2} &\alpha_0'(x_0, x_1, z) Q(z, dx_1)Q(x_0, dz) \\
    &=  \frac{1}{p} \delta_{x_0}(dy)\int_{\qsp^2} \int_{\qsp^{p-1}} \frac{ e^{-\Pot(x_0)} }{\frac{1}{p} \left(e^{-\Pot(x_0)} + e^{-\Pot(y)} + \sum_{l=2}^p e^{-\Pot(x_l)}\right)}\prod_{k=2}^p Q(z, dx_k) Q(z, dy) Q(x_0, dz)\\
    &\to  0, \quad \text{ as } p\to \infty.
\end{align*}

Similarly, if we take the limit for $p\to \infty$ of the three terms in the second formulation in  \cref{lemma:Pp1}, we have
\begin{align*}
    T_1(p) \to \int \frac{e^{-\Pot(y)}}{\int e^{-\Pot(u)}  Q(z, d u)}
     Q(z, d y) Q(x_0, dz),
\end{align*}
the second term $T_2(p)$ converges to zero, and, by the strong Law of Large Numbers, we expect that $T_3(p)$ converges to zero as well, so that 
\begin{equation}\label{eq:2:inftypCN0}
    P_p(x_0, dy)\to P_\infty(x_0, dy) := \int \frac{e^{-\Pot(y)}}{\int e^{-\Pot(u)}  Q(z, d u)}
     Q(z, d y) Q(x_0, dz) \quad \text{ as } p \to \infty.
\end{equation}
This limit is rigorously proved in \cite{glatt2025+} under appropriate conditions on the potential $\Phi$, in the sense of weak convergence and in total variation distance, including explicit rates of convergence with respect to the latter. The algorithm associated to the limiting kernel $P_\infty$ in \eqref{eq:2:inftypCN0} is called $\infty$-pCN.

In the rest of the article we will consider a more general version than \eqref{eq:2:inftypCN0} where the algorithmic parameter $\rho$ can be different between the internal and external kernel, namely 
\begin{equation}\label{eq:2:inftypCN}
    P_\infty(x_0, dy) = \int \dfrac{\exp(-\Pot(y)) Q_1(z, d y)}{\int \exp(-\Pot(u)) Q_1(z, d u) } Q_2(x_0, d z),
\end{equation}
where the  kernels $Q_1,\, Q_2$ are defined as 
\begin{equation*}
    Q_i(x, dz) = \cN(\rho_i x, (1-\rho_i^2)\cC)(dz), \quad i=1,2,
\end{equation*}
and $\rho_i\in [0,1)$, $i=1,2$.
In addition, we denote
\begin{equation}\label{eq:2:baQ}
    \baQ_1(z, dy):= \dfrac{\exp(-\Pot(y)) Q_1(z, d y)}{\int \exp(-\Pot(u)) Q_1(z, d u) },
\end{equation}
so that 
\begin{equation*}
     P_\infty(x_0, dy) = \int \baQ_1(z, dy) Q_2(x_0, dz).
\end{equation*}
\begin{Remark}
    According to \cite{glatt2025+}, 
    choosing $\rho_1 = \rho_2$ provides an unbiased algorithm.  The situation
    where these two parameters are different would be expected to violate reversibility in general, but they are retained here for generality.
\end{Remark} 

\begin{Remark}
    Observe that $\baQ_1(z, \cdot)$ is the invariant measure of the following Langevin dynamics 
\begin{equation}\label{eq:Langevin1}
    dY_t = -\left[ Y_t - \rho_1 z+ \frac{(1 - \rho_1^2)}{2} \cC \nabla \Pot(Y_t)\right]\, dt + \sqrt{(1 - \rho_1^2) \cC} dW_t.
\end{equation}
We will not directly exploit this connection with the Langevin dynamics \eqref{eq:Langevin1} in the results we present for the $\infty$-pCN, nonetheless we show in \cref{app:langevin:lyap} how to ensure one of the key ingredients for mixing, the Lyapunov structure, thanks to the mixing properties of the Langevin dynamics. 
\end{Remark}

\subsection{Multiple Try pCN}
We turn next to \cref{alg:mtpcn} and again fix $p \geq 1$ as the number of proposals, the algorithmic parameter $\rho\in [0,1),$ and a starting state $x_0\in \qsp$. The next state in the MTpCN Markov chain can be expressed as follows. Consider independent random variables $\xi_j \sim \mu_0$, $j=1, \ldots, p$, $\txi_k \sim \mu_0$, $k=1, \ldots, p-1$, and $U, \tU \sim \cU([0,1])$. Then, recalling the definition of $F$ in \eqref{eq:2:defF}, we set
\begin{align*}
    X_j = F(x_0, \xi_j) = \rho x_0 + \rr \xi_j, \quad j=1, \ldots, p,
\end{align*}
and, similarly to \eqref{eq:2:mpcn:chain} for mpCN, define 
\begin{equation*}
    Y = \sum_{j = 1}^p X_j \mathbbm{1}_{U\in \bar{I}_j},  \quad  \bar{I}_j = \left[ \sum_{k = 1}^{j-1} \beta(X_1, \ldots ,X_p ), \sum_{k = 1}^{j} \beta(X_1, \ldots ,X_p)\right),
\end{equation*}
with $\beta$ as given in \eqref{eq:intro:mtpcn:beta}.

Next, we define the auxiliary variables $Z_1, \ldots, Z_{p-1}$ drawn independently from $Q(Y, \cdot)$, namely 
\begin{equation*}
    Z_k = F(Y, \txi_k) = \rho Y + \sqrt{1 - \rho^2} \txi_k, \quad k = 1, \ldots, p-1.
\end{equation*}
Finally, the next state of the chain constructed via \cref{alg:mtpcn} reads
\begin{equation}\label{eq:mtpcn:X1}
    X^{(1)} = Y \mathbbm{1}_{\tU \leq \balpha(x_0, X_1, \ldots, X_p, Z_1, \ldots, Z_{p-1})} + x_0 \mathbbm{1}_{\tU > \balpha(x_0,X_1, \ldots, X_p, Z_1, \ldots, Z_{p-1})},
\end{equation}
with $\balpha$ as defined in \eqref{eq:intro:mtpcn:balpha}.

The MTpCN transition kernel $P_p: \qsp\times \cB(\qsp)\to [0,1]$ associated to the chain just constructed has been rigorously derived in \cite{glatt2025+}, and it explicitly reads
\begin{multline}\label{eq:2:mtpcn:kernel}
    P_p(x_0, dy) = \sum_{j = 1}^p \int_{\qsp^{2p-1}} \beta_j(x_1, \ldots, x_p) \left[ \balpha(x_0, \ldots, x_p, z_1, \ldots, z_{p-1}) \delta_{x_j}(dy)\right.\\ \left.
    +  (1 -  \balpha(x_0, \ldots, x_p, z_1, \ldots, z_{p-1})) \delta_{x_0}(dy)\right]\prod_{k=1}^{p-1}Q(x_j, dz_k)\prod_{m=1}^{p}Q(x_0, dx_m),
\end{multline}
where we recall that
\begin{equation}\label{eq:2:mtpcn:balpha}
    \balpha(x_0, \ldots, x_p, z_1, \ldots, z_{p-1}) = 1\wedge \frac{\sum_{l=1}^p\exp(-\Pot(x_l))}{\exp(-\Pot(x_0)) + \sum_{l = 1}^{p-1}\exp(-\Pot(z_l))},
\end{equation}
and 
\begin{equation}\label{eq:2:mtpcn:beta}
      \beta_j(x_1, \ldots, x_p) := \frac{\exp(-\Pot(x_j))}{
        \sum_{l=1}^p \exp(-\Pot(x_l))}.
\end{equation}

We clearly see from \eqref{eq:mtpcn:X1} that this algorithm has elements of a classic accept-reject Metropolis-Hastings mechanism, as it reads similarly to the chain for pCN in \eqref{eq:2:pcn:X1}, but with a proposal $Y$ similar to mpCN as in \eqref{eq:2:mpcn:chain}. The strategy to show spectral gap results for this kernel will then be a mixture of the approach for the single proposal pCN developed in \cite{hairer2014spectral} and the approach we will develop for multiproposal pCN. This will also be true for the limit kernel $P_\infty$ that we now introduce for Multiple Try pCN. 

From the convergence result in \cite{glatt2025+} regarding a general form of a Multiple Try Metropolis algorithm, it follows that the $p \to \infty$ limit of the MTpCN kernel \eqref{eq:2:mtpcn:kernel} is given by 
\begin{equation}\label{eq:2:mtminfty:kernel}
    P_\infty(x_0, dy) = \alpha(x_0, y)\baQ(x_0, dy) + \delta_{x_0}(dy)\int \left(1 - \alpha(x_0, u)\right) \baQ(x_0, du),
\end{equation}
where $\baQ$ is defined similarly as in \eqref{eq:2:baQ}, namely 
\begin{equation*}
    \baQ(x_0, dy) =  \dfrac{\exp(-\Pot(y)) Q(x_0, d y)}{\int \exp(-\Pot(u)) Q(x_0, d u) },
\end{equation*}
and the acceptance probability is
\begin{equation}\label{eq:2:mtminfty:alpha}
    \alpha(x_0, y) = 1 \wedge \frac{\int \exp(-\Pot(z))Q(x_0, dz)}{\int \exp(-\Pot(u)) Q(y, du)}.
\end{equation}

\subsection{Wasserstein contraction}\label{sec:prelim}

Let $d:\qsp\times  \qsp\to \R_+$ be a distance-like function on a Polish space $\qsp$, namely a symmetric and lower semi-continuous mapping such that $d(x,y) = 0$ if and only if $x = y$. Then $d$ induces a Wasserstein semimetric $W_d: \cM_1(\qsp)\times \cM_1(\qsp) \to \R_+$ on the space of probability measures $\cM_1(\qsp)$, which is defined for any two measures $\nu_1, \nu_2 \in  \cM_1(\qsp)$ as
\begin{equation}\label{eq:intro:wass}
    W_{d}(\nu_1, \nu_2) = \inf_{\pi\in \mathfrak{C}(\nu_1, \nu_2)} \int_{\qsp^2} d(x,y)\, \pi(dx, dy),
\end{equation}
where $\mathfrak{C}(\nu_1, \nu_2)$ denotes the set of all couplings of $\nu_1$ and $\nu_2$, namely all probability measures on the product space $\qsp \times \qsp$ with marginals $\nu_1$ and $\nu_2.$ Note that when $d$ additionally satisfies the triangular inequality property, so that it is a bona fide metric, then $W_d$ coincides with the 1-Wasserstein distance associated to $d$. 

We also consider the space of Lipschitz continuous functions with respect to the distance-like function $d$, denoted $\Lip(d)$, and its associated seminorm, defined for any $f \in \Lip(d)$ as
\begin{equation}\label{eq:lipschitz:norm}
	\|f \|_d := \sup_{x\neq y}\frac{|f(x) - f(y)|}{d(x,y)}.
\end{equation}

Recall that the action of a Markov kernel $P: \qsp \times \cB(\qsp) \to [0,1]$ on a measurable function $f: \qsp \to \R$ is defined as
\begin{align*}
	 P f(x) = \int f(y) \, P(x, dy), \quad x \in \qsp,
\end{align*}
whereas the action of $P$ on a measure $\nu \in \cM_1(\qsp)$ is given by
\begin{align*}
    \nu P(A) = \int P(x, A)\, \nu(dx), \quad A \in \cB(\qsp).
\end{align*}

In order to state the weak Harris theorem, we first recall the following definitions from \cite{HMS11}, starting from a generalisation of small sets: 
\begin{Definition}[$d$-small set]\label{def:smallness}
    Let $P$ be a Markov kernel on $\qsp$ and $d: \qsp \times \qsp \to [0,1]$ a distance-like function. Then, a set $S\subset \qsp$ is a $d$-small set if there exists $0 < s < 1$ such that 
    \begin{equation}
        W_d(P(x, \cdot), P(y, \cdot))\leq s \quad \mbox{for all }\,\, x, y \in S.
    \end{equation}
\end{Definition}
Typical candidates for small sets are sub-level sets of the so-called Lyapunov functions:
\begin{Definition}[Lyapunov function]\label{def:lyapunov}
    A measurable function $V: \qsp \to [0, \infty)$ over a Polish space $\qsp$ is a Lyapunov function for a Markov kernel $P$ on $\qsp$ if there exist constants $0\leq l_V< 1$ and $K_V>0$ such that   
    \begin{equation}\label{eq:2:lyap}
        P^nV(x) = \int V(y) P^n(x, dy) \leq l_V^n V(x) + K_V 
    \end{equation}
    for all $x\in \qsp$ and $n\in \N$.
\end{Definition}
Note that, for the Lyapunov property to hold, it is enough to show \eqref{eq:2:lyap} is satisfied with $n = 1$, since upon iterating it is easy to deduce that \eqref{eq:2:lyap} is valid for any $n \in \N$, with a possibly different constant $K_V$. More precisely,
\begin{equation*}
    P^nV(x) \leq l_V^n V(x) + \frac{K_V}{1 - l_V}.
\end{equation*}
In the classic theory of Markov Chains in finite dimension, showing the existence of a Lyapunov function with a sub-level set that is small with respect to the total variation distance is enough to apply the standard Harris theorem (see e.g. \cite{meyn_tweedie, YetAnother}). This ensures that there exists $\lambda <1$ and $C_1= C_1(x)>0$ such that for all $x\in \qsp$
\begin{equation*}
    d_{TV}(P^n(x, \cdot), \mu) \leq C_1(x) \lambda^n \quad \text{for all }n \in \N,
\end{equation*}
where $d_{TV}$ denotes the total variation distance. This property is typically referred to as \textit{geometric ergodicity} for the kernel $P$. 
However, in the infinite dimensional context, the smallness in total variation distance can easily fail: $P(x, \cdot), P(y, \cdot)$ with different initial points, $x\neq y$, are typically mutually singular, leading to $d_{TV}(P(x, \cdot), P(y, \cdot))=1$. The more general definition of smallness given in \cref{def:smallness} allows for alternative notions of distance that are more suitable in the infinite dimensional setting. 

Moreover, a third condition, which ensures contraction directly at small scales, is required: 
\begin{Definition}[$d$-contraction]
    Let $P$ be a Markov kernel on $\qsp$ and $d: \qsp \times \qsp \to [0,1]$ a distance-like function. Then $P$ is $d$\textit{-contracting} if there exists $0< c<1$ such that 
    \begin{equation}
         W_d(P(x, \cdot), P(y, \cdot))\leq c \,d(x,y)
    \end{equation}
    for all $x,y\in \qsp$ such that $d(x,y)<1$. 
\end{Definition}

We are ready to state the following extended version of the weak Harris theorem:
\begin{Theorem}
    \label{thm:our_harris}
        Let $P$ be a Markov kernel over a Polish space $\qsp$ with invariant measure $\mu$ and with continuous Lyapunov function $V$. Define the sub-level set $S = \lbrace x\in \qsp \, : \, V(x) \leq 4K_V \rbrace$. Suppose there exists a distance-like function $d: \qsp \times \qsp \to [0,1]$ and $n_0>0$ such that for all $n\leq n_0$
        \begin{enumerate}
            \item there exists $c(n) >0$ such that for all $x,y\in \qsp$ with $d(x,y)<1$ 
            \[ W_d(P^n(x, \cdot), P^n(y, \cdot)) \leq c(n) d(x,y);\]
            \item there exists $s(n) >0$ such that for all $x,y\in S$ 
            \[  W_d(P^n(x, \cdot), P^n(y, \cdot)) \leq s(n); \]
        \end{enumerate} 
        and such that $c(n_0)<1$ and $s(n_0)<1$, namely $P^{n_0}$ is $d$-contracting and has $S$ as $d$-small set.
        
         Then $P$ has at most one invariant measure. Moreover, there exists $n_1 > n_0$ and $\lambda<1$ such that for all $n \geq n_1$
        \begin{equation}
        \label{eq:2:harris:contraction}
           W_{\td}(\nu_1 P^{n}, \nu_2 P^{n})\leq \lambda^n W_{\td} (\nu_1,\nu_2),\quad \nu_1, \nu_2\in \cM_1(\qsp),
        \end{equation}
         where $\td(x,y)^2 = d(x,y)(1 + V(x) + V(y))$, so that, for each $x\in \qsp$ there exists $C= C(x)$ such that 
        \begin{equation}\label{eq:2:harris:exp}
          W_{\td}(P^{n}(x, \cdot), \mu )\leq C(x) \lambda^n, \quad \text{for all } n \geq n_1.
        \end{equation}
       In addition, for observables $f\in \Lip(\td)$, it holds
        \begin{align}\label{eq:2:harris:exp1}
            \| P^nf - \mu(f) \|_{\td} \leq \lambda^n\|f - \mu(f)\|_{\td} \quad \text{for all } n\geq n_1,
        \end{align}
        where $\|\cdot \|_{\td}$ is the Lipschitz seminorm associated to $\td$ as in \eqref{eq:lipschitz:norm} and $\mu(f):= \int f \,d\mu$.
    \end{Theorem}

The original version of the weak Harris theorem as stated in \cite{HMS11} requires the kernel to have a Lyapunov function, and $d$-contractivity and $d$-smallness at a time $n_0>0$ to obtain Wasserstein contraction for $P^n$ for a \textit{fixed} $n\in \N$. In \cref{app:harris} we will show how conditions 1-2 for all $n\leq n_0$ in \cref{thm:our_harris} ensure the contraction for all large enough times.

\section{Main Results}\label{sec:results}

We endow the state space $\qsp$ with the distance 
\begin{equation}\label{eq:d_eps}
    d_\varepsilon(x, \tx) = 1\wedge \frac{\|x- \tx\|}{\varepsilon}, \quad x,\tx\in \qsp,
\end{equation}
dependent on an arbitrary parameter $\varepsilon\in (0,1)$ and define 
 \begin{equation}\label{eq:td}
     \td_\varepsilon (x, \tx) := \sqrt{d_\varepsilon(x, \tx)(1 + V(x)+ V(\tx))}
 \end{equation} 
 where $V$ will belong to either one of the following three class of functions:
 \begin{align}
     V(x) &= \| x\|^k, \hspace{-2cm}&&\hspace{-2cm} k\in \N, \notag\\
     V(x)&= \exp(v\|x\|),\hspace{-2cm}&&\hspace{-2cm} v>0 \label{eq:choiceV}\\
      V(x) &= \exp(v\|x\|^2), \hspace{-2cm}&&\hspace{-2cm}\text{with } v \text{ sufficiently small}.\notag
     \end{align}

The conditions to establish Wasserstein contraction in $\td_\varepsilon$ for both the mpCN and MTpCN kernels are detailed in the following theorem, proved in detail for the two algorithms respectively in \cref{sec:proof_mpc_finite} and \cref{sec:proof_mtm_finite}.

\begin{Theorem}\label{thm:3:mpcn1}
     Let $P_p$ be either the mpCN kernel \eqref{eq:2:mpcn:kernel} or the MTpCN kernel \eqref{eq:2:mtpcn:kernel} for a fixed  $\rho\in [0,1)$ and $p\geq 1$. Assume that the potential function $\Pot$ is globally bounded and globally Lipschitz with respect to the $\qsp$ norm $\|\cdot \|$ with constant $L_{\Pot}$. Then there exists $\varepsilon_* \in (0,1)$ so that, for any $\varepsilon\in (0,\varepsilon_*)$, there are $\lambda \in (0,1)$ and $n_1 \in \N$ such that, for every $\nu_1, \nu_2\in \cM_1(\qsp)$
    \begin{equation}
        \sup_{p \geq 1} W_{\td_\varepsilon}\left(\nu_1 P_p^n, \nu_2 P^n_p\right) \leq \lambda^{n} W_{\td_\varepsilon}\left( \nu_1, \nu_2\right) \quad \text{for all } n\geq n_1,
    \end{equation}
    with $\td_\varepsilon$ as in \eqref{eq:td} for any of the $V$ in \eqref{eq:choiceV}. Moreover, the parameters $\varepsilon_*, n_1$ and $\lambda$ are dependent on $\rho, \|\Pot\|_\infty, L_{\Pot}$, and, notably, bounded independently of the number of proposals $p\in \N$.
\end{Theorem}

As discussed previously, the assumptions on the function $\Pot$ can be relaxed to accommodate for unbounded functions. The price to pay will be the loss of uniform boundedness in the number of proposals, leading to a result increasingly worse with the number of the proposals. We present and prove this alternative result in \cref{app:unbound} for the mpCN algorithm.

For unbounded potentials, the Wasserstein contraction obtained for the multiproposal kernels does not, under the current approach, pass to the limit as the number of proposals increases. However, this limitation can be circumvented by analyzing the limiting kernel $P_\infty$ of mpCN directly, which allows one to establish contraction under weaker assumptions, requiring only Lipschitz continuity of the potential. This is formalized by the following theorem later proved in \cref{sec:proof_mpc_infinite}. 

\begin{Theorem}\label{thm:3:inftypcn}
     Let $P_\infty$ be the $\infty$-pCN kernel as in \eqref{eq:2:inftypCN} for a fixed  $\rho_1, \rho_2\in [0,1)$. Assume that the function $\Pot$ is globally Lipschitz with respect to the norm $\|\cdot \|$ with constant $L_{\Pot}$. Then there exists $\varepsilon^* \in (0,1)$ so that for any $\varepsilon\in (0,\varepsilon^*)$ there are $ \lambda \in (0,1)$ and $n_1 \in \N$ such that, for every $\nu_1, \nu_2\in \cM_1(\qsp)$
    \begin{equation}
        W_{\td_\varepsilon}\left(\nu_1 P_\infty^n, \nu_2 P^n_\infty\right) \leq \lambda^{n} W_{\td_\varepsilon}\left( \nu_1, \nu_2\right) \quad \text{for all } n\geq n_1,
    \end{equation}
    with $\td_\varepsilon$ as in \eqref{eq:td} for any of the functions $V$ in \eqref{eq:choiceV}. 
\end{Theorem}

This mechanism appears to be specific to the multiproposal setting: for the Multiple Try algorithm, an analogous extension to the infinite–proposal kernel with unbounded potentials is not captured by the present analysis, although we do not exclude that such a result may hold under a different approach. 

\begin{Theorem}\label{thm:3:inftyMTpcn}
     Let $P_\infty$ be the $\infty$-MTpcN kernel as in \eqref{eq:2:mtminfty:kernel} for a fixed  $\rho\in [0,1)$. Assume that the function $\Pot$ is globally bounded and globally Lipschitz with respect to the norm $\|\cdot \|$ with constant $L_{\Pot}$. Then there exists $\varepsilon^* \in (0,1)$ so that for any $\varepsilon\in (0,\varepsilon^*)$ there are $\lambda \in (0,1)$ and $n_1 \in \N$ such that, for every $\nu_1, \nu_2\in \cM_1(\qsp)$
    \begin{equation}
        W_{\td_\varepsilon}\left(\nu_1 P_\infty^n, \nu_2 P^n_\infty\right) \leq \lambda^{n} W_{\td_\varepsilon}\left( \nu_1, \nu_2\right) \quad \text{for all } n\geq n_1,
    \end{equation}
    with $\td_\varepsilon$ as in \eqref{eq:td} for any of the functions $V$ in \eqref{eq:choiceV}. 
\end{Theorem}
Refer to \cref{sec:proof_mtm_infinite} for a full proof of this result.

As previously mentioned, we invoke \cref{thm:our_harris} to show all of the results above. In particular, in order to verify the $d$-smallness and $d$-contraction properties, one typically constructs a suitable coupling of the kernels. This is usually the most delicate part of the proof for any Markov kernel and it is explicitly constructed in \cref{prop:contr:fin:p:1} and \cref{prop:mtm:contr:fin:p:1} for mpCN and MTpCN, respectively.

For the $\infty$-pCN algorithm, the accept--reject mechanism is absent and a different strategy is devised. Specifically, we developed a coupling of the internal kernel $\bar{Q}_1$ \eqref{eq:2:baQ} which reintroduces an accept--reject step. More importantly, this construction turns out to be applicable to any pair of measures linked via an invertible mapping in the following sense:

\begin{Proposition}
    \label{lemma:AR_coupling}
    Let $\mu_1,\, \mu_2$ be two probability measures on a Polish space $\qsp$ and $T:\qsp \to \qsp$ an invertible map such that $T^*\mu_1$ is equivalent to $\mu_2$. Define the function $\beta: \qsp\times \qsp \to [0,1]$ as
    \begin{equation}
        \label{eq:beta}
        \beta(u,v) := 1 \wedge \frac{dT^* \mu_1}{d\mu_2}(u)\frac{d\mu_2}{dT^*\mu_1}(v).
    \end{equation}
    Then the following probability measure on the product space $\qsp\times \qsp$ 
    \begin{equation}
        \label{eq:general-coupling}
        \pi_\beta(du, dv) = \int [ \beta(w, T(u))\,\delta_{T(u)}(dv) + (1 - \beta(w, T(u)))\,\delta_{w}(dv) ] \, \mu_1(du)\mu_2(dw)
    \end{equation}
    is a coupling of $\mu_1$ and $\mu_2$. 
\end{Proposition}

\begin{proof}
    We must verify that the marginals of $\pi_\beta$ coincide with $\mu_1$ and $\mu_2$. Denoting by $\Pi_1,\Pi_2: \qsp^2\to \qsp$ the projection mappings onto the first and second components, respectively, this means verifying that
    \[
        \Pi_1\pi_\beta(du) = \mu_1(du)\quad \mbox{and} \quad \Pi_2 \pi_\beta(dv) = \mu_2(dv).
    \]
     Fix any bounded and measurable function $\varphi: \qsp \to \R$. Regarding the first identity, we have
    \begin{align*}
        &\int \varphi(u) \,\Pi_1\pi_\beta(du)  
        = \iint \varphi(u) \, \pi_\beta(du, dv) \\
        &\qquad\qquad= \iiint \varphi (u) \beta(w, T(u))\,\delta_{T(u)}(dv) \mu_1(du)\mu_2(dw)
        + \iiint \varphi (u) (1 - \beta(w, T(u)))\,\delta_{w}(dv)  \mu_1(du)\mu_2(dw)\\
        &\qquad\qquad= \int \varphi (u) \int [\beta(w, T(u)) + (1 - \beta(w, T(u))) ] \mu_2(dw) \, \mu_1(du) = \int \varphi(u)\mu_1(du).
    \end{align*}
    For the second marginal, we have
    \begin{align*}
         &\int \varphi(v) \,\Pi_2\pi_\beta(dv) 
         = \iint \varphi(v) \, \pi_\beta (du, dv) \\
         &\qquad= \iiint \varphi (v) \beta(w, T(u))\, \delta_{T(u)}(dv) \mu_1(du)\mu_2(dw)
         + \iiint \varphi(v) (1 - \beta(w, T(u)))\delta_{w}(dv)\, \mu_1(du)\mu_2(dw)\\
        &\qquad= \iint \varphi (T(u)) \beta(w, T(u))  \mu_1(du)\mu_2(dw) + \iint \varphi(w) (1 - \beta(w, T(u)))\,\mu_1(du)\mu_2(dw).
    \end{align*}
    Then we are left to show that 
    \begin{equation*}
        \iint \varphi (T(u)) \beta(w, T(u))  \mu_1(du)\mu_2(dw) = \iint  \varphi(w) \beta(w, T(u))\, \mu_1(du)\mu_2(dw)
    \end{equation*}
    or, equivalently, 
     \begin{equation*}
        \iint \varphi (u) \beta(w, u) \, T^*\mu_1(du)\mu_2(dw) = \iint  \varphi(w) \beta(w, u)\, T^*\mu_1(du) \mu_2(dw).
    \end{equation*}
    By the definition of $\beta$ in \eqref{eq:beta}, it follows that
    \begin{align*}
         \iint \varphi (u) \beta(w, u) &\, T^*\mu_1(du)\mu_2(dw) = \iint \varphi (u) \left(1 \wedge \frac{dT^* \mu_1}{d\mu_2}(w)\frac{d\mu_2}{dT^*\mu_1}(u) \right)\, T^*\mu_1(du)\mu_2(dw)\\
         &=  \iint \varphi (u) \left(1 \wedge \frac{dT^* \mu_1}{d\mu_2}(w)\frac{d\mu_2}{dT^*\mu_1}(u) \right)\, \frac{d(T^*\mu_1)}{d\mu_2}(u)\mu_2(du) \frac{d\mu_2}{d(T^*\mu_1)}(w)(T^*\mu_1)(dw)\\
         &= \iint \varphi (u) \left(\frac{d(T^*\mu_1)}{d\mu_2}(u) \frac{d\mu_2}{d(T^*\mu_1)}(w) \wedge  1 \right)\, \mu_2(du) (T^*\mu_1)(dw)\\
          &= \iint \varphi (u) \beta(u,w ) \, \mu_2(du) (T^*\mu_1)(dw),
    \end{align*}
    as desired.
\end{proof}

\begin{Remark}\label{rem:couplingRV}
    The explicit expression for the coupling in \eqref{eq:general-coupling} can also be written in terms of random variables as follows. Given $X\sim \mu_1$ and an independent $Z\sim \mu_2$, define
    \begin{equation}\label{eq:couplingRV}
        Y = \mathbbm{1}_{U\leq \beta(Z, T(X))} T(X) + \mathbbm{1}_{U > \beta(Z, T(X))} Z,
    \end{equation}
    where $U\sim \cU(0,1)$. Then $(X,Y)\sim \pi_\beta$.
\end{Remark}

Looking back at \cref{thm:intro:spectralgap}, the Wasserstein contraction provided by \cref{thm:3:mpcn1} ensures the convergence result in item 1. As a consequence, the SLLN and CLT for $\Lip(\td)$ observables and any deterministic initial condition as described in item 3 of \cref{thm:intro:spectralgap} also holds (see \cref{app:harris} for the complete statements and e.g.~\cite[Appendix~A]{glatt2021mixing} for the proof details). 

We observe that $\Lip(\td)$, the space of Lipschitz continuous functions with respect to $\td$, is contained in $L^2_\mu$. Indeed, if $f\in \Lip(\td)$ then 
\begin{equation*}
    |f(x)|\leq |f(0)| +  \|f\|_{\td} \, \td(x, 0) \leq |f(0)| + \|f\|_{\td} \sqrt{1 + V(x) + V(0)} \quad \text{for all }x \in \qsp,
\end{equation*}
and, consequently,
\begin{equation*}
    \|f\|_2^2\leq 2|f(0)|^2 + 2 \|f\|_{\td}^2 \left( 1 + V(0) + \int V(x) \mu(dx) \right).
\end{equation*}
The integrability of Lyapunov functions with respect to the invariant measure $\mu$ (see e.g. \cite{butkovsky2014}) then ensures that $f\in L^2_\mu$ as desired. 

The following result, derived in \cite[Section~2.2.2]{hairer2014spectral}, connects the Wasserstein contraction obtained in \cref{thm:our_harris} with the $L^2_\mu$-spectral gap. 
\begin{Theorem}\label{thm:L2gap}
    Let $P$ be a Markov operator with invariant probability measure $\mu$. Assume $P$ is $\mu$-reversible, namely
    \[ P(u, d \tu)\mu (d u) = P(\tu, du) \mu(d\tu).
    \]
    Suppose \cref{thm:our_harris} is satisfied with $\td$ such that $\Lip(\td)\cap L^\infty_\mu$ is a dense subset of $L^2_\mu$. Then 
    \begin{equation}\label{eq:specgapL2}
        \|P^n f - \mu(f)\|_2 \leq \lambda^n\|f - \mu(f)\|_2 \quad \text{ for all } n\in \N,
    \end{equation}
    for all $f\in L^2_\mu$.
\end{Theorem} 

Note that this result is stronger than \eqref{eq:2:harris:exp1} in \cref{thm:our_harris} not only for the space of functions to which it applies, but also because it is valid for \textit{all} $n\in \N$, rather than only for a large enough number of iterations. The elegant argument behind this implication goes back to \cite{hairer2014spectral}, where it is attributed to a private communication and presented in condensed form. Because several of its measure-theoretic steps are left implicit, and because this passage from Wasserstein contraction to an $L^2_\mu$ spectral gap carries important consequences, we provide a complete and self-contained proof in \cref{app:L2}. 

Ultimately, we want to ensure the convergence of the Monte Carlo error \eqref{eq:intro:error} for a desirable class of observables. From the $L^2_\mu$ spectral gap, we can derive the strong Law of Large Numbers and central limit theorem when the chain starts in stationarity, namely accounting for appropriate burn-in. 
We provide a full statement and proof of the SLLN in \cref{thm:SLLN:L2} for completeness, and in \cref{thm:CLT:L2} we recall the precise statement of the CLT as derived in \cite{Kipnis86,latuszynski2013clts} to ease reference.

Furthermore, more practically useful information on tail behavior can be obtained from non-asymptotic results, such as concentration inequalities. A recent work \cite{Hoeffding} provides conditions under which \textit{Hoeffding's inequality} holds for Markov chains on general, potentially infinite-dimensional state spaces such as the one considered here.

\begin{Theorem}[Hoeffding's inequality  \cite{Hoeffding}]
    Let $P$ be a Markov operator for which \cref{thm:L2gap} holds.
    Then, for any $m\in \N$ and for any $t\in \R$, uniformly for all bounded functions $f_k: \qsp \to [a_k, b_k]\subset \R$, $k = 1, \ldots, m$
    \begin{equation*}
        \E \exp\left[ t \sum_{k = 1}^{m} \left( f_k(X_{k}) - \mu(f_k)\right)\right] \leq \exp\left( \frac{1- \lambda}{1 + \lambda}\frac{t^2}{2}\sum_{k = 1}^{m} \frac{(b_k - a_k)^2}{4}\right) .
    \end{equation*}
     It follows that for $\varepsilon>0$
    \begin{equation*}
        \bbP \left(\sum_{k = 1}^{m} \left( f_k(X_{k}) -  \mu(f_k)\right) > \varepsilon\right) \leq \exp\left(- \frac{1+ \lambda}{1 - \lambda}\frac{2\varepsilon^2}{\sum_{k = 1}^{m} (b_k - a_k)^2} \right).
    \end{equation*}
\end{Theorem}
A generalisation of the Hoeffing's inequality derived from the Wasserstein contraction rather than $L^2_\mu$ spectral gap is not present in the literature to the best of our knowledge and it will be subject of further studies.

\section{Proofs of the main results}\label{sec:proofs}

\subsection{Proofs for multiproposal pCN}\label{subsec:proofs_mpcN}
We start by treating the multiproposal pCN algorithm (\cref{alg:mpcn}) with kernel $P_p$ as defined in \eqref{eq:2:mpcn:kernel} for a finite number of proposals $p$, to then give the proof of \cref{thm:3:inftypcn} for the infinite proposal limit.

\subsubsection{Finite number of proposals}\label{sec:proof_mpc_finite}
First we show that the functions in \eqref{eq:choiceV} are indeed Lyapunov functions for $P_p$.
\begin{Proposition}\label{prop:mpcn:lyap}
    Fix $p \geq 1$, and assume the potential function $\Pot: \qsp \to \R$ is bounded. Then the functions $V(x) = \|x\|^n$, $n \in \N$, $V(x) = \exp(v\|x\|)$, $v>0$, and $V(x) = \exp(v\|x\|^2)$ with $v$ small enough, are Lyapunov functions as in \cref{def:lyapunov} for the mpCN Markov kernel $P_p$ with constants $l_V(p)$ and $K_V(p)$ that are uniformly bounded in the number of proposals $p$.
\end{Proposition}
\begin{proof}
  By definition of the multiproposal kernel, we can write, using the expression \eqref{eq:2:mpcn:chain} for the first step of the associated chain,
    \begin{align}\label{eq:PVx0:mpCN}
        PV(x_0) = \E V(X^{(1)}) = \E \sum_{j = 1}^p V(X_j)\mathbbm{1}_{U \in I_j }+ V(x_0) \bbP(U \in I_0).
    \end{align}
    Let us first assume $0 < \rho < 1$. By definition of the proposals $X_j$ in \eqref{eq:2:proposals} and Young's inequality, we can derive the bounds 
    \begin{align*}
        \|X_j\|^n 
        &\leq (1 + \delta)\rho^{2n} \|x_0\|^n + C_\delta (1 - \rho^2)^{n/2} \| \rho \xi_0 + \xi_j\|^n, && \hspace{-1cm }\text{for any  }\delta>0,\\
        e^{v\|X_j\|}&\leq \tfrac{1}{q}e^{q v\rho^2\|x_0\|} + \tfrac{q-1}{q} e^{\frac{v q}{q-1}\sqrt{1 - \rho^2} \|\rho \xi_0 + \xi_j\|}, && \hspace{-1cm }\text{for any }q>1,\\
        e^{v\|X_j\|^2}&\leq \tfrac{1}{q}e^{ q v  (1 + \delta) \rho^4 \|x_0\|^2} + \tfrac{q-1}{q} e^{\frac{q }{q-1} C_\delta v (1 - \rho^2) \|\rho \xi_0 + \xi_j\|^2 }, && \hspace{-1cm }\text{for any } q > 1 \text{ and } \delta > 0,
        \end{align*}
    where $C_\delta$ is a positive constant. It follows that, for the three different options of $V$,
    \begin{equation}\label{eq:mpcn:lyap:0}
        V(X_j) \leq l_1 V(x_0) + G(\xi_0, \xi_j).
    \end{equation}
    with $l_1<1$ and $G(\xi_0, \xi_j)$ defined as
    \begin{align}\label{eq:mpcn:lyap:l1}
        l_1 &= (1 + \delta)\rho^{2n},\quad \text{with } 0< \delta <  \rho^{-2n} - 1, &&  G(\xi_0, \xi_j) = 
             C_\delta (1 - \rho^2)^{n/2} \| \rho \xi_0 + \xi_j\|^n ;
            \\
         l_1 &=  q^{-1},\quad \text{with } 1< q < \rho^{-2}, &&  G(\xi_0, \xi_j) = \tfrac{q-1}{q} e^{\frac{v q}{q-1}\sqrt{1 - \rho^2} \|\rho \xi_0 + \xi_j\|} ; \label{eq:mpcn:lyap:l1:2}\\
         l_1 &=  q^{-1},\quad \text{with } 1< q < ( (1 + \delta)\rho^4)^{-1} \text{ and } \delta < \rho^{-4} - 1, &&  G(\xi_0, \xi_j) = \tfrac{q-1}{q} e^{\frac{q }{q-1} C_\delta v (1 - \rho^2) \|\rho \xi_0 + \xi_j\|^2 };\label{eq:mpcn:lyap:l1:3}
    \end{align}
    for the three candidate Lyapunov functions in \eqref{eq:choiceV}, respectively. 
    Then 
    \begin{align*}
        PV(x_0) &\leq l_1 V(x_0) \sum_{j = 1}^p \bbP(U \in I_j) + \sum_{j = 1}^p \E  [G(\xi_0, \xi_j)\mathbbm{1}_{U \in I_j } ] + V(x_0) \bbP(U \in I_0)\\
        &= V(x_0) \left[ 1 - (1 - l_1) \sum_{j = 1}^p \E \alpha_j (x_0, X_1, \ldots, X_p)\right] + \sum_{j = 1}^p \E [ G(\xi_0, \xi_j)\mathbbm{1}_{U \in I_j }].
    \end{align*}
    As $\Pot$ is assumed globally bounded, then, for any $j= 1, \ldots, p$ and $(x_0, x_1, \ldots, x_p) \in \qsp^{p+1}$,
    \begin{equation}\label{eq:bound_alphaj}
         \frac{e^{ -2\|\Pot\|_\infty}}{p+1} \leq \alpha_j(x_0, x_1, \ldots, x_p) = \frac{e^{- \Pot(x_j)}}{\sum_{k = 0}^p e^{- \Pot(x_k)}} \leq \frac{e^{ 2\|\Pot\|_\infty}}{p+1},
    \end{equation} 
    giving
    \begin{equation}\label{ineq:PV:x0:mpCN}
        PV(x_0) \leq V(x_0) \left[ 1 - (1 - l_1)\frac{p e^{-2\|\Pot\|_\infty}}{p+1}\right]+ \sum_{j = 1}^p \E [ G(\xi_0, \xi_j)\mathbbm{1}_{U \in I_j }].
    \end{equation}
    We are left to ensure that the second term on the right hand side of \eqref{ineq:PV:x0:mpCN} is well defined and bounded uniformly in $x_0\in \qsp$. 
    For $V(x) = \|x\|^n$, we have 
\begin{equation}\label{eq:G:V:x:pn}
     \sum_{j = 1}^p \E [ G(\xi_0, \xi_j)\mathbbm{1}_{U \in I_j } ] =  C_\delta (1 - \rho^2)^{n/2}\sum_{j = 1}^p \E [ \| \rho \xi_0 + \xi_j\|^n \mathbbm{1}_{U \in I_j }].
\end{equation}
Denoting by $\cF(\xi_0, \ldots, \xi_p)$ the $\sigma$-algebra generated by $\xi_0, \ldots, \xi_p$, observe that we can write the expectation terms in the right-hand side of \eqref{eq:G:V:x:pn} as
    \begin{align}\label{eq:mpcn:lyap:111}
        \E [ \| \rho \xi_0 + \xi_j\|^n \mathbbm{1}_{U \in I_j } ]
        &= \E \left[ \E[ \| \rho \xi_0 + \xi_j\|^n \mathbbm{1}_{U \in I_j } |\cF(\xi_0, \ldots, \xi_p)  ] \right]
        = \E \left[  \| \rho \xi_0 + \xi_j\|^n   \E[ \mathbbm{1}_{U \in I_j } |\cF(\xi_0, \ldots, \xi_p)  ]  \right]
        \notag\\
        &= \E [  \| \rho \xi_0 + \xi_j\|^n  \alpha_j(x_0, \Fbar(x_0, \xi_0, \xi_1, \ldots, \xi_p) ) ]
        \notag\\
        &= \int_{\qsp^{p+1}}  \|\rho w_0 + w_j\|^n \alpha_j(x_0, \Fbar(x_0, w_0, w_1, \ldots, w_p) ) \mu_0(dw_0)\ldots \mu_0(dw_p),
    \end{align}
    with $\Fbar$ as defined in \eqref{eq:mpcn:Fbar}. Invoking \eqref{eq:bound_alphaj}, it thus follows that, for all $j=1, \ldots,p$,
    \begin{align}\label{ineq:E:xi0:xij:Ij}
    	\E [\| \rho \xi_0 + \xi_j\|^n \mathbbm{1}_{U \in I_j } ]
    	\leq \frac{e^{ 2\|\Pot\|_\infty}}{(p+1)}\int_{\qsp^2} \|\rho w_0 + w_1\|^n \mu_0(dw_0)\mu_0(dw_1) = \frac{e^{ 2\|\Pot\|_\infty}}{(p+1)}\int_\qsp \|y\|^n \mu_1(dy),
    \end{align}
    where $\mu_1 := \cN(0, (1 + \rho^2)\cC)$. Hence, denoting the $n$-th moment of $\mu_1$ as $M_n = \int_\qsp \|y\|^n \mu_1(dy)$, we deduce from \eqref{eq:G:V:x:pn} and \eqref{ineq:E:xi0:xij:Ij} that
     \begin{equation}\label{ineq:G:V1}
    	\sum_{j = 1}^p \E [ G(\xi_0, \xi_j)\mathbbm{1}_{U \in I_j } ] 
    	\leq  C_\delta (1 - \rho^2)^{n/2}\frac{p e^{ 2\|\Pot\|_\infty}}{(p+1)}M_n.
    \end{equation}    

    Next, for $V(x) =\exp(v\|x\|)$, we have according to \eqref{eq:mpcn:lyap:l1:2} and by employing a similar argument as in the previous case that
    \begin{align}\label{ineq:G:V2}
     \sum_{j = 1}^p \E [ G(\xi_0, \xi_j)\mathbbm{1}_{U \in I_j } ]
     &=  \frac{q-1}{q} \sum_{j = 1}^p \E [ e^{\frac{v q}{q-1}\sqrt{1 - \rho^2} \|\rho \xi_0 + \xi_j\|} \mathbbm{1}_{U \in I_j } ] \\
     &\leq  \frac{q-1}{q}\frac{p e^{ 2\|\Pot\|_\infty}}{(p+1)} \int e^{\frac{v q}{q-1}\sqrt{1 - \rho^2} \|\rho w_0 + w_1\|} \mu_0(dw_0)\mu_0(dw_1) \\
     &= \frac{q-1}{q}\frac{p e^{ 2\|\Pot\|_\infty}}{(p+1)} \int e^{\frac{v q}{q-1}\sqrt{1 - \rho^2} \|y\|} \mu_1(d y)
     =:\frac{q-1}{q}\frac{p e^{ 2\|\Pot\|_\infty}}{(p+1)} M_{exp}.
    \end{align}

    Lastly, for $V(x) =\exp(v\|x\|^2)$, we have from \eqref{eq:mpcn:lyap:l1:3} that
\begin{align*}
     \sum_{j = 1}^p \E [ G(\xi_0, \xi_j)\mathbbm{1}_{U \in I_j } ]  
     &=  \frac{q-1}{q} \sum_{j = 1}^p \E [e^{\frac{q }{q-1} C_\delta v (1 - \rho^2) \|\rho \xi_0 + \xi_j\|^2 }  \mathbbm{1}_{U \in I_j } ] 
     \\
     &\leq \frac{q-1}{q} \frac{p e^{ 2\|\Pot\|_\infty}}{(p+1)} \int e^{\frac{q}{q-1} C_\delta v (1 - \rho^2) \|y\|^2} \mu_1(dy) =: \frac{q-1}{q} \frac{p e^{ 2\|\Pot\|_\infty}}{(p+1)} M_{exp}^{(2)}.
\end{align*}
Here we note that, as a consequence of Fernique's theorem, $M_{exp}^{(2)}$ is guaranteed to be finite as long as $v$ is chosen suitably small. Moreover, this further implies that the algebraic and exponential moments $M_n$ and $M_{exp}$ in \eqref{ineq:G:V1} and \eqref{ineq:G:V2} are finite for any $n \in \N$ and $v >  0$, respectively.

    In summary, we have then showed that for all $x_0\in \qsp$
    \begin{align}\label{eq:mpcn:lyap:p}
         PV(x_0) \leq  l(p) V(x_0) + K(p),
    \end{align}
    with 
    \begin{align}
        l(p) &=  1 - (1 - l_1)\frac{p e^{-2\|\Pot\|_\infty}}{p+1} \to 1 - (1 - l_1)e^{-2\|\Pot\|_\infty}&& \hspace{-1cm }\text{as }p\to \infty,
        \\
        K(p) &= C_1 \frac{p e^{ 2\|\Pot\|_\infty}}{(p+1)} \to C_1 e^{ 2\|\Pot\|_\infty} && \hspace{-1cm }\text{as }p\to \infty,
    \end{align}
    where $l_1$ is as in \eqref{eq:mpcn:lyap:l1}-\eqref{eq:mpcn:lyap:l1:3} and 
    \begin{align*}
        C_1 = \begin{cases}
             C_\delta (1 - \rho^2)^{n/2} M_n, &\quad 0< \delta< \rho^{-2n} - 1,\\
            \frac{q-1}{q} M_{exp},& \quad 1 < q< \rho^{-2}, \\
           \frac{q-1}{q} M_{exp}^{(2)}, &\quad  1< q < ( (1 + \delta)\rho^4)^{-1} \text{ and } \delta < \rho^{-4} - 1
        \end{cases}
    \end{align*}
    for the three functions in \eqref{eq:choiceV}, respectively.

    Finally, if $\rho=0$, then the proposals $X_j$ are nothing but independent draws $\xi_j$ from the reference gaussian measure $\mu_0$, with acceptance probabilities 
    \[\alpha_j(x_0, \xi_1, \ldots, \xi_p) = \frac{e^{-\Pot(\xi_j)}}{e^{-\Pot(x_0)} + \sum_{k = 1}^p e^{-\Pot(\xi_k)}} .\]
    Thus, from \eqref{eq:PVx0:mpCN} and the lower bound in \eqref{eq:bound_alphaj}, it follows that    
     \begin{equation}\label{ineq:PV:x0:mpCN:0}
        PV(x_0) \leq V(x_0) \left[ 1 - \frac{p e^{-2\|\Pot\|_\infty}}{p+1}\right]+ \sum_{j = 1}^p \E [ V(\xi_j)\mathbbm{1}_{U \in I_j }].
    \end{equation}
    Using a similar double conditional expectation argument as in \eqref{eq:mpcn:lyap:111} and the upper bound from \eqref{eq:bound_alphaj}, we obtain that 
    \[
    \sum_{j = 1}^p \E [ V(\xi_j)\mathbbm{1}_{U \in I_j }] \leq \frac{p e^{2\|\Pot\|_\infty}}{p+1}\E V(\xi_1),
    \]
    with $\E V(\xi_1)$ well defined as $\xi_1\sim \mu_0$. Then \eqref{eq:mpcn:lyap:p} holds with $l(p) = 1 - \frac{p e^{-2\|\Pot\|_\infty}}{p+1}$ and $K(p) =\frac{p e^{2\|\Pot\|_\infty}}{p+1}\E V(\xi_1)$.
\end{proof}

Next we show that the mpCN kernel \eqref{eq:2:mpcn:kernel} is $d$-contracting with respect to the distance \eqref{eq:d_eps}, $d_\varepsilon(x,y) = 1 \wedge \varepsilon^{-1}\|x - y\|$, for appropriate choices of the parameter $\varepsilon$. 

\begin{Proposition}\label{prop:contr:fin:p:1}
	Fix $p \geq 1$, $\varepsilon>0$, and assume the potential function $\Pot: \qsp \to \R$ is bounded and globally Lipschitz with Lipschitz constant $L_{\Pot}$. Then, for every $x_0, \tx_0 \in \qsp$ with $d_\eps(x_0, \tx_0) < 1$, we have
	\begin{align}\label{ineq:contr:fin:p:a} 
		\Wass_{d_\varepsilon}(\mk_p(x_0, \cdot), \mk_p(\tx_0, \cdot)) 
		\leq  d_\eps(x_0, \tx_0)  \left[  C \frac{1 + \rho^2 p}{(p+1)} \varepsilon + C_p \right],
	\end{align}
	where $C_p = 1 - (1 - \rho^2) \frac{p}{p+1} e^{-2 \| \Pot \|_\infty}$, and for some positive constant $C = C(\|\Pot\|_\infty, L_{\Pot})$. Consequently, for any fixed $\varepsilon$ satisfying $0 < \varepsilon < \frac{(1 - \rho^2) }{C(1 + \rho^2)} e^{-2 \| \Pot \|_\infty}$ and every $x_0, \tx_0 \in \qsp$ with $d_\eps(x_0, \tx_0) < 1$, it holds
	\begin{align}\label{ineq:contr:fin:p:b} 
		\sup_{p \geq 1} \Wass_{d_\varepsilon}(\mk_p(x_0, \cdot), \mk_p(\tx_0, \cdot)) 
		\leq
		\kappa d_\eps(x_0, \tx_0),
	\end{align}
	with $\kappa = \kappa (\varepsilon, \rho) \in (0,1)$.
\end{Proposition}
\begin{proof}
Fix $x_0, \tx_0 \in \qsp$ such that $d_\eps(x_0, \tx_0) < 1$. We start by defining a suitable coupling of $\mk^p(x_0, \cdot)$ and $\mk^p(\tx_0, \cdot)$. Recalling the formulation of the kernel $\mk^p$ given in \eqref{eq:P:mpCN:pfwd:mu0:U}, let $U$ be a uniform variable over $[0,1]$, i.e. $U\sim \cU([0,1])$, and let $\Xi = \left(\xi_0, \xi_1, \ldots, \xi_p\right) \sim \mu_0^{\otimes (p+1)}$. We then consider a synchronous coupling of the proposals, by setting
\begin{align}\label{eq:mpcn:proposal1}
    X_j = F(F(x_0, \xi_0), \xi_j), \quad \tX_j=  F(F(\tx_0, \xi_0), \xi_j), \quad j =1, \ldots, p,
\end{align}
where $F(x,w) =\rho x + \rr w$ as in \eqref{eq:2:defF}.

Consider also
\begin{align}\label{eq:alphahat}
	\hat{\alpha}_j = \hat{\alpha}_j(x_0, \tx_0; \Xi) 
	= \min\{\alpha_j(x_0, X_1, \ldots, X_p), \alpha_j(\tx_0, \tX_1, \ldots, \tX_p)\}, \quad j = 0, \ldots, p,
\end{align}
and
\begin{align}\label{eq:def:sj}
	s_{-1} = 0, \quad s_j = s_j(x_0, \tx_0; \Xi) = \sum_{i=0}^j \hat{\alpha}_i, \quad j = 0, \ldots, p.
\end{align}
Next, we define the non-negative quantities $\beta_j$ and $\tbeta_j$, for $j=0, \ldots, p$, as
\begin{equation}\label{def:betam:coupl}
\begin{split}
	\beta_j&= \max\{\ar_{j}(x_0,  X_1, \ldots, X_p) - \ar_{j}(\tx_0, \tX_1, \ldots, \tX_p), 0\}, \\
	\tbeta_j&=  \max\{\ar_j(\tx_0, \tX_1, \ldots, \tX_p) - \ar_j (x_0,  X_1, \ldots, X_p), 0 \}.
    \end{split}
\end{equation}
Note that $s_p + \sum_{j=0}^p \beta_j = s_p + \sum_{j=0}^{p} \tbeta_j = 1$, and define the real intervals
\begin{equation}\label{eq:Jintervals}
\begin{split}
	&J_0 = [s_p, s_p + \beta_0), \quad\tJ_0 = [s_p, s_p + \tbeta_0),\\
	&J_k = \left[s_p + \sum_{j = 0}^{k-1} \beta_j, s_p + \sum_{j=0}^k \beta_j \right), \; k = 1, \ldots, p-1, \\
    &\tJ_l = \left[s_p + \sum_{j = 0}^{l-1} \tbeta_j, s_p + \sum_{j=0}^l \tbeta_j \right), \; l = 1, \ldots, p-1,\\
    &J_p = \left[s_p + \sum_{i = 0}^{p-1} \beta_i, s_p + \sum_{i=0}^p \beta_i \right], \quad \tJ_{p} = \left[s_p + \sum_{i = 0}^{p-1} \tbeta_i, s_p + \sum_{i=0}^{p} \tbeta_i \right],
\end{split}
\end{equation}
so that
\begin{align}\label{eq:mpcn:[sp,1]}
	[s_p, 1] = \bigcup_{k=0}^p J_k = \bigcup_{l=0}^{p} \tJ_l 
	= \bigcup_{\{(k,l) \,:\, J_k \cap \tJ_l \neq \emptyset \}} ( J_k \cap \tJ_l ) .
\end{align}
We have then constructed the following partition of the unit interval
\begin{equation}\label{eq:partition[0,1]}
    [0,1] = \bigcup_{j = 0}^p [s_{j-1}, s_{j}) \cup  [s_p, 1] = \bigcup_{j = 0}^p [s_j, s_{j+1}) \cup  \bigcup_{\{(k,l) \,:\, J_k \cap \tJ_l \neq \emptyset \}} ( J_k \cap \tJ_l ).
\end{equation}
Then, we define 
\begin{equation}
    (X, \tX) = \begin{cases}
        (x_0, \tx_0) & \quad \mbox{ if }U \in [0, s_0),\\
        (X_j, \tX_j) & \quad \mbox{ if }U \in [s_{j-1}, s_{j}), \quad j = 1, \ldots p\\
        (X_{k}, \tX_{l}) & \quad \mbox{ if } U \in J_k \cap \tJ_l \mbox{ with } k,l= 0,\ldots, p \mbox{ such that }  J_k \cap \tJ_l \neq \emptyset,
    \end{cases}\label{eq:mpcn:coupling}
\end{equation}
where if $k$ or $l$ are zero, we read $X_{k}$ and $ \tX_{l}$ as $x_0$ and $ \tx_0$, respectively. See \Cref{fig:coupling} for a visualization of the construction above for the case of three proposals.

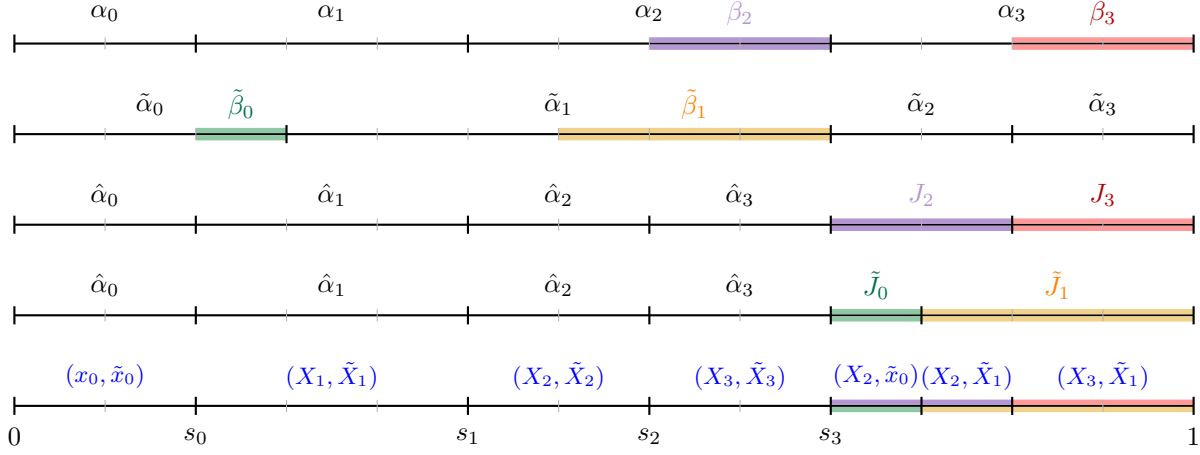
\begin{figure}
\centering
\begin{tikzpicture}[xscale=1.5, scale=0.8]

\draw[thick] (0,0) -- (13,0);

\draw[mgPurple, line width=2pt] (7,0.06) -- (9,0.06);
\draw[mgPurple, line width=2pt] (7,-0.06) -- (9,-0.06);

\draw[red!40, line width=2pt] (11,0.06) -- (13,0.06);
\draw[red!40, line width=2pt] (11,-0.06) -- (13,-0.06);

\foreach \x in {0,...,13} {
    \ifnum\x=0
        \draw[thick] (\x,0.15) -- (\x,-0.15);
    \else\ifnum\x=2
        \draw[thick] (\x,0.15) -- (\x,-0.15);
    \else\ifnum\x=5
        \draw[thick] (\x,0.15) -- (\x,-0.15);
    \else\ifnum\x=9
        \draw[thick] (\x,0.15) -- (\x,-0.15);
    \else\ifnum\x=13
        \draw[thick] (\x,0.15) -- (\x,-0.15);
    \else
        \draw[thin,gray!60] (\x,0.10) -- (\x,-0.10);
    \fi\fi\fi\fi\fi
}

\node at (1,0.5) {$\alpha_0$};
\node at (3.5,0.5) {$\alpha_1$};
\node at (7,0.5) {$\alpha_2$};
\node at (11,0.5) {$\alpha_3$};
\node[text=mgPurple] at (8,0.5) {$\beta_2$};
\node[text=red!70!black] at (12,0.5) {$\beta_3$};

\draw[thick] (0,-1.5) -- (13,-1.5);

\draw[mgGreen, line width=2pt] (2,-1.44) -- (3,-1.44);
\draw[mgGreen, line width=2pt] (2,-1.56) -- (3,-1.56);

\draw[mgGold, line width=2pt] (6,-1.44) -- (9,-1.44);
\draw[mgGold, line width=2pt] (6,-1.56) -- (9,-1.56);

\foreach \x in {0,...,13} {
    \ifnum\x=0
        \draw[thick] (\x,-1.35) -- (\x,-1.65);
    \else\ifnum\x=3
        \draw[thick] (\x,-1.35) -- (\x,-1.65);
    \else\ifnum\x=9
        \draw[thick] (\x,-1.35) -- (\x,-1.65);
    \else\ifnum\x=11
        \draw[thick] (\x,-1.35) -- (\x,-1.65);
    \else\ifnum\x=13
        \draw[thick] (\x,-1.35) -- (\x,-1.65);
    \else
        \draw[thin,gray!60] (\x,-1.40) -- (\x,-1.60);
    \fi\fi\fi\fi\fi
}

\node at (1.5,-1) {$\tilde{\alpha}_0$};
\node[text=darkgreen] at (2.5,-1) {$\tilde{\beta}_0$};
\node at (6,-1) {$\tilde{\alpha}_1$};
\node[text=orange] at (7.5,-1) {$\tilde{\beta}_1$};
\node at (10,-1) {$\tilde{\alpha}_2$};
\node at (12,-1) {$\tilde{\alpha}_3$};

\draw[thick] (0,-3) -- (13,-3);

\draw[mgPurple, line width=2pt] (9,-2.94) -- (11,-2.94);
\draw[mgPurple, line width=2pt] (9,-3.06) -- (11,-3.06);

\draw[red!40, line width=2pt] (11,-2.94) -- (13,-2.94);
\draw[red!40, line width=2pt] (11,-3.06) -- (13,-3.06);

\foreach \x in {0,...,13} {
    \ifnum\x=0
        \draw[thick] (\x,-2.85) -- (\x,-3.15);
    \else\ifnum\x=2
        \draw[thick] (\x,-2.85) -- (\x,-3.15);
    \else\ifnum\x=5
        \draw[thick] (\x,-2.85) -- (\x,-3.15);
    \else\ifnum\x=7
        \draw[thick] (\x,-2.85) -- (\x,-3.15);
    \else\ifnum\x=9
        \draw[thick] (\x,-2.85) -- (\x,-3.15);
    \else\ifnum\x=11
        \draw[thick] (\x,-2.85) -- (\x,-3.15);
    \else\ifnum\x=13
        \draw[thick] (\x,-2.85) -- (\x,-3.15);
    \else
        \draw[thin,gray!60] (\x,-2.90) -- (\x,-3.10);
    \fi\fi\fi\fi\fi\fi\fi
}

\node at (1,-2.5) {$\hat{\alpha}_0$};
\node at (3.5,-2.5) {$\hat{\alpha}_1$};
\node at (6,-2.5) {$\hat{\alpha}_2$};
\node at (8,-2.5) {$\hat{\alpha}_3$};
\node[text=mgPurple] at (10,-2.5) {$J_2$};
\node[text=red!70!black] at (12,-2.5) {$J_3$};

\draw[thick] (0,-4.5) -- (13,-4.5);

\draw[mgGreen, line width=2pt] (9,-4.44) -- (10,-4.44);
\draw[mgGreen, line width=2pt] (9,-4.56) -- (10,-4.56);

\draw[mgGold, line width=2pt] (10,-4.44) -- (13,-4.44);
\draw[mgGold, line width=2pt] (10,-4.56) -- (13,-4.56);

\foreach \x in {0,...,13} {
    \ifnum\x=0
        \draw[thick] (\x,-4.35) -- (\x,-4.65);
    \else\ifnum\x=2
        \draw[thick] (\x,-4.35) -- (\x,-4.65);
    \else\ifnum\x=5
        \draw[thick] (\x,-4.35) -- (\x,-4.65);
    \else\ifnum\x=7
        \draw[thick] (\x,-4.35) -- (\x,-4.65);
    \else\ifnum\x=9
        \draw[thick] (\x,-4.35) -- (\x,-4.65);
    \else\ifnum\x=10
        \draw[thick] (\x,-4.35) -- (\x,-4.65);
    \else\ifnum\x=13
        \draw[thick] (\x,-4.35) -- (\x,-4.65);
    \else
        \draw[thin,gray!60] (\x,-4.40) -- (\x,-4.60);
    \fi\fi\fi\fi\fi\fi\fi
}

\node at (1,-4) {$\hat{\alpha}_0$};
\node at (3.5,-4) {$\hat{\alpha}_1$};
\node at (6,-4) {$\hat{\alpha}_2$};
\node at (8, -4) {$\hat{\alpha}_3$};
\node[text=darkgreen] at (9.5,-4) {$\tilde{J}_0$};
\node[text=orange] at (11.5,-4) {$\tilde{J}_1$};

\draw[thick] (0,-6) -- (13,-6);

\draw[mgPurple, line width=2pt] (9,-5.94) -- (10,-5.94);
\draw[mgGreen, line width=2pt] (9,-6.06) -- (10,-6.06);

\draw[mgPurple, line width=2pt] (10,-5.94) -- (11,-5.94);
\draw[mgGold, line width=2pt] (10,-6.06) -- (11,-6.06);

\draw[red!40, line width=2pt] (11,-5.94) -- (13,-5.94);
\draw[mgGold, line width=2pt] (11,-6.06) -- (13,-6.06);

\foreach \x in {0,...,13} {
    \ifnum\x=0
        \draw[thick] (\x,-5.85) -- (\x,-6.15);
    \else\ifnum\x=2
        \draw[thick] (\x,-5.85) -- (\x,-6.15);
    \else\ifnum\x=5
        \draw[thick] (\x,-5.85) -- (\x,-6.15);
    \else\ifnum\x=7
        \draw[thick] (\x,-5.85) -- (\x,-6.15);
    \else\ifnum\x=9
        \draw[thick] (\x,-5.85) -- (\x,-6.15);
    \else\ifnum\x=10
        \draw[thick] (\x,-5.85) -- (\x,-6.15);
    \else\ifnum\x=11
        \draw[thick] (\x,-5.85) -- (\x,-6.15);
    \else\ifnum\x=13
        \draw[thick] (\x,-5.85) -- (\x,-6.15);
    \else
        \draw[thin,gray!60] (\x,-5.90) -- (\x,-6.10);
    \fi\fi\fi\fi\fi\fi\fi\fi
}

\node[text=blue,font=\small] at (1,-5.5) {$(x_0, \tx_0)$};
\node[text=blue,font=\small] at (3.5,-5.5) {$(X_1, \tX_1)$};
\node[text=blue,font=\small] at (6,-5.5) {$(X_2, \tX_2)$};
\node[text=blue,font=\small] at (8,-5.5) {$(X_3, \tX_3)$};
\node[text=blue,font=\small] at (9.5,-5.5) {$(X_2, \tx_0)$};
\node[text=blue,font=\small] at (10.5,-5.5) {$(X_2, \tX_1)$};
\node[text=blue,font=\small] at (12,-5.5) {$(X_3, \tX_1)$};

\node at (0,-6.5) {0};
\node at (2,-6.5) {$s_0$};
\node at (5,-6.5) {$s_1$};
\node at (7,-6.5) {$s_2$};
\node at (9,-6.5) {$s_3$};
\node at (13,-6.5) {1};

\end{tikzpicture}
\caption{Example of the coupling of the acceptance step for $p =3$ given by the partition of the unit interval \eqref{eq:partition[0,1]}. The first and second lines from the top visualize the acceptance probabilities of the chain starting at $x_0$ and the chain starting at $\tx_0$, respectively. Here we use the shorter notations $\alpha_j$ and $\tilde{\alpha}_j$ for $\alpha_j(x_0, X_1, \ldots, X_p)$ and $\alpha_j(\tx_0,\tX_1, \ldots, \tX_p)$ respectively. We highlight the non-zero differences \eqref{def:betam:coupl} of the acceptance probabilities with different colors. The third and fourth lines show how to partition $[s_3,1]$ in two different ways as in \eqref{eq:mpcn:[sp,1]}. In the last line, we have the whole partition of the unit interval as in \eqref{eq:partition[0,1]}, so the last three intervals are respectively $J_2\cap \tJ_0$, $J_2\cap \tJ_1$ and $J_3\cap \tJ_1$, as the color scheme indicates. Here we also indicate in blue the realizations of the coupling \eqref{eq:mpcn:coupling} over the different intervals where the uniform variable $U$ may fall.}\label{fig:coupling}
\end{figure}

It is not difficult to show from this construction that $(X, \tX)$ is a coupling of $\mk^p(x_0, \cdot)$ and $\mk^p(\tx_0, \cdot)$. Hence, 
\begin{align*}
	\Wass_{d_\varepsilon}(\mk^p(x_0, \cdot), \mk^p(\tx_0, \cdot)) 
	&\leq \E d_\eps(X, \tX)\\
    &= \E d_\eps(x_0, \tx_0)\mathbbm{1}_{U \in [0, s_0)} + \sum_{j = 1}^{p} \E d_\eps(X_j , \tX_j)\mathbbm{1}_{U \in [s_{j-1}, s_{j})} + \E d_\eps(X, \tX) \mathbbm{1}_{U \in [s_p, 1]}.
\end{align*}
From the definition of $X_j$ and $\tX_j$ in \eqref{eq:mpcn:proposal1}, and recalling that $d_\eps(x_0, \tx_0)<1$, it follows that
\begin{equation*}
    d_\eps(X_j, \tX_j) = d_\eps(\rho^2 x_0, \rho^2 \tx_0) = 1 \wedge \frac{\rho^2}{\varepsilon}\|x_0 - \tx_0\| = \rho^2 d_\eps(x_0, \tx_0),
    \quad j =1,\ldots, p.
\end{equation*}
Then, using the fact that $d_\eps \leq 1$ and the definition \eqref{eq:def:sj}, we have 
\begin{align}
   \Wass_{d_\varepsilon}(\mk^p(x_0, \cdot), \mk^p(\tx_0, \cdot)) &\leq  
    d_\eps(x_0, \tx_0) \E s_0 +  \rho^2 d_\eps(x_0, \tx_0) \E  \sum_{j = 1}^{p} \left(s_j - s_{j-1}\right) + \E (1 - s_p) \notag\\
    &= d_\eps(x_0, \tx_0)\E  \left( \har_0 +  \rho^2  \sum_{j = 1}^{p} \har_j \right) + \E (1 - s_p). \label{eq:contr1}
\end{align}
Now we treat the last term in \eqref{eq:contr1}. First observe that for any $(y_0, \ldots, y_p)$ and  $(\ty_0, \ldots, \ty_p)$ in $\qsp^{p+1}$
\begin{align}\label{ineq:sp}
	\sum_{j=0}^p | \ar_j (y_0, \ldots, y_p) - \ar_j (\ty_0, \ldots, \ty_p) |
    &= \sum_{j=0}^p \left| \frac{e^{-\Pot(y_j)}}{\sum_{k=0}^p e^{-\Pot(y_k)}} - \frac{e^{-\Pot(\ty_j)}}{\sum_{k=0}^p e^{-\Pot(\ty_k)}} \right| 
	\notag \\
	&= \sum_{j=0}^p \left|  \frac{e^{-\Pot(y_j)} \sum_{k=0}^p e^{-\Pot(\ty_k)}  -  e^{-\Pot(\ty_j)} \sum_{k=0}^p e^{-\Pot(y_k)} }{\sum_{k,l=0}^p e^{-\Pot(y_k)} e^{- \Pot(\ty_l)}}  \right| 
	\notag\\
	&\leq \frac{\sum_{j=0}^p \left| e^{-\Pot(y_j)} - e^{-\Pot(\ty_j)} \right| }{\sum_{k=0}^p e^{-\Pot(y_k)}}
	+
	\frac{\left| \sum_{k=0}^p (e^{-\Pot(y_k)} - e^{-\Pot(\ty_k)} )  \right| }{\sum_{k,l = 0}^p e^{-\Pot(y_k)} e^{- \Pot(\ty_l)}} \sum_{j=0}^p e^{-\Pot(\ty_j)} 
	\notag \\
	&\leq 2  \frac{\sum_{j=0}^p \left| e^{-\Pot(y_j)} - e^{-\Pot(\ty_j)} \right| }{\sum_{k=0}^p e^{-\Pot(y_k)}}.
\end{align}
By invoking the Mean Value theorem together with the boundedness and Lipschitzianity of $\Pot$, we have
\begin{align}\label{ineq:num}
	\left| e^{-\Pot(y_j)} - e^{-\Pot(\ty_j)} \right|
	\leq e^{\|\Pot\|_\infty} |\Pot(y_j) - \Pot(\ty_j)|
	\leq \,\, e^{\|\Pot\|_\infty} L_{\Pot} \|y_j - \ty_j\| .
\end{align}
Moreover,
\begin{align}\label{ineq:den}
	\sum_{k=0}^p e^{-\Pot(y_k)} \geq \sum_{k=0}^p e^{-\|\Pot \|_\infty} =(p+1) e^{-\|\Pot \|_\infty}.
\end{align}
Then
\begin{equation*}
    \sum_{j=0}^p | \ar_j (y_0, \ldots, y_p) - \ar_j (\ty_0, \ldots, \ty_p) | \leq \frac{2e^{2\|\Pot\|_\infty} L_{\Pot}  }{(p+1) }\sum_{j = 0 }^p \|y_j - \ty_j\|.
\end{equation*}
Recalling the definition of $s_p$ in \eqref{eq:def:sj}, we have 
\begin{align}
   \E \left( 1 - s_p \right) 
   &\leq \sum_{j=0}^p \E | \ar_j (x_0,X_1 \ldots, X_p) - \ar_j (\tx_0, \tX_1, \ldots, \tX_p) | \label{eq:mpcn:1-sp_2}\\
   &\leq  \frac{2e^{2\|\Pot\|_\infty} L_{\Pot}  }{(p+1) }\left( \|x_0 - \tx_0\| + \sum_{k = 1}^p \E \|X_k - \tX_k\|\right), \notag 
\end{align}
and using again \eqref{eq:mpcn:proposal1} yields
\begin{align}\label{eq:mpcn:1-sp}
	 \E \left( 1 - s_p \right) \leq C \frac{1 + \rho^2 p}{p+1} \|x_0 - \tx_0\| = C \frac{1 + \rho^2 p}{p+1} \varepsilon d_\eps(x_0, \tx_0) ,
\end{align}
for some constant $C = C(\| \Pot\|_\infty,L_{\Pot}) > 0$.

From \eqref{eq:contr1}, we then showed that 
\begin{align}\label{ineq:contr:fin:p:0}
    W_{d_\varepsilon} (P_p(x_0, \cdot), P_p(\tx_0, \cdot))  \leq   d_\eps(x_0, \tx_0)  \left[  C \frac{1 + \rho^2 p}{(p+1)} \varepsilon +  \E \left( \har_0 +  \rho^2  \sum_{j = 1}^{p} \har_j  \right) \right].
\end{align}

Note that, for any $(y_0, \ldots, y_p) \in \qsp^{p+1}$ and $j \in \{0,\ldots, p\}$,
\begin{align*}
	\alpha_j (y_0, \ldots, y_p) = \frac{e^{-\Pot(y_j)}}{\sum_{k=0}^p e^{-\Pot(y_k)}} \geq \frac{e^{-\| \Pot\|_\infty}}{(p+1) e^{ \| \Pot \|_\infty }}
	= \frac{e^{- 2 \| \Pot \|_\infty }}{p+1}.
\end{align*}
Thus, for any choice of $\Xi \sim \mu_0^{\otimes (p+1)}$,
\begin{align}\label{eq:bound:har}
	 \hat{\alpha}_j(x_0, \tx_0; \Xi) \geq \frac{e^{- 2 \| \Pot \|_\infty }}{p+1}, \quad j =0, \ldots, p,
\end{align}
and
\begin{align*}
	 \har_0 + \rho^2 \sum_{j=1}^p \har_j 
	&\leq 1 - \sum_{j=1}^p \har_j + \rho^2 \sum_{j=1}^p  \har_j 
	=  1 - (1 - \rho^2) \sum_{j=1}^p  \har_j
	\\
	&\leq 1 - (1 - \rho^2) \frac{p}{p+1} e^{-2 \| \Pot \|_\infty} =: C_p < 1.
\end{align*}
Plugging this estimate into \eqref{ineq:contr:fin:p:0}, we deduce \eqref{ineq:contr:fin:p:a}. The proof of \eqref{ineq:contr:fin:p:b} then follows by estimating
\begin{align}\label{ineq:epsilon:Cp}
	 C \frac{1 + \rho^2 p}{(p+1)} \varepsilon + C_p 
\leq C \frac{1+ \rho^2}{2} \varepsilon + 1 - \frac{(1 - \rho^2)}{2} e^{- 2 \| \Pot\|_\infty}
\end{align}
and selecting $\varepsilon > 0$ so that the right-hand side of \eqref{ineq:epsilon:Cp} is strictly less than $1$, namely
\begin{equation}
    \varepsilon < \frac{(1 - \rho^2) }{C(1 + \rho^2)}e^{-2 \| \Pot \|_\infty}.
\end{equation}

\end{proof}

Last we show that $P_p^n$ satisfies condition 2 in \cref{thm:our_harris} and $d$-smallness for $n$ large enough.
\begin{Proposition}\label{prop:mpcn:small}
	Fix $\varepsilon>0$, $p \geq 1$, and assume the potential function $\Pot: \qsp \to \R$ is bounded and globally Lipschitz with Lipschitz constant $L_\Pot$. Define $S= \lbrace V(x) \leq 4 K_V\rbrace $ with $V$ being any of the functions in \cref{prop:mpcn:lyap}, and let $r_s >0$ be such that $S \subset B(0, r_s)$. Then, for every $x_0, \tx_0\in S$, it holds
    \begin{equation}\label{eq:mpcn:small}
        \Wass_{d_\varepsilon} (P_p^n(x_0, \cdot), P_p^n(\tx_0, \cdot))\leq 1 - \left(\frac{pe^{-2\|\Pot\|_\infty}}{p+1}\right)^n\left( 1 - \frac{2r_s \rho^{2n}}{\varepsilon} \right) \quad n\in \N.
    \end{equation}
    Consequently, for any $n>  \frac{ \log 2r_s - \log \varepsilon}{- 2 \log \rho}$ and every $x_0, \tx_0\in S$, it holds 
    \begin{equation}\label{eq:mpcn:small:s}
     \Wass_{d_\varepsilon} (P_p^n(x_0, \cdot), P_p^n(\tx_0, \cdot))\leq s
    \end{equation}
    with $s = s(\varepsilon, \rho, p, n) \in (0,1)$.
\end{Proposition}

\begin{proof}
    Let $x_0, \tx_0\in S$ and consider the same coupling \eqref{eq:mpcn:coupling} of $P(x_0, \cdot)$ and $P(\tx_0, \cdot)$ as constructed in the proof of \cref{prop:contr:fin:p:1}. Given $U^{(1)}\sim \mathcal{U}([0,1])$, $\Xi^{(1)} = (\xi_0^{(1)}, \ldots, \xi_p^{(1)})\sim \mu_0^{\otimes (p +1)}$, and the proposals $X^{(1)}_j$, $\tX^{(1)}_j$, $j = 1, \ldots, p$ as in \eqref{eq:mpcn:proposal1}, we set 
    \begin{equation*}
        \left(X^{(1)}, \tX^{(1)}\right) = \begin{cases}
            (x_0, \tx_0) \quad &\text{if } U^{(1)} \in [0, s_0^{(1)})\\
            (X^{(1)}_j, \tX^{(1)}_j) \quad &\text{if } U^{(1)} \in [s_{j-1}^{(1)}, s_j^{(1)})  \quad j =1, \ldots, p,
        \end{cases}
    \end{equation*}
    and, decorating the intervals in \eqref{eq:Jintervals} with the index 1, 
    \begin{equation*}
            \left(X^{(1)}, \tX^{(1)}\right) = \left(X^{(1)}_{k}, \tX^{(1)}_{l}\right) \quad \mbox{ if } U^{(1)} \in J^{(1)}_{k} \cap \tJ^{(1)}_{l}
    \end{equation*}
    for $k,l \in \{0,\ldots, p\} $ such that $J_{k}^{(1)} \cap \tJ^{(1)}_{l} \neq \emptyset$.

    Next, define 
    $$
    A^{(1)}=\left\lbrace U^{(1)}\in \left[s_0^{(1)}, s_p^{(1)}\right)\right\rbrace 
    $$
    the event in which two proposals with same index get accepted, so that there is an index $j= 1,\ldots, p$ such that 
    \begin{align*}
    \Wass_{d_\varepsilon} (\mk_p(x_0, \cdot), \mk_p(\tx_0, \cdot)) &\leq \E d\left( X^{(1)}, \tX^{(1)}\right) \\
    &= \E d\left( X^{(1)}, \tX^{(1)}\right) \mathbbm{1}_{A^{(1)}} + \E d\left( X^{(1)}, \tX^{(1)}\right) \mathbbm{1}_{(A^{(1)})^c} \\
    &= \E d\left(X_j^{(1)}, \tX_j^{(1)} \right) \mathbbm{1}_{A^{(1)}} + \E d\left( X^{(1)}, \tX^{(1)}\right) \mathbbm{1}_{(A^{(1)})^c}.
    \end{align*}
    Using the definition of the proposals \eqref{eq:mpcn:proposal1} and the fact that $d\leq 1$ by definition, we get
    \begin{align*}
          \Wass_{d_\varepsilon} (\mk_p(x_0, \cdot), \mk_p(\tx_0, \cdot)) &\leq d_\eps(\rho^2 x_0, \rho^2 \tx_0)\bbP\left(A^{(1)}\right) + 1 - \bbP\left(A^{(1)}\right).
    \end{align*}
    Now as $x_0, \tx_0\in S\subset B(0, r_s)$
     \begin{align*}
          \Wass_{d_\varepsilon} (\mk^p(x_0, \cdot), \mk^p(\tx_0, \cdot)) &\leq 1 - \bbP\left(A^{(1)}\right)\left( 1 - \frac{2r_s\rho^2}{\varepsilon}\right)
    \end{align*}
    and $\bbP\left(A^{(1)}\right) = \E \left( s_p^{(1)} - s_0^{(1)} \right)$. Observe that we cannot ensure that the bound on the right hand side is strictly smaller than one without imposing conditions, for example on $\rho$ to be very small. However, if we take enough steps of the chains, we expect the parameter to decrease and eventually be smaller than one as desired. We will now show this intuition by iterating the argument and constructing a coupling for $\mk_p^n(x_0, \cdot),\, \mk_p^n(\tx_0, \cdot)$. 
    
    Define $\Xi^{(n)} = (\xi^{(n)}_0, \ldots, \xi^{(n)}_p) \sim \mu_0^{\otimes (p+1)}$ independent of $\Xi^{(1)}, \ldots, \Xi^{(n-1)}$, and set the proposals at the $n$-th step as 
    \begin{align*}
        X_j^{(n)} = F(F(X^{(n-1)}, \xi_0^{(n)}), \xi_j^{(n)}) \quad 
        \text{and} \quad \tX_j^{(n)} = F(F(\tX^{(n-1)}, \xi_0^{(n)}), \xi_j^{(n)}).
    \end{align*}
    Then define the minimum of the acceptance probabilities
    \begin{align*}
	\hat{\alpha}_j^{(n)} &= \hat{\alpha}_j(X^{(n-1)}, \tX^{(n-1)}; \Xi^{(n)}) \\
	&= \min\{\alpha_j(X^{(n-1)}, X_1^{(n)}\ldots, X_p^{(n)}), \alpha_j(\tX^{(n-1)}, \tX_1^{(n)}\ldots, \tX_p^{(n)})\}, \quad j = 0, \ldots, p,
\end{align*}
and set, as in \eqref{eq:def:sj},
\begin{align}\label{eq:def:sj:n}
	s_j^{(n)} = s_j^{(n)}\left(x_0, \tx_0; \Xi^{(n)}\right) = \sum_{i=0}^j \hat{\alpha}^{(n)}_i, \quad j = 0, \ldots, p.
\end{align}
Now, given $U^{(n)}\sim \mathcal{U}([0,1])$, independent of $U^{(1)},\ldots, U^{(n-1)}$ and of $\Xi^{(1)}\ldots \Xi^{(n)}$, we set 
    \begin{equation*}
        \left(X^{(n)}, \tX^{(n)}\right) = 
        \begin{cases}
            (X^{(n-1)}, \tX^{(n-1)}) &\quad \text{if } U^{(n)} \in [0, s_0^{(n)})\\
            (X^{(n)}_j, \tX^{(n)}_j)& \quad \text{if } U^{(n)} \in [s_{j-1}^{(n)}, s_j^{(n)}), \quad j =1, \ldots, p
        \end{cases}
    \end{equation*}
    and 
    \begin{equation*}
        \left(X^{(n)}, \tX^{(n)}\right) =  \left(X^{(n)}_{k}, \tX^{(n)}_{l}\right) \quad \mbox{if } U^{(n)} \in J^{(n)}_{k} \cap \tJ^{(n)}_{l}
    \end{equation*}
    for $k,l \in \{0,\ldots, p\}$ such that $ J^{(n)}_{k} \cap \tJ^{(n)}_{l} \neq \emptyset$. 
    It can be verified that this is indeed a coupling of $P_p^n(x_0, \cdot)$ and $P_p^n(x_0, \cdot)$. 
    
    Next consider the event in which two proposals with same index get accepted at the $n$-th iteration, namely
    \[ A^{(n)}=\left\{ U^{(n)}\in \left[s_0^{(n)}, s_p^{(n)}\right)\right\} \]
     and define 
    \begin{equation*}
        \Lambda^{(n)} = \bigcap_{j = 1}^n A^{(j)}
    \end{equation*}
    the event for which, for $n$ steps in a row, the chains accept two proposal with the same index in $1, \ldots, p$, with this index possibly changing among the $n$ steps. Then we can write 
    \begin{align*}
    \Wass_{d_\varepsilon} (P_p^n(x_0, \cdot), P_p^n(\tx_0, \cdot)) &\leq \E d\left( X^{(n)}, \tX^{(n)}\right) \\
    &= \E d\left( X^{(n)}, \tX^{(n)}\right) \mathbbm{1}_{ \Lambda^{(n)}} + \E d\left( X^{(n)}, \tX^{(n)}\right) \mathbbm{1}_{\left(\Lambda^{(n)}\right)^c} \\
    &= d_\eps(\rho^{2n}x_0, \rho^{2n}\tx_0) \bbP( \Lambda^{(n)}) + 1 - \bbP( \Lambda^{(n)}).
    \end{align*}
    Again, as $x_0, \tx_0\in B(0, r_s)$, being in the small set, 
    \begin{align}\label{eq:mpcn:small1}
        \Wass_{d_\varepsilon} (P_p^n(x_0, \cdot), P_p^n(\tx_0, \cdot)) &\leq 1 - \bbP\left( \Lambda^{(n)}\right) \left( 1 - \frac{2r_s \rho^{2n}}{\varepsilon}\right). 
    \end{align}

    Let us now focus on estimating the probability of the event $\Lambda^{(n)}$. It is easy to see that 
    \begin{align*}
        \bbP\left( \Lambda^{(n)}\right) &= \bbP\left(A^{(1)}\right) \prod_{k = 2}^n\bbP\left( A^{(k)}  \left|\, \bigcap_{j =1}^{k-1}  A^{(j)}\right.\right) \\
        &=\prod_{k=1}^n \E \left( s_p^{(k)} - s_0^{(k)}\right) =    \prod_{k=1}^n \E \sum_{j = 1}^p \hat{\alpha}^{(k)}_j.
    \end{align*}
    Given the boundedness assumption on the potential $\Pot$, it follows, as in \eqref{eq:bound:har}, that
    \begin{equation*}
        \E \hat{\alpha}^{(k)}_j > \frac{e^{-2\|\Pot\|_\infty}}{p+1}
    \end{equation*}
    hence 
    \begin{equation}\label{eq:boundLambda}
        \bbP\left( \Lambda^{(n)}\right) > \left( \frac{pe^{-2\|\Pot\|_\infty}}{p+1}\right)^n.
    \end{equation}
    Plugging \eqref{eq:boundLambda} in \eqref{eq:mpcn:small1} we obtain \eqref{eq:mpcn:small}. The proof of \eqref{eq:mpcn:small:s} follows by choosing $n$ so that $ 2r_s \rho^{2n} > \varepsilon$. 
\end{proof}

    Thanks to \cref{prop:mpcn:lyap}, \cref{prop:contr:fin:p:1} and \cref{prop:mpcn:small} the assumptions of \cref{thm:our_harris} are fulfilled, giving proof of \cref{thm:3:mpcn1} for mpCN.

\subsubsection{Infinite number of proposals}\label{sec:proof_mpc_infinite}
We now want to prove \cref{thm:3:inftypcn}. Given the algorithmic parameters $\rho_1$, $\rho_2$ we define the functions, similarly to \eqref{eq:2:defF},
\begin{equation}\label{def:F:i}
    F_i (x, w) = \rho_i x + \sqrt{1 - \rho_i^2} w ,\quad i = 1,2
\end{equation}
so that $Q_i (x, dy) = F(x, \cdot)^*\mu_0(dy)$. Moreover, we set $\bar{F}_{12}(x, \cdot): \qsp^{2}\to \qsp $
\begin{equation*}
    \bar{F}_{12}(x, w_1, w_2) := F_1(F_2(x, w_2), w_1)
\end{equation*}
so that 
\begin{equation}\label{eq:infty:Qtrans}
    \int_\qsp  Q_1(z, dy) Q_2(x, dz) = \bar{F}_{12}(x, \cdot)^* \mu_0^{\otimes 2}(dy)
\end{equation}

We now follow the same pathway as the previous subsection, starting from the Lyapunov functions. 

\begin{Proposition}\label{prop:inftypcn:lyap}
    Assume the potential function $\Pot: \qsp \to \R$ is Lipschitz with constant $L_\Pot$. Then the functions $V(x) = \|x\|^n$, $n \in \N$, $V(x) = \exp(v\|x\|)$, $v>0$ and $V(x) = \exp(v\|x\|^2)$ with $v$ small enough, are Lyapunov functions as in \cref{def:lyapunov} for the $\infty$-pCN Markov kernel $P_\infty$.
\end{Proposition}
\begin{proof}
    For simplicity of notation we drop the decoration $\infty$ on the kernel and we denote $\sqrt{1-\rho_i^2}= \trho_i$, $i =1,2$. By definition of the kernel we write the quantity we desire to bound as 
    \begin{align}
        (PV)(x) &= \int V(y) \, P(x, dy) = \iint V(y) \frac{\exp(-\Pot(y))}{\int \exp(- \Pot(u)) Q_1(z, du)} \, Q_1(z, dy) Q_2(x, dz) \notag\\
        &= \iint V\left(\bar{F}_{12}(x, w_1, w_2)\right) \frac{\exp(-\Pot\left(\bar{F}_{12}(x, w_1, w_2)\right))}{\int \exp(- \Pot(\left(\bar{F}_{21}(x, \tu, w_2)\right))) \mu_0(\tu)} \, \mu_0(dw_1)\mu_0(dw_2).\label{eq:infty:lyap0}
    \end{align}
    where in the second line we used the relation \eqref{eq:infty:Qtrans}.
    
    We start by bounding $V\left(\bar{F}_{12}(x, w_1, w_2)\right)$ in \eqref{eq:infty:lyap0} for the three different candidates \eqref{eq:choiceV} for $V$. By Cauchy--Schwartz and Young's inequalities, for each $\delta>0$, there is $C_\delta>0$ such that
    \begin{equation*}
       \|\bar{F}_{12}(x, w_1, w_2)\|^n = \|\rho_1 \rho_2 x + \rho_1 \trho_2 w_2 + \trho_1 w_1\|^n \leq  (1 + \delta)  \|\rho_1 \rho_2 x\|^n + C_\delta\| \rho_1 \trho_2 w_2 + \trho_1 w_1\|^n.
    \end{equation*}
    Next, thanks to the triangular and Young's inequalities, for any $q > 1$
    \begin{align*}
        \exp(v \|\bar{F}_{12}(x, w_1, w_2) \|)
        &\leq \tfrac{1}{q}\exp\left(q v\rho_1\rho_2 \|x\|\right) + \tfrac{q-1}{q} \exp\left(\tfrac{qv}{q-1}\|\rho_1 \trho_2 w_2 + \trho_1 w_1 \|\right)\\
        & = \tfrac{1}{q}\exp\left(v\|x\|\right)\exp\left((v(q\rho_1\rho_2 -1)\|x\|\right) + \tfrac{q-1}{q} \exp\left(\tfrac{qv}{q-1}\|\rho_1 \trho_2 w_2 + \trho_1 w_1 \|\right).
    \end{align*}
    Similarly for an arbitrary $\delta>0$ there is $C_\delta>0$ such that for any $q>1$ 
    \begin{align*}
        \exp(v \|\bar{F}_{12}(x, w_1, w_2) \|^2)
        &\leq \tfrac{1}{q}\exp\left(q v( 1 + \delta) \rho_1\rho_2 \|x\|^2\right) + \tfrac{q-1}{q} \exp\left(\tfrac{qv  C_\delta}{q-1}\|\rho_1 \trho_2 w_2 + \trho_1 w_1 \|^2\right)\\
        & = \tfrac{1}{q}\exp\left(v\|x\|^2\right)\exp\left[v(q\rho_1\rho_2(1 + \delta) -1)\|x\|^2\right] + \tfrac{q-1}{q} \exp\left(\tfrac{qvC_\delta}{q-1}\|\rho_1 \trho_2 w_2 + \trho_1 w_1 \|^2\right).
    \end{align*}
    Then, for $\rho>0$, it follows that
    \begin{equation}\label{eq:infty:lyap1}
        V(\bar{F}_{12}(x, w_1, w_2) ) \leq l_1 V(x) + G(w_1, w_2)
    \end{equation}
 with $l_1<1$ and $G(w_1,w_2)$ defined as
    \begin{align}\label{eq:inftypcn:lyap:l1}
        l_1 &=  (1 + \delta)\rho_1^n\rho_2^n,\quad 0< \delta <  (\rho_1\rho_2)^{-n} -1; &&  G(w_1, w_2) = 
             C_\delta\|\rho_1 \trho_2 w_2 + \trho_1 w_1\|^n
            \\
         l_1 &=  q^{-1},\quad1< q < (\rho_1\rho_2)^{-1}; &&  G(w_1, w_2)=  \tfrac{q-1}{q} \exp\left(\tfrac{qv}{q-1}\|\rho_1 \trho_2 w_2 + \trho_1 w_1 \|\right)  \label{eq:inftypcn:lyap:l2}\\
         l_1 &=  q^{-1},\quad 1< q< (( 1 + \delta)\rho_1\rho_2)^{-1}, \, 0< \delta <  (\rho_1\rho_2)^{-n} -1; &&  G(w_1, w_2) =  \tfrac{q-1}{q} \exp\left(\tfrac{q v C_\delta}{q-1}\|\rho_1 \trho_2 w_2 + \trho_1 w_1 \|^2\right).\label{eq:inftypcn:lyap:l1:3}
    \end{align}
    for the three candidate Lyapunov functions respectively. 
Then, using \eqref{eq:infty:lyap1} in \eqref{eq:infty:lyap0} yields 
\begin{equation*}
    (PV)(x) \leq l_1 V(x) + \iint G(w_1, w_2) \frac{\exp(-\Pot\left(\bar{F}_{12}(x, w_1, w_2)\right))}{\int \exp(- \Pot(\left(\bar{F}_{21}(x, \tu, w_2)\right))) \mu_0(\tu)} \, \mu_0(dw_1)\mu_0(dw_2) = l_1 V(x) + \mathcal{I}_G.
\end{equation*}
We are now left to ensure $\mathcal{I}_G$ is suitably bounded to define $K_V>0$. 
First we note that by the Lipschitzianity of $\Pot$ we have 
    \begin{align}
      &\frac{\exp(-\Pot(\bar{F}_{12}(x, w_1, w_2)))}{\int \exp(-\Pot(\bar{F}_{12}(x, \tu, w_2)))\mu_0(d\tu)} \notag\\
        &\quad = \left( \int \exp(\Pot(\rho_1 \rho_2 x + \rho_1 \trho_2 w_2 + \trho_1 w_1) -\Pot(\rho_1 \rho_2 x + \rho_1 \trho_2 w_2 + \trho_1\tu))\mu_0(d\tu)\right) ^{-1} \notag\\
       &\quad \leq \left( \int \exp(-L_\Pot \trho_1 \|w_1 - \tu \| )\mu_0(d\tu)\right)^{-1} \leq \frac{1}{M_1}\exp(L_\Pot \trho_1 \|w_1\|)
        \end{align}
    where in the last inequality we set $ M_1 = \int \exp(-L_\Pot \trho_1 \|\tu \| )\mu_0(d\tu)$. Therefore 
    \begin{equation*}
        \mathcal{I}_G \leq \frac{1}{M_1}\iint G(w_1, w_2) \exp(L_\Pot \trho_1 \|w_1\|) \, \mu_0(dw_1)\mu_0(dw_2).
    \end{equation*}
For $G$ as in \eqref{eq:inftypcn:lyap:l1}-\eqref{eq:inftypcn:lyap:l2} it is easy to see that this integral is well defined for any choice of the parameters since $\mu_0$ is Gaussian, so algebraic and exponential moments are finite.
For the last definition of $G$ in \eqref{eq:inftypcn:lyap:l1:3} we have to ensure to be able to use Fernique's theorem. Indeed
  \begin{align}
       \mathcal{I}_G &= \frac{q-1}{q M_1} \iint \exp\left(\tfrac{qv C_\delta }{q-1}\|\rho_1 \trho_2 w_2 + \trho_1 w_1 \|^2\right) \exp( L\trho_1\|w_1 \| )\,\mu_0(dw_1) \,  \mu_0 (dw_2)\notag\\
     &\leq 
     \frac{q-1}{q M_1} \int \exp\left(\tfrac{2qC_\delta (\rho_1 \trho_2)^2}{q-1}v \|w_2\|^2 \right)\, \mu_0(dw_2) 
     \int \exp\left( \tfrac{2qC_\delta \trho_1^2}{q-1}v\|w_1\|^2+ L_\Pot\trho_1\| w_1 \|\right) \,\mu_0 (dw_1) , \label{eq:infty:lyap:ex2}
  \end{align}
  so choosing the parameter $v$ small enough to ensure the finiteness of the Gaussian integrals in \eqref{eq:infty:lyap:ex2} we have the desired result for $\rho_1, \rho_2$ positive. With similar and simpler arguments the result holds also for $\rho_1 = 0$ and/or $\rho_2 = 0$. 
\end{proof}

We continue by establishing the $d$-contraction for close enough initial points. To construct a coupling of $P_\infty(x,\cdot)$, $P_\infty(\tx, \cdot)$ for different $x, \tx$ in $\qsp$ we start by using \cref{lemma:AR_coupling}.
\begin{Lemma}\label{lemma:eta_coupling}
   Consider $\baQ_1(z, \cdot)$ and $\baQ_1(\tz, \cdot)$ as in \eqref{eq:2:baQ}, namely 
   \begin{equation}\label{eq:4:baQ}
    \baQ_1(z, dy):= \dfrac{\exp(-\Pot(y)) Q_1(z, d y)}{\int \exp(-\Pot(u)) Q_1(z, d u) }.
    \end{equation}
    Define the function $T_{\bz}: \qsp \to \qsp$, with $\bz = ( z , \tz)$, as 
    \begin{equation}\label{eq:Tz}
        T_{\bz}(y) := y - \rho_1(z - \tz) .
    \end{equation}
   Then the measure $ \pi_\beta(\bz,\, \cdot)$ on $\qsp^2$ defined as 
    \begin{equation}\label{eq:ARcoupling_eta}
     \pi_\beta(\bz; dy, d\ty) = \int_{\qsp^2} \left[\beta(\bz; w, T_{\bz}(y))\,\delta_{T(y)}(d\ty) + (1 - \beta(\bz; w, T_{\bz}(y)))\,\delta_{w}(d\ty)\right]\, \baQ_1(z, dy)\baQ_1(\tz, dw)
\end{equation}
with  
\begin{equation}\label{eq:acceptanceAR}
       \beta(\bz; w, T_{\bz}(y)) = 1 \wedge \frac{\exp(-\Pot(w) + \Pot(T_{\bz}^{-1}(w)))}{\exp(-\Pot(T_{\bz}(y)) + \Pot(y))}.
    \end{equation}
is a coupling of $\baQ_1(z, \cdot)$ and $\baQ_1(\tz, \cdot)$.
\end{Lemma}
\begin{proof}
We want to show that $T_{\bz}^*\baQ_1(z, \cdot)$ is equivalent to $\baQ_1(\tz, \cdot)$.  Given $\psi:\qsp\to \R$, using $F_1$ defined in \eqref{def:F:i} and, setting $\cI_{\tz} = \int \exp(-\Pot(u))Q_1(\tz, du)$, we have
    \begin{align*}
        \int \psi(y)\, \baQ_1(\tz, dy) &= \int \psi (y) \frac{\exp(-\Pot(y))}{\cI_{\tz}} \, Q_1(\tz, dy)= \int \psi(F_1(\tz, w)) \frac{\exp(-\Pot(F_1(\tz, w))}{\cI_{\tz}} \, \mu_0(dw).
    \end{align*} 
Note that by definition \eqref{eq:Tz} of $T_{\bz}$
\begin{equation*}
    F_1(\tz, w) = \rho_1 \tz + \trho_1 w = \rho_1 z + \trho_1 w - \rho_1(z - \tz) = T_{\bz} ( F_1(z , w))
\end{equation*}
so that 
\begin{align}
        \int \psi(y)\, \baQ_1(\tz, dy) &= \int \psi (T_{\bz} (y)) \frac{\exp(-\Pot(T_\bz (y)))}{\cI_{\tz}} \, Q_1(z, dy)     \label{eq:ar:coup:1}
\end{align}
and, recalling the definition of the kernel $\baQ_1(z, \cdot)$,
\begin{align*}
     \eqref{eq:ar:coup:1} &= \int \psi (y) \frac{\exp(-\Pot(T_{\bz}(y)))}{\cI_{\tz}} \frac{\cI_{z}}{\exp(-\Pot(y))}\, \baQ_1 (z, dy)\\
     &= \int \psi (y) \frac{\cI_{z}}{\cI_{\tz}} \exp(-\Pot(y) + \Pot(T_{\bz}^{-1}(y)))\, (T_{\bz}^*\baQ_1) (z, dy)
\end{align*}
where $T_{\bz}^{-1}(y) = y+ \rho_1(z - \tz)$. We then showed that $\baQ_1 (\tz, \cdot)$ is absolutely continuous with respect to $T_{\bz}^*\baQ_1(z, \cdot)$ and 
    \begin{equation*}
        \frac{d\baQ_1 (\tz, \cdot)}{d T_{\bz}^*\baQ_1(z, \cdot)}(y) =  \frac{\cI_{z}}{\cI_{\tz}} \exp(-\Pot(y) + \Pot(T_{\bz}^{-1}(y))).
    \end{equation*}
    A symmetric proof gives the opposite direction and corresponding Radon--Nikodym derivative
    \begin{equation*}
       \frac{d T_{\bz}^*\baQ_1(z, \cdot)}{d \baQ_1 (\tz, \cdot)}  (\ty) = \left( \frac{d\baQ_1 (\tz, \cdot)}{d T_{\bz}^*\baQ_1(z, \cdot)}\right)^{-1}(\ty) = \frac{\cI_{\tz}}{\cI_{z}} \exp(\Pot(\ty) - \Pot(T_{\bz}^{-1}(\ty))).
    \end{equation*}
    Then, thanks to \cref{lemma:AR_coupling}, the measure $\pi_\beta$ in \eqref{eq:ARcoupling_eta} with 
    \begin{equation*}
    \begin{split}
        \beta(\bz; y, \ty) &= 1 \wedge \frac{d\baQ_1 (\tz, \cdot)}{d T_{\bz}^*\baQ_1(z, \cdot)}(y)  \frac{d T_{\bz}^*\baQ_1(z, \cdot)}{d \baQ_1 (\tz, \cdot)}(\ty) =  1 \wedge \frac{\exp(-\Pot(y) + \Pot(T_{\bz}^{-1}(y)))}{\exp(-\Pot(\ty) + \Pot(T_{\bz}^{-1}(\ty)))}
        \end{split}
    \end{equation*}
    is a coupling of $\baQ_1(z, \cdot)$, $\baQ(\tz, \cdot)$. 
\end{proof}

\begin{Proposition}\label{prop:infty:contr}
	Fix $\varepsilon>0$ and assume the potential function $\Pot: \qsp \to \R$ is globally Lipschitz with Lipschitz constant $L_{\Pot}$. Then, for every $x, \tx \in \qsp$ with $d_\eps(x, \tx) < 1$, we have
	\begin{align}\label{ineq:contr:infty:a} 
		\Wass_{d_\varepsilon}(\mk_\infty(x, \cdot), \mk_\infty(\tx, \cdot)) 
		\leq  d_\eps(x, \tx) \rho_1\rho_2(1 + 2L_\Pot\varepsilon).
	\end{align}
    Consequently, for any fixed $\varepsilon$ satisfying $0 < \varepsilon < \frac{1}{2L_\Pot}\left(\frac{1}{\rho_1\rho_2}-1\right)$ and every $x, \tx \in \qsp$ with $d_\eps(x, \tx) < 1$, it holds
	\begin{align}\label{ineq:contr:infty:b} 
		\Wass_{d_\varepsilon}(\mk_\infty(x, \cdot), \mk_\infty(\tx, \cdot)) 
		\leq
		\kappa d_\eps(x, \tx),
	\end{align}
	with $\kappa = \kappa (\varepsilon, \rho_1, \rho_2) \in (0,1)$.
\end{Proposition}
\begin{proof}
 Set the notations $ \bx = (x, \tx)$, for the initial points, and  $\by = (y, \ty)$,  $\bz = (z, \tz)$ for some auxiliary variables. Let $\pi_s(\bx;\, d\bz)$ be the synchronous coupling of $Q_2(x, dz)$ and $Q_2(\tx, d\tz)$ and $\pi_\beta(\bz;\, d\by)$ be as in \eqref{eq:ARcoupling_eta}. Then the measure
 $\pi_\infty(\bx, \cdot)$ on $\qsp^2$ defined as
    \begin{equation}
    \label{eq:ARcontrol_P}
        \pi_\infty(\bx;\, d\by) := \iint \pi_\beta(\bz;\, d\by)\,  \pi_s(\bx;\, d\bz) 
    \end{equation}
is a coupling of $P_\infty(x, \cdot)$ and $P_\infty(\tx, \cdot)$. As highlighted in \cref{rem:couplingRV}, we can also express the coupling $\pi_\beta(\bz;\, d\by)$ as the couple of random variables $(Y, \tY)$ where $Y\sim \baQ_1(z, \cdot)$ and $\tY\sim \baQ_1(\tz, \cdot)$ is as in \eqref{eq:couplingRV}, namely 
    \begin{align*}
        \tY  = \mathbbm{1}_{U\leq \beta(\bz;\hat{Y}, T_{\bz}(Y))} T_{\bz}(Y) + \mathbbm{1}_{U > \beta(\bz;\hat{Y}, T_{\bz}(Y))} \hat{Y},
    \end{align*}
    with $\hat{Y}\sim \baQ_1(\tz, \cdot)$, drawn independently of~$Y$, and $U\sim \cU(0,1)$.
    Then 
    \begin{equation*}
          W_{d_\varepsilon}(P_\infty(x, \cdot ), P_\infty(\tx, \cdot)) \leq \int d_\eps(y, \ty) \, \pi_\infty(\bz; d\by) = \int \E\, d_\eps(Y, \tY) \, \pi_s(\bx;\, d\bz). 
    \end{equation*}
    By definitions of $\tY$ and $T_{\bz}$, \eqref{eq:Tz}, and using the fact that $d\leq 1$
    \begin{align}
      \E\, d_\eps(Y, \tY) &= \E d_\eps(Y, T_{\bz}(Y))\mathbbm{1}_{U\leq \beta(\bz;\hat{Y}, T_{\bz}(Y))} +  
        \E d_\eps(Y, \hat{Y}) \mathbbm{1}_{U > \beta(\bz;\hat{Y}, T_{\bz}(Y))} \notag \\
        &\leq  d_\eps(\rho_1 z, \rho_1\tz)\bbP(U\leq \beta(\bz;\hat{Y}, T_{\bz}(Y)))  +  \bbP(U > \beta(\bz;\hat{Y}, T_{\bz}(Y)))\notag \\
       &= d_\eps(\rho_1 z, \rho_1\tz) \E \beta(\bz;\hat{Y}, T_{\bz}(Y))) + 1 - \E \beta(\bz;\hat{Y}, T_{\bz}(Y))).\label{eq:infty:contr0}
    \end{align}
     As the function $\Pot$ is Lipschitz with constant $L_\Pot$, then the following upper bound holds for any choice of $\bz, \by$
    \begin{align}
         |\beta(\bz; y, \ty) - 1 | &= \left|  1 \wedge \frac{\exp(-\Pot(y) + \Pot(y + \rho_1(z - \tz)))}{\exp(-\Pot(\ty) + \Pot(\ty+ \rho_1(z - \tz)))} -1\right| \nonumber\\
         &\leq |-\Pot(y) + \Pot(y + \rho_1(z - \tz)| +| \Pot(\ty) - \Pot(\ty+ \rho_1(z - \tz))|
         \leq 2L_\Pot\rho_1 \|z - \tz\|.\label{eq:infty:beta-1}
    \end{align}
    Therefore, as $\beta\leq 1$
    \begin{align*}
       \E d_\eps(Y, \tY) \leq d_\eps(\rho_1 z, \rho_1\tz)  + 2L_\Pot\varepsilon \rho_1 \frac{\|z - \tz\|}{\varepsilon}
    \end{align*}
    and recalling that $\pi_s(\bx;\, d\bz)$ is the synchronous coupling of $Q(x, dz)$ and $Q(\tx, d\tz)$
    \begin{align*}
         W_{d_\varepsilon}(P_\infty(x, \cdot ), P_\infty(\tx, \cdot)) &\leq \int d_\eps(\rho_1 z, \rho_1\tz)  + 2L_\Pot\varepsilon \rho_1 \frac{\|z - \tz\|}{\varepsilon} \, \pi_s(\bx;\, d\bz)  \\
          &= d_\eps(\rho_1\rho_2 x , \rho_1\rho_2 \tx) + 2L_\Pot\varepsilon \rho_1\rho_2 \frac{\|x - \tx\|}{\varepsilon}\\
         &= \rho_1\rho_2(1 + 2L_\Pot\varepsilon) d_\eps(x, \tx).
    \end{align*}
    where in the last equality we used that $d_\eps(x, \tx) <1$. We then pick $\varepsilon< \frac{1}{2L_\Pot}\left(\frac{1}{\rho_1\rho_2}-1\right)$ to derive \eqref{ineq:contr:infty:b}. 
\end{proof}

Last we will use the coupling we just constructed to show condition 2 in \cref{thm:our_harris} and obtain $d$-smallness. 
\begin{Proposition}\label{prop:inftypcn:small}
	Fix $\eps>0$ and assume the potential function $\Pot: \qsp \to \R$ is globally Lipschitz with constant $L_\Pot$. Define $S= \lbrace V(x) \leq 4 K_V\rbrace $ with $V$ any of the functions in \eqref{eq:choiceV}, and let $r_s >0$ be such that $S \subset B(0, r_s)$. Then, for every $x, \tx\in S$, it holds
    \begin{equation}\label{eq:inftypcn:small}
        \Wass_{d_\varepsilon} (P_\infty^n(x, \cdot), P_\infty^n(\tx, \cdot))\leq  1 - \exp\left(-2L_\Pot r_s \rho_1\rho_2(1 - \rho_1^n\rho_2^n)\right)\left(1 - \frac{r_s\rho_1^n\rho_2^n }{\varepsilon} \right),  \quad n\in \N.
    \end{equation}
    Consequently, for any $n>  \frac{\log\varepsilon - \log r_s}{ \log \rho_1\rho_2}$ and every $x, \tx\in S$, it holds 
    \begin{equation}\label{eq:inftypcn:small:s}
     \Wass_{d_\varepsilon} (P_\infty^n(x, \cdot), P_\infty^n(\tx, \cdot))\leq s
    \end{equation}
    with $s = s(\varepsilon, \rho, n) \in (0,1)$.
\end{Proposition}
\begin{proof}
   We denote $\trho_i = \sqrt{1 - \rho_i^2}$, $i = 1,2$ and construct an iterative argument. We start by noticing that the coupling  $\pi_\infty(\bx; \cdot)$, defined in \eqref{eq:ARcontrol_P}, can also be written as follows: given two independent random variables $\xi_0\sim \mu_0$ and $U^{(1)}\sim \cU([0,1])$, set 
\begin{equation*}
    Z^{(1)} = F_2(x, \xi_0) = \rho_2 x + \trho_2 \xi_0, \quad \tZ^{(1)} = F_2(\tx, \xi_0) = \rho_2 \tx + \trho_2 \xi_0, \label{def:Z_1}
\end{equation*}
so that $\bZ^{(1)} = (Z^{(1)}, \tZ^{(1)}) \sim \pi_s(\bx; \cdot)$. Then, draw independently $Y^{(1)}\sim \baQ_1(Z^{(1)}, \cdot)$ and $\hat{Y}^{(1)} \sim \baQ_1(\tZ^{(1)}, \cdot)$ and, given the event 
\begin{equation*}
   A^{(1)}= \left\lbrace U^{(1)}\leq \beta\left(\hat{Y}^{(1)}, T_{\bZ^{(1)}}\left(Y^{(1)}\right)\right)\, \right\rbrace,
\end{equation*}
define
\begin{equation*}
    \tY^{(1)} =   T_{\bZ^{(1)}}(Y^{(1)}) \mathbbm{1}_{A^{(1)}}+ \hat{Y}^{(1)} \mathbbm{1}_{ (A^{(1)})^c}.
\end{equation*}
Consequently $\left(Y^{(1)}, \tY^{(1)}\right)\sim \pi_\infty(\bx, \cdot)$. Then, similarly to the argument carried out to obtain \eqref{eq:infty:contr0}, it follows that
    \begin{align*}
         W_{d_\varepsilon}(P_\infty(x, \cdot ), P_\infty(\tx, \cdot)) & \leq \E d_\eps(Y^{(1)}, T_{\bZ^{(1)}}(Y^{(1)}))\mathbbm{1}_{A^{(1)}} + \E d_\eps(Y^{(1)}, \hat{Y}^{(1)})\mathbbm{1}_{A^{(1)}} \\
        & \leq  \E d_\eps(\rho_1 Z^{(1)}, \rho_1 \tZ^{(1)}) \mathbbm{1}_{A^{(1)}} +  1 - \bbP(A^{(1)})\\
         &=  1 - \bbP(A^{(1)}) (1 - d_\eps(\rho_1 \rho_2 x, \rho_1 \rho_2 \tx)).
    \end{align*}
    By the definition \eqref{eq:acceptanceAR} of the acceptance probability $\beta$, and the the Lipschitzianity of $\Pot$,
    \begin{align*}
        \E \beta(\bZ; \hat{Y}, T_{\bZ}(Y)))
        &= \E 1 \wedge \frac{\exp(-\Pot(\hat{Y}) + \Pot(T_{\bZ}^{-1}(\hat{Y})))}{\exp(-\Pot(T_{\bZ}(Y)) + \Pot(Y))} )\\
        &\geq \E  \exp(-2L\rho_1\|Z- \tZ\|) = \exp(-2L\rho_1\rho_2\|x- \tx\|) .
    \end{align*}
   Last, as $x, \tx\in S\subset B(0, r_s)$ then $\|x - \tx\|\leq 2r_s$ and we have 
    \begin{equation*}
         W_{d_\varepsilon}(P_\infty(x, \cdot ), P_\infty(\tx, \cdot)) \leq  1 -  \exp\left(- 2L\rho_1\rho_2 r_s\right)\left( 1 - \frac{\rho_1\rho_2 r_s}{\varepsilon}\right).
    \end{equation*}
    This is not sufficient to ensure $d_\eps$-smallness unless we impose conditions on the algorithmic parameters $\rho_1$, $\rho_2$. To avoid this requirements we can iterate the argument, constructing a coupling for $P_\infty^n(x, \cdot)$, $P_\infty^n(\tx, \cdot)$ for~$n>1$ in a similar fashion. 
    
    Draw $\xi^{(n-1)} \sim \mu_0$ independently of $\xi^{(0)}, \ldots, \xi^{(n-2)}$ and define $\bZ^{(n)} =(Z_n, \tZ_n)$ to be the synchronous coupling of $Q_2(Y_{n-1},\cdot )$ and $Q_2(\tY_{n-1}, \cdot)$, namely 
\begin{align*}
    & Z^{(n)}= F_2(Y^{(n-1)},\xi^{(n-1)}), \quad \tZ^{(n)} =F_2(\tY^{(n-1)},\xi^{(n-1)}). 
\end{align*}
Drawn independently $Y_n\sim \baQ_1(Z_n, \cdot)$ and $\hat{Y}_n\sim \baQ_1(\tZ^{(n)}, \cdot)$ and, for $U^{(n)} \sim \cU([0,1])$ independent of $U^{(1)}, \ldots, U^{(n-1)}$, define the event 
\begin{equation*}
   A^{(n)} = \lbrace U^{(n)}\leq \beta(\hat{Y}^{(n)},T_{\bZ^{(n)}}(Y^{(n)}))\rbrace.
\end{equation*}
Then it is easy to see that $(Y^{(n)}, \tY^{(n)})$ with 
\begin{align}
    &\tY_n = T_{\bZ_n}(Y^{(n)})  \mathbbm{1}_{A^{(n)}}  +  \hat{Y}^{(n)}  \mathbbm{1}_{(A^{(n)})^c} \notag
\end{align}
is a coupling of $P_\infty^n(x, \cdot)$ and $P_\infty^n(\tx, \cdot)$. 

Next, define the event for which at at each iteration $j =1, \ldots, n$, the variable $\tY^{(j)}$ is exactly $T_{\bZ^{(j)}}(Y^{(j)})$, the shift of $Y^{(j)}$, namely
     \begin{equation*}
         \Lambda^{(n)}  := \bigcap_{j=1}^n A^{(j)}.
     \end{equation*}
    Thanks to these definitions, it follows
    \begin{align*}
        W_{d_\varepsilon}( P_\infty^n(x, \cdot ), P_\infty^n(\tx, \cdot)) &\leq 
        \E d_\eps(Y^{(n)}, \tY^{(n)}) \\
         &\leq \E d_\eps(Y^{(n)},T_{\bZ^{(n)}}(Y^{(n)}))\mathbbm{1}_{\Lambda^{(n)} } + 1 - \bbP(\Lambda^{(n)}) \notag\\
        & \leq  1 - \bbP(\Lambda^{(n)} )\left( 1 - d_\eps(\rho_1^n\rho_2^n x, \rho_1^n\rho_2^n \tx)\right)\notag
    \end{align*}
    and, as $x$ and $\tx$ are assumed to be in the small set $S\subset B(0, r_s)$, 
     \begin{align}
          W_{d_\varepsilon}( P_\infty^n(x, \cdot ), P_\infty^n(\tx, \cdot)) \leq  1 - \bbP(\Lambda^{(n)} )\left( 1 - \frac{\rho_1^n\rho_2^n r_s}{\varepsilon}\right) \label{eq:W_small_infty}. 
     \end{align}
    As last step we have to ensure that $\Lambda^{(n)}$ has positive probability:
\begin{align*}
    \bbP(\Lambda^{(n)} )& = \bbP\left(A^{(1)}\right) \prod_{k=2}^n \bbP\left( A^{(k)}\, \left|\, \bigcap_{j=1}^{k-1} A^{(j)} \right.\right) \\
    &=\exp(-2L\rho_1 \rho_2 r_s) \prod_{k=2}^n \exp(-2L\rho_1^{n}\rho_2^n r_s) = \exp\left(-2Lr_s \sum_{k=1}^n\rho_1^{k}\rho_2^k\right)  \\
    &= \exp\left(-2Lr_s \left( \frac{1 - (\rho_1\rho_2)^{n+1}}{1 - \rho_1\rho_2}-1\right)\right)\geq \exp\left(-2Lr_s\rho_1\rho_2(1 - \rho_1^n\rho_2^n)\right).
\end{align*}
Therefore 
\begin{align*}
        W_{d_\varepsilon}( P_\infty^n(x, \cdot ), P_\infty^n(\tx, \cdot)) &\leq  1 - \exp\left(-2Lr_s\rho_1\rho_2(1 - \rho_1^n\rho_2^n)\right) \left( 1 - \frac{\rho_1^n\rho_2^n r_s}{\varepsilon}\right) =:s
    \end{align*}
   and, choosing $n > \frac{\log\varepsilon - \log r_s}{ \log \rho_1\rho_2}=:n^*$, then $s<1$.
\end{proof}

\subsection{Proofs for Multiple-Try pCN}\label{sec:proof_mtm}
In this section we will prove the results \cref{thm:3:mpcn1} for \cref{alg:mtpcn} and \cref{thm:3:inftyMTpcn} for the limit kernel $P_\infty$ \eqref{eq:2:mtminfty:kernel} of the multiple-try pCN. Again we do so ensuring all the conditions of the weak Harris theorem are fulfilled in a sequence of propositions. Given the structure of the kernels, the proofs will take elements from both the results for mpCN and $\infty$-pCN.

\subsubsection{Finite number of proposals}\label{sec:proof_mtm_finite}
\begin{Proposition}\label{prop:mtm:lyap}
     Given \cref{alg:mtpcn} with Markov kernel $P_p$ as in \eqref{eq:2:mtpcn:kernel} for fixed $p\geq 1$, assume the potential $\Pot:\qsp \to \R$ is globally bounded. Then the functions $V(x) = \|x\|^n$, $n \in \N$, $V(x) = \exp(v\|x\|)$, $v>0$ and $V(x) = \exp(v\|x\|^2)$ for $v$ small enough, are Lyapunov functions, as in \cref{def:lyapunov}, for the MTpCN Markov kernel $P_p$ \eqref{eq:2:mtpcn:kernel} with constants $l_V$ and $K_V$ independent of the number of proposals $p$.
\end{Proposition}

\begin{proof}
    Using the formulation \eqref{eq:mtpcn:X1} of the first step of the chain $X^{(1)}$ generated by $P_p$, with $\bX(p) = (X_1, \ldots, X_p)$ and $\bZ(p-1) = (Z_1, \ldots, Z_{p-1})$, we have
    \begin{align*}
        PV(x_0) &= \E V(X^{(1)}) = \E V(Y) \mathbbm{1}_{U \leq \balpha(x_0,\bX(p),\bZ(p-1))} +   V(x_0)\bbP(U > \balpha(x_0,\bX(p),\bZ(p-1)))\\
        &= \sum_{j = 1}^p \E V(X_j)  \mathbbm{1}_{\tU \in I_j}  \mathbbm{1}_{U\leq \balpha(x_0, \bX(p),\bZ(p-1))} + V(x_0)\left( 1 - \E \balpha(x_0, \bX(p),\bZ(p-1))\right),
    \end{align*}
    where $X_j = \rho x_0 + \sqrt{1 - \rho^2} \xi_j$, $j =1, \ldots, p$.
    If $V(x) = \|x\|^n$, $n\in \N$, we use as usual the bound $V(X_j) \leq (1 + \delta) \rho^n V(x_0) + C_\delta \|\xi_j\|^n$ for an arbitrary $\delta>0$, to get 
    \begin{align}
        PV(x_0) &\leq (1 + \delta) \rho^n V(x_0) \sum_{j= 1}^p \E  \mathbbm{1}_{U \in I_j}  \mathbbm{1}_{\tU\leq \balpha} +  C_\delta \sum_{j= 1}^p \E \|\xi_j\|^n \mathbbm{1}_{U \in I_j}  \mathbbm{1}_{\tU\leq \balpha} + V(x_0)\left( 1 - \E \balpha \right) \notag\\
        &= V(x_0) \left[ 1 - \left(1 - \rho^n(1 + \delta)\right) \E \balpha \right] + C_\delta \sum_{j= 1}^p  \E \|\xi_j\|^n \mathbbm{1}_{U \in I_j}  \mathbbm{1}_{\tU\leq \balpha}.\label{eq:mtm:lyap1}
    \end{align}
    As $\Pot$ is assumed bounded then, the acceptance probabilities $\balpha$ and $\beta_j$, $j = 1, \ldots, p$, defined in \eqref{eq:2:mtpcn:balpha} and \eqref{eq:2:mtpcn:beta}, are bounded as 
    \begin{equation*}
        \balpha \geq e^{-2\|\Pot\|_\infty} \quad \text {and}\quad \beta_j\leq p^{-1}e^{2\|\Pot\|_\infty}.
    \end{equation*}
    Then, for $\rho>0$, we set the arbitrary parameter $\delta$ to be such that $(1 + \delta) \rho^n<1$, namely $\delta < \rho^{-n}-1$, so that 
    \begin{equation*}
        1 - \left(1 - \rho^n(1 + \delta)\right) \E \balpha \leq 1 - \left(1 - \rho^n(1 + \delta)\right) e^{-2\|\Pot\|_\infty} =:l_V. 
    \end{equation*}
    Last we treat the second term in \eqref{eq:mtm:lyap1}
    \begin{align*}
        C_\delta \sum_{j =1}^p \E \|\xi_j\|^n \mathbbm{1}_{U \in I_j}  \mathbbm{1}_{\tU\leq \balpha} = \sum_{j =1}^p \E \E\left[\|\xi_j\|^n \mathbbm{1}_{U \in I_j}  \mathbbm{1}_{\tU\leq \balpha} \right|\left. \xi_j \right]\\
         = C_\delta \sum_{j = 1}^p \E \|\xi_j\|^n \beta_j(\xi_1, \ldots, \xi_p)
         \leq C_\delta e^{2 \|\Pot\|_\infty} \E \|\xi_1\|^n =: K_V,
    \end{align*}
    where we used that $\xi_j$ are i.i.d. Gaussians. Therefore
    \begin{equation*}
        PV(x_0)\leq l_V V(x_0) + K_V. 
    \end{equation*}

    With a similar argument we can show that $V(x) = \exp(v\|x\|)$ for any $v>0$, and $V(x) = \exp(v\|x\|^2)$, for $v$ small enough, are Lyapunov functions using bounds similar to those used in \cref{prop:mpcn:lyap}. 
    In these cases, then, setting $q\in (1, \rho^{-1})$ and $q \in (1, (1 + \delta)\rho^{-2})$, with $\delta< \rho^{-2} - 1$, we have 
    \begin{equation*}
        l_V = 1 - \left(1 - \frac{1}{q}\right) e^{-2\|\Pot\|_\infty} 
    \end{equation*}
    and 
    \begin{equation*}
        K_V = \frac{q-1}{q} e^{2 \|\Pot\|_\infty} \E \exp\left( \frac{q}{q-1}v\sqrt{1 - \rho^2}\|\xi_1\|\right) 
    \end{equation*}
    or, for $v$ appropriately small to use Fernique's theorem,
     \begin{equation*}
        K_V = \frac{q-1}{q} e^{2 \|\Pot\|_\infty} \E \exp\left( C_\delta \frac{q}{q-1}v\sqrt{1 - \rho^2}\|\xi_1\|^2\right), \quad C_\delta = 1 + \frac{1}{4\delta}, 
    \end{equation*}
    for the two functions respectively. For $\rho= 0$ the argument follows in a similar and even simpler way as the proposal are independent draws from the reference gaussian $X_j = \xi_j\sim \mu_0$.
\end{proof}

In the next proposition we establish contraction with respect to the distance-like function $d_\varepsilon$, defined in \eqref{eq:d_eps}, making use of the coupling constructed for mpCN in \cref{prop:contr:fin:p:1}.
\begin{Proposition}\label{prop:mtm:contr:fin:p:1}
	Given \cref{alg:mtpcn} with Markov kernel $P_p$ as in \eqref{eq:2:mtpcn:kernel} for fixed $p\geq 1$, assume the potential function $\Pot: \qsp \to \R$ is bounded and globally Lipschitz with Lipschitz constant $L_{\Pot}$. Fix $\varepsilon>0$, then, for every $x_0, \tx_0 \in \qsp$ with $d_\eps(x_0, \tx_0) < 1$, we have
	\begin{align}\label{ineq:mtm:contr:fin:p:a} 
		\Wass_{d_\varepsilon}(\mk_p(x_0, \cdot), \mk_p(\tx_0, \cdot)) 
		\leq  \kappa d_\eps(x_0, \tx_0),
	\end{align}
	with $\kappa=\kappa (\varepsilon, \rho , p, \|\Pot\|_\infty, L_\Pot) $ as in \eqref{eq:mtm:kappa}. Moreover, $\kappa \in (0,1)$ whenever $0 < \varepsilon < \frac{(1 - \rho)c_1}{\rho c_2 + C_p}$, with $C_p = c_3 + \frac{c_4}{p} + \frac{c_5(p-1)}{p}$ for some positive constants $c_i = c_i(\|\Pot\|_\infty, L_\Pot)$.
\end{Proposition}

\begin{proof}
   Fix $x_0, \tx_0 \in \qsp$ with $d_\eps(x_0, \tx_0) < 1$. We want to find a coupling of $P(x_0, \cdot)$ and $P(\tx_0, \cdot)$. We will couple the four stages of the algorithm either synchronously or taking advantage of the couplings introduced for the multiproposal algorithm in \cref{prop:contr:fin:p:1}. 

    For the first $p$ proposals we consider a synchronous coupling, namely, given $\Xi = (\xi_1, \ldots, \xi_p)\sim\mu_0^{\otimes p}$, the proposals are $X_j = F(x_0, \xi_j)$ and $\tX_j = F(\tx_0, \xi_j)$ with $F$ as in \eqref{eq:2:defF}. We denote the minimum acceptance probability for the $j$-th proposals
    \begin{equation*}
        \hat{\beta}_j(x_0, \tx_0, \Xi) =  \beta_j(X_1, \ldots, X_p) \wedge \beta_j(\tX_1, \ldots, \tX_p)
    \end{equation*}
    and 
    \begin{equation*}
       s_0 = 0, \quad  s_j= s_j (x_0, \tx_0, \Xi) = \sum_{i=1}^j \hat{\beta}_i(x_0, \tx_0, \Xi), \quad j = 1, \ldots, p.
    \end{equation*}
    Over the interval $[s_{j-1}, s_j]$ both chains accept the $j$-th element of the cloud as preliminary steps $Y$, $\tY$. In the interval $[s_p, 1]$ we consider the same construction as in \eqref{eq:mpcn:[sp,1]}. 
    Then, given $U\sim \cU([0,1])$ we define
\begin{align}\label{eq:mtm:(Y,tY)}
	(Y, \tY)  = \begin{cases}
        (X_j, \tX_j) \quad &\mbox{ if } U \in [s_{j-1}, s_j],\; j = 1, \ldots, p\\
	    (X_{k}, \tX_{l})
        ; &\mbox{ if } U \in J_k \cap \tJ_l \mbox{ with } k,l = 1, \ldots,p \mbox{ s.t. }  J_k \cap \tJ_l \neq \emptyset.
	\end{cases} 
\end{align}
Next, to couple the auxiliary variables $Z_j$, $\tZ_j$, $j = 1, \ldots, p-1$ we use a synchronous coupling $Z_j = F(Y, \txi_j)$ and $\tZ_j = F(\tY, \txi_j)$,
where $\tXi = (\txi_1, \ldots, \txi_{p-1})~\mu_0^{\otimes (p-1)}$ is independent of $\Xi$ and $U$.
To ease the notation we write $ \balpha(x_0,X_1, \ldots, X_p, Z_1, \ldots, Z_{p-1}) = \balpha(x_0, Y, \Xi, \tXi)$ and set
 \begin{align*}
     \har(Y, \tY, \Xi,  \tXi) := \balpha(x_0, Y, \Xi, \tXi) \wedge  \balpha(\tx_0, \tY, \Xi, \tXi).
 \end{align*}
Finally, given $\tU \sim \cU([0,1])$ independent of $U$, $\Xi$ and $\tXi$, the desired coupling of the $P_p$ kernels is 
 \begin{equation}\label{eq:mtm:coupling}
        (X, \tX) = \begin{cases}
            (Y, \tY) \quad &\text{if } \tU \leq \har(Y, \tY, \Xi,  \tXi) \\
            (x, \tx)  \quad &\text{if } \tU \geq \balpha(x_0, Y, \Xi, \tXi) \vee \balpha(\tx_0, \tY, \Xi, \tXi)\\
            (x, \tY)   \quad &\text{if } \balpha(x_0, Y, \Xi, \tXi)  \leq \tU \leq  \balpha(\tx_0, \tY, \Xi, \tXi) \\
            (Y, \tx)  \quad &\text{if } \balpha(\tx_0, \tY, \Xi, \tXi) \leq \tU \leq \balpha(x_0, Y, \Xi, \tXi).
        \end{cases}
    \end{equation}

 Recalling that $d_\eps\leq 1$ we can then write 
 \begin{equation}\label{eq:mtm:contr1}
     \E d_\eps(X, \tX) \leq \E d_\eps(Y, \tY) \mathbbm{1}_{\tU \leq \har(Y, \tY, \Xi,  \tXi) }+ d_\eps(x, \tx) \left( 1 - \bbP(\tU \leq \har(Y, \tY, \Xi,  \tXi) ) \right) + \E |\balpha(x_0, Y, \Xi, \tXi) - \balpha(\tx_0, \tY, \Xi, \tXi)| 
 \end{equation}
 We start analyzing the first term on the right hand side. By definition \eqref{eq:mtm:(Y,tY)} we have
 \begin{align*}
     \E d_\eps(Y, \tY) \mathbbm{1}_{\tU \leq \har(Y, \tY, \Xi,  \tXi) }& \leq \E \sum_{j =1}^p d_\eps(X_j, \tX_j) \mathbbm{1}_{\tU \in [s_{j-1}, s_j)}\mathbbm{1}_{\tU \leq \har(Y, \tY, \Xi,  \tXi)} + \E \mathbbm{1}_{U \in [s_p,1]}\mathbbm{1}_{\tU \leq \har(Y, \tY, \Xi,  \tXi)} \\
      &=  d_\eps(\rho x_0, \rho \tx_0)\E \mathbbm{1}_{U \in [0, s_p)}\mathbbm{1}_{\tU \leq \har(Y, \tY, \Xi,  \tXi)} + \E \mathbbm{1}_{U \in [s_p,1]}\mathbbm{1}_{\tU \leq \har(Y, \tY, \Xi,  \tXi)} .
 \end{align*}
 Then if $\mathcal{F}_0$ is the $\sigma$-algebra generated by $(Y, \tY)$
 \begin{align}
     \E \mathbbm{1}_{U \in [s_p,1]}\mathbbm{1}_{\tU \leq \har(Y, \tY, \Xi,  \tXi)} &= \E \E \left[\mathbbm{1}_{U \in [s_p,1]}\mathbbm{1}_{\tU \leq \har(Y, \tY, \Xi,  \tXi)}\left| \mathcal{F}_0\right.\right] \notag\\
     &= \E \left[ \mathbbm{1}_{\tU \leq \har(Y, \tY, \Xi,  \tXi)} \E \left[ \mathbbm{1}_{U \in [s_p,1]} \left| \mathcal{F}_0\right.\right]\right] \notag\\
      &= \E (1 - s_p) \har(Y, \tY, \Xi,  \tXi).\label{eq:mtm:EE}
 \end{align}
 Similarly to \eqref{ineq:sp}
 \begin{equation*}
     1 - s_p = 2\frac{\sum_{k = 1}^p|e^{-\Pot(X_k)} - e^{- \Pot(\tX_k)}| }{\sum_{j= 1}^p e^{-\Pot(X_k)}}
 \end{equation*}
 and for bounded and globally Lipschitz $\Pot$, thanks to \eqref{ineq:num}-\eqref{ineq:den}, it follows 
 \begin{equation}\label{eq:mtm:sp}
     1 - s_p \leq  C_1\rho \|x_0 - \tx_0\| = \varepsilon C_1 \rho d_\eps(x_0,\tx_0)
 \end{equation}
 for some constant $C_1 = C_1(\|\Pot\|_\infty, L_\Pot)>0$. Therefore we obtained 
 \begin{align}
      \E d_\eps(Y, \tY) \mathbbm{1}_{\tU \leq \har(Y, \tY, \Xi,  \tXi) } &\leq d_\eps(\rho x_0, \rho \tx_0)\E \mathbbm{1}_{U \in [0, s_p)}\mathbbm{1}_{\tU \leq \har(Y, \tY, \Xi,  \tXi)} + \varepsilon C_1 \rho d_\eps(x_0,\tx_0)\E \har(Y, \tY, \Xi,  \tXi)  \notag\\
      &\leq \rho d_\eps( x_0, \tx_0)\E \har(Y, \tY, \Xi,  \tXi) \left( 1 + \varepsilon C_1 \right) \label{eq:mtm:contr2}
 \end{align}
 where we have also used that $d_\eps(x_0, \tx_0)<1$. 

 We are left to show that $ \balpha$ is Lipschitz in $x_0$ to bound the last term in  \eqref{eq:mtm:contr1}:
 \begin{align*}
     \E |\balpha(x_0, Y, \Xi, \tXi)& - \balpha(\tx_0, \tY, \Xi, \tXi)| \\
     &= \E \left|  1\wedge \frac{\sum_{l=1}^p e^{-\Pot(X_l)}}{e^{-\Pot(x_0)} + \sum_{l = 1}^{p-1}e^{-\Pot(Z_l)}} - 1\wedge \frac{\sum_{l=1}^pe^{-\Pot(\tX_l)}}{e^{-\Pot(\tx_0)} + \sum_{l = 1}^{p-1}e^{-\Pot(\tZ_l)}}\right|\\
     &\leq   \E \left| \frac{\sum_{l=1}^pe^{-\Pot(X_l)}}{e^{-\Pot(x_0)} + \sum_{l = 1}^{p-1}e^{-\Pot(Z_l)}} -  \frac{\sum_{l=1}^p e^{-\Pot(\tX_l)}}{-e^{-\Pot(\tx_0)} + \sum_{l = 1}^{p-1}e^{-\Pot(\tZ_l)}}\right|\\
      &= \E \left| \frac{\left(e^{-\Pot(\tx_0)} + \sum_{l = 1}^{p-1}e^{-\Pot(\tZ_l)}\right)\sum_{l=1}^pe^{-\Pot(X_l)} - \left( e^{-\Pot(x_0)} + \sum_{l = 1}^{p-1}e^{-\Pot(Z_l)}\right) \sum_{l=1}^pe^{-\Pot(\tX_l)}}{\left( e^{-\Pot(x_0)}+ \sum_{l = 1}^{p-1}e^{-\Pot(Z_l)} \right)\left( e^{-\Pot(\tx_0)} + \sum_{l = 1}^{p-1}e^{-\Pot(\tZ_l)}\right) }\right|
 \end{align*}
 and with some simple manipulations
 \begin{equation*}
     \leq \E \sum_{l=1}^p\frac{|e^{-\Pot(X_l)} - e^{-\Pot(\tX_l)}|}{e^{-\Pot(\tx_0)}
     + \sum_{l = 1}^{p-1}e^{-\Pot(\tZ_l)}} 
     + \E \frac{|e^{-\Pot(x_0)} - e^{-\Pot(\tx_0)}| + \sum_{k =1}^{p-1} |e^{-\Pot(Z_k)} -e^{-\Pot(\tZ_k)}|}{\left( e^{-\Pot(x_0)} + \sum_{l = 1}^{p-1}e^{-\Pot(Z_l)} \right)\left( e^{-\Pot(\tx_0)}+ \sum_{l = 1}^{p-1}e^{-\Pot(\tZ_l)}\right) } \sum_{l=1}^p e^{-\Pot(X_l)}.
 \end{equation*}
 Again, by \eqref{ineq:num}-\eqref{ineq:den}, it follows 
 \begin{multline*}
       \E |\balpha(x_0, Y, \Xi, \tXi) - \balpha(\tx_0, \tY, \Xi, \tXi)| 
       \leq 
       e^{2\|\Pot\|_\infty} L_\Pot \rho\|x_0 - \tx_0\| 
       + \frac{1}{p}e^{4\|\Pot\|_\infty} L_\Pot \|x_0 - \tx_0\| 
       \\
       + \frac{e^{2\|\Pot\|_\infty}}{p} \sum_{l=1}^{p-1} \E |\Pot(Z_l) - \Pot(\tZ_l)|.
 \end{multline*}
Next as the $Z_l, \, \tZ_l$ are synchronous and $\Pot$ is bounded
 \begin{align*}
      \E |\Pot(Z_l) - \Pot(\tZ_l)| &=  \E |\Pot(Z_l) - \Pot(\tZ_l)|\sum_{j=1}^p \mathbbm{1}_{U \in [s_{j-1}, s_j)} + \E |\Pot(Z_l) - \Pot(\tZ_l)|\mathbbm{1}_{U \in [s_p, 1]}\\
      &\leq L_\Pot\E \|Z_l - \tZ_l \|\sum_{j=1}^p \mathbbm{1}_{U \in [s_{j-1}, s_j)} + \|\Pot\|_\infty\E \mathbbm{1}_{U \in [s_p, 1]}\\
       &= L_\Pot\rho\E \|Y - \tY \|\sum_{j=1}^p \mathbbm{1}_{U \in [s_{j-1}, s_j)} + \|\Pot\|_\infty\E [1 - s_p]\\
        &\leq  L_\Pot\rho^2\|x_0 - \tx_0\|\E s_p + \|\Pot\|_\infty C_1\rho \|x_0 - \tx_0\|,
 \end{align*}
 where we have used \eqref{eq:mtm:sp} again. By construction we know that $\E s_p<1$, hence
 \begin{equation}\label{eq:mtm:diffalpha}
     \E |\balpha(x_0, Y, \Xi, \tXi) - \balpha(\tx_0, \tY, \Xi, \tXi)| 
       \leq C_p \varepsilon d_\eps(x_0, \tx_0)
 \end{equation}
 with 
 \begin{equation*}
     C_p = e^{2\|\Pot\|_\infty} L_\Pot\rho
       + \frac{1}{p}e^{4\|\Pot\|_\infty} L_\Pot 
       + \frac{e^{2\|\Pot\|_\infty}}{p} (p-1) L_\Pot\rho^2 + \|\Pot\|_\infty C_1\rho .
 \end{equation*}

 Putting together \eqref{eq:mtm:contr1},\eqref{eq:mtm:contr2} and \eqref{eq:mtm:diffalpha} we showed 
\begin{equation*}
     \E d_\eps(X, \tX) \leq  d_\eps(x_0, \tx_0) \left[ 1 - \left( 1 - \rho \left( 1 + \varepsilon C_1 \right)\right) \E \har(Y, \tY, \Xi,  \tXi)\right] + C_p \varepsilon d_\eps(x_0, \tx_0)
\end{equation*}
so that
 \begin{equation}
     W_{d_\varepsilon}(P(x_0, \cdot), P(\tx_0, \cdot))\leq \kappa d_\eps(x_0, \tx_0)
 \end{equation}
 with 
 \begin{equation}\label{eq:mtm:kappa}
     \kappa = 1 - \left( 1 - \rho \left( 1 + \varepsilon C_1 \right)\right) \E \har(Y, \tY, \Xi,  \tXi) + C_p\varepsilon
 \end{equation}
 which stays smaller than one for appropriate choice of $\varepsilon$, namely 
 \begin{equation*}
     \varepsilon < \frac{(1 - \rho)\E \har(Y, \tY, \Xi,  \tXi)}{\rho C_1 \E \har(Y, \tY, \Xi,  \tXi) + C_p} 
 \end{equation*}
 Note that this choice is well defined as by definition \eqref{eq:2:mtpcn:balpha}, $e^{-2\|\Pot\|_\infty}\leq \E \har(Y, \tY, \Xi,  \tXi)\leq e^{2\|\Pot\|_\infty}$
 so it is enough to consider 
 \begin{equation*}
     \varepsilon\leq \frac{(1 - \rho)e^{-2\|\Pot\|_\infty}}{\rho C_1 e^{2\|\Pot\|_\infty} + C_p}.
 \end{equation*}

In summary, the constants $c_i$, $i =1, \ldots, 5$ in the statement of the theorem are defined as $
    c_1 = e^{-2\|\Pot\|_\infty}$, $c_2 = C_1 e^{2\|\Pot\|_\infty}$,
    $c_3 = e^{2\|\Pot\|_\infty} L_\Pot\rho + \|\Pot\|_\infty C_1\rho$, $c_4 =  e^{4\|\Pot\|_\infty} L_\Pot$, $c_5 = e^{2\|\Pot\|_\infty} L_\Pot\rho^2$.
\end{proof}

Last we show that $P_p$ satisfies condition 2 in \cref{thm:our_harris} and it is $d_\varepsilon$-small after a large enough number of iterations.

\begin{Proposition}\label{prop:mtm:small}
    Given \cref{alg:mtpcn} with Markov kernel $P_p$ as in \eqref{eq:2:mtpcn:kernel} for fixed $p\geq 1$, assume the potential function $\Pot: \qsp \to \R$ is bounded and globally Lipschitz. 
    Define $S= \lbrace V(x) \leq 4 K_V\rbrace $ with $V$ Lyapunov function as in \cref{prop:mtm:lyap}, and let $r_s>0$ be such that $S \subset B(0, r_s)$. Then, for every $x_0, \tx_0\in S$, it holds
    \begin{equation*}
        \Wass_{d_\varepsilon} (P_p^n(x_0, \cdot), P_p^n(\tx_0, \cdot))\leq  1 - e^{-4n\|\Pot\|_\infty}\left(1 - \frac{2 r_s \rho^n}{\varepsilon}\right)\quad n \in \N.
    \end{equation*}
    Consequently, for any $n >\frac{ \log 2r_s - \log \varepsilon}{- \log \rho}$ and every $x_0, \tx_0\in S$, it holds
    \begin{equation*}
        W_{d_\varepsilon}(P_p^{n}(x_0, \cdot), P_p^{n}(\tx_0, \cdot)) \leq s
    \end{equation*}
    with $s = s(\varepsilon, \rho, n)\in (0,1)$.
\end{Proposition}
\begin{proof}
    Fix $x_0$, $\tx_0\in S$. As usual, we construct a coupling for the iterated kernels $P_p^n(x_0, \cdot)$, $P_p^n(\tx_0, \cdot)$, $n \in \N$. For $n = 1$ we consider the same coupling construction used in the $d_\eps$-contraction proof, namely \eqref{eq:mtm:coupling}, so that
    \begin{equation*}
        W_{d_\varepsilon}(P_p(x, \cdot), P_p(\tx_0,\cdot)) \leq \E d (X^{(1)}, \tX^{(1)})  \\
        \leq \E d_\eps(Y^{(1)}, \tY^{(1)}) \mathbbm{1}_{\tU^{(1)}\leq \har^{(1)}} + 1 - \bbP(\tU^{(1)}\leq \har^{(1)})
    \end{equation*}
    where $\har^{(1)} = \har(Y^{(1)}, \tY^{(1)}, \xi^{(1)}, \txi^{(1)})$ and we used that $d_\eps\leq 1$. Next, we recall the definition \eqref{eq:mtm:(Y,tY)} of $(Y^{(1)},\tY^{(1)})$ to get 
    \begin{align*}
        W_{d_\varepsilon}(P_p(x, \cdot), P_p(\tx_0,\cdot))  &\leq \E d_\eps(Y^{(1)},\tY^{(1)}) \mathbbm{1}_{U^{(1)}\leq s_p^{(1)}}\mathbbm{1}_{U\leq \har^{(1)}}  +  \E \mathbbm{1}_{U^{(1)}\geq s_p^{(1)}}\mathbbm{1}_{U\leq \har^{(1)}}  + 1 - \bbP(U\leq \har^{(1)})\\
        &\leq d_\eps(\rho x_0, \rho \tx_0) \E  \mathbbm{1}_{U^{(1)}\leq s_p^{(1)}}\mathbbm{1}_{U\leq \har^{(1)}} + 1 -  \E \mathbbm{1}_{U^{(1)}\geq s_p^{(1)}}\mathbbm{1}_{\tU^{(1)}\leq \har^{(1)}}
    \end{align*}
    where $s_p^{(1)} = \sum_{k = 1}^p\hat{\beta}_k^{(1)}$, with $\hat{\beta}_k^{(1)} = \hat{\beta}_k(x_0, \tx_0, \xi^{(1)})$. Since $x_0$, $\tx_0\in S\subset B(0, r_S)$, then $d_\eps(\rho x_0, \rho \tx_0)\leq \frac{2r_s \rho}{\varepsilon}$ and we obtained 
    \begin{equation*}
         W_{d_\varepsilon}(P_p(x, \cdot), P_p(\tx_0,\cdot)) \leq 1 - \left( 1 - \frac{2r_s \rho}{\varepsilon} \right) \E  \mathbbm{1}_{U^{(1)}\leq s_p^{(1)}}\mathbbm{1}_{U\leq \har^{(1)}}.
    \end{equation*}
    Last, with the double condition argument (see e.g. \eqref{eq:mtm:EE}) to get
    \begin{equation*}
        \E  \mathbbm{1}_{U^{(1)}\leq s_p^{(1)}}\mathbbm{1}_{U\leq \har^{(1)}} = \E s_p^{(1)}\har^{(1)} >  e^{-4\|\Pot\|_\infty},
    \end{equation*}
    since $\Pot$ is assumed bounded.

    Now we iterate the argument just carried on setting 
    \begin{align*}
    \begin{array}{ll}
          X_j^{(n)} = F(X^{(n-1)}, \xi_j^{(n)}) \quad j =1, \ldots, n, &  \;Y^{(n)} = \sum_{j = 1}^p X_j^{(n)} \mathbbm{1}_{U^{(n)}\leq \beta_j^{(n)}}  \\
         Z_k^{(n)} = F(Y^{(n)},\tXi_k^{(n)}) \quad k = 1, \ldots, n-1, & \;X^{(n)} = Y^{(n)} \mathbbm{1}_{\tU^{(n)}\leq \har^{(n)}} + X^{(n-1)}\mathbbm{1}_{\tU^{(n)}>\har^{(n)}},
    \end{array}
    \end{align*}
    with $X^{(0)} = x_0$, and similarly for the second chain, with same uniform variables $\tU^{(n)}$, $U^{(n)}$ and Gaussian variables $\Xi^{(n)}$, $\tXi^{(n)}$. Moreover, we define 
    \begin{align*}
        A^{(n)} = \bigcap_{j =1}^n \left\lbrace \tU^{(j)}\leq \har^{(j)}\right\rbrace, \quad B^{(n)} = \bigcap_{j =1}^n \left\lbrace U^{(j)}\leq s_p^{(j)}\right\rbrace 
    \end{align*}
    respectively the event when the $Y$ and $\tY$ are accepted $n$ times in a row in the last step of the algorithm and that when $Y$ and $\tY$ are made of the proposals with the same index $n$ times in a row (not necessarily the same index at each iteration). Then, when we are in $ A^{(n)}\cap B^{(n)}$ we have
    \begin{equation*}
        d_\eps(X^{(n)}, \tX^{(n)}) = d_\eps(Y^{(n)}, \tY^{(n)}) = d_\eps(\rho X^{(n-1)}, \rho X^{(n-1)}) = \ldots = d_\eps(\rho^n x_0, \rho^n x_0) \leq \frac{2 r_s \rho^n}{\varepsilon}.
    \end{equation*}
    It follows 
    \begin{align*}
          W_{d_\varepsilon}(P^n_p(x, \cdot), P_p^n(\tx_0,\cdot)) &\leq \E   d_\eps(X^{(n)}, \tX^{(n)}) \leq  \E  d_\eps(X^{(n)}, \tX^{(n)}) \mathbbm{1}_{ A^{(n)}\cap B^{(n)}} + 1 - \bbP( A^{(n)}\cap B^{(n)})\\
          &\leq 1 - \left(1 - \frac{2 r_s \rho^n}{\varepsilon}\right) \bbP( A^{(n)}\cap B^{(n)}).
    \end{align*}
    Taking $n$ so that $2 r_s \rho^n< \varepsilon$, we have the desired result, as long as $\bbP\left( A^{(n)}\cap B^{(n)}\right)$ is positive. This follows again the boundedness of the potential $\Pot$ as 
    \begin{equation*}
        \bbP( A^{(n)}\cap B^{(n)}) = \bbP( A^{(1)}\cap B^{(1)})\prod_{k =2}^n \bbP\left( A^{(k)}\cap B^{(k)} \left| \bigcap_{j=1}^{k-1}  A^{(j)}\cap B^{(j)}\right.\right) > e^{-4n\|\Pot\|_\infty}.
    \end{equation*}
    so $ W_{d_\varepsilon}(P(x, \cdot)^n, P(\tx_0,\cdot)^n) \leq s$ with $s = s(n) = 1 - e^{-4n\|\Pot\|_\infty}\left(1 - \frac{2 r_s \rho^n}{\varepsilon}\right)$ for any $x, \tx_0\in S$.
\end{proof}

\subsubsection{Infinite number of proposals}\label{sec:proof_mtm_infinite}

In this section we provide the proof of \cref{thm:3:inftyMTpcn} by showing that the limiting kernel $P_\infty$ of the Multiple Try pCN algorithm, defined as
\begin{equation}\label{eq:mtminfty:kernel}
    P_\infty(x, dy) = \alpha(x, y)\baQ(x, dy) + \delta_{x}(dy)\int \left(1 - \alpha(x, u)\right) \baQ(x, du)
\end{equation}
with 
\begin{equation}\label{eq:mtminfty:alpha}
    \alpha(x, y) = 1 \wedge \frac{\int e^{-\Pot(z)}Q(x, dz)}{\int e^{-\Pot(u)}Q(y, du)}, \quad \baQ(x, dy) = \frac{e^{-\Pot(y)}}{\int e^{- \Pot(u)}Q(x, du)} Q(x, dy)
\end{equation}
fulfills all the conditions of the weak Harris theorem \cref{thm:our_harris}.
\begin{Proposition}
      Assume $\Pot:\qsp\to \R$ is globally bounded. Then the functions $V(x) = \|x\|^n$, $n \in \N$, $V(x) = \exp(v\|x\|)$, $v>0$ and $V(x) = \exp(v\|x\|^2)$ for $v$ small enough, are Lyapunov functions, as in \cref{def:lyapunov}, for the Markov kernel $P_\infty$ \eqref{eq:mtminfty:kernel}.
\end{Proposition}
\begin{proof} We want to establish the existence of $l_V<1$ and $K_V>0$ so that $PV(x)\leq l_V V(x) + K_V$ for all $x\in \qsp$ where 
    \begin{align}\label{mtm:lyap1}
        PV(x) = \int V(y)\alpha(x,y)\baQ(x, dy) + V(x) \int 1 - \alpha(x,z)\, \baQ(x, dz) .
    \end{align}
    Recall that $F(x, w)= \rho x + \sqrt{1 - \rho^2} w$ and $Q(x, \cdot ) = F(x, \cdot)^* \mu_0$, so for the first term we have 
  \begin{align*}
      \int V(y)\alpha(x,y)\baQ(x, dy) &= \int V(y)\alpha(x,y)\frac{e^{-\Pot(y)}}{\int e^{-\Pot(z)}Q(x, dz)}Q(x, dy)\\
      &= \int V(F(x, w))\alpha(x,F(x, w))\frac{e^{-\Pot(F(x, w))}}{\int e^{-\Pot(z)}Q(x, dz)}\mu_0(dw).
  \end{align*}
  If $V(x) = \|x\|^n$, then $V(F(x, w)) = \|\rho x + \sqrt{1 - \rho^2} w\|^n$ and, by Young's inequality, $V(F(x, w)) \leq  (1+\delta)\rho^n \|x\|^n + C(\delta)(1 - \rho)^{n/2} \|w\|^n$, for an arbitrary $\delta>0$. We can then write:
  \begin{multline}\label{mtm:lyap2}
      \int V(y)\alpha(x,y)\baQ(x, dy) \leq (1+\delta)\rho^n V(x) \int  \alpha(x,y)\baQ(x, dy) \\+ C(\delta, \rho)\int \|w\|^n \alpha(x,F(x, w))\frac{e^{-\Pot(F(x, w))}}{\int e^{-\Pot(z)}Q(x, dz)}\mu_0(dw).
  \end{multline}
  As $\Pot$ is bounded, then for any $x, y\in \qsp$ we have the bounds $\alpha(x,y)\geq e^{-2\|\Pot\|_\infty}$ and 
  \begin{align}\label{eq:mtminfty:alphabound}
     \alpha(x,y)\frac{e^{-\Pot(y)}}{\int e^{-\Pot(z)}Q(x, dz)}
     &= \left[1 \wedge \frac{\int e^{-\Pot(z)} Q(x, dz)}{\int e^{-\Pot(u)} Q(y, du)} \right] \frac{e^{-\Pot(y)}}{\int e^{-\Pot(z)}Q(x, dz)} \leq  \frac{e^{-\Pot(y)}}{\int e^{-\Pot(u)} Q(y, du)} \leq e^{2\|\Pot\|_\infty}.
  \end{align}
  As a consequence, from \eqref{mtm:lyap2} 
    \begin{equation*}
         \int V(y)\alpha(x,y)\baQ(x, dy) \leq (1+\delta)\rho^n V(x) \int  \alpha(x,y)\baQ(x, dy) + C(\delta, \rho, \|\Pot\|_\infty)\int \|w\|^n \mu_0(dw)
    \end{equation*}
and looking back at \eqref{mtm:lyap1} it follows
    \begin{equation}\label{eq:mtminf:lyap}
        PV(x) \leq V(x)\left \{ 1 - \left[1 -(1+\delta)\rho^n\right]\int  \alpha(x,y)\baQ(x, dy) \right\} + C(\delta, \rho, \|\Pot\|_\infty)\int \|w\|^n \mu_0(dw).
    \end{equation}
    For $\rho>0$ we than choose $\delta$ so that $1 -(1+\delta)\rho^n >0$ and use the fact that $\alpha(x,y)\geq e^{-2\|\Pot\|_\infty}$ to close the argument with constants
    \begin{align*}
        l_V = 1 - e^{-2\|\Pot\|_\infty}\left[1 -(1+\delta)\rho^n\right], \quad K_V = C(\delta, \rho, \|\Pot\|_\infty)\int \|w\|^n \mu_0(dw).
    \end{align*}

      With a similar argument we can show that $V(x) = \exp(v\|x\|)$ for any $v>0$, and $V(x) = \exp(v\|x\|^2)$ for $v$ small enough, are Lyapunov functions. We use the bounds 
    \begin{align*}
       \exp(v\|F(x,w)\|) &\leq \frac{1}{q}e^{v\|x\|}e^{v(q\rho - 1)\|x\|} + \frac{q-1}{q} e^{\frac{q}{q-1}v\sqrt{1 - \rho^2}\|w\|}\\
       \exp(v\|F(x,w)\|^2) &\leq \frac{1}{q}e^{v\|x\|^2}e^{v((1 + \delta)q\rho^2 - 1)\|x\|^2} + \frac{q-1}{q} e^{C_\delta\frac{q}{q-1}v\sqrt{1 - \rho^2}\|w\|^2}, 
    \end{align*}
    for any $q>1$ and $0<\delta<\rho^2 -1$ with $C_\delta =1 + (4\delta)^{-1}$. Then choosing $1<q<\rho^{-1}$ and $1 < q< ((1 + \delta)\rho)^{-1}$ respectively, the analogous of \eqref{eq:mtminf:lyap} are
    \begin{align*}
          PV(x) \leq V(x)\left \{ 1 - \left[1 - \frac{1}{q}\right]\int  \alpha(x,y)\baQ(x, dy) \right\} + C(q, \rho, \|\Pot\|_\infty)\int \exp\left(\frac{q}{q-1}v\sqrt{1 - \rho^2}\|w\|\right) \mu_0(dw)
        \end{align*}
    and for $v$ small enough to apply Fernique's theorem 
        \begin{align*}
        PV(x) \leq V(x)\left \{ 1 - \left[1 - \frac{1}{q}\right]\int  \alpha(x,y)\baQ(x, dy) \right\} + C(q, \rho, \|\Pot\|_\infty)\int \exp\left(C_\delta\frac{q}{q-1}v\sqrt{1 - \rho^2}\|w\|^2\right) \mu_0(dw).
    \end{align*}
    The desired result then follows again from the boundedness of $\Phi$, which provides the bound \eqref{eq:mtminfty:alphabound} for the acceptance probability, and from the fact that the reference measure $\mu_0$ is Gaussian. Note that for $\rho = 0$ the argument is similar and even simpler leading to 
    \begin{equation*}
        l_V = 1 - e^{-2\|\Phi\|_\infty}, \quad K_V = C(\|\Phi\|_\infty) \int V(w) \, \mu_0(dw).
    \end{equation*}
\end{proof}

\begin{Proposition}\label{prop:mtminfty:contr}
	Assume the potential function $\Pot: \qsp \to \R$ is globally Lipschitz with Lipschitz constant $L_{\Pot}$. Fix $\varepsilon>0$, then, for every $x, \tx \in \qsp$ with $d_\eps(x, \tx) < 1$, we have
	\begin{align}\label{ineq:contr:mtminfty:a} 
		\Wass_{d_\varepsilon}(\mk_\infty(x, \cdot), \mk_\infty(\tx_0, \cdot)) 
		\leq  d_\eps(x, \tx)\kappa(\varepsilon, \rho, \|\Pot\|_\infty, L_\Pot)
	\end{align}
    with $\kappa$ as in \eqref{eq:mtminfty:kappa}. Moreover, there is $\varepsilon^* = \varepsilon^*( \rho, \|\Pot\|_\infty, L_\Pot)$ such that $\kappa(\varepsilon, \rho, \|\Pot\|_\infty, L_\Pot)\in (0,1)$ for any $\varepsilon< \varepsilon^*$.
\end{Proposition}

\begin{proof}
   Let $\bx=(x, \tx)$ be such that $d_\eps(x, \tx)< 1$. 
    We start by constructing the coupling $(X,\tX)$ of $P(x, \cdot), P(\tx, \cdot)$. First we want to construct a coupling of the proposal kernels  $\baQ(x, \cdot)$ and $\baQ(\tx, \cdot)$. We use the coupling \eqref{eq:ARcoupling_eta} $\pi_\beta(\bx; \cdot)\in \mathfrak{C}(\baQ(x, \cdot), \baQ(\tx, \cdot))$ as developed in \cref{lemma:eta_coupling}, namely $(Y, \tY)\sim \pi_\beta(\bx; \cdot)$ with $Y\sim \baQ(x, \cdot)$ and
    \begin{equation}\label{eq:mtm:tY}
        \tY = T_{\bx}(Y) \mathbbm{1}_{U\leq \beta(\hat{Y}, T(Y))} + \hat{Y} \mathbbm{1}_{U\geq \beta(\hat{Y}, T(Y))}
    \end{equation}
    where $\hat{Y}\sim \baQ(\tx, \cdot)$ is independent of~$Y$, $U\sim\cU([0,1])$ and $ T_{\bx}(Y)= Y - \rho(x - \tx)$.
    Next, define $\tU\sim \cU([0,1])$, independent of $U$, and set 
    \begin{equation}\label{eq:mtm:coupling1}
        (X, \tX) = \begin{cases}
            (Y, \tY) \quad &\text{if } \tU \leq \alpha(x, Y)\wedge\alpha(\tx, \tY)=: \har(Y, \tY)\\
            (x, \tx)  \quad &\text{if } \tU \geq \alpha(x, Y)\vee\alpha(\tx, \tY)\\
            (x, \tY)   \quad &\text{if }\alpha(x, Y) \leq \alpha(\tx, \tY) \text{ and } \alpha(x, Y) \leq \tU \leq \alpha(\tx, \tY)\\
            (Y, \tx)  \quad &\text{if } \alpha(\tx, \tY) \leq \alpha(x, Y) \text{ and } \alpha(\tx, \tY) \leq \tU \leq \alpha(x, Y),
        \end{cases}
    \end{equation}
    It is easy to see \eqref{eq:mtm:coupling1} is a coupling of $P(x, \cdot), P(\tx, \cdot)$.
    It follows, using that $d_\eps \leq 1$, 
    \begin{align}
        W_{d_\varepsilon}&(P_\infty(x, \cdot), P_\infty(\tx, \cdot))\leq \E d (X, \tX) \notag\\
        &\leq \E d_\eps(Y, \tY)\mathbbm{1}_{\tU \leq \har(Y, \tY)} + d_\eps(x, \tx) \left[ 1 - \bbP\left( \tU \leq \har(Y, \tY)\right) \right] + \E |\alpha(x, Y)- \alpha(\tx, \tY)|. \label{eq:mtminfty:contr:1}
    \end{align}
    We start analyzing the first term on the rhs using the explicit representation \eqref{eq:mtm:tY} of $\tY$, so that 
    \begin{align*}
        \E d_\eps(Y, \tY)\mathbbm{1}_{\tU \leq \har(Y, \tY)} =  \E d_\eps\left( Y, T_{\bx}(Y)\right) \mathbbm{1}_{U\leq \beta(\hat{Y}, T_{\bx}(Y))}\mathbbm{1}_{\tU \leq \har(Y, \tY)} + \E d_\eps(Y, \hat{Y}) \mathbbm{1}_{U\geq \beta(\hat{Y}, T_{\bx}(Y))}\mathbbm{1}_{\tU \leq \har(Y, \tY)}.
    \end{align*}
    On the first term we use the definition of $T_{\bx}$ and on the second we drop the distance as $d_\eps\leq 1$: 
\begin{align*}
        \E d_\eps(Y, \tY)\mathbbm{1}_{\tU \leq \har(Y, \tY)} &\leq  d\left( \rho x, \rho\tx\right) \E \mathbbm{1}_{U\leq \beta(\hat{Y}, T_{\bx}(Y))}\mathbbm{1}_{\tU \leq \har(Y, \tY)} + \E \mathbbm{1}_{U\geq \beta(\hat{Y}, T_{\bx}(Y))}\mathbbm{1}_{\tU \leq \har(Y, \tY)}\\
       & =   d\left( \rho x, \rho\tx\right) \E \beta(\hat{Y}, T_{\bx}(Y))\har(Y, \tY) + \E \left( 1 - \beta(\hat{Y}, T_{\bx}(Y))\right) \har(Y, \tY),
    \end{align*}
   where we used the double conditional expectation argument namely, if $\cF_0$ is the $\sigma$-algebra generated by $Y, \tY$, then by the independence of $\tU$ and $U$ 
   \begin{align*}
       \E \mathbbm{1}_{U\leq \beta(\hat{Y}, T_{\bx}(Y))}\mathbbm{1}_{\tU \leq \har(Y, \tY)} = \E \E \left[\mathbbm{1}_{U\leq \beta(\hat{Y}, T_{\bx}(Y))}\mathbbm{1}_{\tU \leq \har(Y, \tY)}| \cF_0\right] =  \E \beta(\hat{Y}, T_{\bx}(Y))\har(Y, \tY).
   \end{align*}
   As seen in \eqref{eq:infty:beta-1}, by the Lipschitzianity of $\Pot$ 
   \begin{equation}\label{eq:infty:beta-12}
       |\beta(\hat{Y}, T_{\bx}(Y)) - 1 | \leq 2L_\Pot \rho \|x - \tx\| 
   \end{equation}
    then, using the fact that $\beta\leq 1$ and $d_\eps(x, \tx) = \varepsilon^{-1}\|x - \tx\|<1$, it follows
    \begin{align}
         \E d_\eps(Y, \tY)\mathbbm{1}_{\tU \leq \har(Y, \tY)} &
          \leq \rho d_\varepsilon\left(  x, \tx\right) \E \har(Y, \tY) + 2L_\Pot\rho \varepsilon d_\eps(x, \tx) \E \har(Y, \tY) .\label{eq:mtminfty:contr:2}
    \end{align}
    Therefore from \eqref{eq:mtminfty:contr:1}
    \begin{equation*}
         W_{d_\varepsilon}(P_\infty(x, \cdot), P_\infty(\tx, \cdot))\leq d_\eps\left(  x, \tx\right) \left[1 -  \E \har(Y, \tY) \left( 1- \rho(1 + 2L_\Pot \varepsilon) \right) \right] + \E |\alpha(x, Y)- \alpha(\tx, \tY)|. 
    \end{equation*}
    If $\varepsilon< \frac{1}{2L_\Phi}(\frac{1}{\rho}-1)$, since for any $x,y\in \qsp$ the following lower bound holds
    \[
     \har(x, y) = 1 \wedge \frac{\int e^{-\Pot(z)}Q(x, dz)}{\int e^{-\Pot(u)}Q(y, du)} \wedge \frac{\int e^{-\Pot(z)}Q(y, dz)}{\int e^{-\Pot(u)}Q(x, du)} \geq e^{-2\|\Pot\|_\infty},
    \]
    it then follows that
    \begin{equation}\label{eq:mtminfty:contr:21}
         W_{d_\varepsilon}(P_\infty(x, \cdot), P_\infty(\tx, \cdot))\leq d_\eps\left(  x, \tx\right) \left[1 -  e^{-2\|\Pot\|_\infty}\left( 1- \rho(1 + 2L_\Pot \varepsilon) \right) \right] + \E |\alpha(x, Y)- \alpha(\tx, \tY)|. 
    \end{equation}
    Next we analyze the last term in \eqref{eq:mtminfty:contr:21}: we want to find a constant $k$ so that 
    \begin{equation*}
        \E |\alpha(x, Y)- \alpha(\tx, \tY)| \leq k \varepsilon d_\eps(x, \tx).
    \end{equation*}
    As the function $1 \wedge x$ is Lipschitz with constant 1, then 
    \begin{align*}
        \E |\alpha(x, Y)- \alpha(\tx, \tY)| \leq \E \left|\frac{\int e^{-\Pot(z)}\, Q(x, dz)}{\int e^{-\Pot(w)}\, Q(Y, dw)} - \frac{\int e^{-\Pot(z)}\, Q(\tx, dz)}{\int e^{-\Pot(w)}\, Q(\tY, dw)} \right|\\
        \leq \E \left|\frac{\int e^{-\Pot(z)}\, Q(x, dz)\int e^{-\Pot(w)}\, Q(\tY, dw) - \int e^{-\Pot(z)}\, Q(\tx, dz)\int e^{-\Pot(w)}\, Q(Y, dw)}{\int e^{-\Pot(w)}\, Q(Y, dw)\int e^{-\Pot(w)}\, Q(\tY, dw)}  \right|.
    \end{align*}
    Using the definition of the kernels $Q$ and the boundedness of $\Pot$ 
    \begin{align*}
      &\leq e^{2\|\Pot\|_\infty} \E \iint \left| e^{-\Pot(F(x, z))} e^{-\Pot(F(\tY, w))} - e^{-\Pot(F(\tx, z))} e^{-\Pot(F(Y, w))} \right| \mu_0(dz)\mu_0(dw) \\
       &\leq e^{3\|\Pot\|_\infty} \rho L_\Pot\|x - \tx\| + \E \left|\Pot(F(\tY, w)) -\Pot(F(Y, w))\right|
    \end{align*}
    Next, looking at the the second term:
    \begin{align*}
        \E &\left|\Pot(F(\tY, w)) -\Pot(F(Y, w))\right| \\
        &= \E \left|\Pot(F(\tY, w)) -\Pot(F(Y, w))\right|\left( \mathbbm{1}_{U\leq \beta(\hat{Y}, T_{\bx}(Y))} + \mathbbm{1}_{U\geq \beta(\hat{Y}, T_{\bx}(Y))}\right)\\
       & \leq L\rho \E \left\|\tY - Y \right\| \mathbbm{1}_{U\leq \beta(\hat{Y}, T_{\bx}(Y))} + 2\|\Pot\|_\infty \E\mathbbm{1}_{U\geq \beta(\hat{Y}, T_{\bx}(Y))}\\
       & \leq L_\Pot\rho^2 \|x - \tx\|  + 2\|\Pot\|_\infty \left|1 - \E \beta(\hat{Y}, T_{\bx}(Y))\right| \leq L_\Pot\rho^2 \|x - \tx\|  + 4\|\Pot\|_\infty L_\Pot\rho\|x - \tx\|, 
    \end{align*}
    where we have used the Lipschitzianity of $\Pot$, the definition of $\tY$, the fact that $\beta \leq 1$ and \eqref{eq:infty:beta-12}.
    
    Therefore we showed that 
    \begin{equation*}
         \E |\alpha(x, Y)- \alpha(\tx, \tY)| \leq \varepsilon \rho L_\Pot \left( e^{3\|\Pot\|_\infty} + \rho + 4\|\Pot\|_\infty \right)  d_\eps(x, \tx)
    \end{equation*}
    and, with \eqref{eq:mtminfty:contr:21}, it follows that
    \begin{align}
         W_{d_\varepsilon}(P(x, \cdot), P(\tx, \cdot))
         &\leq d_\eps(x, \tx) \left\lbrace 1 - e^{-2\|\Pot\|_\infty} \left[ 1 - \rho\left(1 + 2L_\Pot\varepsilon\right) \right] + \varepsilon \rho L_\Pot \left( e^{3\|\Pot\|_\infty} + \rho + 4\|\Pot\|_\infty \right) \right\rbrace 
         \notag\\
         &= \kappa d_\eps(x, \tx).\label{eq:mtminfty:kappa}
    \end{align}
    Then, choosing 
    \begin{equation*}
        \varepsilon < \frac{e^{-2\|\Pot\|_\infty}(1 - \rho)}{\rho L_\Pot (e^{3\|\Pot\|_\infty} + \rho + 4\|\Pot\|_\infty +2 e^{-2\|\Pot\|_\infty})} = \varepsilon^*,
    \end{equation*}
    the parameter $\kappa$ is strictly smaller than one as desired.
\end{proof}

\begin{Proposition}\label{prop:mtminftypcn:small}
	Assume the potential function $\Pot: \qsp \to \R$ is globally Lipschitz with constant $L_\Pot$. Define $S= \lbrace V(x) \leq 4 K_V\rbrace $ with $V$ any of the functions in \eqref{eq:choiceV}, and let $r_s >0$ be such that $S \subset B(0, r_s)$. Then, for every $x, \tx\in S$, there exists $s(n)>0$ such that 
    \begin{equation}\label{eq:inftyMTpcn:small}
        \Wass_{d_\varepsilon} (P_\infty^n(x, \cdot), P_\infty^n(\tx, \cdot))\leq s(n) \quad n\in \N
    \end{equation}
    and there is $n^* = n^*(\varepsilon, \rho, \|\Pot\|_\infty, L_{\Pot})$ such that $s = s(\varepsilon, \rho, n) \in (0,1)$, for all $n >n^*$.
\end{Proposition}

\begin{proof}
    Let $x, \tx\in S\subset B(0, r_s)$ and consider the same coupling of $P(x, \cdot)$, $P(\tx, \cdot)$ as in \eqref{eq:mtm:coupling1} so that 
    \begin{align}
        W_{d_\varepsilon}(P(x, \cdot), P(\tx, \cdot))
        &\leq \E d_\eps(Y, \tY) \mathbbm{1}_{\tU\leq \har(Y, \tY)} + 1 - \E\har(Y, \tY).\label{eq:mtm:small1}
    \end{align}
   We start by focusing on the first term: using the definition of the coupling $(Y, \tY)$ it follows
    \begin{align*}
        \E d_\eps(Y, \tY) \mathbbm{1}_{\tU\leq \har(Y, \tY)} &= \E d_\eps(Y, T_{\bx}(Y)) \mathbbm{1}_{U\leq \beta(\hat{Y}, T_{\bx}(Y))} \mathbbm{1}_{\tU\leq \har(Y, \tY)}
       + \E d_\eps(Y, \hat{Y}) \mathbbm{1}_{U\geq \beta(\hat{Y}, T(Y))}\mathbbm{1}_{\tU\leq \har(Y, \tY)}\\
        &\leq d_\eps(\rho x, \rho \tx) \E \mathbbm{1}_{U\leq \beta(\hat{Y}, T_{\bx}(Y))} \mathbbm{1}_{\tU\leq \har(Y, \tY)}+ \E \mathbbm{1}_{U\geq \beta(\hat{Y}, T_{\bx}(Y))}\mathbbm{1}_{\tU\leq \har(Y, \tY)}\\
        &\leq  \frac{2r_s \rho }{\varepsilon}\E \beta(\hat{Y}, T_{\bx}(Y))\har(Y, \tY)+ 
        \E \left(1- \beta(\hat{Y}, T_{\bx}(Y))\right) \har(Y, \tY),
    \end{align*}
    where in the last inequality we used that $x, \tx$ are in $B(0, r_s)$. Looking back at \eqref{eq:mtm:small1}, we then showed that 
    \begin{align*}
          W_{d_\varepsilon}(P(x, \cdot), P(\tx, \cdot))
        &\leq  \frac{2r_s \rho }{\varepsilon}\E \beta(\hat{Y}, T_{\bx}(Y))\har(Y, \tY)+ 
        \E \left(1- \beta(\hat{Y}, T_{\bx}(Y))\right) \har(Y, \tY)+  1 - \E\har(Y, \tY)\\
        &=  1 - \E \beta(\hat{Y}, T_{\bx}(Y))\har(Y, \tY) \left( 1 - \frac{2r_s \rho }{\varepsilon}\right),
    \end{align*}
    and, as $\Pot$ is assumed bounded, from the definitions of the acceptance probabilities $\alpha$ \eqref{eq:mtminfty:alpha} and $\beta$ \eqref{eq:acceptanceAR}, the quantity $\E \beta(\hat{Y}, T_{\bx}(Y))\har(Y, \tY) > e^{-6\|\Pot\|_\infty}$ stays greater than zero. 
    
    Next, we iterate the argument to get a similar bound for $P^n(x, \cdot), P^n(\tx, \cdot)$ so that, for $n$ large enough, the factor $1 - \frac{2r_s \rho^n }{\varepsilon}$ is small enough to ensure the desired result.
    Let $n\geq 1$ and $(X^{(0)}, \tX^{(0)}) = (x,\tx)$. Consider the coupling $(Y^{(n)}, \tY^{(n)}) \sim \pi_\beta(X^{(n-1)}, \tX^{(n-1)})$ of the proposals $\baQ^n(x,\cdot), \baQ^n(\tx,\cdot)$, namely $ \tY^{(n)}\sim \baQ(X^{(n-1)}, \cdot)$ and, given the independently drawn $\hat{Y}^{(n)}\sim \baQ(\tX^{(n-1)}, \cdot)$, 
    \begin{equation*}
         \tY^{(n)} = T_{n-1}(Y^{(n)}) \mathbbm{1}_{U^{(n)} \leq \beta(\hat{Y}^{(n)}, T_{n-1}(Y^{(n)}))} + \hat{Y}^{(n)} \mathbbm{1}_{U^{(n)} > \beta(\hat{Y}^{(n)}, T_{n-1}(Y^{(n)}))} ,
    \end{equation*}
    where $U^{(n)}\sim \cU([0,1])$ and $T_{n-1}(y) := T_{\bX^{(n-1)}}(y) = y - \rho(X^{(n-1)} - \tX^{(n-1)})$. 
    Let $\tU^{(n)}\sim \cU([0,1])$ be independent of $U^{(k)}$, $k=1, \ldots, n$, and $\tU^{(k)}$, $k=1, \ldots, n-1$, and define
    \begin{equation}\label{eq:mtm:coupling:n}
        (X^{(n)}, \tX^{(n)}) = 
        \begin{cases}
            (Y^{(n)}, \tY^{(n)}) \qquad &\text{if } \tU^{(n)} \leq \alpha(X^{(n-1)}, Y^{(n)})\wedge\alpha(\tX^{(n-1)}, \tY^{(n)})=: \har_{n-1}(Y^{(n)}, \tY^{(n)})\\
            (X^{(n-1)}, \tX^{(n-1)})  \qquad &\text{if } \tU^{(n)} \geq \alpha(X^{(n-1)}, Y^{(n)})\vee\alpha(\tX^{(n)}, \tY^{(n)})\\
            (X^{(n-1)}, \tY^{(n)})   \qquad &\text{if } \alpha(X^{(n-1)}, Y^{(n)}) \leq \tU^{(n)} \leq \alpha(\tX^{(n-1)}, \tY^{(n)})\\
            (Y^{(n)}, \tX^{(n-1)})  \qquad &\text{if } \alpha(\tX^{(n-1)}, \tY^{(n)}) \leq \tU^{(n)} \leq \alpha(X^{(n-1)}, Y^{(n)}).
        \end{cases}
    \end{equation}
    Define the events
    \begin{equation*}
    \begin{split}
            A^{(k)} &= \lbrace U^{(k)} \leq \beta(\hat{Y}^{(k)}, T_{k-1}(Y^{(k)})) \, \text{ and }\,\tU^{(k)} \leq \har_{k-1}(Y^{(k)}, \tY^{(k)})\rbrace \\
        \Lambda^{(n)}& = \bigcap_{k=1}^n A^{(k)}
    \end{split}
    \end{equation*}
    so that $\Lambda^{(n)}$ is the event over which $(X^{(k)}, \tX^{(k)}) =  (Y^{(k)}, T_{k-1}(Y^{(k)}))$ for any $k= 1,\ldots, n$. Then, since $d_\varepsilon\leq 1$ and from the definition of the shift function $T_{n-1}$, it follows 
    \begin{align*}
        W_{d_\varepsilon}(P^n(x, \cdot), P^n(\tx, \cdot))&\leq \E d_\eps(X^{(n)}, \tX^{(n)}) \left[\mathbbm{1}_{\Lambda^{(n)}} + \mathbbm{1}_{(\Lambda^{(n)})^c}\right]\\
        &\leq \E d_\eps(Y^{(n)}, \tY^{(n)}) \mathbbm{1}_{\Lambda^{(n)}} + 1 - \bbP(\Lambda^{(n)})\\
        &=  \E d_\eps(Y^{(n)}, T_{n-1}(Y^{(n)})) \mathbbm{1}_{\Lambda^{(n)}} + 1 - \bbP(\Lambda^{(n)})\\
        &=  \E d_\eps(\rho X^{(n-1)}, \rho \tX^{(n-1)}) \mathbbm{1}_{\Lambda^{(n)}} + 1 - \bbP(\Lambda^{(n)})\\
        & \; \vdots\\
        & = \E d_\eps(\rho^n x, \rho^n \tx ) \mathbbm{1}_{\Lambda^{(n)}} + 1 - \bbP(\Lambda^{(n)}).
    \end{align*}
    Finally, as $x, \tx\in S \subset B(0, r_s)$
    \begin{align}
         W_{d_\varepsilon}(P^n(x, \cdot), P^n(\tx, \cdot))&\leq 1 - \bbP(\Lambda^{(n)}) \left( 1 - \frac{2 r_s \rho^n}{\varepsilon}\right). 
    \end{align}
    If $\Lambda^{(n)}$ has positive probability, and $n$ is large enough to get $2 r_s \rho^n< \varepsilon$, we then have the desired result. 
    Again by the boundedness of the function $\Pot$ for any $k =1, \ldots, n$
    \begin{equation*}
        \har_{k-1} (Y^{(k)}, \tY^{(k)}) \geq e^{-2\|\Pot\|_\infty}, \quad \beta(\hat{Y}^{(k)}, T_{k-1}(Y^{(k)})> e^{-4 \|\Pot\|_\infty}
    \end{equation*}
    so that 
    \begin{align*}
         \bbP(A^{(1)}) &= \bbP\left\lbrace U^{(1)} \leq \beta(\hat{Y}^{(1)}, T_{k-1}(Y^{(1)}))\, \text{ and }\,  \tU^{(1)} \leq \har_0(Y^{(1)}, \tY^{(1)})\right\rbrace\\
         &= \E  \beta(\hat{Y}^{(1)}, T_{k-1}(Y^{(1)})) \har_0(Y^{(1)}, \tY^{(1)}) > e^{-6\|\Pot\|_\infty}
    \end{align*}
    and 
    \begin{align*}
        \bbP(\Lambda^{(n)}) = \bbP(A^{(1)}) \prod_{k=2}^{n}\bbP \left( A^{(k)} \left| \bigcap_{j = 1}^{k-1} A^{(j)}\right.\right) 
        > e^{-6n\|\Pot\|_\infty}. 
    \end{align*}
    This concludes the proof.
\end{proof}

\section{Numerical Experiments} \label{sec:numerics}
In this section, we benchmark mpCN and MTpCN against pCN on three toy inverse problems with complicated posterior geometries.
All three problems can be mathematically phrased as Bayesian inverse problems \cite{kaipio2006statistical,stuart2010inverse}. We can phrase the problem as follows:  estimating the unknown $x\in \mathcal{X}$ from observations $y\in \R^n$ connected to the unknown through a forward map $f:\mathcal{X}\to\mathbb{R}^n$, 
\begin{equation} \label{eq:inverse_problem}
    y = f(x) + \eta,
\end{equation}
where $\eta \in \mathbb{R}^n$ is the additive observational noise. 

We assume throughout Gaussian noise  with observational noise level $\sigma^2$, defining in turn the Gaussian likelihood 
\begin{equation}\label{eq:lik}
    \mathcal{L}(y; x) \propto \exp \left ( -\frac{1}{2\sigma^2}\left\|f(x)-y\right\|^2\right ) := \exp \left (-\Phi(y;x) \right ), 
\end{equation}
for the data $y$ in~\eqref{eq:inverse_problem}, with $\Phi$ the quadratic potential.

Next, we assume a zero-mean Gaussian prior measure $\mu_0=\mathcal{N}(x; 0, \mathcal{C})$ with covariance operator $\mathcal{C}$. Together with the likelihood, this prior measure defines the  intractable posterior measure 
\begin{equation*}
    \mu(dx) \propto \exp \left ( -{\Phi(y; x)}\right ) \mu_0(dx).
\end{equation*}
In practice, we will represent the functional $x$ with some finite discretization in $\mathbb{R}^d$, in which case the posterior measure can be written as
\begin{equation*}
    \mu(dx) \propto
\exp\left(-\frac{1}{2\sigma^2}\left\|f(x)-y\right\|^2
-\frac{1}{2}x^T \bm{C}^{-1} x\right)\, dx,
\end{equation*}
where we have assumed a $d \times d$ covariance matrix $\bm{C} = \tau^2 \bm{R}$ for the Gaussian prior, with $\tau^2$ the scalar prior marginal variance and $\bm{R}$ a correlation matrix. 

In the first two problems we consider a two-dimensional parameter $x =(x_1, x_2) \in \mathbb{R}^2$, i.e. $d=2$, and corresponding forward maps $f_1, f_2 :\mathbb{R}^2 \to \mathbb{R}^2$ that generate multiwelled and swirl surfaces in $\mathbb{R}^{3}$, respectively.  

The third problem consists of a toy formulation of a functional solute transport inverse problem and is motivated by sparse indirect measurements of a fluid flow under damping and external forcing. In its discretized version, the forward map is given by a matrix inversion. 

To demonstrate the algorithms, we use the so-called \textit{inverse crime} approach \cite{kaipio2006statistical}, and simulate data from the models as
\begin{equation}\label{eq:data_simulation}
x_{\text{true}} \sim \mathcal{N}(0, \bm{C}), \quad y_{\text{obs}} = f(x_{\text{true}}) + \eta, \quad \eta \sim \mathcal{N}(0, \sigma^2 \bm{I}_n).
\end{equation}

We then sample from the model's posterior measure using mpCN, MTpCN, and pCN, for different values for the proposal count $p$ and the proposal aggressiveness $\rho$. We use the toy examples to demonstrate and compare various aspects of the performance of the multiproposal mpCN and MTpCN among themselves and against pCN, considering both single pCN chains and the embarrassingly parallel approach of running $p$ pCN chains simultaneously. 

\subsection{A multiwelled posterior}\label{sec:multiwell} 
The goal with this example is to analyze the performance of the algorithms on a multimodal posterior on $\mathbb{R}^2$, obtained by defining the multiwell forward map
\begin{equation}
    f(x):= f_1(x) = \begin{bmatrix}
(x_1^2 - 1)(x_1^2 - 4) \\
(x_2^2 - 1)(x_2^2 - 4)
\end{bmatrix}
\end{equation}
and $\tau^2=\sigma^2=1$ (so $\bm{C}=\bm{R}=\bm{I}$) in the Gaussian prior and likelihood in~\eqref{eq:lik}.
\cref{fig:lik_contours} (a) shows the likelihood contours obtained with this setup.

\begin{figure}
\begin{subfigure}{0.5\textwidth}
    \centering
    \includegraphics[width=0.7\linewidth]{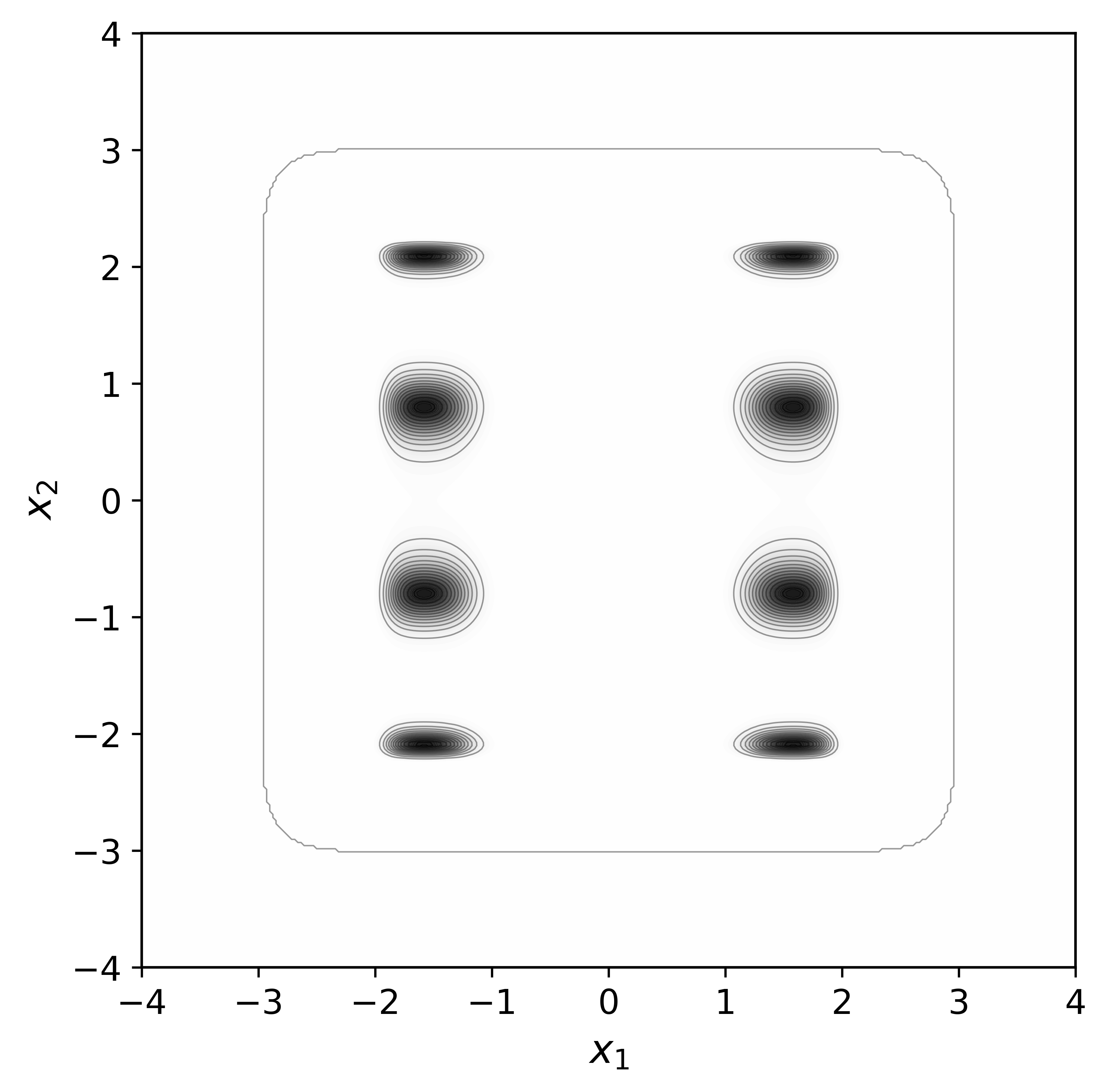}
    \caption{Multiwell.}
    \label{fig:lik_contour_multiwell}
\end{subfigure}
\hfill
\begin{subfigure}{0.5\textwidth}
    \centering
    \includegraphics[width=0.7\linewidth]{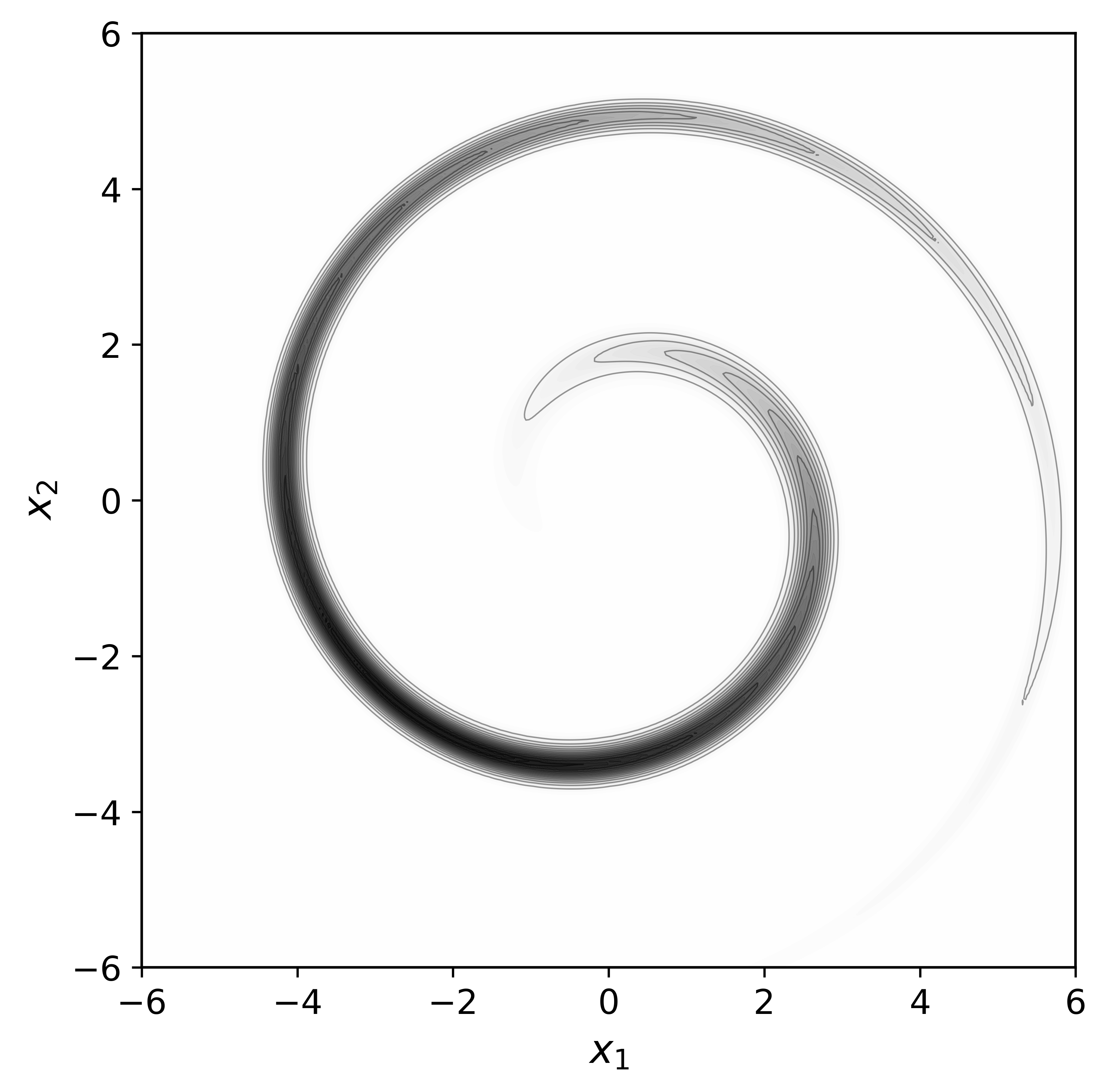}
    \caption{Polar twist.}
    \label{fig:lik_contour_polar}
\end{subfigure}
\caption{Likelihood contours.}
\label{fig:lik_contours}
\end{figure}

In \cref{fig:sweep_multiwell_per_param} we study mixing as a function of $p$ and $\rho$ for $x_1$ and $x_2$. We observe a transition in the mpCN algorithm from globally informed to locally informed exploration as $\rho$ increases. For small $\rho$, the proposal combined with Barker acceptance introduces enough randomness to explore multiple directions and move between modes. As $\rho$ grows, proposed states become increasingly similar to the current state, and the chain explores only local directions.

\begin{figure}
    \centering
    \includegraphics[width=1\linewidth, trim={0 0 0 1cm}, clip]{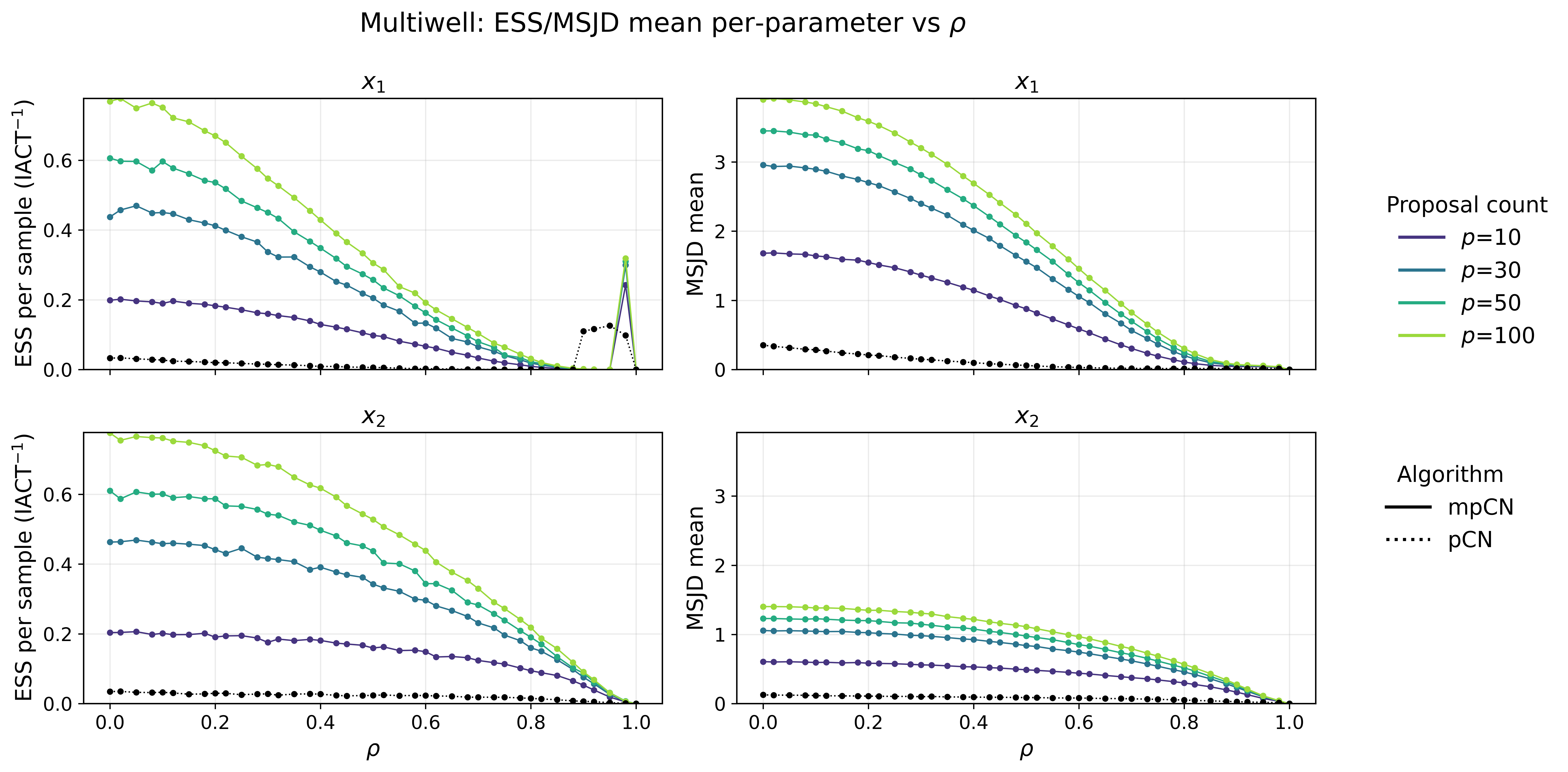}
    \caption{$p$-$\rho$ sweep for ESS and MSJD for the multiwell example, split by parameter. The mixing decays with $\rho$ in general. Since ESS measures local mixing, it can capture the moment when the algorithm transitions from a poor global exploration to exclusively local exploration. This is evidenced in the top-left subplot, where low ESS values are followed by a spike in the ESS.}
    \label{fig:sweep_multiwell_per_param}
\end{figure}

This transition is reflected in spikes in the ESS for $x_1$ (\cref{fig:sweep_multiwell_per_param}, top left). Since ESS is a local diagnostic, it can be high even when the chain is confined to a single mode but mixes well within it. The traceplots in \cref{fig:multiwell_traceplots} support this: as $\rho$ increases, the $x_1$ component exhibits fewer transitions between modes and becomes trapped in one mode. The corresponding $\rho$ values align with the ESS spikes.

\begin{figure}
    \centering
    \includegraphics[width=1\linewidth, trim={0 0 0 1cm}, clip]{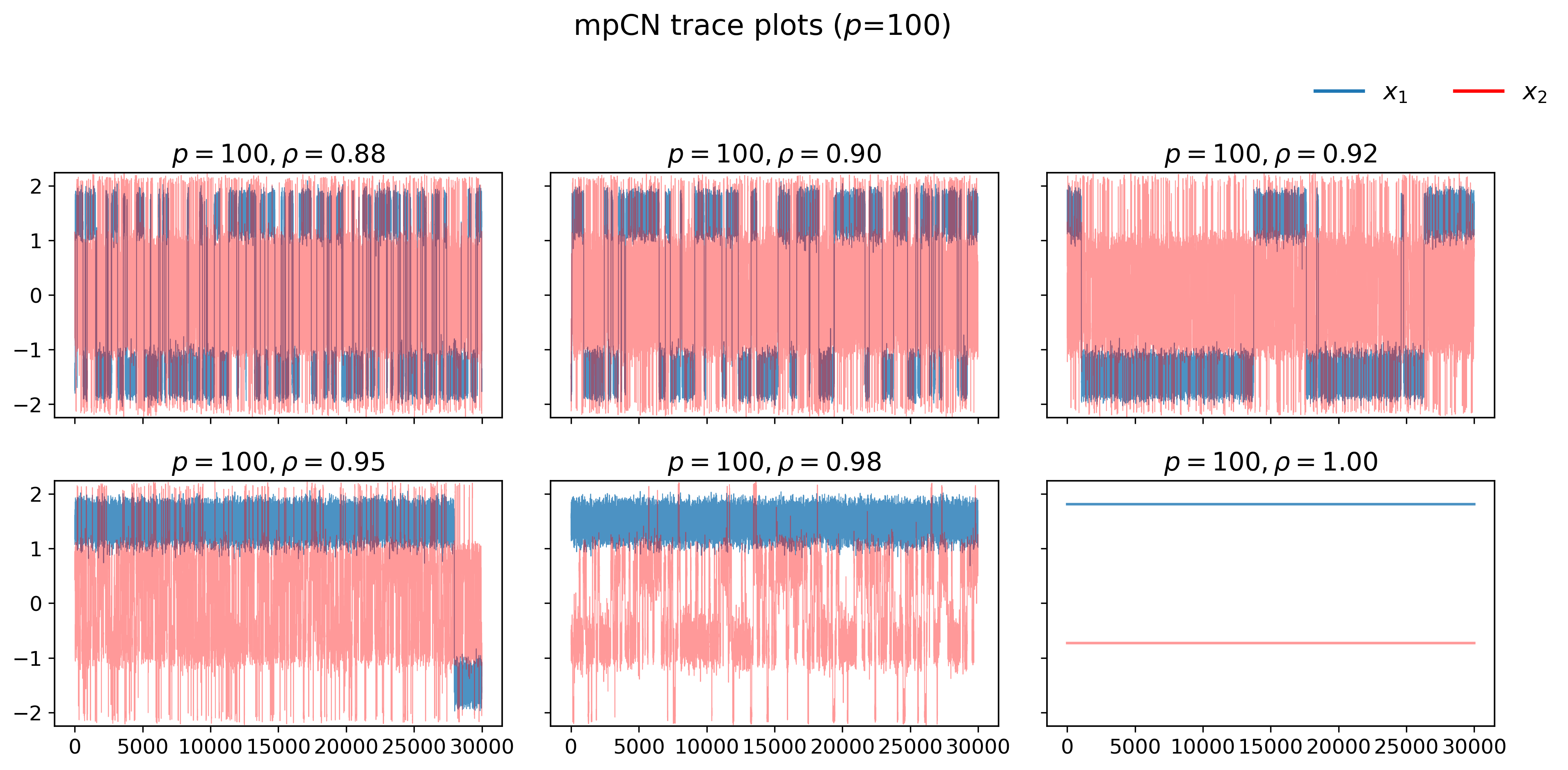}
    \caption{mpCN ($p=100$) traceplots for the multiwell example. The ability to jump between modes decreases with $\rho$; for $\rho$ large enough, the component plotted in blue becomes confined to a single mode.}
    \label{fig:multiwell_traceplots}
\end{figure} 

\subsection{Polar twist posterior}\label{sec:polar}
The goal with this example is to evaluate and compare the performance of mpCN, MTpCN and single-chain pCN on a posterior with complicated correlation structures. To generate these complicated structures, we assume the "polar-twist`` forward map 
$$
f(x):= f_2(x;\alpha) = \begin{bmatrix}
x_1 \cos(\alpha r) - x_2 \sin(\alpha r) \\
x_1 \sin(\alpha r) + x_2 \cos(\alpha r)
\end{bmatrix}
$$
where $\alpha$ is a hyperparameter that determines the strength of the twist and $r = \sqrt{x_1^2 + x_2^2}$. 
We then simulate data from the model as described at the beginning of Section~\ref{sec:numerics}, 
with $\alpha=2$, $\sigma^2=1$, $\tau^2=4$, and
$$
\bm{R} = \begin{bmatrix}1 & 0.3 \\ 0.3 & 0.5\end{bmatrix}.
$$
\cref{fig:lik_contours} (b) shows the likelihood contours obtained with this setup.
\begin{figure}[h!]
    \centering
    \includegraphics[width=0.9\linewidth, trim={0 0 0 0cm}, clip]{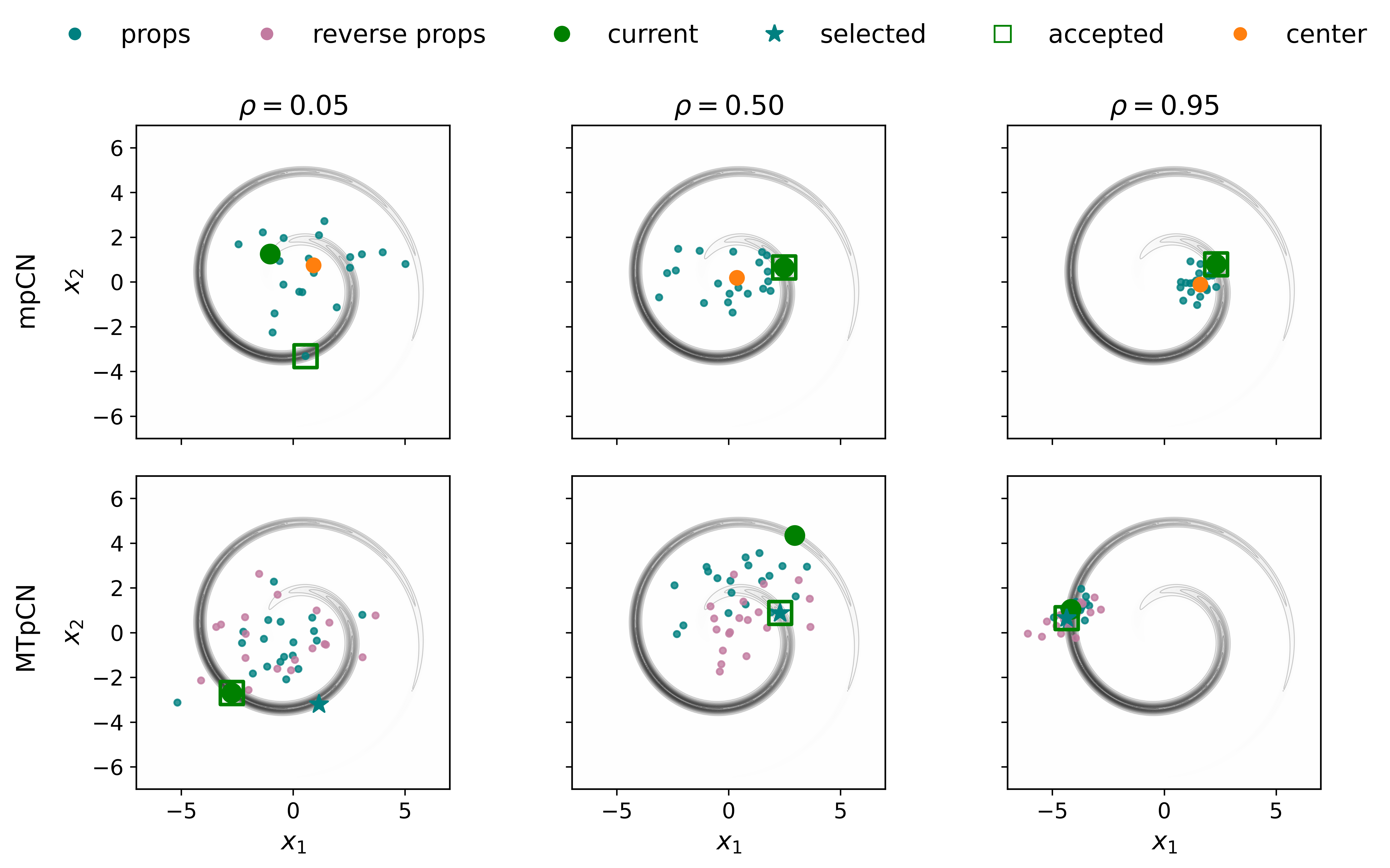}
    \caption{Example proposal clouds for mpCN and MTpCN, for the polar-twist example. Each subplot corresponds to iteration 10000 in the chains defined by $p=10$ and the $\rho$ indicated for each column, considering the same random seed for initializing the six resulting chains.}
    \label{fig:clouds}
\end{figure}

In \cref{fig:sweep_polar} we compare the mixing of the algorithms as a function of the proposal count $p$ and the hyperparameter $\rho$. In the figure, the mixing is expressed in raw ESS and MSJD. 
The mixing decreases with $\rho$ for all algorithms and all values of $p$. The mpCN algorithm has better mixing than MPpCN and pCN in general, and we note that MTpCN has approximately double computational cost as mpCN since it requires twice as many likelihood computations. 
\begin{figure}[h!]
    \centering
    \includegraphics[width=1\linewidth, trim={0 0 0 0cm}, clip]{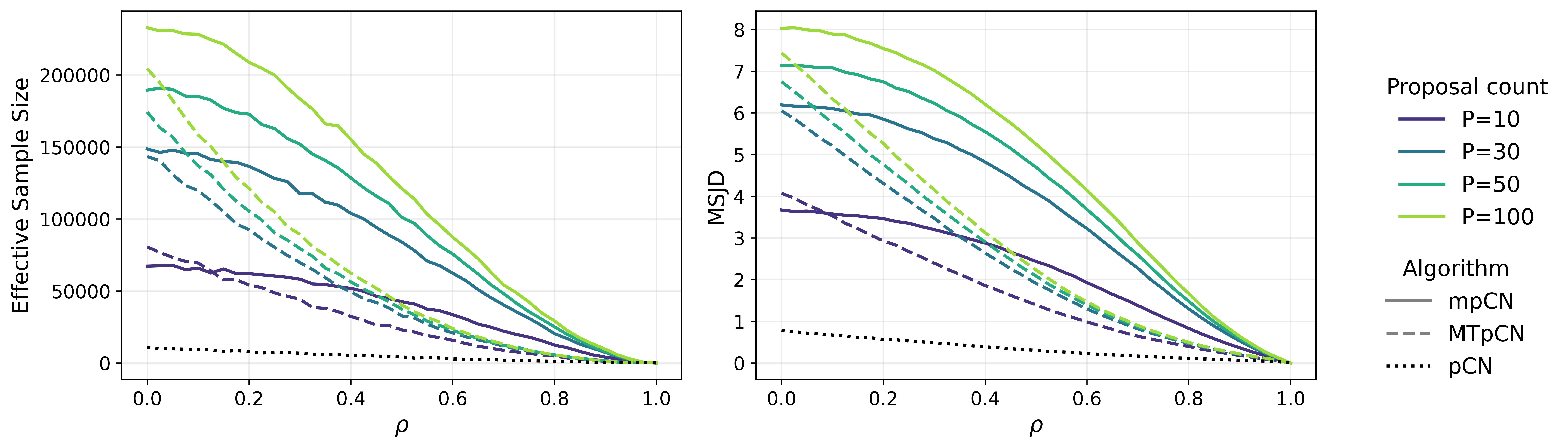}
    \caption{Mixing of mpCN, MTpCN, and pCN for the polar twist forward map. Both plots consider the ESS and MSJD averaged over the $x_1$ and $x_2$ coordinates.}
    \label{fig:sweep_polar}
\end{figure}

\subsection{Toy solute transport problem}\label{sec:matrix_inversion}
We consider the following matrix-based approximation of a PDE for damped transport in fluids:
\begin{equation} \label{eq:ad_eq}
    \frac{d}{dt} \bm{\theta} + (\bm{A} + \kappa \bm{I}) \bm{\theta} = \bm{g}.
\end{equation}
Here, the solution vector $\bm{\theta} \in \mathbb{R}^d$ is the state variable (e.g. temperature, solute concentration), $\bm{A} \in \mathbb{M}^{[d\times d]}$ is a discretized version of the advection (transport) operator, $\kappa>0$ is a damping parameter, and $\bm{g} \in \mathbb{R}^d$ is a time-independent forcing term (e.g. heat source, dye injection or stirring). 
This toy model, first introduced in \cite{senn2026mess} and with preliminary versions introduced in \cite{glatt2024parallel},
mimics some of the features of the transport of a solute by a fluid under damping and external forcing. 

The model is already written in a basis that behaves like Fourier modes, and then the elements $a_{ij}$ of $\bm{A}$ represent the energy transfer from mode $i$ to $j$. To model a natural physical symmetry, $\bm{A}$ is defined to be zero-diagonal and antisymmetric, and is thus specified by the the $m = d(d-1)/2$ non-zero elements in its upper triangle. 
We choose $\bm{g} = \hat{\bm{e}}_1 = (1, 0, \ldots, 0)$ so that the energy injection is done only at mode 1 (the largest scale) and is cascaded to the other modes (smaller scales) by $\bm{A}$.  

We focus on the steady state of equation~\eqref{eq:ad_eq}, which satisfies 
\begin{equation} \label{eq:ad_toy}
    (\bm{A} + \kappa \bm{I}) \bm{\theta} = \bm{g},
\end{equation}
and extends to infinite dimensions assuming appropriate square-summability of the coefficients of $\bm{A}$ and $\bm{g}$ and hence $\bm{\theta}$ (see \cref{ap:solute_like}).
Our goal is to use Bayesian inversion on the toy model in~\eqref{eq:ad_toy} to estimate $\bm{A}$ assuming the partial observational model
\begin{equation}\label{eq:ad_obs_model}
    \bm{y} = \mathcal{P}(\bm{\theta}(\bm{A})) + \bm{\eta}, \quad \bm{\eta} \sim \mathcal{N}_k(0, \sigma^2 \bm{I}_k),
\end{equation}
where $\mathcal{P}:\mathbb{R}^d \to \mathbb{R}^k$ is a projection operation so that if $\bm{\theta} = (\theta_1, \ldots, \theta_{d})$ and $1 \leq d_0 \leq d$, then $\mathcal{P}(\bm{\theta})=(\theta_{d_{0}-k}, \ldots, \theta_{d_{0}})$ and $\sigma^2$ is the observational noise scale. 

We phrase the estimation problem as sampling from the posterior with measure
\begin{equation}\label{eq:ex3:target}
    \mu(d\bm{A}) \propto \exp(-\Phi(\bm{A}))\mu_0(d\bm{A}),
\end{equation}
with potential 
\begin{equation}\label{eq:pot}
    \Phi(\bm{A}) =  \frac{1}{2\sigma^2} \|\cP\left((\bA + \kappa \bI)^{-1}\bg)\right)- \by \|^2,
\end{equation}
that involves computing the solution vector $ \bm{\theta}(\bm{A})=(\bm{A}+\kappa \bm{I})^{-1} \bm{g}$ at each iteration, implying an expensive likelihood evaluation for large $d$. 
We show in \cref{rem:ex3} in Appendix~\ref{ap:solute_like} that the log-likelihood function \eqref{eq:pot} falls under the category analyzed in \cref{sec:results}, as it can be proved to be globally bounded and Lipschitz, ensuring theoretical benchmarks for robustness of mixing with increasing dimension. 

Regarding the prior modeling of the $m$ unknown coefficients elements $a_{ij}, \; 0 \leq i < j \leq d-1$ in the non-zero upper triangle of $\bm{A}$,  we choose prior measure  $\mu_0 := \mathcal{N}(\bm{0}, \bm{C})$ with diagonal covariance matrix $\bm{C} := \tau^2 \mbox{diag}(\bm{q})$, where $\bm{q} = (q_{ij})$ is a vector containing the variances the elements of $A$. More precisely, we define the elements of $\bm{q}$ as 
\begin{equation} \label{eq:ad_covariance}
    q_{ij} = \mbox{Var}(a_{ij}) := (ij)^{-\alpha}|i-j|^{-\gamma}, \quad 0 \leq i < j \leq d-1, \quad \gamma, \alpha > 0.
\end{equation}
With the modeling choice in~\eqref{eq:ad_covariance}, the magnitude of each element $a_{ij}$ will decay both with the difference between $i$ and $j$, discouraging long-range energy interactions, and with increasing indices $i$ and $j$, thereby penalizing energy transfer at higher frequencies. 

In the following two sections we analyze first the warm-up phase and then the stationary phase using datasets generated from the model, as described at the beginning of this section.

\subsubsection{Warm-up phase analysis}\label{subsec:warm-up}
For the warm-up phase, we consider $d=40$, corresponding to $m=780$ parameters. We choose the hyperparameter values 
\begin{equation} \label{eq:hyperpars_st}
    \kappa=0.02, \alpha=3, \gamma=2, \sigma^2=0.25, \tau^2=2,
\end{equation}
and set the observational scale with $d_0=12$ and $k=7$, corresponding to observing modes $6, 7, \ldots, 12$ (the 6th, 7th, ... 12th elements of the observational vector $\bm{y}_{40}$). 
This model configuration ensures that the problem is complex enough so that the warm-up phase can be observed with the naked eye. 
\cref{fig:ad_simulated_data_stat} (right) in Appendix~\ref{ap:numerics} illustrates the generated $\bm{A}_{40}$, $\bm{\theta}(\bm{A}_{40})$ and $\bm{y}_{40}$. 

We next demonstrate how mpCN with $p$ proposals has a shorter warm-up phase than a single pCN chain and also than $p$ embarrassingly parallel pCN chains starting at $p$ different initial states. 
To this aim, we will consider the following four observables of the advection matrix $\bm{A}$: 
\begin{equation}\label{eq:observables}
    \varphi_1(\bm{A}) := a_{01},
    \quad
    \varphi_2(\bm{A}) := \mbox{max}_{i<j}|a_{ij}|,
    \quad
    \varphi_3(\bm{A}) := E_2(\bm{A})= \sum_{i<j:|i-j|\leq 2}a_{ij}^2,
    \quad
    \varphi_4(\bm{A}) := \Phi(\bm{A}),
\end{equation}
corresponding respectively to the first element of the unknown vector, the maximum absolute transfer coefficient, the energy in a 2-band, and the potential.

To compare the error evolution between algorithms, we compute running mean estimators for each of the observables in~\eqref{eq:observables}. Namely, and letting $Y = (Y_1, Y_2, \ldots)$ denote an mpCN chain and $X^j = (X_1^j, X_2^j, \ldots)$ denote the $j$-th pCN chain out of a group of $p$ independent pCN chains, we define the running mean estimators
\begin{equation}\label{eq:estimators}
    \hat{\varphi}_{t, mpCN} := \frac{1}{t}\sum_{s=1}^t \varphi(X_s), \quad 
    \hat{\varphi}_{t, pCN} := \frac{1}{t}\sum_{s=1}^t \varphi(X_s^1), \quad
    \hat{\varphi}_{t, EP} := \frac{1}{t}\sum_{s=1}^t \frac{1}{p}\sum_{j=1}^p\varphi(X_s^j),
\end{equation}
for an arbitrary observable $\varphi$.
Next, we define the running mean squared error (MSE) for an arbitrary estimator $\hat{\varphi}$, such as those in~\eqref{eq:estimators}, as
\begin{equation}\label{eq:MSE}
     \mbox{MSE}(t, \hat{\varphi}) := \frac{1}{M}\sum_{m=1}^M \bigl(\hat{\varphi}_t^m - \hat{\mu}(\varphi) \bigr)^2 \approx \int_{\mathcal{X}} \bigl(\hat{\varphi}_t - \mu(\varphi) \bigr)^2 \mu_0(dx),
\end{equation}
where $\hat{\mu}(\varphi)$ is an estimation of the true value $\mu(\varphi)$ of the estimator, computed with a reliable method such as a benchmark MCMC algorithm.

In \cref{fig:mpcn_pcn_traceplots} we compare the traceplots for the four observables in~\eqref{eq:observables} obtained with mpCN ($p=100$) and the embarrassingly parallel (EP) approach of simultaneously running $p=100$ independent pCN chains. The traceplots suggest mpCN has a shorter warm-up phase than embarassingly parallel pCN and therefore single-chain pCN. 
\begin{figure}
    \centering
    \includegraphics[width=1\linewidth]{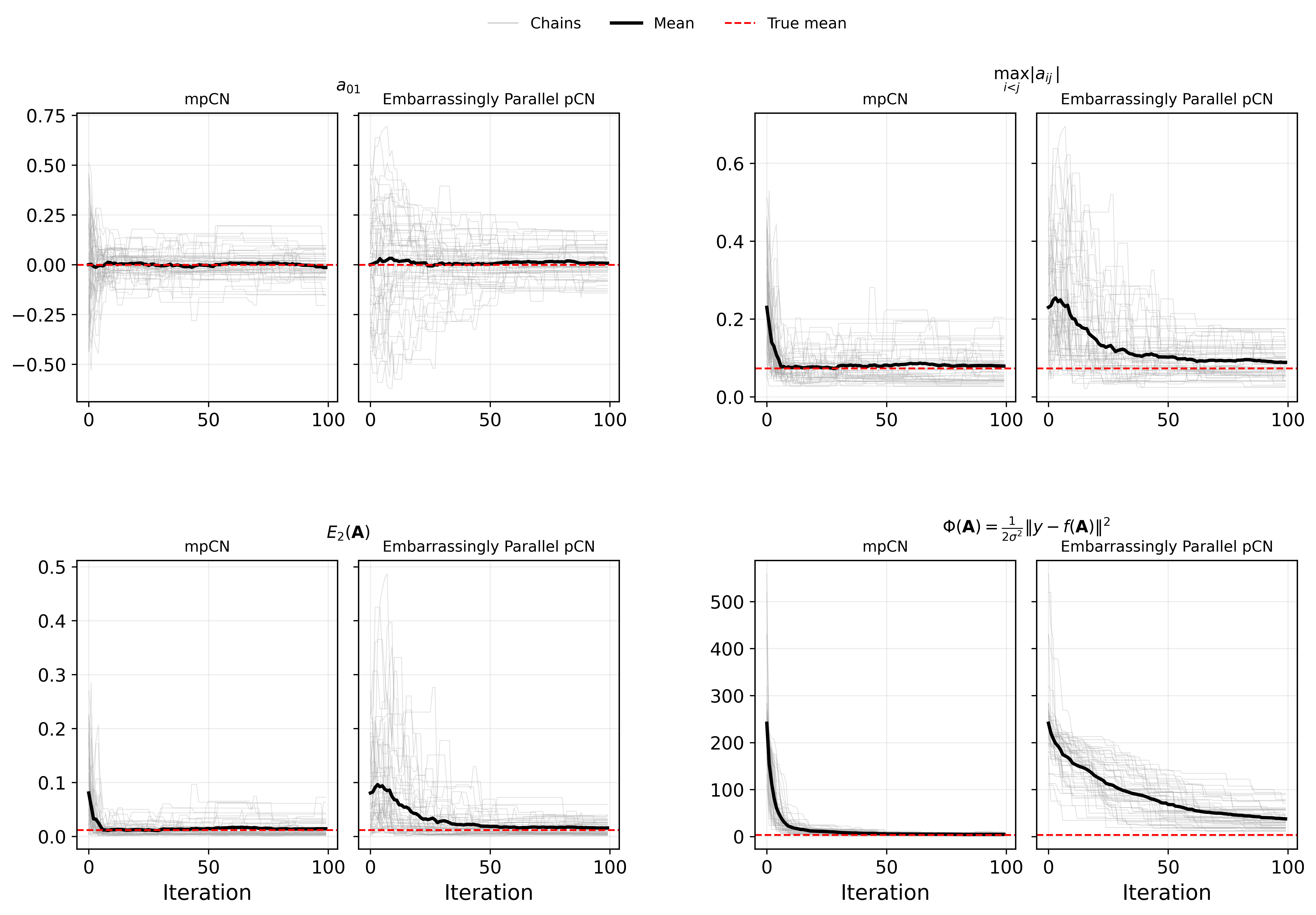}
    \caption{Traceplots for the first 100 iterations for the observables in~\eqref{eq:observables}, for mpCN with $p=100$ (left column of eaach subplot) and EP-pCN with $p=100$ chains (right column of each subplot). Each gray lines is a realization of a chain. We consider $M=50$, meaning that there are 50 mpCN independent chains, and 50 replicates of the EP-pCN experiment with $p=100$ chains each, making a total of 5000 pCN chains. Each chain is initialized at a random value from the prior. The solid black line represents the mean of all gray lines.}
    \label{fig:mpcn_pcn_traceplots}
\end{figure}
\cref{fig:running_mse} compares the running MSE with $M=50$ replicates for the same four observables using the three different estimators in~\eqref{eq:estimators}. 
The error decreases the fastest with mpCN for all observables except for $\varphi_1$, which corresponds to the first element of the unknown elements in the estimated matrix $\bm{A}$. 
Here the effect of using $100\times$ as many starting points, as compared to the mpCN and single-pCN, provides an advantage, because the starting points are sampled from the prior centered around zero and the true value for $a_{01}$ is very close to zero. 

\begin{figure}
    \centering
    \includegraphics[width=1\linewidth]{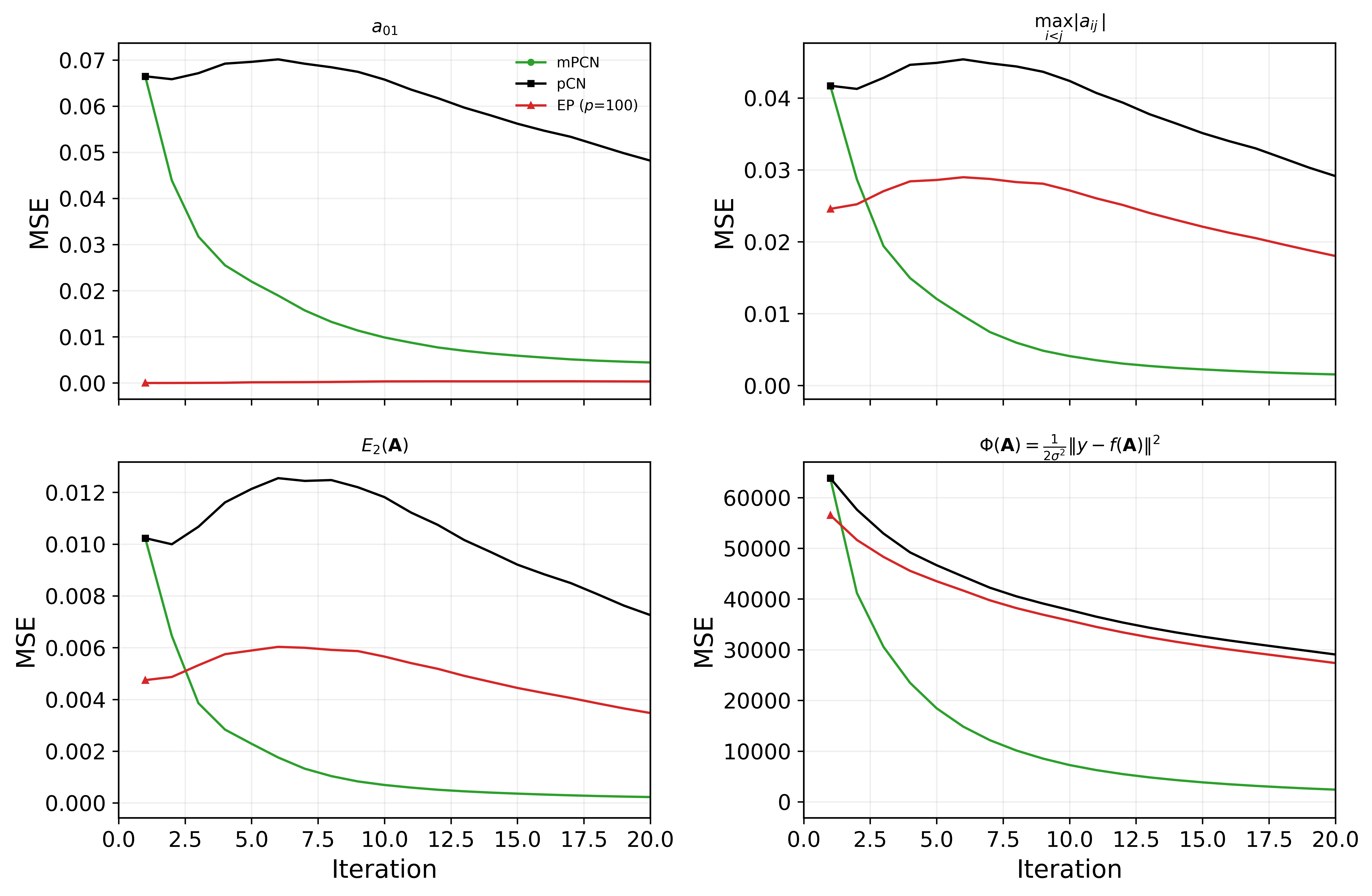}
    \caption{Running MSE for each of the estimators in~\eqref{eq:estimators}, for the four observables in~\eqref{eq:observables}, for $M=50$ replicates and $T=20$ iterations. The 50 mpCN and pCN chains are initialized at the same random 50 samples from the prior and therefore the error at the first iteration is the same. For all observables considered, the error decreases the fastest with mpCN. For the observable $a_{01}$, EP-pCN has an advantage over single chain pCN due to the increased number of prior-sampled starting points, which match the true value of $a_{01}$ well.}
    \label{fig:running_mse}
\end{figure}

\subsubsection{Stationary phase analysis}\label{subsec:solute:statio}
For the stationary phase, we set $d=10$ and generate datasets from the model using hyperparameter values as in~\eqref{eq:hyperpars_st}. Then, we configure the observational scale with $d_0=6$ and $k=2$, corresponding to observing modes 4, 5, and 6 (the 4th, 5th and 6th elements of the observational vector $\bm{y}_{10}$). A similar setup was used in the numerical experiments in \cite{senn2026mess}. Since in this example we do not require visual observation of the warm-up phase, the dimensions of the problem can be chosen smaller compared to those in Section~\ref{subsec:warm-up}, with the corresponding increased computational speed.

\cref{fig:ad_simulated_data_stat} (left) in Appendix~\ref{ap:numerics} illustrates the generated $\bm{A}_{10}$, $\bm{\theta}(\bm{A}_{10})$ and $\bm{y}_{10}$ and \cref{fig:pairplots_stat_mpcn} illustrates the resulting complex posterior geometry.
This last visualization includes marginal histograms for $a_{01}, a_{02}, a_{03}, a_{12}$, and $a_{13}$ on the diagonal, and corresponding pairwise density plots off-diagonal. The highly non-Gaussian marginal posteriors and complex correlation structures are a result of the non-linear map in~\eqref{eq:ad_toy}. 
\begin{figure}
    \includegraphics[width=0.9\linewidth, trim={0 0 0 2cm}, clip]{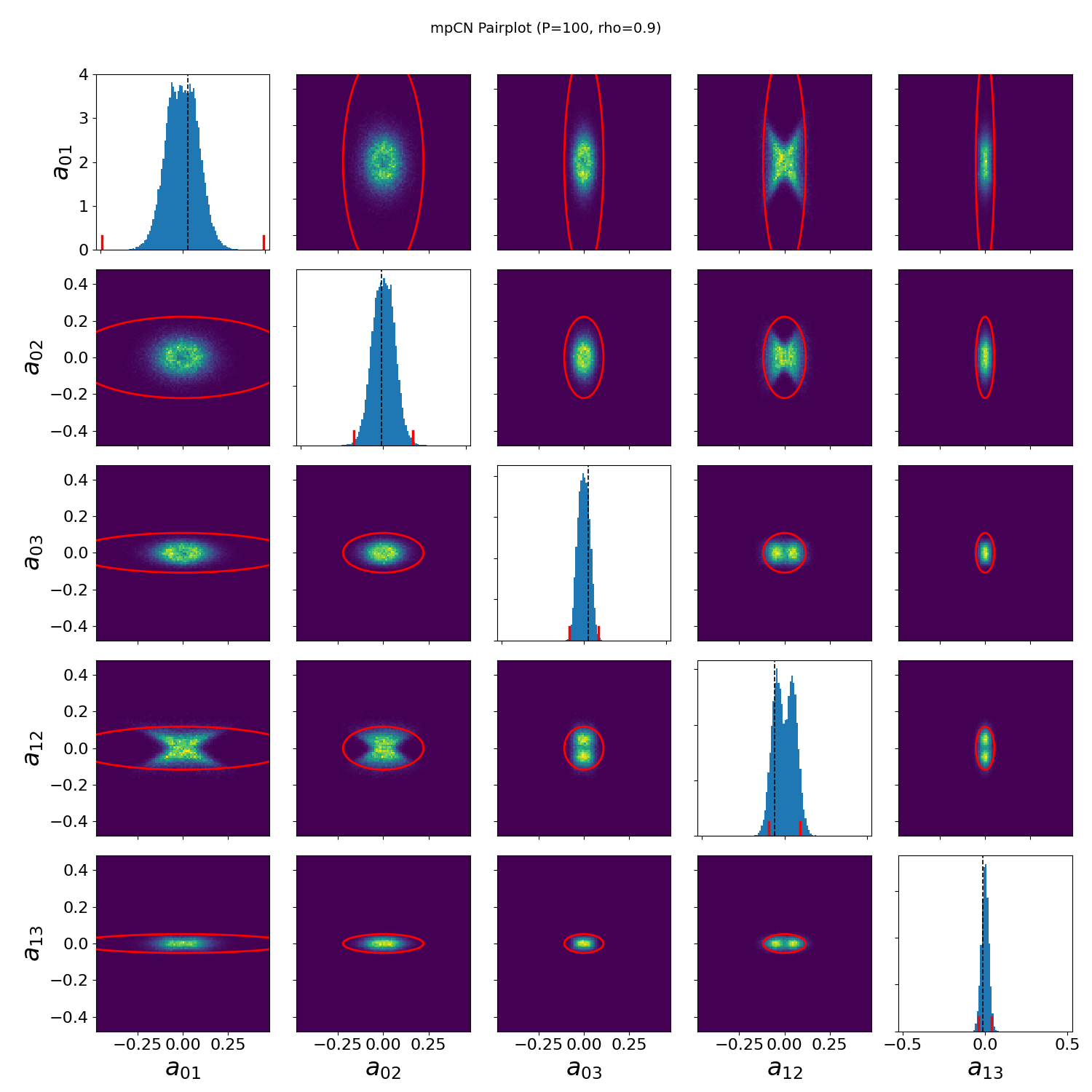}
\caption{Posterior marginal histograms and two-dimensional density plots for the components $a_{01}, a_{02}, a_{03}, a_{12}, a_{13}$, computed on a single mpCN chain with $p=100$ proposals and 300k iterations.}
\label{fig:pairplots_stat_mpcn}
\end{figure}

In \cref{fig:sweep_solute} we analyze the mixing in the stationary regime as a function of $p$ and $\rho$ simultaneously. To quantify mixing, we compute the ESS and MSJD using samples from each chain, after burn-in. We consider a $p-\rho$ configuration grid with $9 \times 40$ combinations, i.e. for $p \in \{10, 20, 30, 40, 100, 300, 500, 1000, 2000\}$ and $\rho \in \{0, 0.025, \ldots, 0.975, 1\}$. We then run one mpCN chain for each of the resulting 360 configurations. Then, we compute the ESS and MSJD for 1) the raw coordinates $a_{01}, a_{02}, a_{03}, a_{04}, a_{12}, a_{13}, a_{14}$, 2) the potential, and 3) the 2-band energy. 
For the raw coordinates, we report the computed ESS and MSJD averaged over the seven coordinates.
The ESS and MSJD curves are shown in the first and second rows, respectively, and each column is one of the three observables just mentioned.

In general, the mixing improves with $\rho$, until the proposal becomes so conservative at very large $\rho$ that the chain eventually stops moving. This behaviour could suggest a mismatch between the prior and the posterior. Increasing $p$ partially mitigates this, as we observe that we achieve maximum mixing at smaller $\rho$. In fact, we observe that the mixing curves become flatter with increased $p$, that is, that increasing $p$ also increases the range of values of $\rho$ for which the observed mixing remains close to optimal. This is because large $p$ increases the possibility of hitting a high probability region, thus allowing for less conservative $\rho$s. 
The saturation point is not clear, and for this example higher $p$ would probably lead to better mixing. 

\begin{figure}
    \centering
    \includegraphics[width=1\linewidth, trim={0 0 0 0cm}, clip]{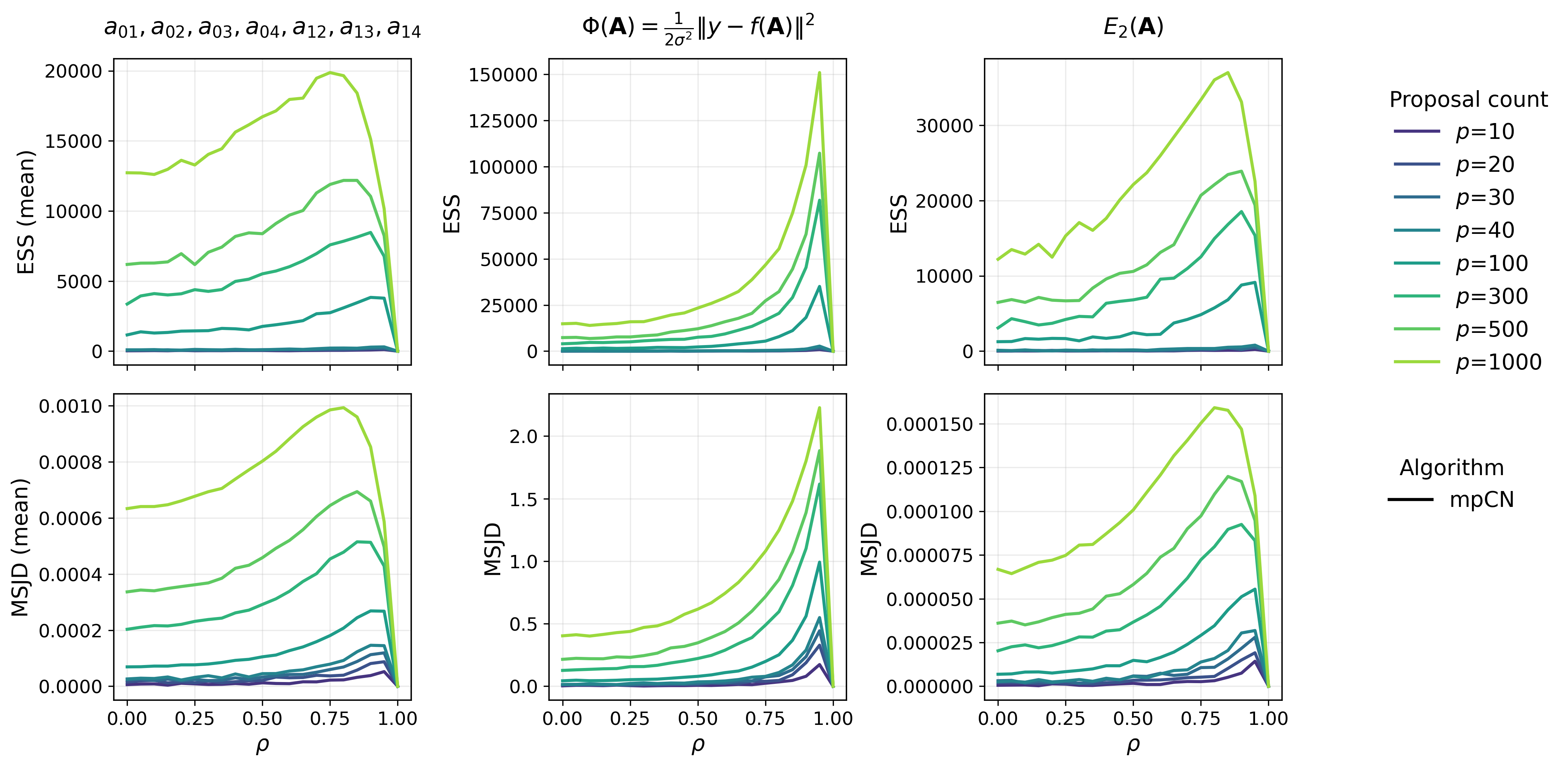}
    \caption{$p$-$\rho$ sweep for ESS and MSJD for the solute transport model. The ESS and MSJD in the first column has been averaged over the seven coordinates in the column label. The mixing increases with $\rho$ and then decreases again for very large $\rho$, independently of the proposal count.}
    \label{fig:sweep_solute}
\end{figure}

\cref{fig:fraction} shows more clearly how increasing $p$ reduces the sensitivity of mpCN to tuning ot $\rho$. In this figure, the y-axis represents the fraction of values of $\rho$ for which the obtained ESS (MSJD) lies within 25\% of the maximum ESS (MSJD) obtained, and this fraction is plotted as a function of the proposal count $p$ in the x-axis. 
\begin{figure}
    \centering
    \includegraphics[width=0.5\linewidth, trim={0 0 0 0cm}, clip]{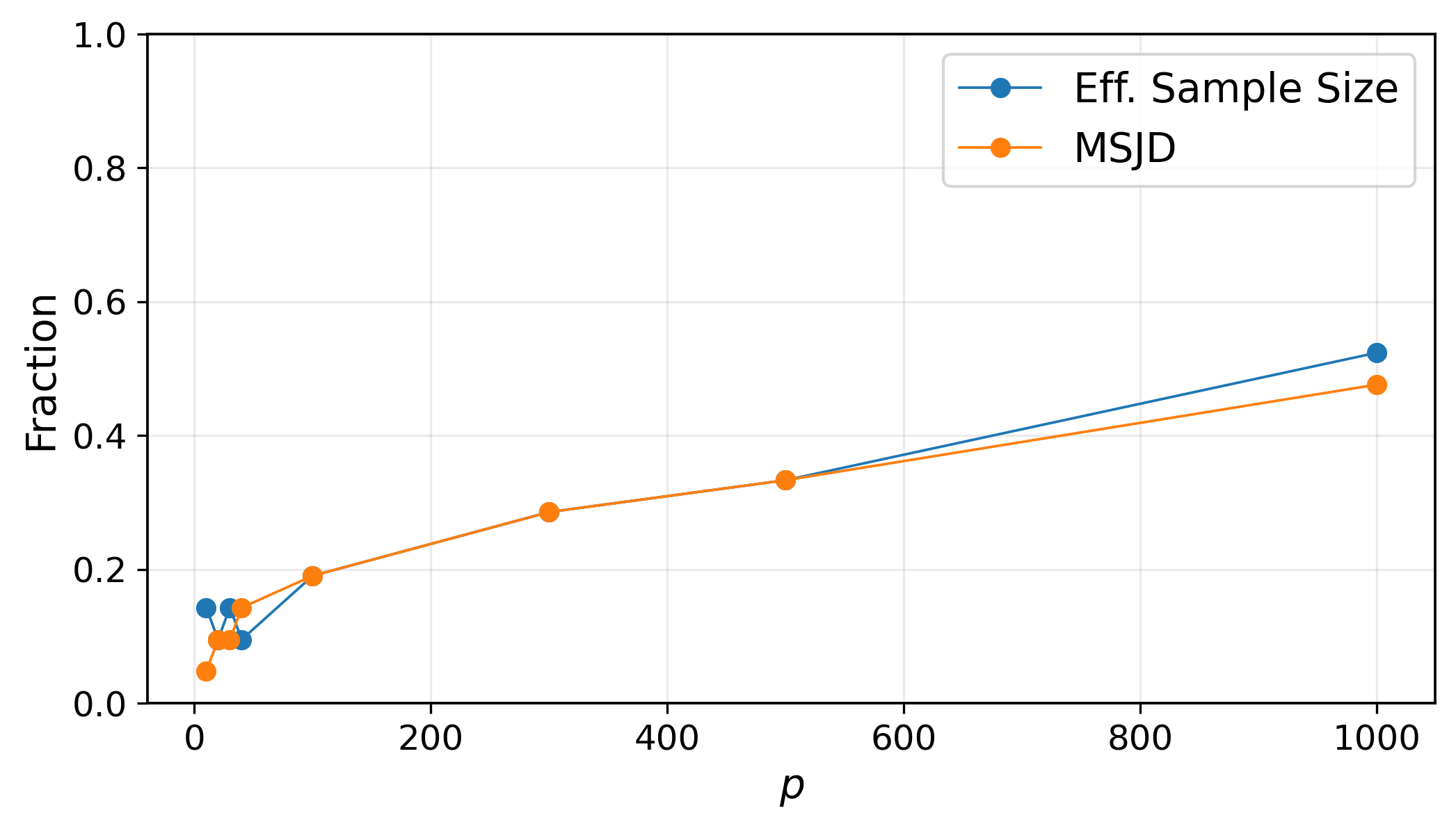}
    \caption{Fraction of $\rho$ within 25\% of max. Increasing the number of proposals also increases robustness to the choice of values of $\rho$.}
    \label{fig:fraction}
\end{figure}

In \cref{fig:sweep_solute_ep} we compare the mixing of mpCN with $p$ proposals to the mixing obtained with $p$ pCN chains, ran independently in parallel, and thinned every $p$ samples. The comparison is thus fair in terms of computational time and memory storage. Thinning the pCN chains every $p$ samples returns in turn a new Markov chain where the transition operator is equivalent to the 1-step transition operator of the pCN chain, applied $p$ steps. 
We do this comparison for $p \in \{10, 40, 100\}$ and $\rho \in \{0, 0.025, \ldots, 0.975, 1\}$. 
The results show that the mixing at stationarity of mpCN is bounded above by the mixing obtained with the $p$ thinned chains. 
\cref{fig:pairplots_stat_ep} in Appendix~\ref{ap:numerics} includes a visualization of the posterior analogous to that in \cref{fig:pairplots_stat_mpcn}. The figure is constructed from samples obtained from 100 pCN chains thinned every 100 samples, providing thus a graphical comparison between the two approaches that is fair in terms of computational budget and wall-clock time. The results enforce the interpretation arising from \cref{fig:sweep_solute_ep}. 

\begin{figure}
    \centering
    \includegraphics[width=1\linewidth, trim={0 0 0 0cm}, clip]{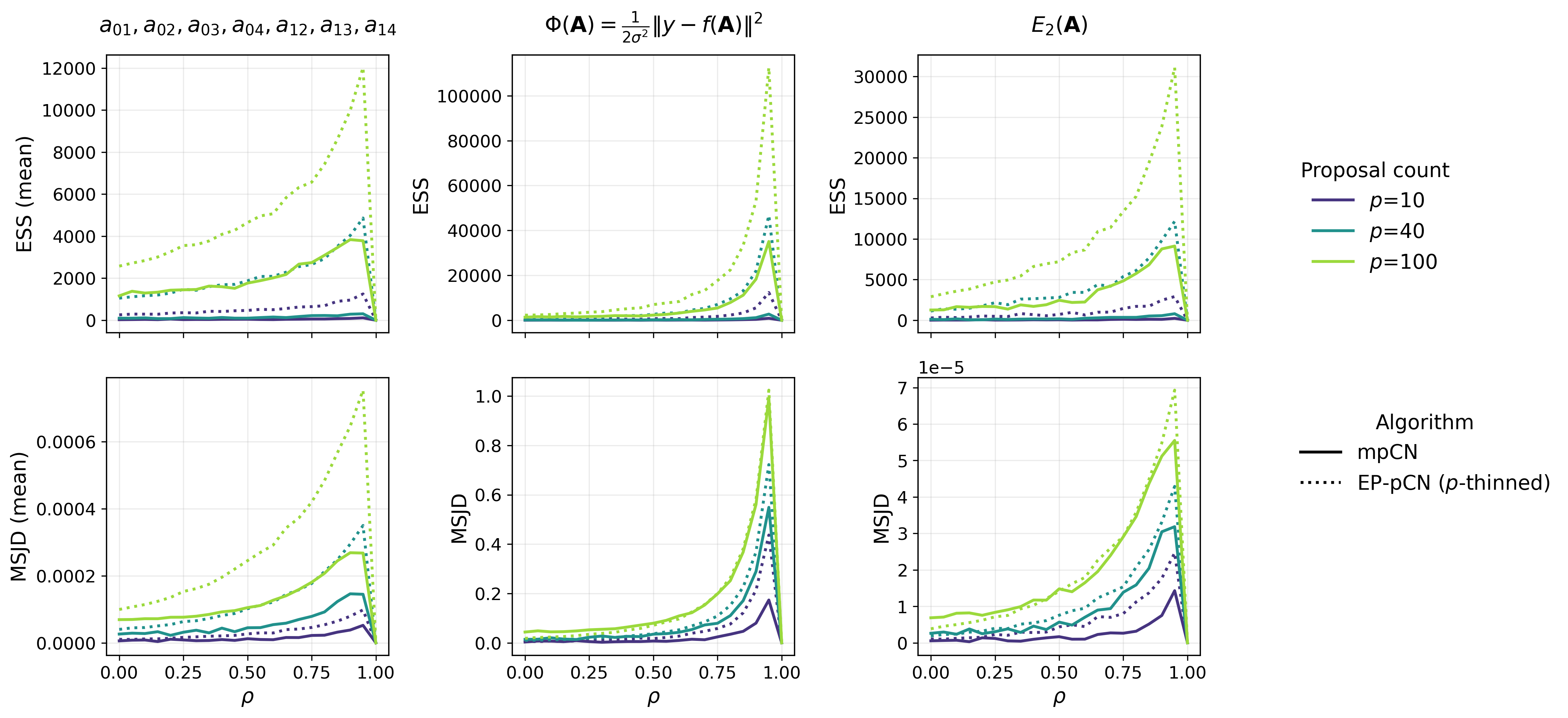}
    \caption{$p$-$\rho$ sweep for ESS and MSJD for the solute transport model, comparing the performance at stationarity of mpCN with $p$ independent pCN chains, thinned every $p$ samples. The mixing of the latter is superior to that of mpCN.}
    \label{fig:sweep_solute_ep}
\end{figure}

\section{Summary and Outlook}\label{sec:outlook}

This work is a systematic analysis of two recently discovered multiproposal, preconditioned, gradient-free methods, mpCN and MTpCN, which are applicable to high dimensional problems defined around a Gaussian reference measure.  We provide a rigorous proof of dimension (and proposal size) free independent mixing rates for these methods.  Moreover, we have begun to accrue some very interesting preliminary numerical evidence for two precise advantages of mpCN and MTpCN: in overcoming difficulties in the burn-in phase and in the robustness of mixing as a function of the specification of algorithmic parameters. These two properties of these methods are potentially decisive advantages over other status quo anti methods in the same Hilbert space family in important situations of practical interest. 

Notwithstanding these contributions we should emphasize that our progress herein represent some initial steps in a wider program rather than a set of definitive conclusions for this research area; many fascinating issues remain wide open.  Specifically, one immediate open question is to determine the precise way in which mpCN and MTpCN should be tuned as a function of the algorithmic parameters $p$ and $\rho$ and of the structure of the target $\mu$.  More broadly it remains to develop a more complete accounting of the relative advantages and disadvantages of various approaches now available within the wider family of Hilbert space methods and multiproposal methods.  This wider question of scope comes into sharp focus in view of other recently derived methods at the Multiproposal Hilbert space intersection (\cite{glatt2024sacred, glatt2025+,senn2026mess}) not covered by our discussions herein.

Regarding the first direction:  determining algorithmic parameter tuning for mpCN and MTpCN, our numerical experiments  only go so far as to suggest that increasing the number of proposals $p$ makes these algorithms more robust with respect to the choice of the parameter $\rho$. This raises a natural question: how should one optimally scale $\rho$ as $p$ increases? Here of course, $p$, the number of proposals is fundamentally constrained by the available computational resources, while $\rho$ is the  parameter to tune for the given $p$ accordingly.
While we intend to carry out a more systematic set of numerical studies 
for general guidance and to formulate clear conjectures, it would of course be
desirable to analyze this $\rho(p)$ optimality relationship on a rigorous basis.

One possible rigorous approach to address this $\rho-p$ scaling question is to draw on the strategy developed in \cite{andrieu2024explicit}, where upper as well as lower bounds on the spectral gap are obtained for the Random Walk Metropolis (RWM) and pCN methods with precise algorithmic parameter dependence through Cheeger's inequality. Lower bounds in our mpCN context here would be particularly valuable in order to obtain a more precise understanding of the dependence of the $L^2_{\mu}$ spectral gap on $p$ and $\rho$. Preliminary work in progress reveal a rich and technically delicate structure in the multiproposal setting. Further work is needed to understand whether these estimates can be made tractable enough to produce useful tuning principles. 

We note however that, without diminishing the significance of the innovative new approaches developed in \cite{andrieu2024explicit}, this `Functional-inequalities' direction comes with some important limitations.  In particular it requires strong uni-modal assumptions on the target distribution.   The methods of \cite{andrieu2024explicit} therefore applies to a much narrower class of examples than we are able to address within the weak Harris framework adopted herein.  The tension is that our Harris approach, does not seem to yield any meaningful dependence in the upper bound mixing constants on the important problem parameters; not least the dependence on $\rho-p$.   Notably however our proofs in this current work do suggest explicit coupling strategies of mpCN chains.  These coupling suggest possible avenues for developing the ideas in e.g. \cite{biswas2019estimating} toward a novel, semi-analytical approach to estimating the $p-\rho$ relationship.  Significantly this semi-analytical route would allow us to estimate this $p-\rho$ relationship in a number of non-equivalent Wasserstein metrics.  Note that the potential to address different metric within a single general framework underlines the subtile point that $\rho(p)$ optimality may depend
on the class of observables under consideration.

Zooming out, the second wider direction is to clarify how mpCN and MTpCN fit amongst other established and newly emerging methods designed for high- and infinite-dimensional sampling problems. In particular the algorithms studied herein represent two ways of extending the pCN method to a cloud of proposals considered at once, but they are not the only possible way to exploit parallelism or nonlocal proposal mechanisms to address high dimensional problems, even within the specific pCN paradigm. It would therefore be useful to compare mpCN more directly with other approaches built around Gaussian reference measures, including a local variant of mpCN, as introduced in \cite{glatt2025+}, auto-tuning methods such as the Multiproposal Elliptic Slice Samplier (MESS) \cite{senn2026mess}, as well as approaches based on approximate or surrogate trajectories \cite{glatt2020accept,  glatt2024sacred}. As far as we can tell the rigorous analysis of mixing for each of these methods, via either the weak Harris or Functional-Inequalites routes remains essential wide open.  Of course various benchmark problems including those consider for mpCN and MTpCN here in \cref{sec:numerics}
(as well as other related PDE informed Bayesian models used previously in e.g. \cite{stuart2010inverse, glatt2024parallel, glatt2024sacred}) provide a starting point for systematic comparative numerical case studies which could be of great value.

\section*{Acknowledgments}

Our efforts are supported under the grants NSF-DMS-2108790,
NSF-DMS-2510856 (NEGH), DMS-2239325 (CFM). GC and MS gratefully acknowledge the Department of Mathematics at Drexel University for hosting them on separate occasions, which provided valuable opportunities to collaborate on this paper. We
would like to thank Andrew Holbrook, Justin Krometis and Andrew Warren for inspiring discussions and helpful feedback on this work.

\begin{footnotesize}
	\addcontentsline{toc}{section}{References}
	\bibliographystyle{alpha}
	\bibliography{refs}
\end{footnotesize}

\appendix

\section{Wasserstein contraction for unbounded potentials}\label{app:unbound}
In this section we verify conditions 1 and 2 of the weak Harris theorem, \cref{thm:our_harris}, for the multiproposal pCN algorithm, \cref{alg:mpcn}, under the assumption that the potential $\Pot$ is Lipschitz, without requiring boundedness.

At first sight, by analogy with the spectral gap results for the single-proposal pCN algorithm established in \cite{hairer2014spectral}, one might expect Lipschitz continuity alone, either global or local, to suffice in order to prove the desired Wasserstein contraction properties also in the multiproposal setting. In the classical pCN algorithm, the acceptance mechanism is given by the Metropolis--Hastings probability
\[
    \alpha(x,y) = 1 \wedge \exp(-\Pot(y) + \Pot(x)),
\]
which is known to be Peskun optimal and therefore minimizes the asymptotic variance among a broad class of reversible acceptance rules. By contrast, the multiproposal setting behaves differently. As discussed in \cite{glatt2025+}, the natural and effective choice of acceptance probabilities is instead of Barker type, namely
\[
    \alpha(x,y) = \frac{\exp(-\Pot(y))}{\exp(-\Pot(x)) + \exp(-\Pot(y))},
\]
and more generally as in \eqref{eq:intro:alpha}. This structural difference introduces additional analytical difficulties in the coupling arguments required for the weak Harris framework, particularly in the absence of boundedness assumptions on $\Pot$.

In the proof of $d_\eps$-contraction and $d_\eps$-smallness for mpCN we made essential use of the boundedness of the potential function $\Pot$ in estimating the acceptance probabilities. Indeed, under the assumption that $\Pot$ is bounded, we have
\[
    \alpha_j(x_0, x_1, \ldots, x_p)
    =
    \frac{\exp(-\Pot(x_j))}{\sum_{k = 0}^p \exp(- \Pot(x_k))}
    \in
    \left[
    \frac{\exp(- 2\|\Pot\|_\infty)}{p+1},
    \frac{\exp(2\|\Pot\|_\infty)}{p+1}
    \right],
    \qquad j =0, \ldots, p.
\]
These bounds are particularly useful when estimating differences of acceptance probabilities in the proof of $d_\eps$-contraction, for example in \eqref{ineq:sp}. In that setting, they allow us to establish Lipschitz continuity of the acceptance probabilities with respect to the starting point, with a Lipschitz constant that remains uniformly bounded in the number of proposals $p$. 
This mechanism appears to break down if one assumes only that $\Pot$ is Lipschitz continuous. The difficulty stems from the Barker-type structure \eqref{eq:intro:alpha} of the acceptance probabilities, which introduces denominators involving sums involving $\Pot$. 

To overcome this issue, we impose an alternative condition, analogous in spirit to Assumption~2.10 in \cite{hairer2014spectral}. Heuristically, this assumption ensures that, even when the chain starts far from the origin, there remains a sufficiently large probability of accepting at least one proposal, thereby preventing the dynamics from becoming effectively frozen in regions where the potential is large. More precisely we impose:
 
 \begin{Assumption}\label{assumptionalpha}
    There is $R>0$, $\alpha^*>0$ and a function $r:\R^+\to \R^+$ with the property $r(t) \leq (1 - \rho^2)t/2$ for all $|t|\geq R$ such that for at least one $j =1 \ldots, p$ the following is true: for all $x\in B(0,R)^c$ and all $z_1, \ldots, z_{j-1}, z_{j+1}, \ldots, z_n\in \qsp$ 
    \begin{equation}
        \inf_{z_j \in B(\rho^2 x, r(\|x\|)) }\alpha_j(x, z_1, \ldots,z_j, \ldots z_p) > \alpha^* .
    \end{equation}
\end{Assumption}

Then we have the following result for the mpCN kernel, alternative to \cref{thm:3:mpcn1}.

\begin{Theorem}\label{thm:3:mpcn2}
    Let $P_p$ be as in \eqref{eq:2:mpcn:kernel} for a fixed  $\rho\in [0,1)$ and $p>1$ with function $\Pot$ globally Lipschitz with respect to the norm $\|\cdot \|$ with constant $L_{\Pot}$. If \cref{assumptionalpha} is satisfied, then the result of \cref{thm:3:mpcn1} holds with $\varepsilon^*= \varepsilon^*(p)\to 0$, $n_1(p)\to \infty$ and $\lambda = \lambda(p)\to 1$ when $p\to \infty$.
\end{Theorem}

To use the weak Harris theorem \cref{thm:our_harris}, we will use the functions found in \cref{prop:mpcn:lyap}, which can be showed still to be Lyapunov under some modifications of the proof, and we show alternatives of the $d_\eps$-contraction result \cref{prop:contr:fin:p:1} and the $d_\eps$-smallness result \cref{prop:mpcn:small}.

\begin{Proposition}
    Fix $1\leq p <\infty$, assume the potential function $\Pot: \qsp \to \R$ is globally Lipschitz with constant $L_\Pot$, and that \cref{assumptionalpha} holds. Then there exists $\kappa = \kappa(\rho, p, \eps, L_\Pot)>0$ such that 
    \begin{align*}
    	\Wass_{d_\eps}(\mk_p(x_0, \cdot), \mk_p(\tx_0, \cdot)) 
    	\leq \kappa d_\eps(x_0, \tx_0)
    \end{align*}
    for all $x_0, \tx_0 \in \qsp$ satisfying $d(x_0, \tx_0) < 1$.
    Moreover, there exists $\varepsilon^*= \eps^*(\rho, p, L_\Phi)>0$ such that $\kappa<1$ for all $\eps< \eps^*$. Finally, as $p \to \infty$, then the parameters saturate $\eps^*\to 0$ and $\kappa\to 1$.
\end{Proposition}

\begin{proof}
Let $R > 0$ be as in \cref{assumptionalpha}. Fix $x_0, \tx_0 \in \qsp$ with $d_\eps(x_0, \tx_0) < 1$. Then, for $\varepsilon < 1$, it follows that either $x_0, \tx_0 \in B(0, R+1)$, or $x_0, \tx_0 \in B(0,R)^c$.

First, suppose that $x_0, \tx_0 \in B(0, R+1)$. Define the set
\begin{align}\label{eq:ap:setA}
	A = \{ (z_0, \ldots, z_p) \in \qsp^{p+1} \,:\, \sqrt{1 - \rho^2} \| \rho z_0 + z_j \| \leq R + 1 \mbox{ for all } j=1, \ldots, p\}.
\end{align}
We proceed with the same steps as in the beginning of the proof of \cref{prop:contr:fin:p:1} leading to \eqref{eq:contr1}, namely 
\begin{equation}\label{eq:contr1:ap}
   \Wass_{d_\varepsilon}(\mk^p(x_0, \cdot), \mk^p(\tx_0, \cdot)) \leq  d_\eps(x_0, \tx_0)\E  \left( \har_0 +  \rho^2  \sum_{j = 1}^{p} \har_j \right) + \E (1 - s_p). 
\end{equation}
Recall, using the synchronous coupling for the proposals, we set 
\[
    \har_j = \har_j(x_0, \tx_0;\Xi) = \min\{\alpha_j(x_0, X_1, \ldots, X_p), \alpha_j(\tx_0,\tX_1, \ldots, \tX_p)\}, \, \quad j=0, \ldots, p
\]
where $\Xi = (\xi_0, \ldots, \xi_p)\sim \mu_0^{\otimes (p+1)}$ and $s_p = \sum_{j = 0}^p \har_j\leq 1$.
Focusing on the first term on the right hand side of \eqref{eq:contr1:ap}, we write
\begin{align}
	\E \left( \har_0 + \rho^2 \sum_{j =1}^p \har_j\right) &= \E \left( \har_0 + \rho^2 \sum_{j =1}^p \har_j\right)\mathbbm{1}_{\Xi\in A} + \E \left( \har_0 + \rho^2 \sum_{j =1}^p \har_j\right)\mathbbm{1}_{\Xi\in A^c} \notag\\
    &\leq   \E \left( \har_0 + \rho^2 \sum_{j =1}^p \har_j\right)\mathbbm{1}_{\Xi\in A} +\mu_0^{\otimes (p+1)} (A^c).\label{eq:I:II}
\end{align}

Recalling the notation \eqref{eq:mpcn:proposal1} for the synchronously coupled proposals, it follows that if $\Xi\in A$, then 
\begin{align*}
	\| X_j \| \leq \rho^2 \| x_0 \| + \sqrt{1 - \rho^2} \| \rho \xi_0 + \xi_j \| 
	\leq \rho^2 (R+ 1) + R+1 < 2(R + 1), \quad j = 1, \ldots, p,
\end{align*}
and, similarly, $\| \tX_j \| \leq 2 (R + 1)$, $j = 1, \ldots, p$. Therefore, for all $\Xi\in A$, $j = 1, \ldots, p$
\begin{align*}
	\alpha_j(x_0, X_1, \ldots, X_p) 
	= \frac{e^{-\Pot(X_j)}}{e^{-\Pot(x_0)}+\sum_{k=1}^p e^{-\Pot(X_k)}} 
	\geq \frac{\exp (- \sup \{ \Pot(z) \,:\, \| z \| \leq 2 (R+1) \})}{ (p+1) \exp ( - \inf \{ \Pot(z) \,:\, \|z \| \leq 2 (R + 1)\})}
	=: \frac{C_R}{p+1},
\end{align*}
and similarly for $\alpha_j(\tx_0,\tX_1, \ldots, \tX_p)$, so that
\begin{align*}
	\har_j =\har_j(x_0, \tx_0;\Xi) \geq \frac{C_R}{p+1} \quad \mbox{for all } \,\, \Xi \in A, \,\, j = 1, \ldots, p.
\end{align*}

In \eqref{eq:I:II} we can then estimate
\begin{align*}
	 \har_0  + \rho^2 \sum_{j=1}^p \har_j 
	&\leq 1 - \sum_{j=1}^p \har_j + \rho^2 \sum_{j=1}^p \har_j
	= 1 - (1 - \rho^2) \sum_{j=1}^p \har_j  \notag \\
	&\leq 1 - ( 1 - \rho^2 ) \frac{p}{p+1} C_R,
\end{align*}
hence 
\begin{equation}\label{ineq:int:I}
    \E \left( \har_0 + \rho^2\sum_{j =1}^p \har_j\right)\mathbbm{1}_{\Xi\in A} \leq \left( 1 - ( 1 - \rho^2 ) \frac{p}{p+1} C_R\right) \mu_0^{\otimes(p+1)} (A).
\end{equation}
It follows from \eqref{eq:contr1:ap}, \eqref{eq:I:II}, and \eqref{ineq:int:I} that
\begin{align}\label{Wass:contr:nb:0}
	\Wass_d (\mk_p(x_0, \cdot), \mk_p(\tx_0, \cdot))
	&\leq \left[ \left(1 - ( 1 - \rho^2 ) \frac{p}{p+1} C_R \right) \mu_0^{\otimes (p+1)} (A) + \mu_0^{\otimes (p+1)} (A^c)\right]  d(x_0, \tx_0) 
	\notag \\
	&\qquad\qquad\qquad\qquad \qquad\qquad \qquad\qquad\qquad + \E (1 - s_p)
	\notag \\
	&=: K_1(p) d(x_0, \tx_0) + \E (1 - s_p)
\end{align}
with $K_1(p) = 1 - ( 1 - \rho^2 ) \frac{p}{p+1} C_R \mu_0^{\otimes (p+1)} (A) < 1$ for all $p\geq 1$.

We now proceed to estimate the last term in \eqref{Wass:contr:nb:0}. Since here we do not assume $\Pot$ to be bounded, we must obtain a different estimate than \eqref{eq:mpcn:1-sp}. First, recall from the construction in the proof of \cref{prop:contr:fin:p:1} that \eqref{eq:mpcn:1-sp_2} holds, namely 
\begin{align}\label{ineq:sp:b}
	\E \left(1 - s_p \right)\leq \sum_{j=0}^p\E | \ar_j (x_0, X_1, \ldots, X_p) - \ar_j (\tx_0, \tX_1, \ldots, \tX_p) |.
\end{align}

For each $j \in \{0,\ldots, p\}$, denote
\begin{align*}
	G_j(u_0,\ldots, u_p) = \frac{e^{-u_j}}{\sum_{k=0}^p e^{-u_k}}, \quad (u_0, \ldots, u_p) \in \R^{p+1},
\end{align*}
and let $G = (G_0, \ldots, G_p): \R^{p+1} \to \R^{p}$. Moreover, denote $y_j = \Pot(x_j)$ and $\ty_j = \Pot(\tx_j)$ for $j=0, \ldots, p$. With this notation, we have
\begin{align*}
	&\sum_{j=0}^p | \ar_j (x_0, \ldots, x_p) - \ar_j (\tx_0, \ldots, \tx_p) |
	=
	\sum_{j=0}^p |G_j(y_0, \ldots, y_p) - G_j(\ty_0, \ldots, \ty_p) | 
	\\
	&\qquad \qquad \leq 
	\sqrt{p} \left( \sum_{j=0}^p |G_j(y_0, \ldots, y_p) - G_j(\ty_0, \ldots, \ty_p) |^2 \right)^{1/2}
	= \sqrt{p} \| G(y_0, \ldots, y_p) - G(\ty_0, \ldots, \ty_p) \|_{\R^p}, 
\end{align*}
where $\| \cdot \|_{\R^p}$ denotes the Euclidean norm in $\R^p$. By a direct calculation, we have
\begin{align*}
	\partial_k G_j (u_0, \ldots, u_p) = G_j(u_0, \ldots, u_p) [ G_k(u_0, \ldots, u_p) - \delta_{kj}], \quad k = 0,\ldots, p, \,\, j = 0, \ldots, p,
\end{align*}
and then it is not difficult to show that $\| DG(u_0, \ldots, u_p)\| \leq \sqrt{2}$ for all $(u_0, \ldots, u_p) \in \R^{p+1}$.

By the mean value theorem and the Lipschitzianity of $\Phi$, it follows that 
\begin{align*}
	1 - s_p(Z)
	&\leq 
	\sqrt{p} \| DG \|_\infty \| (y_0, \ldots, y_p)  - (\ty_0, \ldots, \ty_p) \|
	= \sqrt{p} \| DG \|_\infty \left( \sum_{k=0}^p ( \Pot(x_k) - \Pot(\tx_k))^2 \right)^{1/2} 
	\\
	&\leq \sqrt{p} \| DG \|_\infty L_{\Pot} \left( \sum_{k=0}^p \|x_k - \tx_k\|^2 \right)^{1/2}.
\end{align*}

Therefore, using the definition of the proposals \eqref{eq:mpcn:proposal1}, we deduce that for any $\Xi\sim \mu_0^{\otimes(p+1)}$
\begin{align}
	\E\left(  1 -s_p(\Xi) \right)  
    &\leq \sqrt{2p} L_{\Pot} \sqrt{1 + \rho^2 p} \|x_0 - \tx_0\| =  C \sqrt{p (1 +\rho^2 p)} \varepsilon d(x_0, \tx_0),\label{ineq:sp:c}
\end{align}
with $C = C(L_{\Pot}) > 0$. Plugging into \eqref{Wass:contr:nb:0}, we have
\begin{align}\label{eq:app:K1}
	\Wass_d (\mk^p(x_0, \cdot), \mk^p(\tx_0, \cdot))
	\leq 
	[K_1(p) + C \sqrt{p (1 +\rho^4 p)} \varepsilon ]d(x_0, \tx_0).
\end{align} 

Next, we analyze the case when $x_0, \tx_0 \in B(0,R)^c$. Here, we define
\begin{align*}
	B_p = \left\{ Z = (z_0, \ldots, z_p) \in \qsp^{p+1} \,:\, \sqrt{1 - \rho^2} \| \rho z_0 + z_j \| \leq r(\| x_0 \|) \wedge r(\| \tx_0 \| ) \mbox{ for all } j =1, \ldots,p \right\}.
\end{align*}
Thus, for all $\Xi \in B_p$, the corresponding $X_j, \tX_j$ satisfy $X_j \in B(\rho^2 x_0, r(\| x_0\|))$ and $\tX_j \in B(\rho^2 \tx_0, r(\| \tx_0\| ))$, $j = 1, \ldots, p$. By \cref{assumptionalpha}, it follows that there exists $l \in \{1,\ldots, p\}$ such that 
\begin{align}\label{lbound:har}
	\har_l (\Xi) > \alpha^* > 0 \quad \mbox{for all } \Xi\in B_p.
\end{align}

Similarly as in the previous case, we have
\begin{align*}
	&\Wass_d (\mk^p(x_0, \cdot), \mk^p(\tx_0, \cdot)) \leq d_\eps(x_0, \tx_0)\E \left( \har_0 + \rho^2\sum_{j =1}^p \har_j\right) + \E (1 - s_p) \\
    &\leq  d_\eps(x_0, \tx_0) \E \left( \har_0 + \rho^2\sum_{j =1}^p \har_j\right)\mathbbm{1}_{\Xi\in B_p} + d_\eps(x_0, \tx_0) \mu_0^{\otimes (p+1)}(B_p^c)  + \E (1 - s_p) .
\end{align*}
From \eqref{lbound:har}, it follows similarly as in \eqref{ineq:int:I} that for all $\Xi \in B_p$
\begin{align*}
	 \har_0(\Xi) + \rho^2 \sum_{j=1}^p \har_j(\Xi)
	 \leq 1 - (1 - \rho^2) \sum_{j=1}^p \har_j(\Xi)
	 < 1 - (1 - \rho^2) \alpha^*.
\end{align*}
Thus, together with \eqref{ineq:sp:c}, we obtain that
\begin{align}
	\Wass_d (\mk_p(x_0, \cdot), \mk_p(\tx_0, \cdot))
	&\leq \left[ (1 - (1 - \rho^2) \alpha^*) \mu_0^{\otimes (p+1)} (B_p) + \mu_0^{\otimes (p+1)} (B_p^c) \right] d(x_0, \tx_0) + C \sqrt{p (1 +\rho^4 p)} \varepsilon d(x_0, \tx_0)
	\notag\\
	&\quad =: K_2(p) d(x_0, \tx_0) + C \sqrt{p (1 +\rho^4 p)} \varepsilon d(x_0, \tx_0), \label{eq:app:K2}
\end{align}
where note that $K_2(p) < 1$ for each fixed $p<\infty$. 

In summary we showed the estimates \eqref{eq:app:K1} for $x_0, \tx_0\in B(0,R)$ and \eqref{eq:app:K2} for $x_0, \tx_0\in B(0, R)^c$, then in general for all $x_0, \tx_0$ such that $d(x_0, \tx_0)<1$ it follows 
\begin{equation*}
    W_{d_\eps} (P_p(x_0, \cdot) , P_p(\tx_0, \cdot)) \leq \kappa d_\eps(x_0,\tx_0)
\end{equation*}
with 
\begin{equation*}
    \kappa = \kappa(\varepsilon, p) = \max\left(K_1(p) , K_2(p)\right)  + C \sqrt{p (1 +\rho^4 p)} \varepsilon 
\end{equation*}
so that $\kappa(\eps, p)<1$ for any $\eps$ smaller than
\begin{equation}\label{eq:app:eps*}
     \eps_* = \frac{1 - \max\left(K_1(p), K_2(p)\right)}{ C \sqrt{p (1 +\rho^4 p)}}
\end{equation}
which is well defined as $K_1, K_2$ are smaller than one. 
\end{proof}

Last we show how to adapt the $d$-smallness proof for unbounded $\Pot$.
\begin{Proposition}
    Set the number of proposals $p< \infty$ and consider $P_p$ with $\Pot$ at least continuous. Let $S= \lbrace V(x) \leq 4 K_V\rbrace $ with $V$ any of the Lyapunov functions in \cref{prop:mpcn:lyap}, and let $r_s>0$ be such that $S \subset B(0, r_s)$. Then there exists $s = s(n, p, \rho, \varepsilon, r_s)$ such that 
    \begin{equation}
        \Wass_{d_\eps} (P_p^n(x_0, \cdot), P_p^n(\tx_0, \cdot))\leq s,
    \end{equation}
    for all $x_0, \tx_0\in S$. Moreover, there exists $n^* = n^*(\rho, \varepsilon, r_s)$ such that $s<1$ for all $n > n^*$. Finally, as $p \to \infty$ the parameters saturate $s\to 1$ and $n^*\to \infty$. 
\end{Proposition}
\begin{proof}
    We will use the same coupling as in the proof of the $d$-contraction but with a shifted uniform distribution. More precisely, fix $x_0, \tx_0 \in S$. Recalling the definition of the proposals \eqref{eq:mpcn:proposal1}, and $\hat{\alpha}_k$, the minima between the acceptance probabilities, \eqref{eq:alphahat} we now define 
\begin{align*}
     s_j := s_j(x_0, \tx_0; \Xi) = \sum_{k=0}^j \hat{\alpha}_k, \quad j = 0, \ldots, p.
\end{align*}

Next, denote, for $j = 0, \ldots, p$,
\begin{gather*}
	\beta_j= \max\{\ar_{j} (x_0, X_1, \ldots, X_p) - \ar_{j_n}(\tx_0, \tX_1, \ldots, \tX_p), 0\}, \\
	\tbeta_j=  \max \{ \ar_{j}(\tx_0, \tX_1, \ldots, \tX_p) - \ar_{j} (x_0, X_1,\ldots, X_p),0\},
\end{gather*}
and define the real intervals
\begin{gather*}
	J_0 = [s_p, s_p + \beta_0), \quad
	J_k = \left[s_p + \sum_{i = 0}^{k-1} \beta_i, s_p + \sum_{i=0}^k \beta_i \right), \,\, k = 1, \ldots, p-1, \quad J_p = \left[s_p + \sum_{i = 0}^{p-1} \beta_i, s_p + \sum_{i=0}^p \beta_i \right], \\
	\tJ_0 = [s_p, s_p + \tbeta_0), \quad
	\tJ_l = \left[s_p + \sum_{i = 0}^{l-1} \tbeta_i, s_p + \sum_{i=0}^l \tbeta_i \right), \,\,l= 1, \ldots, p-1, \quad \tJ_{p} = \left[s_p + \sum_{i = 0}^{p-1} \tbeta_i, s_p + \sum_{i=1}^{p} \tbeta_i \right].
\end{gather*}
Note that $s_0 + s_p + \sum_{i=0}^p \beta_i = s_0+ s_p + \sum_{i=0}^{p} \tbeta_i = 1$, so that
\begin{align*}
	[s_p, 1- s_0] = \bigcup_{k=0}^p J_k = \bigcup_{l=0}^{p} \tJ_l 
	= \bigcup_{\{(k,l) \,:\, J_k \cap \tJ_l \neq \emptyset \}} ( J_k \cap \tJ_l ) .
\end{align*}
We have then constructed the following partition of the interval $[-s_0, 1 - s_0]$
\begin{equation}\label{eq:partition[-s_0,1-s_0]}
    [-s_0, 1 - s_0] = [-s_0, 0) \cup [0,s_1) \cup \bigcup_{j = 2}^p[s_{j-1}, s_j)\cup  \bigcup_{\{(k,l) \,:\, J_k \cap \tJ_l \neq \emptyset \}} ( J_k \cap \tJ_l ). 
\end{equation}
Then, given $U\sim \cU([-s_0, 1 - s_0]) $, we define the variables
\begin{align*}
(X, \tX) = 
	\begin{cases}
		(x_0, \tx_0) \quad &\mbox{ if } \,\, U \in [-s_0, 0) \\
        (X_1, \tX_1) \quad &\mbox{ if } \,\, U \in [0, s_1) \\
		(X_j, \tX_j) \quad &\mbox{ if } \,\, U \in \left[ s_{j-1},  s_j\right), \,\, j  =2, \ldots, p,\\
        (X_k, \tX_l) \quad &\mbox{ if } \,\, U \in J_k\cap J_l,\mbox{ with } k,l=0, \ldots, p \mbox{ such that }  J_k \cap \tJ_l \neq \emptyset.
	\end{cases}
\end{align*}
It is not difficult to show from this construction that $(X,\tX)$ is still a coupling of $\mk_p(x_0, \cdot)$ and $\mk_p(\tx_0, \cdot)$. Hence, 
\begin{align*}
\Wass_{d_\eps}(\mk_p(x_0, \cdot), \mk_p(\tx_0, \cdot)) 
	&\leq \E d_\eps (X, \tX)
    \\
    \leq  d_\eps(x_0, \tx_0)& \E s_0 + \E d_\eps(X_1, \tX_1)\mathbbm{1}_{U\in [0, s_1]} + \sum_{j = 2}^p \E d_\eps(X_j, \tX_j)\mathbbm{1}_{U \in [s_{j-1}, s_j]} + \E (1 - s_0-s_p)
\end{align*}
and, using the definition of the proposals which are synchronously coupled, and the fact that $d_\eps\leq 1$
\begin{align*}
     \Wass_{d_\eps}(\mk_p(x_0, \cdot), \mk_p(\tx_0, \cdot)) 
     &\leq  d_\eps(\rho^2 x_0, \rho^2 \tx_0)\E s_1 + d_\eps(\rho^2 x_0, \rho^2 \tx_0) + \sum_{j =2}^p \E (s_j - s_{j-1}) + \E (1 - s_p)\\
     & = d_\eps(\rho^2 x_0, \rho^2 \tx_0) \E s_p  + \E (1 - s_p) .
\end{align*}
As $x_0, \tx_0\in S \subset B(0,r_s)$ then we showed that 
\begin{equation*}
      \Wass_{d_\eps} (\mk_p(x_0, \cdot), \mk_p(\tx_0, \cdot))\leq  1 - \left( 1 - \frac{2r_s\rho^2}{\varepsilon} \right) \E s_p
\end{equation*}
To ensure that $ \E s_p>0$ we notice that, given $\Xi = (\xi_0, \ldots, \xi_p)\sim \mu_0^{\otimes (p+1)}$, 
\begin{equation*}
    \E s_p = \E s_p(\Xi) > \E s_p(\Xi)\mathbbm{1}_{\Xi \in B_p}
\end{equation*}
where we define
\begin{equation}\label{eq:B_p}
    B_p := \left\lbrace (w_0, \ldots, w_p) \, : \, w_j \in B\left(0, \frac{r_s}{n}\right) \text{ for all } j = 0, \ldots, p\right\rbrace.
\end{equation}  
Given this constraint on $\Xi$ it follows 
\begin{equation}\label{eq:mpCN:small1}
    \|X_j\| = \| \rho^2 x_0 + \rho \sqrt{1 - \rho^2}\xi_0 + \sqrt{1 - \rho^2} \xi_j \| \leq \rho^2 r_s+ \rho \sqrt{1 - \rho^2}\frac{r_s}{n} + \sqrt{1 - \rho^2} \frac{r_s}{n}   
\end{equation}
and for any $n \geq \sqrt{\frac{1 + \rho^2}{1 - \rho^2}}$, it follows $\| X_j\|\leq r_s$, and similarly for $\tX_j$, $j = 1, \ldots, p$. Namely with the constraint on the size of the stochasticity, the proposals will all land in the ball containing the small set. It follows then 
\begin{equation*}
    \hat{\alpha}_j(x_0, \tx_0, \Xi) = \frac{e^{-\Pot(X_j)}}{e^{-\Pot(x_0)} + \sum_{k=1}^p e^{-\Pot(X_k)}} \wedge \frac{e^{-\Pot(\tX_j)}}{e^{-\Pot(\tx_0)} + \sum_{k=1}^p e^{-\Pot(\tX_k)}} > \frac{\exp\left(-\sup_{z \in B(0, r_s)} \Pot(z)\right)}{(p+1)\exp\left(-\inf_{z \in B(0, r_s)} \Pot(z)\right)}
\end{equation*}
and 
\begin{equation*}
   \E s_p(\Xi)\mathbbm{1}_{\Xi \in B_p} > \frac{  \exp\left(-\sup_{z \in B(0, r_s)} \Pot(z)\right)}{\exp\left(-\inf_{z \in B(0, r_s)} \Pot(z)\right)}\left[\mu_0\left(B\left(0, \frac{r_s}{n}\right)\right)\right]^{p+1}.
\end{equation*}
This estimate also shows that this argument would not ``pass to the limit" for $p\to \infty$ as the event $\Xi \in B_p$ would be empty. 

We now iterate the argument constructing a coupling of $P_p^n(x, \cdot)$ and $P_p^n(\tx, \cdot)$. Define $\Xi^{(n)} = (\xi^{(n)}_0, \ldots, \xi^{(n)}_p) \sim \mu_0^{\otimes (p+1)}$ independent of $\Xi^{(1)}, \ldots, \Xi^{(n-1)}$, and set the proposals at the $n$-th step as 
    \begin{align}
        X_j^{(n)} = \rho^2 X^{(n-1)} + \rho^2\sqrt{1 - \rho^2}\xi_0^{(n)} +  \sqrt{1 - \rho^2}\xi_j^{(n)} \\
        \tX_j^{(n)} = \rho^2 \tX^{(n-1)} + \rho^2\sqrt{1 - \rho^2}\xi_0^{(n)} +  \sqrt{1 - \rho^2}\xi_j^{(n)}.
    \end{align}
    Then consider also
    \begin{align*}
	\hat{\alpha}_j^{(n)} = \hat{\alpha}_j(X^{(n-1)}, \tX^{(n-1)}; \Xi^{(n)}) 
	= \min\{\alpha_j(X^{(n-1)}, X_1^{(n)}\ldots, X_p^{(n)}), \alpha_j(\tX^{(n-1)}, \tX_1^{(n)}\ldots, \tX_p^{(n)})\}, \quad j = 0, \ldots, p,
\end{align*}
and 
\begin{align}\label{eq:app:def:sj}
	s_j^{(n)} = s_j^{(n)}(x_0, \tx_0; \Xi^{(n)}) = \sum_{k=0}^j \hat{\alpha}^{(n)}_k, \quad j = 0, \ldots, p.
\end{align}
and define the associated families of intervals $J_0^{(n)}, \ldots, J_p^{(n)}$ and $\tJ_0^{(n)}, \ldots, \tJ_p^{(n)}$ as done for the first step. 
Now, given $U^{(n)}\sim \mathcal{U}([-s_0^{(n)}, 1 - s_0^{(n)}])$, independent of $U^{(1)},\ldots, U^{(n-1)}$ and of $\Xi^{(1)}\ldots \Xi^{(n)}$, we set 
    \begin{equation}
        \left(X^{(n)}, \tX^{(n)}\right) = \begin{cases}
            (X^{(n-1)}, \tX^{(n-1)}) \quad &\text{if } U^{(n)} \in [-s_0^{(n)}, 0)\\
            (X^{(n)}_1, \tX^{(n)}_1) \quad &\text{if } U^{(n)} \in [0, s_1^{(n)}), \\
            (X^{(n)}_j, \tX^{(n)}_j) \quad &\text{if } U^{(n)} \in [s_{j-1}^{(n)}, s_j^{(n)}), \quad j =2, \ldots, p\\
             (X^{(n)}_k, \tX^{(n)}_l) \quad &\mbox{ if } \,\, U \in J^{(n)}_k\cap J^{(n)}_l,\mbox{ with } k,l=0, \ldots, p \mbox{ such that }  J^{(n)}_k \cap \tJ^{(n)}_l \neq \emptyset.
        \end{cases}
    \end{equation}
    with $X^{(n)}_0 = X^{(n-1)}$.
    It can be verified that this is a coupling of $P_p^n(x_0, \cdot)$ and $P_p^n(x_0, \cdot)$. 
    
    Next define the event in which two proposals with same index are accepted at step $n$, namley $A^{(n)} = \lbrace U^{(n)}\in [0, s_p^{(n)}]\rbrace $, and 
    \begin{equation*}
        \Lambda^{(n)} = \bigcap_{j =1}^n A^{(j)},
    \end{equation*}
    namely the event for which for $n$ steps in a row the chains accept two proposal with the same index in $1, \ldots, p$, with this index possibly changing among the n steps. Then we can write 
    \begin{align*}
         \Wass_{d_\eps}(P_p^n(x_0, \cdot), P_p^n(\tx_0, \cdot)) &\leq \E d_\eps\left( X^{(n)}, \tX^{(n)}\right) \\
            &= \E d_\eps\left( X^{(n)}, \tX^{(n)}\right) \mathbbm{1}_{ \Lambda^{(n)}} + \E d_\eps \left( X^{(n)}, \tX^{(n)}\right) \mathbbm{1}_{\left(\Lambda^{(n)}\right)^c} \\
    &= d_\eps(\rho^{2n}x_0, \rho^{2n}\tx_0) \bbP( \Lambda^{(n)}) + 1 - \bbP( \Lambda^{(n)}).
    \end{align*}
    Again, as $x_0, \tx_0\in B(0, r_s)$
    \begin{align*}
       \Wass_{d_\eps} (P_p^n(x_0, \cdot), P_p^n(\tx_0, \cdot))  &\leq 1 - \bbP( \Lambda^{(n)}) \left( 1 - \frac{2r_s \rho^{2n}}{\varepsilon}\right). 
    \end{align*}

    Let us now focus on $\Lambda^{(n)} = \bigcap_{j =1}^n A^{(j)}$. As for the step $n=1$, consider $B_p$ defined in \eqref{eq:B_p}, and if all $\Xi^{(k)}$ for $k = 1, \ldots, n$ are in $B_p$ then, starting from \eqref{eq:mpCN:small1} all the proposals are in the ball $B\left( 0, \frac{r_s}{n}\right)$
    \begin{equation*}\begin{split}
        \| X_j^{(k)} \| &\leq  \rho^2\| X^{(n-1)}\| + \rho^2\sqrt{1 - \rho^2}\|\xi_0^{(n)}\| +  \sqrt{1 - \rho^2}\|\xi_j^{(n)}\| \\
        &\leq  \left( \rho^2  + \rho^2 \frac{\sqrt{1 - \rho^2}}{n} +  \frac{\sqrt{1 - \rho^2}}{n}\right) r_s,
    \end{split}
    \end{equation*}
    as long as $n > \sqrt{\frac{1 + \rho^2}{1 - \rho^2}}$. 
    In this case it follows, for all $k =1, \ldots, n$,
    \begin{align}
        s_p^{(k)} = \sum_{j = 0}^p \har_j^{(k)}  >  \frac{ \exp\left(-\sup_{z \in B(0, r_s)} \Pot(z)\right)}{\exp\left(-\inf_{z \in B(0, r_s)} \Pot(z)\right)}=: \sigma.
    \end{align}
  Therefore we have 
\begin{align*}
    \bbP\left( \Lambda^{(n)}\right) &>  \bbP\left( \Lambda^{(n)}\; \left| \; \bigcap_{k=1}^n \lbrace \Xi^{(k)}\in B_p\rbrace \right.\right)\bbP \left(\bigcap_{k=1}^n \lbrace \Xi^{(k)}\in B_p\rbrace \right)\\
    &>  \bbP\left( \bigcap_{j =1}^n \{ U^{(j)} \leq \sigma  \}\; \left| \; \bigcap_{k=1}^n \lbrace \Xi^{(k)}\in B_p\rbrace\right.\right)\bbP \left(\bigcap_{k=1}^n \lbrace \Xi^{(k)}\in B_p\rbrace\right)\\
    &>  \sigma ^n \mu_0\left( B\left(0, \frac{r_s}{n}\right)\right)^{n(p+1)},
\end{align*}
where in the last inequality we used the mutual independence of the $U^{(1)}, \ldots, U^{(n)}$ and of the $\Xi^{(1)},\ldots, \Xi^{(n)}$.

    Finally, we have showed that 
    \begin{equation}
          \Wass_{d_\eps} (P_p^n(x_0, \cdot), P_p(\tx_0, \cdot)) \leq 1 - \left\lbrace \sigma \mu_0\left( B\left(0, \frac{r_s}{n}\right)\right)^{p+1}\right\rbrace^n \left( 1 - \frac{2r_s \rho^{2n}}{\varepsilon}\right)=:s(n,p) 
    \end{equation}
    so that $s(n,p)<1$ taking
    \begin{equation}\label{eq:app:n*}
        n > \max\left( \sqrt{\frac{1 + \rho^2}{1 - \rho^2}},\frac{ \log 2r_s - \log \varepsilon}{- 2 \log \rho} \right) =: n^*
    \end{equation}
    
    Finally, note that the presented argument is not valid for $ p= \infty$ as the events $\{\Xi^{(k)}\in B_p\}$ would be empty for any $k$.
\end{proof}

\section{Additional proofs}\label{app:proofs}
\begin{proof}[Proof of \cref{lemma:Pp1}]
\textbf{Formulation 1.}
    From the original formulation \eqref{eq:2:mpcn:kernel} of the mpCN Markov kernel, namely
    \begin{equation}\label{eq:app:mpcn:kernel}
      P_p(x_0, dy) = \sum_{j=0}^p \int_{\qsp^{p+1}}\frac{e^{- \Pot(x_j)}}{\sum_{l=0}^p e^{- \Pot(x_l)}} \delta_{x_j}(dy)\prod_{k= 1}^p Q(z, dx_k) Q(x_0, dz),
\end{equation}
    isolating and rearranging the $j =0$ term
    \begin{multline*}
        \int_{\qsp^{p+1}}\frac{e^{- \Pot(x_0)}}{\sum_{l=0}^p e^{- \Pot(x_l)}} \delta_{x_0}(dy)\prod_{k= 1}^p Q(z, dx_k) Q(x_0, dz)\\
        =\int_{\qsp^{2}} e^{- \Pot(x_0)}\delta_{x_0}(dy) \int_{\qsp^{p-1}}\frac{\prod_{k= 2}^p Q(z, dx_k)}{e^{- \Pot(x_0)} + e^{-\Pot(x_1)}+ \sum_{l=2}^p e^{- \Pot(x_l)}} Q(z, dx_1) Q(x_0, dz).
    \end{multline*}
    For the terms $j =1, \ldots, p$
    \begin{multline*}
        \sum_{j=1}^p \int_{\qsp^{p+1}}\frac{e^{- \Pot(x_j)}}{\sum_{l=0}^p e^{- \Pot(x_l)}} \delta_{x_j}(dy)\prod_{k= 1}^p Q(z, dx_k) Q(x_0, dz) \\
        = \sum_{j=1}^p  \int_{\qsp^{2}} e^{- \Pot(x_j)}\delta_{x_j}(dy) \int_{\qsp^{p-1}}\frac{\prod_{k\neq j} Q(z, dx_k)}{e^{- \Pot(x_0)} + e^{-\Pot(x_j)}+ \sum_{l\neq j} e^{- \Pot(x_l)}} Q(z, dx_j) Q(x_0, dz).
    \end{multline*}
    We can now relabel the variables without loss of generality so that 
    \begin{equation*}
         = p \int_{\qsp^{2}} e^{- \Pot(x_1)}\delta_{x_1}(dy) \int_{\qsp^{p-1}}\frac{\prod_{k=2}^p Q(z, dx_k)}{e^{- \Pot(x_0)} + e^{-\Pot(x_1)}+ \sum_{l=2}^p e^{- \Pot(x_l)}} Q(z, dx_1) Q(x_0, dz).
    \end{equation*}
    Therefore setting 
    \begin{equation*}
        \gamma_p(x_0, x_1, z):=\int_{\qsp^{p-1}}\frac{\prod_{k= 2}^p Q(z, dx_k)}{e^{- \Pot(x_0)} + e^{-\Pot(x_1)}+ \sum_{l=2}^p e^{- \Pot(x_l)}}
    \end{equation*}
    we have the desired results.

\textbf{Formulation 2.}
Looking at the acceptance probability (integrand) in \eqref{eq:app:mpcn:kernel} we have 
\begin{align*}
   \frac{e^{- \Pot(x_j)}}{\sum_{l=0}^p e^{- \Pot(x_l)}} 
   &= \frac{e^{- \Pot(x_j)}}{\sum_{l=0}^p e^{- \Pot(x_l)}}\frac{e^{-\Pot(x_0)  } + p \int e^{-\Pot(u)}  Q(z, d u)}{e^{-\Pot(x_0)  } + p \int e^{-\Pot(u)}  Q(z, d u)} \\
   &= \frac{e^{- \Pot(x_j)}}{\sum_{l=0}^p e^{- \Pot(x_l)}}\frac{e^{-\Pot(x_0) } + p \int e^{-\Pot(u)}  Q(z, d u)+ \sum_{k=1}^p e^{- \Pot(x_k)} - \sum_{k=1}^p e^{- \Pot(x_k)}}{e^{-\Pot(x_0)} + p \int e^{-\Pot(u)}  Q(z, d u)} \\
   &=  \frac{e^{- \Pot(x_j)}}{\sum_{l=0}^p e^{- \Pot(x_l)}}\frac{ p \int e^{-\Pot(u)}  Q(z, d u) - \sum_{k=1}^p e^{- \Pot(x_k)}}{e^{-\Pot(x_0)} + p \int e^{-\Pot(u)}  Q(z, d u)}
   + \frac{e^{- \Pot(x_j)}}{e^{-\Pot(x_0)} + p \int e^{-\Pot(u)}  Q(z, d u)}.
\end{align*}
Therefore 
\begin{align}
     P_p(x_0, &dy) = \sum_{j=0}^p \int_{\qsp^{p+1}}\frac{e^{- \Pot(x_j)}}{\sum_{l=0}^p e^{- \Pot(x_l)}} \delta_{x_j}(dy)\prod_{k= 1}^p Q(z, dx_k) Q(x_0, dz) \notag\\
     &= \sum_{j=0}^p \int_{\qsp^{p+1}}\frac{e^{- \Pot(x_j)}}{\sum_{l=0}^p e^{- \Pot(x_l)}}\frac{ p \int e^{-\Pot(u)}  Q(z, d u) - \sum_{k=1}^p e^{- \Pot(x_k)}}{e^{-\Pot(x_0)} + p \int e^{-\Pot(u)}  Q(z, d u)}\delta_{x_j}(dy)\prod_{k= 1}^p Q(z, dx_k) Q(x_0, dz) \label{eq:mpcnT3}\\
   &\;+ \sum_{j=0}^p \int_{\qsp^{p+1}}\frac{e^{- \Pot(x_j)}}{e^{-\Pot(x_0)} + p \int e^{-\Pot(u)}  Q(z, d u)}\delta_{x_j}(dy)\prod_{k= 1}^p Q(z, dx_k) Q(x_0, dz),\label{eq:mpcnT12}
   \end{align}
   and we derived the term $- T_3$ in \eqref{eq:mpcnT3}. Then simple manipulations of \eqref{eq:mpcnT12} give 
   \begin{align*}
   \sum_{j=0}^p \int_{\qsp^{p+1}}&\frac{e^{- \Pot(x_j)}}{e^{-\Pot(x_0)} + p \int e^{-\Pot(u)}  Q(z, d u)} \delta_{x_j}(dy)\prod_{k= 1}^p Q(z, dx_k) Q(x_0, dz)\\
   &= \int_{\qsp}\frac{e^{- \Pot(x_0)}}{e^{-\Pot(x_0)} + p \int e^{-\Pot(u)}  Q(z, d u)}\delta_{x_0}(dy) Q(x_0, dz)\\
    &\quad + \sum_{j=1}^p \int \frac{1}{e^{-\Pot(x_0)} + p \int e^{-\Pot(u)}  Q(z, d u)} \int_{\qsp} e^{- \Pot(x_j)} \delta_{x_j}(dy) Q(z, dx_j) Q(x_0, dz)\\
    &=   \int_{\qsp} \frac{e^{-\Pot(y)} }
      { e^{-\Pot(x_0)  } + p \int e^{-\Pot(u)}  Q(z, d u)} \delta_{x_0}(dy) Q(x_0, dz)\\
    & \quad+ p \int_{\qsp}  \frac{ e^{-\Pot(y)}}
      { e^{-\Pot(x_0)  } + p \int e^{-\Pot(u)}  Q(z, d u)}
     Q(z, d y) Q(x_0, dz)  = T_2 + T_1
\end{align*}
as desired.
\end{proof}

\begin{Proposition}\label{app:langevin:lyap}
    Let $\mu_0$ be the reference measure with trace class covariance operator and $P_\infty$ be the kernel \eqref{eq:2:inftypCN} for $\infty$-pCN. Assume the potential function $\Pot: \qsp \to \R$ is $C^1$ with $\|\nabla \Phi\|_\infty< \infty$. Then the functions $V(x) = \|x\|^n$, $n \in \N$, are Lyapunov functions as in \cref{def:lyapunov} for the $\infty$-pCN Markov kernel $P_\infty$.
\end{Proposition}

\begin{proof}
We are looking for a measurable function $V: \qsp\to [0,\infty)$ for which there exist $K_V>0$ and $0\leq l_V<1$ such that 
\begin{equation}\label{eq:LF_P}
    (PV)(x) = \int V(y)\, P(x, dy) \leq l_V V(x) + K_V \quad \text{for all }x\in \qsp.
\end{equation}
Note that if $V$ is a Lyapunov function for the kernel $\baQ_1$ defined in \eqref{eq:2:baQ}, for some constant $K_1>0$ and $0\leq l_1<1$, and for the the kernel $Q_2$ with constants $K_2>0$ and $0\leq l_2<1$ then \eqref{eq:LF_P} follows. Indeed
\begin{align*}
    (PV)(x) = \int V(y)\, P(x, dy) &= \iint V(y) \baQ_1(z, dy) Q_2(x, dz)\\
    &\leq \int l_1 V(z) Q_2(x, dz) + K_1 \\
    &\leq  l_1l_2V(x) + K_2 + K_1.
\end{align*}
We only need $l_1l_2<1$, so in theory one of the two parameters can be greater or equal to one.

Consider the candidate Lyapunov function $V(x)= \| x \|^2$. Let us start with showing that it is Lyapunov for $\baQ_1(z, \cdot)$. We add and subtract $\rho_1 z$ and use Young's inequality to get 
\begin{align*}
    \int V(y) \, \baQ_1 (z, dy) = \int \|y\|^2 \, \baQ_1 (z, dy) 
    \leq \int k_\delta\|y - \rho_1 z\|^2 \baQ_1 (z, dy) + (1 + \delta)\rho_1^2\| z\|^2 
\end{align*}
where $\delta>0$ is an arbitrarily small parameter, and $k_\delta>0$. To ensure that 
\begin{equation*}
    \int \|y - \rho_1 z\|^2 \baQ_1 (z, dy)
\end{equation*}
stays finite, we study the associated Langevin dynamics 
\begin{equation}\label{eq:Langevin2}
    dY_t = -\left[ Y_t - \rho_1 z+ \frac{(1 - \rho_1^2)}{2} \cC \nabla \Pot(Y_t)\right]\, dt + \sqrt{(1 - \rho_1^2) \cC} dW_t.
\end{equation}
and use its mixing properties (see e.g. \cite{DPZ_green}). 
The equation \eqref{eq:Langevin2} for any initial condition $Y_0$ holds the measure $\baQ_1(z, \cdot)$ invariant and is mixing in the sense that 
\begin{equation*}
    \lim_{t \to \infty} \E g(Y_t) =  \int g(y) \baQ_1(z, dy).
\end{equation*}

From \eqref{eq:Langevin2} it also holds
\begin{equation}\label{eq:Langevin}
    d(Y_t - \rho_1 z)= -\left[ Y_t - \rho_1 z+ (1 - \rho_1^2)\cC \nabla \Phi(Y_t)\right]\, dt + \sqrt{1 - \rho_1^2} \cC dW_t
\end{equation}
which also holds $\baQ_1$ invariant and 
\begin{equation}\label{eq:app:mix:lan}
    \lim_{t \to \infty} \E g(Y_t - \rho_1 z) =  \int g(y) \baQ_1(z, dy).
\end{equation}

Taking the $\qsp$ scalar product with $Y_t - \rho_1 z$ and by It\^{o}'s lemma 
\begin{multline*}
    \|Y_t - \rho_1 z\|^2 = \|Y_0 - \rho_1 z\|^2  -2\int_0^t \|Y_t - \rho_1 z\|^2 \, ds - 2(1 - \rho_1^2)\int_0^t ( \cC \nabla \Pot(Y_s), Y_s - \rho_1 z)\, ds \\ + t(1 - \rho_1^2)\tr \cC + 2\sqrt{1 - \rho_1^2} \int_0^t (Y_s - \rho_1 z, \cC dW_s).
\end{multline*}
Using the fact that $\nabla \Pot$ is assumed globally bounded, it follows 
\begin{equation*}
     \|Y_t - \rho_1 z\|^2 = \|Y_0 - \rho_1 z\|^2  -(1+ \rho_1^2)\int_0^t \|Y_s - \rho_1 z\|^2 \, ds  + \kappa t+ 2\sqrt{1 - \rho_1^2} \int_0^t (Y_s - \rho_1 z, \cC dW_s)
\end{equation*}
with $\kappa = (1 - \rho_1^2)\tr \cC + (1-\rho_1^2)\|\cC\nabla \Phi\|_\infty$.
Then, taking expectations, 
\begin{equation*}
    \E  \|Y_t - \rho_1 z\|^2 \leq \E \|Y_0 - \rho_1 z\|^2  -(1+ \rho^2)\int_0^t \E \|Y_s - \rho_1 z\|^2 \, ds  + \kappa t,
\end{equation*}
and, by the integral version of Gronwall's inequality,
\begin{equation*}
    \E  \|Y_t - \rho_1 z\|^2 \leq e^{-(1 + \rho_1^2)t}\E \|Y_0 - \rho_1 z\|^2 + \frac{\kappa}{(1 + \rho_1^2)} \left(1 - e^{-(1 + \rho_1^2)t}\right).
\end{equation*}
Then we derived a bound that we can use in \eqref{eq:app:mix:lan} with $g(y) = \|y\|^2$, namely 
\begin{equation*}
    \int \|y - \rho_1 z\|^2\baQ (z, dy) = \lim_{t\to \infty} \E  \|Y_t - \rho_1 z\|^2 \leq \frac{\kappa}{(1 + \rho_1^2)}.
\end{equation*}
and so $V$ is a Lyapunov function for $\baQ_1$ with parameters $K_1= k_\delta\frac{\kappa}{(1 + \rho_1^2)}$ and $l_1 = (1 + \delta)\rho_1^2$. 

Now we turn to the proposal kernel $Q_2$ that is simply a Gaussian. It is immediate to derive 
\begin{align*}
    \int \|y \|^2 Q_2(x, dy) &= \int \|\rho_2 x + \sqrt{1- \rho_2^2} y)\|^2 \mu_0(dy)\\
    &\leq (1 + \delta_1) \rho_2^2 \|x \|^2 + k_{\delta_1}(1 - \rho_2^2)\int \| y \|^2 \mu_0(dy).
\end{align*}
Finally, it is enough to pick the auxiliary parameters $\delta$ and $\delta_1$ so that 
\begin{equation*}
    (1 + \delta)(1 + \delta_1) \rho_1^2\rho_2^2 <1,
\end{equation*}
and we have \eqref{eq:LF_P} with $l_V = (1 + \delta)(1 + \delta_1) \rho_1^2\rho_2^2$ and 
\begin{equation*}
\begin{split}
    K_V = K_1+ K_2 &= \frac{k_\delta \kappa}{1 + \rho_1^2} + k_{\delta_1}(1 - \rho_2^2)\tr \cC\\
    &= \left(1 + \frac{1}{\delta}\right)\frac{1 - \rho_1^2}{1 + \rho_1^2}\left(\tr \cC + \|\cC\nabla \Phi\|_\infty\right) +\left(1 + \frac{1}{\delta_1}\right)(1 - \rho_2^2)\tr \cC 
\end{split}
\end{equation*}  
These results are extendable to $V(x)= \|x\|^q$ for any $q>1$, as long as $L^q$ ``energy" bounds for the Langevin dynamics above are attainable.
\end{proof}

\section{Additional numerical results} \label{ap:numerics}

Here we include additional information that can be useful for a more detailed comprehension of the algorithms and the  interpretation of the results in Section~\ref{sec:numerics}. 

\subsection{Polar twist example}
\cref{fig:polar_clouds_burst} illustrates the behaviour of one mpCN and one MTpCN chain with $p=20$ proposals each and $\rho=0.5$, for 5 contiguous iterations of each chain, on the polar twist example in Section~\ref{sec:polar}. The chains are started from the same initial value. The teal clouds ("props``) represent the location of the proposed values in the forward cloud in MTpCN and the only cloud of mpCN. The pink dots ("reverse props``) shows the location of the proposed points in the reference cloud in MTpCN. 
\begin{figure}
    \centering
    \includegraphics[width=1\linewidth]{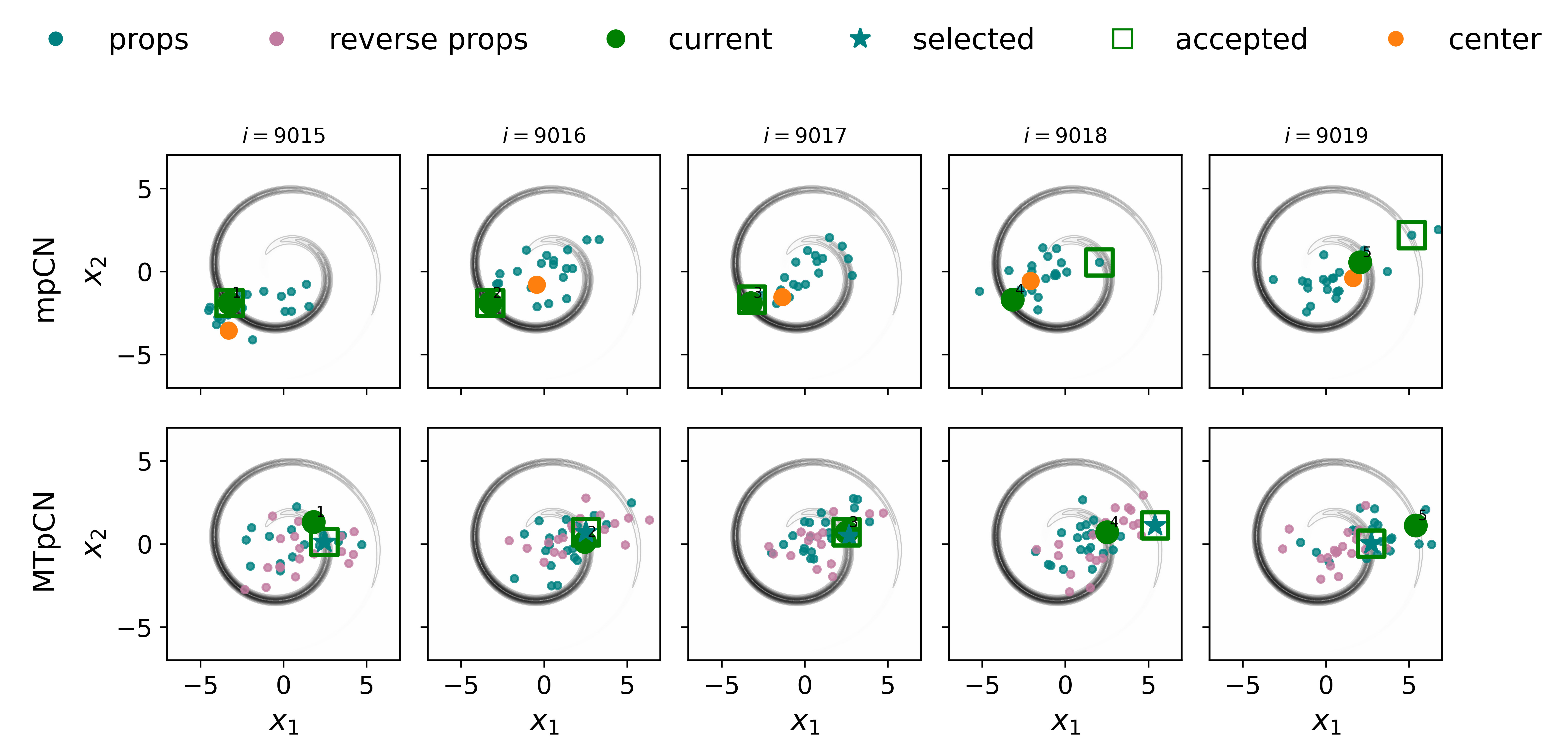}
    \caption{Point clouds illustration for the polar twist, for $\rho=0.5$, $p=20$, and 5 iterations starting from iteration number 9015, for mpCN and MTpCN.}
    \label{fig:polar_clouds_burst}
\end{figure}

\subsection{Solute transport example}

In the first part of this section, we revisit the inverse problem introduced in \cref{sec:matrix_inversion}, but formulate it on $\ell^2(\R)$. This viewpoint makes explicit its infinite-dimensional nature and allows finite-dimensional matrix models to be interpreted as truncations of the underlying problem. In particular, increasing the matrix size in \eqref{eq:ad_toy} corresponds to increasing the dimension of the approximation space.

We conclude by including plots that illustrate the data generation from the model, and the posterior geometry obtained with this data at stationarity. 

\subsubsection{Infinite dimensional formulation}\label{ap:solute_like}
We consider the following problem set in the space of sequences on $\R$, $\ell^2(\R)$, which we simply denote $\ell^2$:
\begin{equation}
     (\bm{A} + \kappa \bm{I}) \bm{\theta} = \bm{g},
\end{equation}
where $\bm{\theta}, \bm{g} \in \ell^2$, $\kappa>0$ and $\bm{A}$ is an Hilbert-Schmidt antisymmetric operator on $\ell^2$, namely \[
    \sum_{j = 1}^\infty \| \bm{A}e_k\|_{\ell^2}^2 <\infty
\]
where \(\lbrace \bm{e}_k, \; k \in \N \rbrace \) is an orthonormal basis of $\ell^2$, and $\langle \bm{A}x, y\rangle_{\ell^2}= - \langle x, \bA^* y\rangle_{\ell^2}$, for any $x,y\in \ell^2$.
We want to estimate $\bm{A}$ assuming we have finite-dimensional observations of $\bm{\theta}$:
\begin{equation*}
    y = \cP(\bm{\theta}(\bm{A})) + \bm{\eta},\; \bm{\eta}\sim \cN_k(0, \sigma^2 \bm{I}_k) 
\end{equation*}
where $\cP: \ell^2 \to \R^k$ is a projection operator. 
The parameter space $\qsp$ is then the collection of all Hilbert-Schmidt antisymmetric operators on $\ell^2$. We consider a centered Gaussian prior on $\qsp$, $\mu_0 = \cN(0, \bm{C})$, with covariance $\bm{C}$ to be specified below, and the log-likelihood defined by the observation model
\begin{equation}\label{eq:app:pot}
    \Phi(\bm{A}) =  \frac{1}{2\sigma^2} \|\cP\left((\bA + \kappa \bI)^{-1}\bg)\right)- \by \|_{\R^k}^2,
\end{equation}
so the Bayesian posterior on $\qsp$ takes form 
\begin{equation}\label{eq:ex1:target}
    \mu(d\bA) \propto \exp\left( - \frac{1}{2\sigma^2}  \|\cP\left((\bA + \kappa \bI)^{-1}\bg)\right)- \by \|_{\R^k}^2 \right) \, \mu_0(\bA).
\end{equation}
Note that the likelihood is well defined as $\bm A+\kappa\bm I$ is invertible, thanks to $\bm A$ being antisymmetric as we show in the following remark.

\begin{Remark}[Invertibility]
Let $\|\cdot \|$ without decorations denote the operator norm. Since $\bm A$ is assumed Hilbert-Schmidt, then $\|\bm A\|\leq \|\bm A\|_{HS}$ and
\begin{align*}
    \|(\bm A + \kappa \bm I ) \btheta\|_{\ell^2} \leq \| \bm A\|\| \btheta \|_{\ell^2} + \kappa \|\btheta\|_{\ell^2}\leq \|\bm \left(\|\bm A\|_{HS} + \kappa \right) \|\btheta\|_{\ell^2}.
\end{align*}
This implies that $\bm A + \kappa \bm I$ is a bounded operator with operator norm satisfying 
\(\|\bm A + \kappa \bm I \|\leq \|\bm A\|_{HS} + \kappa\).
We can also derive a lower bound thanks to the fact that $\bm A$ is antisymmetric. In fact, as $\langle \bm A \btheta , \btheta \rangle_{\ell^2} =0$, it follows
\[\langle (\bm A+\kappa\bm I)\btheta,\btheta\rangle_{\ell^2} = \kappa \|\btheta\|_{\ell^2}^2, \]
and as a consequence, thanks to Cauchy--Schwartz inequality,
\[
    \kappa\|\btheta\|_{\ell^2} \leq \|(\bm A+\kappa\bm I)\btheta\|_{\ell^2}, \qquad \btheta\in\ell^2.
\]
We conclude that
\begin{equation}\label{app:A+kI:bounds}
    \kappa \leq \|\bm A + \kappa \bm I \|\leq \|\bm A\|_{HS} + \kappa.
\end{equation}
This implies that $\bm A+\kappa\bm I$ is injective and that its inverse on
$\operatorname{Rg}(\bm A+\kappa\bm I)$ satisfies
\begin{equation}\label{eq:bound_loglik}
\|(\bm A+\kappa\bm I)^{-1}\|
\le \kappa^{-1}
\end{equation}
Moreover, \eqref{eq:bound_loglik} implies that
$\operatorname{Rg}(\bm A+\kappa\bm I)$ is closed. Indeed, if
$(\tilde{\btheta}_n)_n \subset \operatorname{Rg}(\bm A+\kappa\bm I)$ is
Cauchy and $\tilde{\btheta}_n=(\bm A+\kappa\bm I)\btheta_n$, then
$(\btheta_n)_n$ is Cauchy in $\ell^2$, hence converges to some
$\btheta\in\ell^2$. As $\bm A+\kappa\bm I$ is bounded, it is continuous and 
\[
\tilde{\btheta}_n \to (\bm A+\kappa\bm I)\btheta
\in \operatorname{Rg}(\bm A+\kappa\bm I).
\]
\end{Remark}

Further in the next remark we show that the log-likelihood \eqref{eq:app:pot} satisfies the assumptions of the theoretical results established in \cref{sec:results}, namely it is globally bounded and Lipschitz continuous. This ensures both mpCN and MTpCN algorithms are mixing to the target measure \eqref{eq:ex1:target} uniformly in the number of proposals.
\begin{Remark}\label{rem:ex3}
Thanks to the estimate \eqref{eq:bound_loglik}, it follows that the log-likelihood \eqref{eq:app:pot} is bounded as
\begin{equation*}
    \Phi(\bA) \leq \frac{1}{2\sigma^2} \left(\frac{\|\bg\|_{\ell^2}}{\kappa} +\|\by\|_{\R^k}\right) ^2 \quad \text{for all }\bA\in \qsp.
\end{equation*}
To prove global Lipschitz continuity, observe that
\begin{align}
    &|\Phi(\bA) - \Phi(\tbA)| = \frac{1}{2\sigma^2} \left| \|\cP\left((\bA + \kappa \bI)^{-1}\bg\right)- \by \|_{\R^k}^2 - \|\cP\left((\tbA + \kappa \bI)^{-1}\bg\right)- \by \|_{\R^k}^2 \right| \notag \\
    &\leq \frac{1}{2\sigma^2} \|\cP\left((\bA + \kappa \bI)^{-1}\bg - (\tbA + \kappa \bI)^{-1}\bg\right)\|_{\R^k}\left[ \|\cP\left((\bA + \kappa \bI)^{-1}\bg\right)- \by \|_{\R^k} + \|\cP\left((\tbA + \kappa \bI)^{-1}\bg\right)- \by \|_{\R^k} \right]. \label{eq:ex1:Phi1}
\end{align}
Thanks to the following classic identity for resolvents
\begin{equation*}
    (\bA + \kappa \bI )^{-1} -  (\tbA + \kappa \bI )^{-1} =  (\bA + \kappa \bI )^{-1}(\bA - \tbA)(\tbA + \kappa \bI )^{-1} 
\end{equation*}
and \eqref{eq:bound_loglik}, it follows that 
\begin{equation*}
     \|(\bA + \kappa \bI )^{-1}\bg -  (\tbA + \kappa \bI )^{-1}\bg\|_{\ell^2} \leq  \|(\bA + \kappa \bI )^{-1}\| \|\bA - \tbA\| \|(\tbA + \kappa \bI )^{-1}\|\|\bg\|_{\ell^2} 
     \leq \frac{\|\bg\|_{\ell^2}}{\kappa^2} \|\bA - \tbA\|.
\end{equation*}
Then, from \eqref{eq:ex1:Phi1}, using the triangular inequality, \eqref{eq:bound_loglik} and the fact that the projection operator has norm one, it follows 
\begin{align*}  
   |\Phi(\bA) - \Phi(\tbA)| &\leq \frac{\|\bg\|_{\ell^2}}{\kappa^2} \|\bA - \tbA\| \left[ \|\cP\left((\bA + \kappa \bI)^{-1}\bg\right)\|_{\R^k}+ 2\|\by\|_{\R^k}+\|\cP\left((\tbA + \kappa \bI)^{-1}\bg\right)\|_{\R^k} \right] \\
    &\leq \frac{\|\bg\|_{\ell^2}}{\kappa^2} \|\bA - \tbA\| \left( \frac{2\|\bg\|_{\ell^2}}{\kappa}+ 2\|\by\|_{\R^k} \right),
\end{align*}
namely $\Phi$ is globally Lipschitz with constant $2\kappa^{-2}\|\bg\|_{\ell^2} \left( \kappa^{-1}\|\bg\|_{\ell^2}+ \|\by\|_{\R^k} \right)$.
\end{Remark}

Last, we describe the structure of the prior covariance operator $\bm{C} = \E \left[ (\bm{A} - \E \bm{A})\otimes (\bm{A} - \E\bm{A})\right]$. Let $\lbrace E_{i,j},\, i,j \in \N\rbrace$ be an orthonormal basis of the space of Hilbert-Schmidt operators on $\ell^2$, where 
\begin{equation*}
    \langle E_{i,j}, E_{k,l}\rangle_{HS} = \tr\left(E_{k,l}^*E_{i,j}\right) = \sum_{m\in \N} \langle E_{i,j}e_m, E_{k,l}e_m\rangle_{\ell^2} = \delta_{i,k}\delta_{j,l}, \quad i,j,k,l\in \N
\end{equation*}
for $\lbrace e_m, \, m\in \N\rbrace$ basis of $\ell^2$. The parameter space $\qsp$ is the subset of antisymmetric Hilbert-Schmidt operators with orthonormal basis $\lbrace F_{i,j},\, i,j \in \N, i<j\rbrace$ where
\begin{equation*}
    F_{i,j} := \frac{1}{\sqrt{2}} \left( E_{i,j} - E_{j,i}\right).
\end{equation*}
Then, any $\bm A \in \qsp$ can be written as 
\begin{equation*}
    \bm A = \sum_{i<j} a_{i,j} F_{i,j}
\end{equation*}
with $\|\bm A\|_{HS} = \sum_{i<j} |a_{i,j}|^2 <\infty$.
We can describe the covariance $\bm{C}$ by its action on the basis
\begin{align*}
    \bm{C}F_{i,j} &= \E \left[ (\bm{A} - \E \bm{A}) \langle \bm{A} - \E \bm{A}, F_{i,j} \rangle_{HS} \right] \\
    &= \E \left[ (\bm{A} - \E \bm{A}) (a_{i,j} - \E a_{i,j}) \right]\\
    &= \sum_{k<l} \E \left[ (a_{i,j} - \E a_{i,j})(a_{k,l} - \E a_{k,l})\right] F_{k,l}.
\end{align*}
We assume the following correlation structure
\begin{align}
     q_{(i,j)(k,l)} &:= \operatorname{Corr}(a_{i,j}, a_{k,l}) = \E \left[ (a_{i,j} - \E a_{i,j})(a_{k,l} - \E a_{k,l})\right] = 0 \quad \text{for } (i,j) \neq (k,l)\\
    q_{i,j} = q_{(i,j)(i,j)} &:= \Var (a_{i,j}) = \tau (ij)^{-\alpha}|i-j|^{-\gamma}. 
\end{align}
with $\tau$, $\alpha$ and $\gamma$ positive parameters. Tuning the parameters $\alpha$ and $\gamma$ we can ensure ensure the covariance is trace class, namely
\begin{equation*}
    \tr \bm{C} = \sum_{i<j} q_{i,j} = \tau \sum_{i<j} (ij)^{-\alpha}|i-j|^{-\gamma}< \infty.
\end{equation*}
This series converges when $\alpha$ and $\gamma$ are chosen to that $2\alpha + \gamma>2$.

\subsubsection{Numerics}

\cref{fig:ad_simulated_data_stat} shows the datasets generated from the model for the warm-up phase in Section~\ref{subsec:warm-up} (right column) and for the stationary phase in Section~\ref{subsec:solute:statio} (left column). The same random seed is used for both datasets. Since we use the inverse crime approach described at the beginning of Section~\ref{sec:numerics}, we first generate the advection matrices $\bm{A}_{10}$ and $\bm{A}_{40}$ (reported in the first row of the figure) from the prior distribution, using the hyperparameter values in~\eqref{eq:hyperpars_st} for both datasets. Both matrices illustrate how the  magnitude of their coefficients goes to zero when the distance to the diagonal or between elements increases, corresponding to the penalization behaviour modeled in the prior covariance matrix. Then, we generate $\bm{\theta}(\bm{A}_{10})$ and $\bm{\theta}(\bm{A}_{40})$, represented by the blue lines, and finally, add Gaussian noise to obtain the data vectors $\bm{y}_{10}$, plotted with orange dots. The decay in the data amplitude, expected from the modeling choices, is especially visible for the $d=40$ example. 

\begin{figure}
    \centering
    \includegraphics[width=0.8\linewidth]{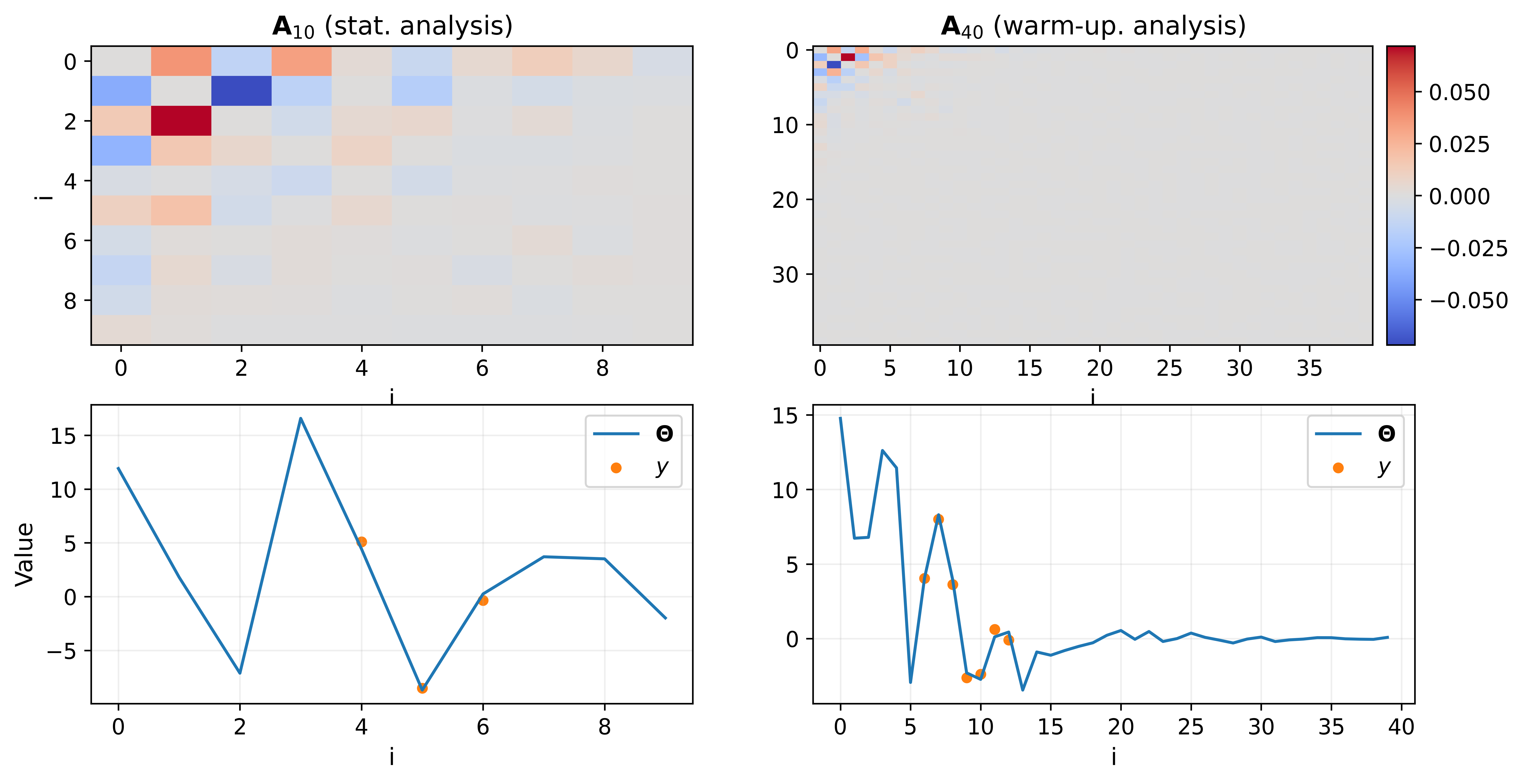}
    \caption{Advection matrices, solution vectors and observations simulated from the model for the example datasets used in the warm-up and stationary phase analyses in Section~\ref{subsec:warm-up} and Section~\ref{subsec:solute:statio}.}
    \label{fig:ad_simulated_data_stat}
\end{figure}

\cref{fig:pairplots_stat_ep} provides an additional comparison between mpCN with 100 proposals and running 100 independent and simultaneous pCN chains, and then thinning the latter every $p$ samples. The resulting figure is computed on the same number of samples as \cref{fig:pairplots_stat_mpcn}, that in addition have been obtained with the same computational budget and the same wall-clock time. The results enforce the interpretation arising from \cref{fig:sweep_solute_ep}. 
\begin{figure}
    \includegraphics[width=0.7\linewidth, trim={0 0 0 2cm}, clip]{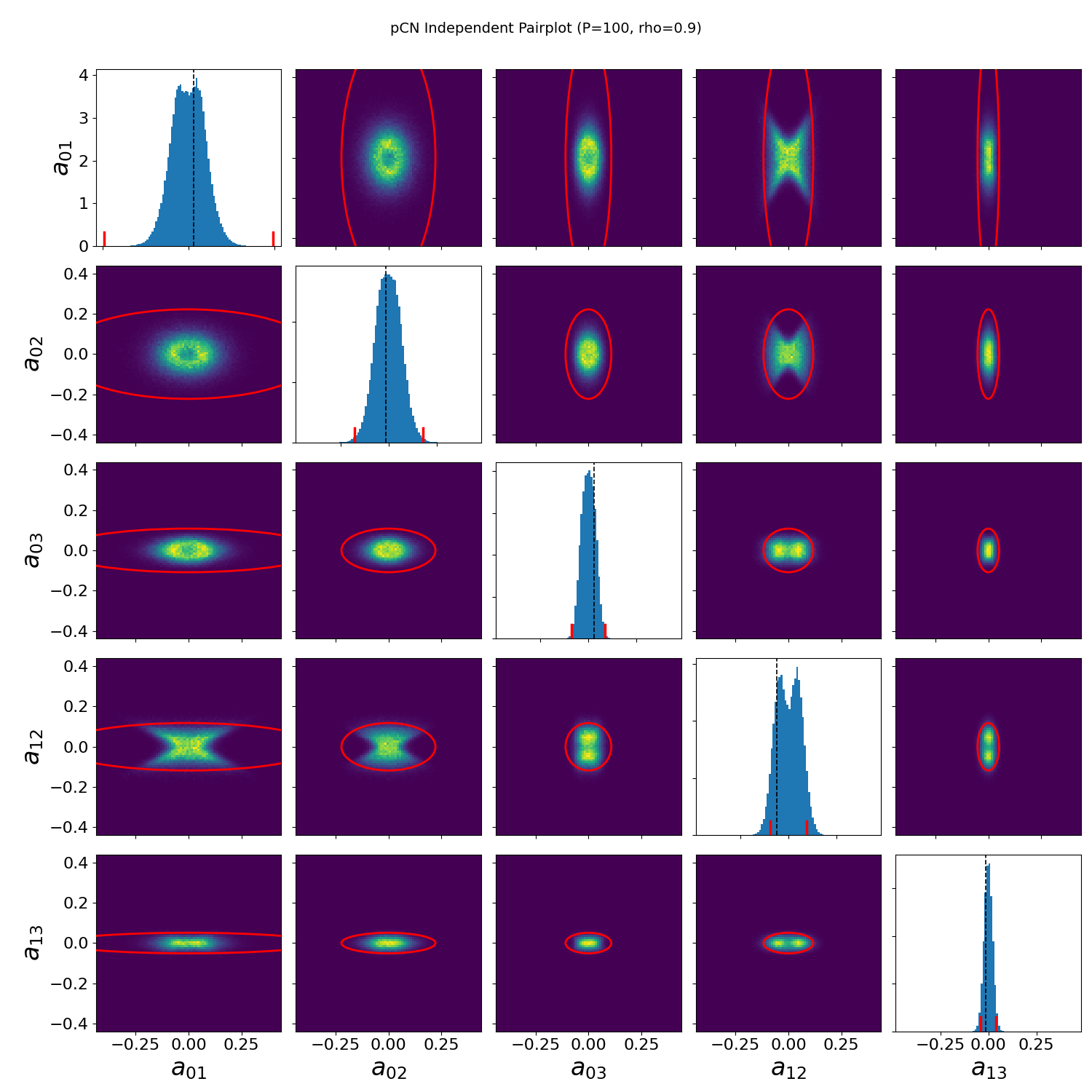}
\caption{Posterior marginal histograms and pairwise density plots for five components of the unknown vector. The figures are computed using samples from 100 embarassingly parallel pCN chains with 300k iterations each, thinned every 100 samples.}
\label{fig:pairplots_stat_ep}
\end{figure}

\section{About the $L^2_{\mu}$ spectral gap}\label{app:L2}

\begin{proof}[Proof of \cref{thm:L2gap}]
   
    \textbf{Step 1}: First we show that thanks to the Wasserstein contraction in $\td$ given by \cref{thm:our_harris}, for functions $f\in \Lip(\td)\cap L^\infty_\mu \subset L^2_\mu$ we have the following partial result: there exists a positive constant $C(f)$ such that 
    \begin{equation}\label{A:specgap1}
        \|P^n f - \mu(f)\|_2^2 \leq C(f)\lambda^{2n}.
    \end{equation}
    for all $n \geq n_1/2$, where $n_1$ and $\lambda$ are as in \cref{thm:our_harris}.
    
    Let us first consider non-negative observables $f\geq 0$. By some simple manipulation and the fact that $P$ is $\mu$-invariant we see that
    \begin{align*}
        \|P^n f - \mu(f) \|_2^2 &= \int \left( P^nf(x) - \mu(f)\right)^2\, \mu(dx)\\
        &= \int \left[ P^nf(x)\right]^2\, \mu(dx) + \mu(f)^2 - 2\mu(f) \int P^nf(x)\, \mu(dx) \\
        &= \int \left[ P^nf(x)\right]^2\, \mu(dx) - \mu(f)^2. 
    \end{align*}
    As $P$ is $\mu$-reversible, it is self-adjoint on $L^2_\mu$ and
    \begin{align*}
        \int \left[ P^nf(x)\right]^2\, \mu(dx) = \int f(x) P^{2n}f(x) \mu(dx) = \mu(P^{2n}f) \int f(x) \frac{P^{2n}f(x)}{\mu(P^{2n}f)} \mu(dx) .
    \end{align*}
    By the positivity of $f$ and the Markov property, the measure $\nu(dx) := \frac{P^{2n}f(x)}{\mu(P^{2n}f)} \mu(dx)$ is a well defined probability measure so that, using again the $\mu$-invariance of $P$,
    \begin{equation*}
        \int \left[ P^nf(x)\right]^2\, \mu(dx) - \mu(f)^2 = \mu(f) \nu(f) - \mu(f)^2 = \mu(f) \left[ \nu(f) - \mu(f)\right].
    \end{equation*}
    For any measure $\pi$ on $\qsp\times \qsp$ coupling of $\nu$ and $\mu$ then 
    \begin{align*}
        \nu(f) - \mu(f)= \iint f(x) \pi(dx, dy) - \iint f(y) \pi(dx,dy) \leq L_f \iint \td(x,y)\,\pi(dx, dy)
    \end{align*}
    where we used that $f$ is Lipschitz with respect to the semidistance $\td$ with constant $L_f$.
    Since this bound holds for any coupling of $\nu$ and $\mu$, in particular it holds for the Wasserstein distance associated to $\td$, leading to 
    \begin{equation}\label{eq:app:l2wass}
        \|P^n f - \mu(f) \|_2^2 \leq \mu(f) L_f W_{\td}(\nu, \mu).
    \end{equation}
    In order to use the spectral gap in Wasserstein distance, we want to write $\nu$ as $P\nu_1$ for some $\nu_1$. By definition of $\nu$, for a generic observable $\varphi$,
    \begin{align*}
        \int \varphi(x) \nu(dx) &= \int \varphi(x)\frac{P^{2n}f(x)}{\mu(P^{2n}f)} \mu(dx) \\
        &= \int P^{2n}\varphi(x) \frac{f(x)}{\mu(f)} \mu(dx) \quad \text{(by reversibility and invariance)}\\
        &= \int P^{2n}\varphi(x) \nu_f(dx) = \int \varphi(x) \left(\nu_f P^{2n}\right)(dx) 
    \end{align*}
    where we defined $\nu_f(dx) = \frac{f(x)}{\mu(f)} \mu(dx)$. Next, note that by \cref{thm:our_harris}, there exists $n_1>0$ and $\lambda\in (0,1)$ such that, for all $n\geq \frac{n_1}{2}$ it holds
    \begin{equation*}
        W_{\td}(\nu, \mu) = W_{\td}(\nu_f P^{2n}, \mu P^{2n}) \leq \lambda^{2n} W_{\td}(\nu_f, \mu) 
    \end{equation*}
    and, finally, from \eqref{eq:app:l2wass},
    \begin{equation*}
        \|P^n f - \mu(f) \|_2^2 \leq \lambda^{2n} \mu(f) L_f  W_{\td}(\nu_f, \mu) =:\lambda^{2n} C(f).
    \end{equation*}

    Last, if $f$ does not have a prescribed sign we can always consider the positive functions $f^+, f^-$. If $f\in \Lip(\td)\cap L^\infty_\mu$ it is easy to see that also its positive and negative parts are in such a space so that the first part of Step 1 holds. As a consequence we can write 
    \begin{align*}
        \|P^nf - \mu(f) \|_2^2 &= \|P^n(f^+ - f^-) - \mu (f^+ - f^-)\|_2^2\leq 
        2\|P^nf^+ - \mu(f^+) \|_2^2  + 2\|P^nf^- - \mu(f^-) \|_2^2 \\
        &\leq \lambda^{2n}\left[ \mu(f^+) L_{f^+}  W_{\td}(\nu_{f^+}, \mu) +  \mu(f^{-}) L_{f^{-}}  W_{\td}(\nu_{f^{-}}, \mu)\right] = \lambda^{2n}C(f)
    \end{align*}
    giving the desired bound \eqref{A:specgap1} for any function in $\Lip(\td)\cap L^\infty_\mu$.

     \textbf{Step 2}: Next we show that \eqref{A:specgap1} is sufficient to ensure  
    \begin{equation*}
        \|P^n f - \mu(f)\|_2^2 \leq \lambda^{2n}\|f - \mu(f)\|^2, \quad n\in \N
    \end{equation*}
    for all $f \in \Lip(\td)\cap L^\infty_\mu$.
    As $P$ is a $\mu$-reversible Markov kernel on the Hilbert space $\qsp$, it is bounded and selfadjoint in $L^2_\mu(\qsp)$. Then the spectral theorem for bounded selfadjoint operators on Hilbert spaces (see e.g. \cite[Chapter 7]{reed1980functional}) ensures that there exists a measurable space $(\mathcal{S}, \Sigma, \nu)$, an essentially bounded measurable function $m: \mathcal{S}\to \R$ and a unitary operator $U: \qsp\to L_\nu^2(\mathcal{S})$ such that 
    \begin{equation}\label{eq:app:PUTU}
        P = U^* T U \quad \text{with } (T\varphi)(x) = m(x) \varphi(x), \quad x \in \mathcal{S},
    \end{equation}
    namely $\left[ (UPU^*)\varphi\right](x) = m(x)\varphi(x)$. Note also that $P^n=(U^* T U)^n = U^* T^n U$ since $U$ is unitary.
    
    Set $f_0 = f - \mu(f)$ to be the centered version of $f$, then 
    \begin{align}\label{eq:app:sg1}
         \|P^n f - \mu(f)\|_2^2 = \int_\qsp \left[ (P^n f)(y) - \mu(f)\right]^2\, \mu(dy) = 
         \int_\qsp \left[(P^nf_0)(y)\right]^2\, \mu(dy)
    \end{align}
    and, by the spectral theorem just invoked,
    \begin{equation}\label{A:specgap2}\begin{split}
        \int_\qsp \left[(P^nf_0)(y)\right]^2\, \mu_0(dy) &=  \int_\qsp \left[ \left(U^* T^n U f_0\right)(y)\right]^2\, \mu_0(dy)\\
          &= \langle U^* T^n U f_0, U^* T^n U f_0 \rangle_{L^2_\mu(\qsp)} \\
          &=  \langle T^n U f_0, U U^* T^n Uf_0 \rangle_{L^2_\nu(\mathcal{S})} \\
        &=\langle T^n U f_0, T^n Uf_0 \rangle_{L^2_\nu(\mathcal{S})} \end{split}\end{equation}
        where we have used that $U$ is unitary. Then, using the definition of the multiplication operator $T$ \eqref{eq:app:PUTU}, it follows 
        \begin{align}\label{A:specgap3}
            \langle T^n U f_0, T^n Uf_0 \rangle_{L^2_\nu(\mathcal{S})} = \int_{\mathcal{S}} \left[ T^n (U f_0)(x)\right]^2 \, \nu(dx) =  \int_{\mathcal{S}}  \left[ m^n(x) (U f_0)(x)\right]^2 \, \nu(dx).
        \end{align}
        We introduce the factor $\nu((U f_0)^2) =\int (U f_0)^2(x)\, \nu(dx)$ so that 
        \begin{equation}\label{eq:app:nu0}
            \frac{(U f_0)^2(x)}{\nu((U f_0)^2)}\nu(dx) =:\nu_0 (dx)
        \end{equation}
        is a probability measure on $\mathcal{S}$. Moreover, note that 
        \begin{equation}\label{eq:app:sg4}
            \int_{\mathcal{S}} (U f_0)^2(x)\, \nu(dx) = \langle Uf_0, Uf_0 \rangle_{L^2_\nu(\mathcal{S})}= \langle f_0, U^* U f_0\rangle_{L^2_\mu(\qsp)} =  \int_{\qsp} f_0^2(x)\, \mu(dx).
        \end{equation}
        In summary from \eqref{eq:app:sg1},\eqref{A:specgap2},\eqref{A:specgap3}, using the definition \eqref{eq:app:nu0} we can write 
        \begin{align*}
            \|P^n f - \mu(f)\|_2^2 =  \int_{\mathcal{S}}  \left[ m^n(x) (U f_0)(x)\right]^2 \, \nu(dx) = \nu((U f_0)^2) \int_{\mathcal{S}}  m^{2n}(x) \, \nu_0 (dx) =  \mu(f_0^2) \int_{\mathcal{S}}  m^{2n}(x) \, \nu_0 (dx).
        \end{align*}
        where in the last equality we used \eqref{eq:app:sg4} with the usual compact notation for the integrals. 
        
       Next, thanks to Jensen's inequality, for any $k\in \N$ it holds 
        \begin{align*}
            \mu(f_0^2) \int_{\mathcal{S}}  m^{2n}(x) \, \nu_0 (dx)& = \mu(f_0^2) \int_{\mathcal{S}}  m^{\frac{2n(2n + 2k)}{2n + 2k}}(x) \nu_0 (dx)\\
            &\leq \mu(f_0^2)\left( \int_{\mathcal{S}}  m^{2n + 2k}(x) \, \nu_0(dx)\right)^{\frac{n}{n+ k}},
        \end{align*}
        and, recalling the definition \eqref{eq:app:nu0} of $\nu_0$ and \eqref{eq:app:sg4}
\begin{align*}
    = \mu(f_0^2)\left( \int_{\mathcal{S}}  m^{2n + 2k}(x) \,  \frac{(U f_0)^2(x)}{\mu(f_0^2)}\nu(dx)\right) ^{\frac{n}{n+ k}}\\
           & = \mu(f_0^2)^{1 - \frac{n}{n+k}}\left( \int_{\mathcal{S}}  m^{2n + 2k}(x)(U f_0)^2(x) \nu(dx)\right) ^{\frac{n}{n+ k}}.
\end{align*}
        
       Then with the same argument as in \eqref{A:specgap2}\eqref{A:specgap3} it can be easily derived that 
       \begin{equation*}
           \int_{\mathcal{S}}  m^{2n + 2k}(x)(U f_0)^2(x) \nu(dx) = \int_\qsp \left[(P^{n + k}f_0)(y)\right]^2\, \mu(dy) = \|P^{n + k}f - \mu(f)\|_2^2
       \end{equation*}
      hence 
       \begin{equation*}
            \|P^n f - \mu(f)\|_2^2\leq \mu(f_0^2)^{\frac{k}{n+k}}\|P^{n + k}f - \mu(f)\|_2^{\frac{2n}{n+k}}.
       \end{equation*}
        In particular choosing the arbitrary parameter $k$ to be larger than $n_1/2$ so that $n + k>n_1/2$ for any $n\in \N$, we can then use \eqref{A:specgap1} to get 
        \begin{equation*}
            \|P^n f - \mu(f)\|_2^2\leq \mu(f_0^2)^{\frac{k}{n+k}}C(f)^{\frac{n}{n+k}} \lambda^{2n}
        \end{equation*}
        and taking the limit $k\to \infty$
         \begin{equation*}
            \|P^n f - \mu(f)\|_2^2\leq  \lambda^{2n} \mu(f_0^2)= \lambda^{2n} \| f - \mu(f)\|_2^2,  \quad n \in \N
        \end{equation*}
        as desired. 

    \textbf{Step 3}: By step 1 and 2 we showed that for any $f\in \Lip(\td)\cap L^\infty_\mu$ it holds 
    \begin{equation*}
        \|P^n f - \mu(f)\|_{2}^2\leq \lambda^{2n} \|f - \mu(f)\|_2^2.
    \end{equation*}
    By the density of $\Lip(\td)\cap L^\infty_\mu$ in $L^2_{\mu}$ (see \cite[Theorem~2.15]{hairer2014spectral}), we can then extend the result to all $L^2_\mu$ observables. Indeed, for $f \in L^2_\mu$, let $(f_k)_{k\in \N}\subset \Lip(\td)\cap L^\infty_\mu$ be such that $\|f_k - f\|_2\to 0$ as $k \to \infty$. Then $\|P^n f - P^n f_k\|_2\to 0$ for $k\to \infty$, as $P^n$ is a bounded operator, and $|\mu(f_k) - \mu(f)| \to 0$ from the convergence in $L^2_\mu$, so that 
    \begin{align*}
         \|P^n f - \mu(f)\|_2 &\leq \|P^n f - P^n f_k\|_2 + \|P^n f_k - \mu(f_k)\|_2 + \|\mu(f_k) - \mu(f)\|_2  \\
         &\leq \lambda^{n} \lim_{k \to \infty}\|f_k - \mu(f_k) \|_2  =  \lambda^{n} \|f - \mu(f) \|_2
    \end{align*}
    for any $n\in \N$, as desired. 
\end{proof}

Next we state the strong Law of Large numbers for observables in $L^2_\mu$ under the assumption of $L^2_\mu$ spectral gap. We provide a proof as it does not seems readily available in the literature under this assumption. 

\begin{Theorem}[Strong Law of Large Numbers]\label{thm:SLLN:L2}
     Let $P$ be a Markov kernel with invariant probability measure $\mu$ with associated chain $(X_n)_{n\in \N}$. Assume $P$ is such that the $L^2_\mu$ spectral gap \eqref{eq:specgapL2} holds. Then, if $X_0\sim \mu$,
        \begin{equation*}
         \lim_{n\to \infty} \left|\frac{1}{n} \sum_{k = 0}^{n-1} f (X_k) -  \int f \, d\mu   \right| =0 \quad \bbP\text{-a.s.}
        \end{equation*}
    for any $f\in L^2_\mu$.
\end{Theorem}
\begin{proof}
    Let $\baf$ be $f - \mu(f)$, then we want to show that 
    \begin{equation*}
        \lim_{n\to \infty}\frac{1}{n} \sum_{k = 0}^{n-1} \baf(X_k)  = 0, \quad \bbP\text{-a.s.}
    \end{equation*}
    Since $P$ exhibits a spectral gap $1 - \lambda$ in $L^2_\mu$, it follows that $1 - P$ is an invertible operator on the subset of $\mu$-centered functions, which we denote as $L^2_0(\mu)$. Then there exists a unique $g\in L^2_0(\mu)$ such that $\baf = (1 - P)g$ and $\|g\|_2 \leq (1 - \lambda)^{-1} \|\baf\|_2$. Set $D_{k+1} = g(X_{k+1}) - Pg(X_k)$, then we can write 
    \begin{equation*}
        \baf(X_k) = g(X_k) - Pg(X_k) = g(X_{k}) - g(X_{k+1}) + D_{k+1} 
    \end{equation*}
    and 
    \begin{equation}\label{eq:app:slln1}
         \lim_{n\to \infty}\frac{1}{n} \sum_{k = 0}^{n-1} \baf(X_k) =  \lim_{n\to \infty}\frac{1}{n}\left[ g(X_0) - g(X_n) \right] +  \lim_{n\to \infty}\frac{1}{n} \sum_{k = 0}^{n-1} D_{k+1}.
    \end{equation}
    We want to show that $\lim_{n \to \infty}g(X_n)/n$ and $\lim_{n\to \infty} \frac{1}{n} \sum_{k = 0}^{n-1} D_{k+1}$ are null. 

    We start from the first limit. Define $A_n = \lbrace |g(X_n)|>\varepsilon n\rbrace$ and, since we start in stationarity, 
    \begin{align*}
        \sum_{n = 1}^\infty \bbP(A_n) &= \sum_{n =1}^\infty \mu(|g|>\varepsilon n) \leq \sum_{n = 1}^\infty\frac{\mu(g^2)}{\varepsilon^2 n^2}= \sum_{n = 1}^\infty\frac{\|g\|^2}{\varepsilon^2 n^2}<\infty
    \end{align*}
    by the Markov inequality. 
    Since $g \in L^2_0(\mu)$, by Borel-Cantelli it follows that for any arbitrary $\varepsilon>0$ it holds $\limsup_{n \to \infty} |g(X_n)|/n < \varepsilon$ almost surely, hence $\lim_{n \to \infty} |g(X_n)|/n = 0$ as desired. 

    Next we want to show that $\lim_{n\to \infty} \frac{1}{n} \sum_{k = 0}^{n-1} D_{k+1} = 0$ almost surely. We start by showing that $\sum_{k = 0}^{n-1} D_{k+1} =:M_n$ is a mean-zero square-integrable martingale. By definition, for any $k\in \N$, 
    $$\E D_{k+1} = \E g(X_{k+1}) - \E Pg(X_{k}) = \E g(X_{k+1}) - \E \E\left( g(X_{k+1})| g(X_k)\right) = 0,$$
    hence $\E M_n = 0$. Next, to ensure $M_n$ is a martingale we look at
    \begin{align*}
        \E \left(M_{n+1}\, | \, \mathcal{F}_n\right) &=  \E \left(\sum_{k = 0}^{n} D_{k+1}\, | \, \mathcal{F}_n\right) 
        = \sum_{k = 0}^n \E \left( D_{k+1}\, | \, \mathcal{F}_n\right)\\
       & =  \sum_{k = 0}^{n-1} \E \left( D_{k+1}\, | \, \mathcal{F}_n\right)+ \E \left( D_{n+1}\, | \, \mathcal{F}_n\right).
    \end{align*}
    By Markovianity we can write 
    \begin{align*}
          \E \left(M_{n+1}\, | \, \mathcal{F}_n\right) &=  \E M_n +  \E \left( D_{n+1}\, | \, \mathcal{F}_n\right)\\
          &= \E M_n + \E \left(g(X_{n+1}) \, | \, \mathcal{F}_n\right) - \E \left( Pg(X_n) | \mathcal{F}_n\right) \\
          &= \E M_n
    \end{align*}
    giving the desired result. Next we want to ensure that $\E M_n^2 <\infty$:
    \begin{equation*}
        \E M_n^2 = \E \left( \sum_{j = 1}^n\left( M_j - M_{j -1}\right)\right)^2 \leq C(n) \sum_{j =1}^n \E \left( M_j -M_{j -1}\right)^2 = C(n) \sum_{j =1}^n \E D_j^2
    \end{equation*}
    for some positive constant $C(n)$. Now for any $k\in \N$
    \begin{align*}
        \E D_{k+1}^2 &= \E \left[  g(X_{k+1}) - Pg(X_k)\right]^2 = \E \left[  g(X_{k+1}) - \E \left( g(X_{k+1})| X_k\right) \right]^2\\
        &=\E \Var(g(X_{k +1})| X_k) =  \E \left[ \E \left( g(X_{k+1})^2 | X_k\right) - \E \left(g(X_{k+1} \, | \, X_k\right)^2 \right]\\
        &= \E \left[Pg^2(X_k) - [Pg(X_k)]^2\right]
    \end{align*}
    and, by stationarity and invariance
    \begin{equation*}
        \E D_{k+1}^2 = \int Pg^2(x) - (Pg)^2(x) \, \mu(dx) \leq \mu(g^2) = \|g\|_2^2 < \infty.
    \end{equation*}
   Then 
\begin{equation*}
    \sum_{n =1}^\infty \frac{\E (M_n - M_{n-1})^2}{n^2} =  \sum_{n =1}^\infty \frac{\E D_{n+1}^2}{n^2}  \leq  \sum_{n =1}^\infty \frac{\|g\|_2^2 }{n^2} <\infty
\end{equation*}
    and by a martingale convergence theorem \cite{Chow} it follows that 
    \begin{equation*}
        \lim_{n \to \infty} \frac{M_n}{n} = 0\quad \bbP-\text{a.s.}.
    \end{equation*}
    Finally the limit \eqref{eq:app:slln1} is zero as desired.
\end{proof}

Last we recall for completeness the full statement of the Central Limit Theorem in $L^2_\mu$ which can be traced back to \cite{Kipnis86}.
\begin{Theorem}[Central Limit Theorem in $L^2_\mu$]\label{thm:CLT:L2}
    Let $P$ be a Markov kernel with invariant probability measure $\mu$ with associated chain $(X_n)_{n\in \N}$. Assume $P$ is $\mu$-reversible and exhibits a strectral gap as in \eqref{eq:specgapL2}. Then, for any $f\in L^2_0(\mu)$,
    \begin{equation}\label{eq:sigmaL2}
        \sigma_{f,P}^2 := \left\langle (1 - P)^{-1}(1 + P) f, \, f\right\rangle_2 \leq \frac{2C\mu(f^2)}{1 - e^{-\gamma}} 
    \end{equation}
    and if $X_0\sim \mu$, then
        \begin{equation*}
        \frac{1}{\sqrt{n}} \sum_{k = 0}^{n-1} \left( f (X_k) - \int f \, d\mu \right) \xrightarrow{d} \cN(0, \sigma_{f,P}^2)\quad \text{for } n \to \infty
        \end{equation*}
    for any $f\in L^2_\mu$.
\end{Theorem}
Note that the inverse of $1 - P$ in \eqref{eq:sigmaL2} is well defined on $L^2_0(\mu) = \{ f \in L^2_\mu \; :\; \mu(f) = 0\}$ thanks to the spectral gap assumption. In fact 
\begin{equation*}
    \| (1 - P)^{-1}\|_{L^2_0 \to L^2_0} \leq \sum_{n =0}^\infty \|P^n\|_{L^2_0 \to L^2_0} \leq C \sum_{n=0}^\infty e^{-\gamma n} = \frac{C}{1 - e^{-\gamma}}
\end{equation*}
where we denoted with $\|\cdot \|_{L^2_0 \to L^2_0} $ the norm of an operator on $L^2_0$.
Then, using Cauchy--Schwartz and the fact that $P$ has operator norm one, the bound in \eqref{eq:sigmaL2} follows. 

\section{About the weak Harris theorem}\label{app:harris}
Recall the definitions in \cref{sec:prelim}. Then the original weak Harris theorem in \cite{HMS11} when formulated in discrete times reads:

\begin{Theorem}[Weak Harris Theorem]
    \label{thm:general_harris}
        Let $P$ be a Markov kernel over a Polish space $\qsp$ with invariant measure $\mu$ and with continuous Lyapunov function $V$. Suppose there exist a distance-like function $d: \qsp \times \qsp \to [0,1]$ and $n_* \in \N$ such that $P^{n_*}$ is $d$-contracting, and the sublevel set $\lbrace x\in \qsp \, : \, V(x) \leq 4K_V \rbrace$ is $d$-small for $P^{n_*}$. Then $P$ has at most one invariant measure. Furthermore, defining $\td(x,y)^2 = d(x,y)(1 + V(x) + V(y))$, there exists $n \in \N$ such that
        \begin{equation}
        \label{eq:A:harris:contraction}
           W_{\td}(\nu_1 P^{n}, \nu_2 P^{n})\leq \frac{1}{2} W_{\td}(\nu_1,\nu_2)
        \end{equation}
        for all probability measures $\nu_1$ and $\nu_2$ on $\qsp$.
    \end{Theorem}
It is easy to see that the contraction \eqref{eq:A:harris:contraction} given by the weak Harris theorem can be iterated leading to 
\begin{equation*}
    W_{\td}\left( \nu_1 P^{kn}, \nu_2 P^{kn}\right)\leq \left( \frac{1}{2}\right)^k W_{\td}(\nu_1,\nu_2) \quad \mbox{for all }k \in \N.
\end{equation*}
However, this does not necessarily imply contraction for discrete times that are not multiples of $n$ without further assumptions as those used in \cref{thm:our_harris}. This extension, crucial for our analysis, is suggested in \cite[Remark~4.10]{HMS11}, and showed in details for the Hamiltonian Monte Carlo kernel in \cite[Theorem~6.1]{glatt2021mixing}. Here we describe again the argument for a generic kernel, not linked to a specific algorithm, with discrete times for the sake of completeness.

\begin{proof}[Proof of \cref{thm:our_harris}]
    It is easy to see that the Wasserstein contraction \eqref{eq:2:harris:contraction} implies the spectral gap \eqref{eq:2:harris:exp1} in $\Lip(\td)$. 
In fact by definition of the Lipschitz norm \eqref{eq:lipschitz:norm} 
        \begin{equation}\label{eq:gap-contr1}
            \| P^n f- \mu(f) \|_{\td} 
            =  \sup_{x\neq y} \frac{\left|P^n f(x) - P^n f(y)\right|}{\td(x,y)}\leq \sup_{x\neq y} \frac{\|f\|_d W_{\td}(P^n(x, \cdot), P^n(y, \cdot))}{\td(x,y)}.
        \end{equation}
         Here we have used the weak version of the Kantorovich-Rubinstein formula
         \begin{equation}\label{eq:kant-rub}
              \sup_{\|f\|_{\td}\leq 1} \left| \int f \,d\nu_1  -  \int f\, d \nu_2 \right| \leq  W_{\td}(\nu_1, \nu_2) 
         \end{equation}
         which holds also for $d$ being a distance-like function (the opposite inequality requires $d$ to satisfy the triangular inequality). 
         Then \eqref{eq:2:harris:contraction} with $\nu_1 = \delta_x$ and $\nu_2= \delta_y$ ensures that 
        \begin{equation}\label{eq:gap-contr2}
             \| P^n f - \mu(f) \|_{\td} \leq \|f\|_{\td}  \sup_{x\neq y} \frac{\lambda \td(x,y)}{\td(x,y)} = \lambda \|f - \mu(f)\|_{\td}.
        \end{equation}

    Regarding \eqref{eq:2:harris:exp}, if we take $\nu_1 = \delta_x$ and $\nu_2 = \mu$ the invariant measure in \eqref{eq:2:harris:contraction}, we get
    \begin{equation*}
          W_{\td}(P^{n}(x, \cdot), \mu )\leq \lambda^n W_{\td}(\delta_x,\mu). 
    \end{equation*}
    Using the definition of $\td$ and the fact that $d$ takes values in $[0,1]$ it follows
    \begin{align*}
        W_{\td}(\delta_x,\mu) = \int \td(x, y) \mu(dy) &= \int  \sqrt{d(x,y)(1 + V(x) + V(y))} \mu(dy)\\
         &\leq \left( 1 + V(x) + \int V(y) \, \mu(dy)\right)^{1/2}.
    \end{align*}
   It can be showed that Lyapunov functions $V$ of $P$ are integrable with respect to any invariant measure of $P$ (see e.g. \cite[Lemma~4.1]{butkovsky2014}), hence by \eqref{def:lyapunov} we can have an explicit upper bound of the integral 
    \begin{equation*}
        \mu(V) = \int V(y) \, \mu(dy) \leq \frac{K_V}{1 - l_V}
    \end{equation*}
    and 
    \begin{equation*}
        W_{\td}(\delta_x,\mu) \leq \left( 1 + \frac{K_V}{1 - l_V} + V(x)\right)^{1/2}.
    \end{equation*}
    We have then showed that \eqref{eq:2:harris:exp} holds with $C(x) = \left( 1 + \frac{K_V}{1 - l_V} + V(x)\right)^{1/2}$. 

    We now focus on showing \eqref{eq:2:harris:contraction}, namely we will show that there exists $n_1>0$ and $\lambda<1$ such that for all $n\geq n_1$
    \begin{equation}\label{eq:A:harris1}
        W_{\td}(P^n(x, \cdot), P^n(y, \cdot)) \leq \lambda^n \td(x, y) \quad \text{for all }x, y\in \qsp.
    \end{equation}
    In fact, since the distance-like function $\td$ is lower-semicontinuous and non-negative, it follows that (e.g. \cite[Theorem~4.8]{villani2009optimal})
    \begin{equation*}
         W_{\td}(\nu_1 P^n, \nu_2 P^n) \leq \inf_{\pi\in \mathfrak{C}(\nu_1, \nu_2)} \int W_{\td}(P^n(x, \cdot), P^n(y, \cdot))\, \pi(dx, dy)
    \end{equation*}
    hence if \eqref{eq:A:harris1} holds
    \begin{equation*}
         W_{\td}(\nu_1 P^n, \nu_2 P^n) \leq \lambda^n \inf_{\pi\in \fC(\nu_1, \nu_2)} \int  \td(x, y)\, \pi(dx, dy) = \lambda^n W_{\td}(\nu_1, \nu_2)
    \end{equation*}
    as desired. 

    We start by defining an auxiliary distance, modification of $\td$
    \begin{equation*}
        \td_\beta(x,y)^2 = d(x, y) \left( 1 + \beta V(x) + \beta V(y)\right) 
    \end{equation*}
    with $\beta>0$ to be specified later. It is easy to see that $\td_\beta$ is equivalent to $\td$ as there exist positive constants $k_\beta = \min(1, \beta)^{1/2}$ and $K_\beta = \max(1 , \beta)^{1/2}$ such that $k_\beta \td(x, y) \leq \td_\beta(x,y)\leq K_\beta \td(x,y)$. Then if we show \eqref{eq:A:harris1} for $\td_\beta$ and the associated semimetric $W_{\td_\beta} =: W_\beta$, for some $n_1(\beta)>0$ and $\lambda_\beta<1$, we can derive 
    \begin{equation*}
         W_{\td}(P^n(x, \cdot), P^n(y, \cdot)) \leq \frac{1}{k_\beta} W_{\beta}(P^n(x, \cdot), P^n(y, \cdot)) \leq \frac{1}{k_\beta} \lambda_\beta^n \td_\beta (x, y) \leq \frac{K_\beta}{k_\beta} \lambda_\beta^n\td (x, y),
    \end{equation*}
    for all $n > n_1(\beta)$. Let $n_*$ be the first integer strictly larger than
    \begin{equation*}
        \frac{\log K_\beta - \log k_\beta}{-\log \lambda_\beta}>0
    \end{equation*}
    so that, for all $n\geq n_*$,
    \begin{equation*}
        \lambda_\beta \left( \frac{K_\beta}{k_\beta}\right)^{1/n} \leq \lambda_\beta \left( \frac{K_\beta}{k_\beta}\right)^{1/n_*} < 1.
    \end{equation*}
    This implies that 
    \begin{equation*}
         W_{\td}(P^n(x, \cdot), P^n(y, \cdot)) \leq \lambda^n \td (x, y), \quad \text{for all }x,y\in \qsp,
    \end{equation*}
    with $\lambda = \lambda_\beta \left( \frac{K_\beta}{k_\beta}\right)^{1/n_*}$ being strictly smaller than one for all 
    \begin{equation*}
        n > \max\left( n_1(\beta), \frac{\log K_\beta - \log k_\beta}{-\log \lambda_\beta}\right) =:n_1 .
    \end{equation*}

    We divide the proof for $\td_\beta$ in three cases: (1) $x,y$ close to each other, namely $d(x,y) <1$, (2) $x,y$ such that $d(x,y) =1$ and they are not in the set $S$, (3) $x,y$ such that $d(x,y) =1$ and they are in the set $S$.
    \paragraph{Case 1: $d(x,y) <1$.} We are in the regime where there is $c(n)>0$ such that $c(n_0)<1$ and 
    \begin{equation}\label{eq:A:contract}
        W_{d}(P^n(x, \cdot), P^n(y, \cdot)) \leq c(n) d(x,y).
    \end{equation}
    By the definition of $W_\beta$ it follows 
    \begin{align}
        W_\beta(P^n(x, \cdot), P^n(y, \cdot))^2 &\leq \inf_\pi \int d (x', y') \pi(dx', dy') \int 1 + \beta V(x') + \beta V(y') \, \pi(dx', dy') \label{eq:A:harris2}\\
        &\leq W_{d} (P^n(x, \cdot), P^n(y, \cdot)) \left( 1 + \beta P^n V(x) + \beta P^nV(y)\right), \label{eq:A:harris3}
    \end{align}
    where the infimum is over all couplings $\pi$ of $P^n(x, \cdot), P^n(y, \cdot)$.
    Thanks to \eqref{eq:A:contract} and the definition of Lyapunov function \cref{def:lyapunov} we have
    \begin{align*}
         W_\beta(P^n(x, \cdot), P^n(y, \cdot))^2 \leq c(n) \, d(x,y) \left( 1 + \beta l_V^n V(x) + \beta l_V^n V(y) + 2 \beta K_V\right).
    \end{align*}
    We now select $\beta$ so to reconstruct $\td_\beta$ on the right hand side.  
    Since $l <1$ then for any $n\in \N$
    \begin{equation}\label{eq:A:harris4}
          W_\beta(P^n(x, \cdot), P^n(y, \cdot))^2 \leq c(n) (1 + 2\beta K_V) \, d(x,y) \left( 1 + \beta V(x) + \beta V(y) \right) =: c_1(n) \td_\beta(x,y)^2.
    \end{equation}
    If $n =n_0$ then $c(n_0) <1$ and we can choose $\beta$ small enough so that $c_1(n_0) = c(n_0) (1 + 2\beta K_V)< 1$, namely 
    \begin{equation*}
        \beta < \frac{1 - c(n_0)}{2 c(n_0) K_V}.
    \end{equation*}
    
    \paragraph{Case 2: $\td_\beta(x,y) = 1$ and $V(x) + V(y) \geq 4K_V$.} From \eqref{eq:A:harris2} and the definition of Lyapunov function it follows 
        \begin{align*}
            W_\beta(P^n(x, \cdot), P^n(y, \cdot))^2 &\leq  1 + \beta P^n V(x) + \beta P^nV(y) \\
            &\leq 1 + 2\beta K_V + \beta l_V^n V(x) + \beta l_V^n V(y) \\
            &\leq \frac{1 + 2 \beta K_V}{1 + 3 \beta K_V} \left(1 + 3 \beta K_V) + \beta l^n \left( V(x) + V(y)\right)\right) \\
            &\leq \max\left(  \frac{1 + 2 \beta K_V}{1 + 3 \beta K_V} , \frac{l^n }{4}\right)\left( 1 + 3 \beta K_V + \frac{\beta}{4} \left( V(x) + V(y)\right)\right).
            \end{align*}
        Since $V(x) + V(y) \geq 4K_V$
        \begin{align*}
             W_\beta(P^n(x, \cdot), P^n(y, \cdot))^2 &\leq c_2(n) \td_\beta(x,y)^2,
        \end{align*}
        with 
        \begin{equation*}
             c_2(n)= \max\left(\frac{1 + 2 \beta K_V}{1 + 3 \beta K_V} , \frac{l_V^n}{4}\right)
        \end{equation*}
        being strictly smaller than one for any $n\in \N$ and $\beta>0$.

    \paragraph{Case 3: $\td_\beta(x,y) = 1$ and $V(x) + V(y) \leq 4K_V$.} We are in the regime for which $x,y\in S = \{V\leq 4K_V\}$, the $d$-small set of $P^{n_0}$. Then from \eqref{eq:A:harris3}, 
    \begin{align*}
         W_\beta(P^n(x, \cdot), P^n(y, \cdot))^2&\leq W_d (P^n(x, \cdot), P^n(y, \cdot)) \left( 1 + \beta P^n V(x) + \beta P^nV(y)\right) \\
         &\leq s(n) \left( 1 + 2\beta K_V + \beta l_V^n  V(x) + \beta l_V^n V(y)\right) \\
         &\leq s(n) \left( 1 + 2\beta K_V + 8\beta l^n K_V \right)\\
         &= s(n) \left( 1 + 2\beta K_V( 1 + 4l_V^n) \right)\td_\beta(x,y)^2.
    \end{align*}
    Now, for $n = n_0$, we have $s(n_0) <1$, then there is a choice of $\beta$ for which $c_3(n) := s(n) \left( 1 + 2\beta K_V( 1 + 4l_V^{n}) \right)<1$ namely
    \begin{equation*}
        \beta < \frac{1 - s(n_0)}{2 s(n_0) K_V(1 + 4l_V^{n_0})}.
    \end{equation*}

    Therefore we showed that for any $n \in \N$ 
    \begin{equation}\label{eq:A:harris5}
         W_\beta(P^n(x, \cdot), P^n(y, \cdot)) \leq c_\beta(n) \td_\beta(x,y) \quad \text{for all }x,y\in \qsp,
    \end{equation}
    with 
    \begin{equation*}
        c_\beta(n) = \max(c_1(n, \beta), c_2(n, \beta), c_3(n, \beta))^{1/2}
    \end{equation*}
    and for $n = n_0$ we can choose 
        \begin{equation*}
            \beta < \min\left( \frac{1 - c(n_0)}{2c(n_0)K_V}, \frac{1 - s(n_0)}{2 s(n_0) K_V(1 + l_V^n)} \right)
        \end{equation*}
    so that $c_\beta(n_0)<1$ as desired.
    Next we show that having contractivity just at $n_0$ is enough to ensure it for all large enough times.
    
    If $n$ is a multiple of $n_0$, $n = k n_0$ for $k
    \geq 1$ then iterating \eqref{eq:A:harris5} at $n_0$ one gets 
    \begin{equation*}
        W_\beta(P^{kn_0}(x, \cdot), P^{kn_0}(y, \cdot)) \leq c_\beta(n_0)^k \td_\beta(x,y) \leq \left( c_\beta(n_0)^{\frac{1}{n_0}}\right)^n \td_\beta(x,y).
    \end{equation*}
    
    If instead $n = kn_0 + q$ with $1\leq q\leq n_0-1$, then using again \eqref{eq:A:harris5}
    \begin{align*}
         W_\beta(P^{kn_0+ q}(x, \cdot), P^{kn_0+q}(y, \cdot))\leq c_\beta(n_0)^k  W_\beta(P^{q}(x, \cdot), P^{q}(y, \cdot))
         \leq c_\beta(n_0)^k c_\beta(q) \td_\beta(x,y).
    \end{align*}
    With some simple manipulations we see
    \begin{align*}
      \left[ c_\beta(n_0)^k c_\beta(q) \right]^{1/n} = \left[ c_\beta(n_0)^{\frac{n}{n_0} - q} c_\beta(q) \right]^{1/n} 
        = c_\beta(n_0)^{\frac{1}{n_0}} \left(\frac{c_\beta(q)}{c_\beta(n_0)^q}\right)^{\frac{1}{n}}
    \end{align*}
    then, setting $L_\beta(n_0) = \max_{q\in [1, n_0-1]} c_\beta (q)$,
    \begin{equation*}
         \left[ c_\beta(n_0)^k c_\beta(q) \right]^{1/n}  \leq  c_\beta(n_0)^{\frac{1}{n_0}} \left(\frac{L_\beta(n_0)}{c_\beta(n_0)^{n_0-1}}\right)^{\frac{1}{n}} .
    \end{equation*}
    It is easy to see that the quantity on the right hand side stays strictly smaller than one for all 
    \begin{align}
         n >  \frac{n_0}{- \log c_\beta(n_0) \left[ L_\beta(n_0) -  (n_0-1)\log c_\beta(n_0)\right]}\label{eq:A:harris6}
    \end{align}
    where we have used again that $c_\beta(n_0)<1$ and $c_\beta(q)\geq 1$ for $1\leq q \leq n_0-1$ (otherwise we relabel $n_0$). 

    Next let $m_*$ be the first integer satisfying \eqref{eq:A:harris6} so that for all $n = k n_0 + q \geq m_*$
    \begin{equation*}
        \left[ c_\beta(n_0)^k c_\beta(q) \right]^{1/n} \leq  c_\beta(n_0)^{\frac{1}{n_0}} \left(\frac{L_\beta(n_0)}{c_\beta(n_0)^{n_0-1}}\right)^{\frac{1}{m*}} =: \kappa_\beta <1
    \end{equation*}
    and
    \begin{equation*}
          W_\beta(P^{kn_0+ q}(x, \cdot), P^{kn_0+q}(y, \cdot))
         \leq \kappa_\beta^n \td_\beta(x,y).
    \end{equation*}

    Finally, we showed that, for 
    \begin{equation}
        n > n_1(\beta) = \max\left( n_0; \frac{n_0}{- \log c_\beta(n_0) \left[ L_\beta(n_0) -  (n_0-1)\log c_\beta(n_0)\right]}\right),
    \end{equation}
    it holds 
     \begin{equation*}
          W_\beta(P^{n}(x, \cdot), P^{n}(y, \cdot))
         \leq \lambda_\beta^n \td_\beta(x,y) \quad \text{for all } x,y\in \qsp
    \end{equation*}
    with 
    \begin{equation*}
        \lambda_\beta = \max \left(  c_\beta(n_0)^{\frac{1}{n_0}}; \kappa_\beta\right)<1 .
    \end{equation*}
    \end{proof}

    Last we provide the statements of the Wasserstein SLLN and CLT to ease reference. See e.g.~\cite[Appendix A]{glatt2021mixing} for their proofs.
    
    \begin{Theorem}[Strong Law of Large Numbers]
        Let $P$ be a Markov kernel on $\qsp$ with invariant measure $\mu$ and associated chain $(X_n)_{n\in \N}$ such that \cref{thm:our_harris} holds. Then for any $f \in \Lip(\td)$ and arbitrary $x_0\in \qsp$
        \begin{equation*}
            \lim_{n \to \infty}\left| \frac{1}{n} \sum_{j = 0}^{n-1} f (X_n) - \int f \, d\mu \right| =0 \quad \text{almost surely.}
        \end{equation*}
    \end{Theorem}

    \begin{Theorem}[Central Limit Theorem]\label{thm:clt}
        Let $P$ be a Markov kernel on $\qsp$ with invariant measure $\mu$ and associated chain $(X_n)_{n\in \N}$ such that \cref{thm:our_harris} holds. Suppose that 
        \begin{equation*} 
            W_{\td_2} (\nu_1P^n, \nu_2 P^n) \leq \lambda^n W_{\td_2}(\nu_1, \nu_2) \quad\text{for all }\nu_1, \nu_2\in \cM_1(\qsp)
        \end{equation*}
        holds also for the semidistance $\td_2(x,y) = \sqrt{d(x,y)(1 + V(x)^2 + V(y)^2)}$ with a possibly different $\lambda<1$. 
        Then for $f \in \Lip(\td)$ and every initial condition $x_0\in \qsp$
        \begin{equation*}
            \frac{1}{\sqrt{n}} \sum_{k = 0}^{n-1} \left( f (X_k) - \int f \, d\mu \right) \xrightarrow{d} \cN(0, \sigma(f)^2)\quad \text{for } n \to \infty
        \end{equation*}
        where if $\Bar{f} := f - \int f \, d\mu$
        \begin{equation*}
            \sigma(f)^2 = \int \E \left[ \bar{f} (X_1(x_0))  - \sum_{k = 1}^\infty \left(P^{k}\bar{f} (X_1(x_0)) - P^k \bar{f} (x_0)\right) \right]^2 \, \mu(dx_0).
        \end{equation*}
    \end{Theorem}

\end{document}